\documentclass[11pt]{imsart}

\RequirePackage{amsthm,amsmath,amsfonts,amssymb, bbm, mathrsfs, subcaption}
\RequirePackage[numbers]{natbib}
\RequirePackage[colorlinks,citecolor=blue,urlcolor=blue]{hyperref}
\RequirePackage{graphicx, 
 mathbbol}

\startlocaldefs
\numberwithin{equation}{section}
\usepackage[
     left=1.0in,       
     right=1.0in,      
     top=0.6in,        
     bottom=0.6in,     
     includehead, includefoot, 
]{geometry}
\theoremstyle{plain}
\newtheorem{theorem}{Theorem}[section]

\newtheorem{lemma}{Lemma}[section]
\newtheorem{proposition}{Proposition}[section]
\theoremstyle{remark}
\newtheorem{remark}{Remark}[section]
\newtheorem{definition}{Definition}[section]
\newtheorem{Assumption}{Assumption}[section]
\newcommand{\KL}{{\mathrm{D}_{\mathrm{KL}}}}
\newcommand{\IA}{I^A}
\newcommand{\NellA}{\widetilde{N}^A_{\ell}}
\newcommand{\muA}{\mu^A}
\newcommand{\fA}{f^A}
\newcommand{\FA}{F^A}
\newcommand{\muAM}{\mu^{\pmb{\mathscr{A}}}}
\newcommand{\IAM}{I^{\pmb{\mathscr{A}}}}
\newcommand{\NAM}{\widetilde{N}^{\pmb{\mathscr{A}}}}
\newcommand{\fAM}{f^{\pmb{\mathscr{A}}}}
\newcommand{\FAM}{F^{\pmb{\mathscr{A}}}}
\newcommand{\HI}[5]{H^{(\mathrm{I})}_{#1,#2;#3,#4}(#5)}
\newcommand{\HIM}
{H^{(\mathrm{I})}}
\newcommand{\NI}[1]{\mathcal{N}^{(\mathrm{I})}(#1)}
\newcommand{\NIM}{\mathcal{N}^{(\mathrm{I})}}
\newcommand{\HII}[6]{H^{(\mathrm{II})}_{#1,#2;#3,#4}(#5,#6)}
\newcommand{\HIIM}{H^{(\mathrm{II})}}
\newcommand{\NII}[2]{\mathcal{N}^{(\mathrm{II})}(#1,#2)}
\newcommand{\NIIM}{\mathcal{N}^{(\mathrm{II})}}
\newcommand{\UIM}{U^{(0)}}
\newcommand{\UIIM}{U^{(1)}}
\newcommand{\UI}[2]{U^{(0)}(#1;#2)}
\newcommand{\UII}[3]{U^{(1)}(#1;#2,#3)}
\newcommand{\UIsubscr}[4]{U^{(0)}_{#1,#2}(#3;#4)}
\newcommand{\UIIsubscr}[5]{U^{(1)}_{#1,#2}(#3;#4,#5)}
\newcommand{\Dtilde}{\mathscr{D}}
\newcommand{\ctilde}{e}

\endlocaldefs

\begin{document}

\begin{frontmatter}
\title{Variational Inference for Latent Variable Models in High Dimensions}
\runtitle{Variational Inference for Latent Variable Models in High Dimensions}

\begin{aug}
\author[A]{\fnms{Chenyang}~\snm{Zhong}\ead[label=e1]{cz2755@columbia.edu}},
\author[A]{\fnms{Sumit}~\snm{Mukherjee}\thanksref{t1}\ead[label=e2]{sm3949@columbia.edu}}
\and
\author[A]{\fnms{Bodhisattva}~\snm{Sen}\thanksref{t2}\ead[label=e3]{b.sen@columbia.edu}}

\thankstext{t1}{Supported by NSF grant DMS-2113414.}
\thankstext{t2}{Supported by NSF grants DMS-2311062 and DMS-2515520.}
\address[A]{Department of Statistics, Columbia University\printead[presep={,\ }]{e1,e2,e3}}

\end{aug}

\begin{abstract}
Variational inference (VI) is a popular method for approximating intractable posterior distributions in Bayesian inference and probabilistic machine learning. In this paper, we introduce a general framework for quantifying the statistical accuracy of mean-field variational inference (MFVI) for posterior approximation in Bayesian latent variable models with categorical local latent variables (and arbitrary global latent variables). Utilizing our general framework, we capture the exact regime where MFVI `works' for the celebrated latent Dirichlet allocation  model. Focusing on the mixed membership stochastic blockmodel, we show that the vanilla fully factorized MFVI, often used in the literature, is suboptimal. We propose a partially grouped VI algorithm for this model and show that it works, and derive its exact finite-sample performance. We further illustrate that our bounds are tight for both the above models. Our proof techniques, which extend the framework of nonlinear large deviations, open the door for the analysis of MFVI in other latent variable models.
\end{abstract}



\end{frontmatter}

\section{Introduction}\label{Sect.1}

Bayesian latent variable models are ubiquitous in Bayesian statistics and probabilistic machine learning. In modern applications, these models typically involve a large number of parameters and latent variables, resulting in complex and high-dimensional posteriors that are computationally intractable. For such scenarios, traditional Markov chain Monte Carlo (MCMC) approaches often suffer from lengthy burn-in periods and generally lack scalability \cite{blei2017variational}. 

Recently, variational inference (VI) \cite{jordan1999introduction, bishop2006pattern, wainwright2008graphical, blei2017variational} has emerged as a popular and scalable alternative method for approximating intractable posterior distributions in large-scale applications (where the number of observations and dimensionality are both large) and is typically orders of magnitude faster than MCMC methods. Among the various forms of VI, arguably the most widely used and important is mean-field variational inference (MFVI) \cite{wainwright2008graphical, blei2017variational}, which approximates the intractable posterior by a product distribution. This approach has been widely adopted in statistics and machine learning, thanks to efficient algorithmic implementations based on coordinate ascent variational inference (CAVI) \cite{bishop2006pattern, blei2017variational,durante2019conditionally, bhattacharya2023convergence, arnese2024convergence, lavenant2024convergence, castillo2024variational, kim2024flexible}. 

Despite its empirical success, the statistical accuracy of MFVI in approximating the true posterior distribution remains to be understood. In this paper, we introduce a general framework for analyzing the accuracy of MFVI in Bayesian latent variable models. Our general results, presented in Section \ref{Sect.2}, provide tight explicit finite-sample bounds on the discrepancy between the variational posterior obtained from MFVI and the true posterior. To the best of our knowledge, these are the first results of such kind in the literature. In our results, we allow the relevant parameter\slash latent variable dimension to grow with the sample size, thereby ensuring the applicability of our results to high-dimensional settings.  

To illustrate the effectiveness of our general framework, we apply the theory developed to two representative and widely used (hierarchical) Bayesian latent variable models: the \textit{latent Dirichlet allocation model} and the \textit{mixed membership stochastic blockmodel}. We briefly introduce these two models next. \vspace{0.05in}

\noindent{\bf{Latent Dirichlet allocation model.}} Latent Dirichlet allocation (LDA) is a probabilistic generative model proposed in \cite{pritchard2000inference, falush2003inference} in the context of population genetics and was used in \cite{blei2003latent} to model text corpora. In the context of topic modeling, the goal of LDA is to discover hidden thematic structures (topics) within a large collection of text documents while assuming that each document is composed of multiple (latent) topics in varying proportions. The model posits that each word's presence in a document can be attributed to one of the document's constituent topics. Given a collection of documents, LDA aims to infer both the topic distributions for each document and the topics that generate each word in each document. In the last 20 years, various extensions of this model have been introduced, including correlated topic models \cite{blei2005correlated} and collaborative topic models \cite{wang2011collaborative}; see \cite{vayansky2020review} for a comprehensive survey of generalizations of LDA. In addition to population genetics and topic modeling, LDA and its generalizations have also been applied to various other fields, including musicology \cite{lieck2020tonal}, epidemiology \cite{parker2023results}, and social science \cite{laureate2023systematic}. We refer to \cite{bing2022likelihood, ke2024using} and the references therein for frequentist approaches to topic modeling. \vspace{0.05in}

\noindent{\bf{Mixed mebership stochastic blockmodel.}} The mixed membership stochastic blockmodel (MMSB), introduced in \cite{airoldi2008mixed}, is a probabilistic model for analyzing pairwise measurements and inferring hidden attributes of individuals in networks. Unlike the stochastic blockmodel \cite{wang1987stochastic,snijders1997estimation}, which assigns each individual to a single group, MMSB introduces a membership probability-like vector for each individual and allows the individual to belong to multiple groups or communities. MMSB has been widely employed in network inference tasks, including applications in social networks \cite{xing2010state, snijders2011statistical}, protein interaction networks \cite{airoldi2006mixed,wang2022ppisb}, and gene regulatory networks \cite{xing2010state}. For extensions of MMSB and related models, we refer to \cite[Section 2.4]{abbe2018community} and \cite{song2024survival}. Additionally, frequentist approaches to latent community detection can be found in \cite{LR, zhang2016minimax, abbe2018community, gao2018community} and the references therein.

\subsection{General setup for Bayesian latent variable models and mean-field variational inference}\label{Sect.1.1}

Let $m,n\in\mathbb{N}^{*}:=\{1,2,3,\cdots\}$ and $d_1,d_2,\cdots,d_m,K_1,K_2,\cdots,K_n\in\mathbb{N}^{*}$. In Bayesian latent variable models, we assume a group of \emph{global latent variables} $\pmb{\theta}\equiv (\pmb{\theta}_1,\pmb{\theta}_2,\cdots,\pmb{\theta}_m)$ and a group of \emph{local latent variables} $\mathbf{Z}\equiv (Z_1,Z_2,\cdots,Z_n)$, where $\pmb{\theta}_j\in\mathbb{R}^{d_j}$ for each $j\in[m]:=\{1,2,\cdots,m\}$ and $Z_i\in [K_i]$ for each $i\in[n]$. Here and henceforth, we will use bold symbols to denote vectors and matrices that are random. For notational simplicity, we use the same bold symbols to denote random quantities and dummy variables.

In this paper, we consider the case where the local latent variables are {\it categorical}. The joint distribution $\mathbb{P}(\pmb{\theta},\mathbf{Z})$ of the latent variables $(\pmb{\theta},\mathbf{Z})$ is referred to as the {\it prior distribution}, and we denote by $\mathbb{p}(\pmb{\theta},\mathbf{Z})$ the density of $\mathbb{P}(\pmb{\theta},\mathbf{Z})$ with respect to (w.r.t.) a fixed product base measure $\mathcal{M}\equiv \prod_{j=1}^m \mathcal{M}_{g,j}\otimes\prod_{i=1}^n\mathcal{M}_{l,i}$ 
on $\prod_{j=1}^m\mathbb{R}^{d_j}\times\prod_{i=1}^n[K_i]$. The conditional density of the observed data $\mathbf{X}$ (w.r.t.~a fixed base measure on the data space), given the latent variables $(\pmb{\theta},\mathbf{Z})$, is given by the likelihood function $L(\mathbf{X}\mid\pmb{\theta},\mathbf{Z})$. 

The central object in Bayesian inference and prediction is the \emph{posterior distribution} $\mathbb{P}(\pmb{\theta},\mathbf{Z}\mid\mathbf{X})$, which is the conditional distribution of the latent variables $(\pmb{\theta},\mathbf{Z})$ given the data $\mathbf{X}$. The density of the posterior distribution is given, in terms of the prior density and the likelihood function, by
\begin{equation*}
   \mathbb{p}(\pmb{\theta},\mathbf{Z}\mid\mathbf{X})
   \propto  \, \mathbb{p}(\pmb{\theta},\mathbf{Z})L(\mathbf{X}\mid\pmb{\theta},\mathbf{Z}).
\end{equation*}

\emph{Mean-field variational inference} (MFVI) seeks to approximate the target posterior distribution $\mathbb{P}(\pmb{\theta},\mathbf{Z}\mid\mathbf{X})$ by a mean-field distribution, that is, a fully factorized distribution. For any two probability distributions $P_1$ and $P_2$ defined on the same measurable space, we denote by $\KL(P_1\,\|\;P_2)$ the Kullback-Leibler (KL) divergence from $P_2$ to $P_1$:
\begin{equation*}
    \KL(P_1\,\|\;P_2):=\mathbb{E}_{P_1}\bigg[\log\bigg(\frac{dP_1}{dP_2}\bigg)\bigg],\qquad \text{ if } P_1\ll P_2,
\end{equation*}
and $\KL(P_1\,\|\;P_2):= +\infty$ otherwise; here $\frac{dP_1}{dP_2}$ is the Radon-Nikodym derivative of $P_1$ w.r.t.~$P_2$. Let $\mathcal{P}$ denote the family of mean-field distributions on $(\pmb{\theta},\mathbf{Z})$; i.e., $P\in\mathcal{P}$ if and only if the density of $P$ (w.r.t.~the product base measure $\mathcal{M}$) has the form $\prod_{j=1}^m \xi_j(\pmb{\theta}_j)\cdot \prod_{i=1}^n \zeta_i(Z_i)$, where $\xi_j$ is a density on $\mathbb{R}^{d_j}$ w.r.t. $\mathcal{M}_{g,j}$ for each $j\in[m]$, and $\zeta_i$ is a density on $[K_i]$ w.r.t. $\mathcal{M}_{l,i}$ for each $i\in[n]$. 

MFVI finds a distribution $\hat{P} \in \mathcal{P}$---called a {\it variational posterior}---that is closest in KL divergence to the posterior distribution $\mathbb{P}(\pmb{\theta},\mathbf{Z}\mid\mathbf{X})$, i.e.,
\begin{equation}\label{vari3.1}
\hat{P}\,\in\,\arg \min_{P\in\mathcal{P}} \KL \big(P\, \big \|\; \mathbb{P}(\pmb{\theta},\mathbf{Z}\mid\mathbf{X}) \big),
\end{equation}
where we assume that the minimum is attained. In both our applications (LDA and MMSB) the class of product probability measures $\mathcal{P}$ is supported on a compact set ensuring that minimizers exist (since the KL divergence is lower semicontinuous). It is worth noting that while optimization over the space of probability measures $\mathcal{P}$ may seem formidable, in many cases,~\eqref{vari3.1} reduces to a finite-dimensional optimization problem over Euclidean spaces (see e.g.,~\cite[Section 4.1]{blei2017variational}).

Besides the aforementioned fully factorized MFVI, a variant called \emph{collapsed variational inference} (collapsed VI) is also widely used in the literature \cite{teh2006collapsed, foulds2013stochastic}. In collapsed VI, one considers the family $\mathcal{Q}$ of mean-field (i.e., product) distributions on the local latent variables $\mathbf{Z}$, i.e., $Q\in\mathcal{Q}$ if and only if the density of $Q$ (w.r.t.~the marginal product base measure $\prod_{i=1}^n\mathcal{M}_{l,i}$) has the form $\prod_{i=1}^{n} \zeta_i(Z_i)$, where $\zeta_i$ is a density on $[K_i]$ for each $i\in[n]$. We then approximate the marginal (``collapsed'') posterior distribution $\mathbb{P}(\mathbf{Z}\mid\mathbf{X})$ by a distribution $\hat{Q}$ in $\mathcal{Q}$ that is closest in KL divergence to the true collapsed posterior, i.e.,
\begin{equation}\label{vari3.2}
  \hat{Q}\,\in\,\arg\min_{Q\in\mathcal{Q}} \KL\big(Q\,\big\|\;\mathbb{P}(\mathbf{Z}\mid\mathbf{X})\big).
\end{equation}

In this paper, we investigate the accuracy of the variational posteriors $\hat{P}$ and $\hat{Q}$ in approximating the full and collapsed posteriors, respectively. In Section \ref{Sect.2} (Theorems~\ref{Theorem2.1} and \ref{Theorem2.2}), we quantify this discrepancy by deriving explicit tight bounds on the KL divergences
$\KL \big(\hat{P}\,  \|\; \mathbb{P}(\pmb{\theta},\mathbf{Z}\mid\mathbf{X}) \big)$ and $\KL\big(\hat{Q}\, \|\;\mathbb{P}(\mathbf{Z}\mid\mathbf{X})\big)$ for general Bayesian latent variable models with categorical local latent variables where the conditional distribution of $\pmb{\theta}$ given $\mathbf{Z}$ and $\mathbf{X}$ factorizes (see Assumption~\ref{Assump:Cond-ind}). Given the prevalence of intractable posterior distributions in modern Bayesian statistics and probabilistic machine learning, as well as the widespread use of VI in these settings, our results have important implications for downstream inference and prediction tasks.

For both LDA and MMSB, our general results characterize the exact regime for the validity of VI. Moreover, we derive \emph{tight} non-asymptotic upper and lower bounds on the approximation accuracy of VI in these models (see Sections \ref{Sect.1.2} and \ref{Sect.1.3} for details). Notably, our results remain robust under model misspecification, when the  observed data $\mathbf{X}$ is not generated from the assumed model. Additionally, our results are applicable in the high-dimensional setting, encompassing both sparse and dense parameter regimes, and hold uniformly across all choices of hyperparameters, under only one assumption on the model parameters (see~\eqref{conditionalpha} for LDA and~\eqref{conditionalpha2} for MMSB).

Our results also have direct implications on the mean field approximation of the {\it log-marginal density\slash log evidence}\footnote{In the machine learning literature, $\mathbb{p}(\mathbf{X})$ is usually referred to as the~\textit{evidence}.} $\log \mathbb{p}(\mathbf{X})$ of the data $\mathbf{X}$---the {\it log-partition function\slash log-normalizing constant} for the posterior distribution $\mathbb{P}(\pmb{\theta},\mathbf{Z}\mid\mathbf{X})$. Notice that, for any $P \in \mathcal{P}$ with density $p(\pmb{\theta},\mathbf{Z})$ (w.r.t.~the product base measure $\mathcal{M}$), 
one can easily deduce that 
\begin{equation}\label{eq:VI-ELBO}
\KL \big(P\, \big \|\; \mathbb{P}(\pmb{\theta},\mathbf{Z}\mid\mathbf{X}) \big) = \log \mathbb{p}(\mathbf{X}) - \mathrm{ELBO}(P) \ge 0,
\end{equation}
where $\mathrm{ELBO}(P)$ is called the \textit{evidence lower bound} (ELBO)\footnote{The ELBO plays an important role in the computation of VI: it is an efficiently computable lower bound on the log evidence.} at $P$ and is defined as
\begin{equation}\label{ELBO_def}
   \mathrm{ELBO}(P):=\mathbb{E}_P[\log \mathbb{p}(\pmb{\theta},\mathbf{Z},\mathbf{X})]-\mathbb{E}_P[\log p(\pmb{\theta},\mathbf{Z})], \qquad \mbox{ for all } \;\; P\in\mathcal{P};
\end{equation}
here $\mathbb{p}(\pmb{\theta},\mathbf{Z},\mathbf{X}):=\mathbb{p}(\pmb{\theta},\mathbf{Z})L(\mathbf{X}\mid\pmb{\theta},\mathbf{Z})$ is the complete likelihood of the latent variables and the data. Thus, from~\eqref{eq:VI-ELBO}, the minimization problem in~\eqref{vari3.1} is equivalent to maximizing  ELBO$(P)$ in~\eqref{ELBO_def} over $\mathcal{P}$. Moreover, as\footnote{This viewpoint is crucial in the development of the CAVI algorithm.}  $$\KL \big(\hat{P}\, \big \|\; \mathbb{P}(\pmb{\theta},\mathbf{Z}\mid\mathbf{X}) \big) \approx 0 \qquad \Leftrightarrow \qquad \log \mathbb{p}(\mathbf{X}) \approx \mathrm{ELBO}(\hat{P}),$$ 
our main results immediately imply tight non-asymptotic lower and upper bounds on the log evidence. Note that, due to its connection to the log evidence, the ELBO has been widely used as a model selection criterion (see the discussion in \cite[Section 2.2]{blei2017variational} and the references therein). 

The rest of the Introduction is organized as follows. Sections \ref{Sect.1.2} and \ref{Sect.1.3} present detailed introductions to LDA and MMSB, as well as our main results on the statistical accuracy of VI in these two models. The main contributions of this paper are summarized in Section \ref{Sect.1.4}. Section \ref{Sect.1.5} discusses related work, while Section~\ref{Sect.1.6} outlines the structure of the paper.

\subsection{Latent Dirichlet allocation (LDA)}\label{Sect.1.2}

In this subsection, we discuss LDA~\cite{pritchard2000inference, blei2003latent, falush2003inference}, which is a widely used probabilistic model for topic modeling. 

Assume that there are in total $D\in \mathbb{N}^{*}$ documents and that the vocabulary set is $[V]$ with $V\in\mathbb{N}^{*}$. For every $d\in [D]$, let $n_d\in\mathbb{N}^{*}$ be the number of words in the $d$th document. We let $n :=\sum_{d=1}^D  n_d $ be the total number of words. We assume that there are $K \in\mathbb{N}^{*}$ latent topics.

We fix $\pmb{\alpha}\equiv(\alpha_1,\alpha_2,\cdots,\alpha_K) \in (0,\infty)^K$ and a $K\times V$ matrix $\pmb{\eta} \equiv (\eta_{\ell,r})_{\ell\in [K], r\in [V]}$ such that $\eta_{\ell,r}\geq 0$ for every $\ell\in [K]$ and $ r\in [V]$, and $\sum_{r=1}^V \eta_{\ell,r}=1$ for every $\ell\in [K]$. Here $\pmb{\alpha}$ and $\pmb{\eta}$ are treated as hyperparameters. For each $\ell\in [K]$, $\alpha_{\ell}$ encapsulates the proportion of topic $\ell$ in the documents. For each $\ell\in [K]$ and $r\in [V]$, $\eta_{\ell,r}$ represents the probability of picking the $r$th vocabulary from the $\ell$th topic for a word.

For each document $d\in [D]$, we independently perform the following steps. We first sample a $K$-dimensional probability vector $\pmb{\pi}_d \equiv (\pi_{d,1},\cdots,\pi_{d,K})\sim \text{Dir}(\alpha_1,\alpha_2,\cdots,\alpha_K)$, where by $\text{Dir}(\alpha_1,\alpha_2,\cdots,\alpha_K)$ we mean the Dirichlet distribution of order $K$ with parameters $\alpha_1,\cdots,\alpha_K>0$. For each $d\in [D]$, $\pmb{\pi}_d$ gives the proportion of the $K$ topics in the $d$th document. For each $i\in [n_d]$, we sample a topic $Z_{d,i} \mid \pmb{\pi}_d \stackrel {i.i.d.} \sim \text{Cat}(\pmb{\pi}_d)$, where $\text{Cat}(\pmb{\pi}_d)$ denotes the categorical distribution with parameter $\pmb{\pi}_d$. Here, $Z_{d,i} \in [K]$ denotes the latent topic from which the $i$th word is generated in the $d$th document. Now, for every $i\in [n_d]$, we sample the $i$th word in the $d$th document as: 
\begin{eqnarray*}
   (X_{d,i} \mid Z_{d,i} = \ell) \stackrel{ind}{\sim} \text{Cat}(\pmb{\eta}_\ell),
\end{eqnarray*}
where by $\pmb{\eta}_\ell$ we mean the $\ell$th row of the matrix $\pmb \eta$. The observations are $\mathbf{X}:=(X_{d,i})_{d\in [D],i\in [n_d]}$. The local latent variables are $\mathbf{Z}:=(Z_{d,i})_{d\in [D], i\in [n_d]} \in [K]^n$, whereas the global latent variables are $\{\pmb{\pi}_d\}_{d\in[D]}$.

The full and collapsed posteriors in LDA can be derived as follows. For any $d\in [D],\ell\in [K]$, we let $N_{d,\ell}$ be the number of words in the $d$th document originating from topic $\ell$, i.e., 
\begin{equation}\label{defNdl}
    N_{d,\ell}:=|\{i\in [n_d]: Z_{d,i}=\ell\}|.
\end{equation}
For any $d\in [D], \ell\in [K], r\in [V]$, we let $N_{d,\ell,r}$ be the number of times the $r$th vocabulary word appears from topic $\ell$ in the $d$th document, i.e.,
\begin{equation}\label{defNdlr}
    N_{d,\ell,r}:=|\{i\in [n_d]: Z_{d,i}=\ell, X_{d,i}=r\}|.
\end{equation}
From the above generation process, a simple calculation shows that the posterior distribution of the latent variables given the data $\mathbf{X}$ is
\begin{equation}\label{Eq3.1}
    \mathbb{P} \left(\mathbf{Z},\{\pmb{\pi}_d\}_{d\in [D]}\mid\mathbf{X} \right) 
   \propto \prod_{d=1}^D \prod_{\ell=1}^K \pi_{d,\ell}^{N_{d,\ell}+\alpha_{\ell}-1} \cdot \prod_{d=1}^D \prod_{\ell=1}^K \prod_{r=1}^V \eta_{\ell,r}^{N_{d,\ell,r}}.
\end{equation}
Integrating out $\{\pmb{\pi}_d\}_{d\in [D]}$ in (\ref{Eq3.1}), we obtain the collapsed posterior:
\begin{eqnarray}\label{Eq3.2}
    \mathbb{P}(\mathbf{Z}\mid\mathbf{X})\propto \prod_{d=1}^D\prod_{\ell=1}^K\prod_{r=1}^V \eta_{\ell,r}^{N_{d,\ell,r}}\cdot \prod_{d=1}^D \frac{\prod_{\ell=1}^K \Gamma(N_{d,\ell}+\alpha_{\ell})}{\Gamma\Big(n_d+\sum_{\ell=1}^K \alpha_{\ell}\Big)} .
\end{eqnarray}

In applications of LDA, the full posterior $\mathbb{P}(\mathbf{Z},\{\pmb{\pi}_d\}_{d\in [D]}\mid\mathbf{X})$ and the collapsed posterior 
$\mathbb{P}(\mathbf{Z}\mid\mathbf{X})$ play a crucial role, as they capture key latent information---the topic proportions in each document and the topic assignments for each word. However, as expressed in~\eqref{Eq3.1} and~\eqref{Eq3.2}, these posteriors are complex distributions with intractable normalizing constants. To analyze and approximate these effectively, VI has become an indispensable tool.

\subsubsection{Validity of mean-field variational inference (MFVI)}

In the following, we establish the validity of MFVI for LDA. Let $\hat{P}$ and $\hat{Q}$ denote the full and collapsed variational posteriors, respectively, as defined in (\ref{vari3.1}) and (\ref{vari3.2}), for the LDA model. The following result, proved in Section \ref{Sect.3}, provides a non-asymptotic upper bound on the KL divergence between the variational and true posteriors. In this result, we allow the variables $D$, $K$, $V$, and $n$ to grow arbitrarily, and impose no conditions whatsoever; we only require that the hyperparameters $\{\alpha_{\ell}\}_{\ell\in [K]}$ are uniformly bounded. Since the number of latent variables in the collapsed posterior $\mathbb{P}(\mathbf{Z}\mid\mathbf{X})$ is $n$, we normalize the KL divergence between the variational and true posteriors by $n$. 

\begin{theorem}\label{Theorem_LDA_UBD}
Suppose that for each $r\in [V]$, there exists at least one $\ell\in [K]$ such that $\eta_{\ell,r}>0$.\footnote{Otherwise $r \in [V]$ can be removed from the vocabulary.} Assume that
\begin{equation}\label{conditionalpha}
    \alpha_{\ell}\in (0,C_0],\quad\text{ for every }\ell\in [K],
\end{equation}
where $C_0>0$ is fixed and independent of $n$. Then,
\begin{enumerate}
    \item[(a)] for the full posterior we have \begin{equation*}\label{upper1.3}
    \frac{1}{n} \sup_{\mathbf{X}\in [V]^n} \KL\big(\hat{P} \,\big\| \; \mathbb{P}(\mathbf{Z},\{\pmb{\pi}_d\}_{d\in [D]}\mid\mathbf{X}) \big)\leq  \frac{CDK}{n}\log\Big(\frac{n}{DK}+2\Big);
\end{equation*}
    \item[(b)] for the collapsed posterior we have 
    \begin{equation*}\label{upper1.4}
    \frac{1}{n} \sup_{\mathbf{X}\in [V]^n} \KL\big(\hat{Q} \,\big\| \; \mathbb{P}(\mathbf{Z}\mid\mathbf{X}) \big)\leq \frac{CDK}{n}\log\Big(\frac{n}{DK}+2\Big),
\end{equation*}
\end{enumerate}
where $C$ is a positive constant that only depends on $C_0$.
\end{theorem}

Under the assumptions in Theorem \ref{Theorem_LDA_UBD}, if $DK=o(n)$, the right-hand sides of (a) and (b) converge to $0$ as $n\rightarrow\infty$. Thus, the variational posteriors $\hat{P}$ and $\hat{Q}$ give consistent approximations to the true and collapsed posteriors throughout the regime $DK=o(n)$. 

Theorem \ref{Theorem_LDA_UBD} holds uniformly across all choices of hyperparameters (under mild conditions); in particular, this allows for arbitrarily small values of $\alpha_{\ell}$ and $\eta_{\ell,r}$ (where $\ell\in [K]$ and $r\in [V]$). These parameters are even allowed to be vanishing at arbitrary rates as $n$ increases (such parameters would yield sparse topic vectors and sparse topic proportions). Therefore, our results cover both sparse and dense regimes. We also do not impose any restriction on the observed data $\mathbf{X}$, which allows us to conclude that MFVI is valid even under model misspecification (i.e., when $\mathbf{X}$ is not generated from the assumed model). Using~\eqref{eq:VI-ELBO}, Theorem~\ref{Theorem_LDA_UBD} immediately implies that the gap between the log evidence $\log \mathbb{p}(\mathbf{X})$ and the supremum of the ELBO over $\mathcal{P}$ (recall~\eqref{vari3.1} and~\eqref{ELBO_def}) is upper bounded by $CDK\log\big(\frac{n}{DK}+2\big)$. 

As will be demonstrated in Theorem \ref{Theorem_LDA_LBD} (proved in Appendix \ref{Appendix_C}) below, the condition $DK=o(n)$ is indeed necessary for both the full and collapsed variational posteriors to be valid. Therefore, Theorem \ref{Theorem_LDA_UBD} covers the exact parameter regime where MFVI is valid. 

\begin{theorem}\label{Theorem_LDA_LBD}
Assume that $K\geq 2$, $\alpha_{\ell}=1\slash 2$ for all $\ell\in[K]$, $\eta_{\ell,r}=V^{-1}$ for all $\ell\in [K], r\in [V]$, and $n_1=\cdots = n_D$. Then there exists a positive absolute constant $c_0$, such that when $DK\leq c_0 n$, for any $\mathbf{X}\in [V]^n$, we have
\begin{equation*}
   \frac{1}{n}\KL\big(\hat{P} \,\big\| \; \mathbb{P}(\mathbf{Z},\{\pmb{\pi}_d\}_{d\in [D]}\mid\mathbf{X}) \big)\geq  \frac{1}{n} 
 \KL\big(\hat{Q}\,\big\|\;\mathbb{P}(\mathbf{Z}\mid\mathbf{X})\big)\geq \frac{DK}{5n}\log\Big(\frac{n}{DK}+2\Big).
\end{equation*}
\end{theorem}
Theorem \ref{Theorem_LDA_LBD} shows that the non-asymptotic upper bounds in Theorem \ref{Theorem_LDA_UBD} provide \emph{optimal convergence rates} for the KL divergence between the variational and true posteriors. We expect a similar conclusion (as in Theorem \ref{Theorem_LDA_LBD}) to hold for general $\alpha_{\ell}$'s ($\ell\in [K]$) that are uniformly bounded away from $0$ and $\infty$, possibly at the expense of a more technical proof.
\begin{remark}\label{Remark1.3}
Theorems~\ref{Theorem_LDA_UBD} and~\ref{Theorem_LDA_LBD} together provide a comprehensive characterization of the validity of MFVI in terms of the parameters $D,K,n$. Specifically, Theorem \ref{Theorem_LDA_UBD} establishes that MFVI remains a valid approximation when $DK=o(n)$, while Theorem \ref{Theorem_LDA_LBD} demonstrates that MFVI becomes unreliable when $DK\asymp n$. In this regime, we expect the `TAP correction' to the mean-field approximation to work better, akin to the behavior observed in \cite{ghorbani2019instability}. Moreover, it is plausible that under this scaling, the number of near optimizers of the MFVI problem grows exponentially in $n$, similar to what happens for the optimizers of the Hamiltonian in related models in \cite{chen2018energy}.
\end{remark}

\subsection{Mixed membership stochastic blockmodel (MMSB)}\label{Sect.1.3}

In this subsection, we examine the MMSB, introduced in \cite{airoldi2008mixed}, which provides a probabilistic framework for analyzing complex network structures. Below, we describe the model in more detail. 

Let $n,K\in \mathbb{N}^{*}$ with $n\geq 2$. We fix a $K$-dimensional vector $\pmb{\alpha}\equiv (\alpha_1,\alpha_2,\cdots,\alpha_K)$ and a $K\times K$ matrix $\mathbf{B}\equiv (B_{\ell,\ell'})_{\ell,\ell'\in [K]}$, which are treated as hyperparameters. For each $\ell\in [K]$, $\alpha_{\ell}$ reflects the proportion of members belonging to the $\ell$th group in the network, and we require that $\alpha_{\ell}> 0$. For each $\ell,\ell'\in [K]$, $B_{\ell,\ell'}\in[0,1]$ represents the probability of having an edge between a node from group $\ell$ and a node from group $\ell'$. 

The MMSB model concerns a directed graph with node set $[n]$. Each node $i\in [n]$ represents an individual. The (directed) edges in the graph are generated as follows. For each $i\in [n]$, we draw a $K$-dimensional mixed membership probability vector $\pmb{\pi}_i \equiv (\pi_{i,1},\cdots,\pi_{i,K})\sim \text{Dir}(\alpha_1,\cdots,\alpha_K)$, where for each $\ell\in [K]$, $\pi_{i,\ell}$ gives the probability of node $i$ belonging to group $\ell$. For each pair of nodes $(i,j)\in [n]^2$ with $i\neq j$ we call $i$ the initiator and $j$ the receiver, and we do the following:
\begin{itemize}
    \item Draw membership for the initiator, $Z_{i\rightarrow j}\sim \text{Cat}(\pmb{\pi}_i)$.
    \item Draw membership for the receiver, $Z_{i\leftarrow j}\sim\text{Cat}(\pmb{\pi}_j)$.
    \item Sample $X_{i,j} \mid Z_{i\rightarrow j},Z_{i\leftarrow j} \sim \text{Bernoulli}(B_{Z_{i\rightarrow j},Z_{i\leftarrow j}} )$. If $X_{i,j}=1$, there is a directed edge from $i$ to $j$; otherwise, there is no directed edge from $i$ to $j$. 
\end{itemize}
The observations are $(X_{i,j})_{i,j\in [n]: i\neq j}$, which we denote by $\mathbf{X}$. We denote by
\begin{equation*}
    \mathbf{Z}_{\rightarrow}:=(Z_{i\rightarrow j})_{i,j\in [n]:i\neq j}, \qquad \mbox{and}\qquad \mathbf{Z}_{\leftarrow}:=(Z_{i\leftarrow j})_{i,j\in [n]:i\neq j}.
\end{equation*}
Here $\mathbf{Z}_{\rightarrow}$ and $\mathbf{Z}_{\leftarrow}$ represent the local latent variables, and $\{\pmb{\pi}_i\}_{i=1}^n$ form the global latent variables. The full and collapsed posteriors in MMSB can be derived as follows. 
For any $i\in [n]$ and $\ell\in [K]$, we let
\begin{equation}\label{defNil}
    N_{\rightarrow,i,\ell}:=|\{j\in [n]\backslash \{i\}: Z_{i\rightarrow j}=\ell\}|,\quad \quad N_{\leftarrow,i,\ell}:=|\{j\in [n]\backslash\{i\}: Z_{j\leftarrow i}=\ell\}|,
\end{equation}
and $N_{i,\ell} := N_{\rightarrow,i,\ell} + N_{\leftarrow,i,\ell}$. Note that $N_{\rightarrow,i,\ell}$ is the number of times node $i$ belongs to the $\ell$th group when interacting with the other $n-1$ nodes as the initiator, and $N_{\leftarrow,i,\ell}$ is the number of times node $i$ belongs to the $\ell$th group when interacting with the other $n-1$ nodes as the receiver. For any $\ell,\ell'\in [K]$, we let
\begin{eqnarray*}
    A_{\ell,\ell'} & := & \left|\{(i,j)\in [n]^2:i\neq j, Z_{i \rightarrow j}=\ell, Z_{i\leftarrow j}=\ell'\} \right|, \qquad \mbox{and} \\
    M_{\ell,\ell'} &:=&  \left|\{(i,j)\in [n]^2:i\neq j, Z_{i \rightarrow j}=\ell, Z_{i\leftarrow j}=\ell', X_{i,j}=1\} \right|.
\end{eqnarray*}
Here, $A_{\ell,\ell'}$ represents the number of pairs of nodes whose group memberships are given by $\ell$ and $\ell'$ when they interact, and $M_{\ell,\ell'}$ represents the number of such pairs of nodes that have a directed edge between them. From the above data generating process, a simple calculation shows that the posterior distribution of the latent variables given the data $\mathbf{X}$ is
{\small \begin{equation}\label{Eq4.1}
\mathbb{P}\left(\mathbf{Z}_{\rightarrow},\mathbf{Z}_{\leftarrow},\{\pmb{\pi}_i\}_{i\in [n]}\mid\mathbf{X}\right) \propto \prod_{i=1}^n\prod_{\ell=1}^K \pi_{i,\ell}^{N_{i,\ell}+\alpha_{\ell}-1}\cdot \prod_{\ell=1}^K \prod_{\ell'=1}^K B_{\ell,\ell'}^{M_{\ell,\ell'}} (1-B_{\ell,\ell'})^{A_{\ell,\ell'}-M_{\ell,\ell'}}.
\end{equation}
}Integrating out $\{\pmb{\pi}_{i}\}_{i\in [n]}$ in (\ref{Eq4.1}), we obtain the collapsed posterior
{\small \begin{equation}\label{mmsb} \mathbb{P}\left(\mathbf{Z}_{\rightarrow},\mathbf{Z}_{\leftarrow}\mid\mathbf{X} \right) \propto \prod_{\ell=1}^K \prod_{\ell'=1}^K B_{\ell,\ell'}^{M_{\ell,\ell'}} (1-B_{\ell,\ell'})^{A_{\ell,\ell'}-M_{\ell,\ell'}} \cdot \prod_{i=1}^n \frac{\prod_{\ell=1}^K\Gamma(N_{i,\ell}+\alpha_{\ell})}{\Gamma\big(2n-2+\sum_{\ell=1}^K \alpha_{\ell}\big)}.
\end{equation}
}

The full and collapsed posteriors allow inference about the membership probability-like vector $\pmb{\pi}_i$ for the $i$th individual, as well as the group memberships of individuals when interacting with others, encapsulated by the local latent variables $\mathbf{Z}_{\rightarrow}$ and $\mathbf{Z}_{\leftarrow}$.
Due to the complicated forms of the full and collapsed posteriors, their normalizing constants are intractable to compute, and VI is crucial for providing efficient approximations to these posteriors. \newline

\noindent{\bf{Partially grouped variational inference.}} In the discussions to follow, for each pair $(i,j)\in[n]^2$ with $i\neq j$, we group the variables $(Z_{i\rightarrow j}, Z_{i\leftarrow j})$ and treat $(Z_{i\rightarrow j}, Z_{i\leftarrow j})$ as a single categorical variable with $K^2$ possible values. Thus, the local latent variables in the model are the grouped variables $(Z_{i\rightarrow j}, Z_{i\leftarrow j})$ for distinct $i,j\in [n]$. We refer to the optimization problems (\ref{vari3.1}) and (\ref{vari3.2}) with such grouping of categorical variables as ``\emph{partially grouped variational inference}'' (partially grouped VI). A similar partial grouping strategy was previously employed in \cite{xing2010state} for logistic-normal MMSB (a variation of MMSB).

In Theorem~\ref{Theorem_MMSB_S} below we show that for MMSB the fully factorized VI (which assumes the independence of all individual categorical variables) does not give a close approximation to the true posterior, while the proposed partially grouped VI is shown to be accurate in approximating the true posterior in Theorem~\ref{Theorem_MMSB_UBD}. In fact, in Appendix~\ref{Sect.4}, we illustrate, via simulations, that $Z_{i\rightarrow j}$ and $Z_{i\leftarrow j}$ (where $i,j\in[n]$ and $i\neq j$) are not approximately independent in general. As far as we are aware, this provably superior performance of partially grouped VI for MMSB has not been observed before in the literature (cf.~\cite{gopalan2013efficient,huang2020mixed}). We further provide a CAVI update algorithm for the computation of the partially grouped VI procedure in Appendix \ref{Appendix_F}.  

\subsubsection{Validity of partially grouped variational inference}

In the following, we establish the validity of partially grouped VI for MMSB. Let $\hat{P}$ and $\hat{Q}$ be the full and collapsed partially grouped variational posteriors for the MMSB as defined in (\ref{vari3.1}) and (\ref{vari3.2}). The following result (proved in Appendix \ref{Appendix_B}) provides a non-asymptotic upper bound on the KL divergence between the variational and true posteriors. 

\begin{theorem}\label{Theorem_MMSB_UBD}
Suppose that the $K \times K$ matrix $\mathbf{B}$ is neither the zero matrix nor the matrix of all $1$'s.\footnote{Otherwise the observations would be trivially $X_{i,j}\equiv 0$ or $X_{i,j}\equiv 1$ for all distinct $i,j\in [n]$.} Assume that
\begin{equation}\label{conditionalpha2}
    \alpha_{\ell}\in (0,C_0],\quad\text{ for every }\ell\in [K],
\end{equation}
where $C_0>0$ is fixed and independent of $n$. Then,  
\begin{itemize}
    \item[(a)] for the full posterior we have \begin{equation*}\label{E1.1}
   \frac{1}{n^2} \sup_{\mathbf{X}\in \{0,1\}^{n(n-1)}} \KL\big(\hat{P}\,\big\|\;\mathbb{P}(\mathbf{Z}_{\rightarrow},\mathbf{Z}_{\leftarrow},\{\pmb{\pi}_{i}\}_{i\in [n]}\mid\mathbf{X})\big)\leq \frac{CK}{n}\log\Big(\frac{n}{K}+2\Big);
\end{equation*}
 \item[(b)] for the collapsed posterior we have \begin{equation*}\label{E1.2}
   \frac{1}{n^2} \sup_{\mathbf{X}\in \{0,1\}^{n(n-1)}} \KL\big(\hat{Q}\,\big\|\;\mathbb{P}(\mathbf{Z}_{\rightarrow},\mathbf{Z}_{\leftarrow}\mid\mathbf{X})\big)\leq \frac{CK}{n}\log\Big(\frac{n}{K}+2\Big),
\end{equation*}
\end{itemize}
where $C$ is a positive constant that only depends on $C_0$. 
\end{theorem}
In the above result, we allow $K$ and $n$ to be arbitrary. As the number of variables in the collapsed posterior $\mathbb{P}(\mathbf{Z}_{\rightarrow},\mathbf{Z}_{\leftarrow}\mid\mathbf{X})$ is $2n(n-1)\asymp n^2$, we naturally scale the KL divergence between the variational and true posteriors by $n^2$. Under the assumptions of Theorem \ref{Theorem_MMSB_UBD}, if $K=o(n)$, the right-hand sides of (a) and (b) above converge to $0$ as $n\rightarrow\infty$. Thus, the partially grouped variational distributions $\hat{P}$ and $\hat{Q}$ give close approximations to the full and collapsed posteriors in the regime when $K=o(n)$. 

We note that in Theorem \ref{Theorem_MMSB_UBD}, $\alpha_{\ell}$ (where $\ell\in[K]$) is allowed to be arbitrarily small and to vanish at arbitrary rates as $n$ increases. Such parameters allow for sparse mixed membership vectors, and thus Theorem \ref{Theorem_MMSB_UBD} encompasses both sparse and dense regimes of MMSB. Further, as we do not impose any restrictions on the observed data $\mathbf{X}$, partially grouped VI maintains its validity even under model misspecification. 

We demonstrate in Theorem \ref{Theorem_MMSB_LBD} (proved in Appendix \ref{Appendix_C}) below that the condition $K=o(n)$ characterizes the exact parameter regime where partially grouped VI is valid. Theorem \ref{Theorem_MMSB_LBD} also shows that the non-asymptotic upper bounds in Theorem \ref{Theorem_MMSB_UBD} provide \emph{optimal convergence rates} for the KL divergence between the variational and true posteriors.  
\begin{theorem}\label{Theorem_MMSB_LBD}
Assume that $K\geq 2$, $\alpha_{\ell}=1\slash 2$ for all $\ell\in [K]$, and $B_{\ell,\ell'}= B\in (0,1)$ for all $\ell,\ell'\in [K]$. Then there exists a positive absolute constant $c_0$, such that when $K\leq c_0 n$, for any $\mathbf{X}\in \{0,1\}^{n(n-1)}$, we have
{\small \begin{eqnarray*}
\frac{1}{n^2}\KL\big(\hat{P}\,\big\|\;\mathbb{P}(\mathbf{Z}_{\rightarrow},\mathbf{Z}_{\leftarrow},\{\pmb{\pi}_{i}\}_{i\in [n]}\mid\mathbf{X})\big) 
 \; \geq \;\frac{1}{n^2} \KL\big(\hat{Q}\,\big\|\;\mathbb{P}(\mathbf{Z}_{\rightarrow},\mathbf{Z}_{\leftarrow}\mid\mathbf{X})\big)\; \geq \;\frac{K}{5n}\log\Big(\frac{n}{K}+2\Big).
\end{eqnarray*}}
\end{theorem}
We expect a similar conclusion to hold for general $\alpha_{\ell}$'s ($\ell\in [K]$) that are uniformly bounded away from $0$ and $\infty$, possibly at the expense of a more technical argument. We note in passing that a comment similar to Remark \ref{Remark1.3} also applies here. In particular, the validity of partially grouped VI for the MMSB holds if $K=o(n)$, and breaks down when $K\asymp n$.

\subsubsection{Suboptimality of fully factorized variational inference for MMSB}\label{SuboptimalfullVI}

Consider the family $\mathcal{P}^{\mathrm{ff}}$ of fully factorized mean-field distributions on $(\mathbf{Z}_{\rightarrow},\mathbf{Z}_{\leftarrow},\{\pmb{\pi}_i\}_{i\in [n]})$; i.e., $P\in\mathcal{P}^{\mathrm{ff}}$ if and only if the density of $P$ w.r.t. a product base measure has the form
\begin{equation*}
    \prod_{i,j\in [n]:i\neq j}\zeta_{i \rightarrow j}(Z_{i\rightarrow j})\cdot \prod_{i,j\in [n]: i\neq j}\zeta_{i \leftarrow j}(Z_{i\leftarrow j}) \cdot \prod_{i=1}^n \xi_i(\pmb{\pi}_i).
\end{equation*}
Fully factorized VI finds a distribution $\hat{P}^{\mathrm{ff}}\in\mathcal{P}^{\mathrm{ff}}$ that is closest in KL divergence to the posterior distribution (\ref{Eq4.1}), i.e., 
\begin{equation*}
    \hat{P}^{\mathrm{ff}}\in\arg\min_{P\in \mathcal{P}^{\mathrm{ff}}} \KL\big(P\,\big\|\;\mathbb{P}(\mathbf{Z}_{\rightarrow},\mathbf{Z}_{\leftarrow},\{\pmb{\pi}_{i}\}_{i\in [n]}\mid\mathbf{X})\big).
\end{equation*}
The following result (proved in Appendix \ref{Appendix_D}) gives an example where the fully factorized variational posterior $\hat{P}^{\mathrm{ff}}$ does not yield an accurate approximation of the true posterior. 
\begin{theorem}\label{Theorem_MMSB_S}
Assume that $K=2$, $0<B_{1,1}=B_{2,2} < B_{1,2}=B_{2,1}$, $\alpha_1=\alpha_2=1\slash 2$, and $X_{i,j}=1$ for all $i,j\in [n]$ with $i\neq j$. Then there exists a positive constant $\delta$ that only depends on $\mathbf{B} \equiv (B_{\ell,\ell'})_{\ell,\ell'\in [2]}$, such that for $n$ sufficiently large,
\begin{equation*}
    \frac{1}{n^2}\KL\big(\hat{P}^{\mathrm{ff}}\, \big\|\; \mathbb{P}(\mathbf{Z}_{\rightarrow},\mathbf{Z}_{\leftarrow},\{\pmb{\pi}_{i}\}_{i\in [n]}\mid\mathbf{X}) \big)\geq \delta.
\end{equation*}
\end{theorem}

\subsection{Contributions}\label{Sect.1.4}

Although there has been some work \cite{Alquier-JMLR-2016, pati2018statistical, wang2018frequentist, zhang2020convergence} (see Section \ref{Sect.1.5} below for a detailed description of these references) in studying the behavior of MFVI, most of the existing literature either deals with low-dimensional settings or assumes sparsity conditions. As far as we are aware, this is the first paper that provides upper bounds that quantify the accuracy of MFVI for both the full and collapsed variational posteriors in high-dimensional settings in a wide class of latent variable models, including LDA and MMSB. Further, we provide examples to demonstrate that our bounds are tight, up to constants. Our results do not require a true frequentist data generating model, neither do we need contraction properties of the posterior. They also accommodate both dense and sparse parameter regimes, giving the theory wide applicability. An important feature of our results is that we allow for model misspecification, i.e., the observed data $\mathbf{X}$ need not be generated from the assumed model. Moreover, our results for the validity of MFVI do not need the identifiability of parameters in these latent variable models; for complicated hierarchical Bayesian models, identifiability can be tricky to verify and quite often requires further assumptions on the model (see \cite{dunson2000bayesian, gu2023bayesian} and the references therein).

Our analysis of the accuracy of MFVI builds upon and extends the framework of {\it nonlinear large deviations}. The subject of nonlinear large deviations originates from the work of Chatterjee and Dembo \cite{chatterjee2016nonlinear}.  The goal of nonlinear large deviations is to develop a theory of large deviations for nonlinear functions $f(X_1,\cdots,X_n)$ of independent random variables $X_1,\cdots,X_n$. At the core of this framework is an error bound, expressed in terms of the properties of $f$ (such as the gradient and Hessian of $f$), for approximating the log-partition function\slash log-normalizing constant of the Gibbs distribution with density proportional to $e^{f(X_1,\cdots,X_n)}$. In our context, we use these tools to approximate the log-partition function\slash log-normalizing constant of the posterior distribution---which is nothing but the marginal log-likelihood of the observed data $\mathbf{X}$. Lemma~\ref{Lemma2.1} in Section \ref{Sect.2} directly relates the KL divergence between the variational approximation and the true collapsed posterior, our primary quantity of interest, to the log-partition function.

Chatterjee and Dembo's work primarily focused on applications to large deviations of subgraph counts in Erd\"os-R\'enyi random graphs and statistical physics. See \cite{eldan2018gaussian, yan2020nonlinear, augeri2021transportation} and the references therein for recent developments on nonlinear large deviations. Extending this framework, we introduce a novel set of theoretical tools that yield substantially improved error bounds for approximating the log-partition function compared to previous works. A brief summary of the new proof techniques involved is outlined below:
\begin{itemize}
    \item[(i)] In our analysis, we first relate the approximation accuracy of full VI to the approximation accuracy of collapsed VI. We note that this result does not follow directly from existing results in nonlinear large deviations. To obtain this result, we utilize a novel argument that combines bounds on the accuracy of collapsed VI with the Gibbs variational principle. 
    \item[(ii)] To analyze the approximation accuracy of collapsed VI we use local bounds on the Hessian of the  Hamiltonian\footnote{Following terminology from statistical physics, we refer to $\log \mathbb{p}(\pmb{\theta},\mathbf{Z} ,\mathbf{X})$ in~\eqref{ELBO_def} as the `Hamiltonian' of the posterior; similarly, $\log \mathbb{p}(\mathbf{Z}, \mathbf{X})$ is referred to as the Hamiltonian of the collapsed posterior.}
    of the collapsed posterior, instead of `uniform bounds' as used in~\cite{chatterjee2016nonlinear, yan2020nonlinear}. 
    \item[(iii)] We use error terms $\Delta_1(y)$ and $\Delta_2(y)$ (see~\eqref{eq:error_term1} and~\eqref{eq:error_term2}) that differ from those used in \cite{chatterjee2016nonlinear,yan2020nonlinear}; these error terms are easier to control and can be bounded more precisely. At a high level, our error term $\Delta_1(y)$ is obtained via a two-term Taylor's approximation, whereas the error-bound in \cite{chatterjee2016nonlinear, yan2020nonlinear} corresponds to a one-term Taylor's approximation. 
    \item[(iv)] We bound both the first and second moments of the error terms in (iii), with a more precise control on the first moment. In \cite{chatterjee2016nonlinear,yan2020nonlinear}, only second moment bounds are used. To utilize our first and second moment bounds, we derive a new lemma (see Lemma \ref{Lem2.1}) which may be of independent interest.
\end{itemize}
The combination of (ii)-(iv) above results in substantially smaller `smoothness terms' as introduced in Definition \ref{DefSmoothnessTerm} (cf.~\cite[Theorem 1.6]{chatterjee2016nonlinear}). See Section \ref{Sect.2.2} for an overview of the proof strategy and detailed discussions on the technical innovations of our results.

We note that the reduction from full VI to collapsed VI in (i) above is crucial for applications to LDA, MMSB, and many other latent variable models---direct application of existing nonlinear large deviations results to the full posterior fails to capture the validity of MFVI for the full posterior. The improved bounds in (ii)-(iv) are crucial for establishing the validity of collapsed VI with optimal rates of convergence for both LDA and MMSB in a broad parameter regime which captures the entire regime where MFVI works (as presented in Theorems \ref{Theorem_LDA_UBD} and \ref{Theorem_MMSB_UBD}).

For the MMSB model, we show that fully factorized VI does not give an accurate approximation to the true posterior, and propose a partially grouped VI algorithm which yields a precise approximation to the true posterior. To the best of our knowledge, this is the first work to demonstrate that fully factorized VI is provably suboptimal compared to partially grouped VI for the MMSB model.

\subsection{Related work}\label{Sect.1.5}

We summarize previous theoretical results on the validity of MFVI as follows. Wang and Blei \cite{wang2018frequentist} established frequentist consistency and asymptotic normality of MFVI when the parameter dimension is fixed. Notably, their variational posterior approximates the marginal posterior of the global latent variables, with local latent variables integrated out, rather than the full posterior distribution. Zhang and Gao \cite{zhang2020convergence} study convergence rate of variational posterior distributions to the true parameter in nonparametric and high-dimensional inference. Pati, Bhattacharya, and Yang \cite{pati2018statistical} provide general conditions for obtaining optimal risk bounds for point estimates acquired from MFVI. Both \cite{zhang2020convergence} and \cite{pati2018statistical} focus on estimation accuracy in a frequentist Bayesian framework, assuming the existence of a true parameter, but do not assess how well the variational posterior distribution approximates the true posterior. Ray and Szab\'o \cite{ray2022variational} study estimation consistency of MFVI for high-dimensional Bayesian linear regression with sparse priors. Across these works, the assumptions generally involve either a fixed parameter dimension or sparsity constraints (i.e., where only a small fraction of parameters are nonzero), which are designed to ensure that the posteriors ``contract'' around some frequentist true data generating model. Alquier, Ridgway, and Chopin~\cite{Alquier-JMLR-2016} investigate the theoretical properties of variational approximations in the context of Gibbs posteriors\footnote{Unlike standard Bayesian inference, which is based on a likelihood function, Gibbs posteriors arise in generalized Bayesian inference settings where the likelihood is replaced by an empirical risk function.} (also called PAC-Bayes posteriors) and provide finite-sample upper bounds on the prediction risk for the variational approximation. In the context of high-dimensional Bayesian linear regression with i.i.d. regression coefficients, Mukherjee and Sen \cite{mukherjee2022variational} established the leading-order correctness of the mean-field approximation to the log-normalizing constant of the posterior distribution.

Now we discuss previous theoretical work on MFVI that is related to LDA and MMSB. For LDA, \cite{pati2018statistical} obtained results on the approximation accuracy of points estimates derived from MFVI under a sparsity condition. In particular, \cite{pati2018statistical} assumed that the number of topics with a non-zero proportion is much smaller than $K$ and that the number of words with non-zero proportion within each topic is much smaller than $V$. Additionally, \cite{pati2018statistical} worked with frequentist consistency, i.e., assuming data are generated from a true, fixed parameter and studying the concentration of the variational posterior around the true parameter. In \cite{ghorbani2019instability}, an instability of MFVI for topic models was stated. The model discussed in \cite{ghorbani2019instability} is different from the original model in \cite{blei2003latent}: the model in \cite{ghorbani2019instability} assumes that the topic vectors are generated from a Dirichlet\slash Gaussian prior distribution and that the noise distribution is i.i.d. Gaussian. We believe that our results can be extended to this setup, verifying the validity of MFVI if $DK+KV=o(n)$. As \cite{ghorbani2019instability} considers the parameter regime where $K$ is fixed and $V\asymp n$, it is not surprising that MFVI fails in this setting. 

For the stochastic blockmodel---a simpler version of MMSB---earlier theoretical results on MFVI can be found in \cite{bickel2013asymptotic,zhang2020theoretical,gaucher2021optimality}. We note that these works assume a true, fixed parameter and focus on estimation accuracy.

\subsection{Organization}\label{Sect.1.6}

The rest of the paper is organized as follows. In Section \ref{Sect.2}, we state the meta-theorems (Theorems \ref{Theorem2.1} and \ref{Theorem2.2}) for the validity of MFVI for general latent variable models with categorical local latent variables. In Section \ref{Sect.3}, we apply the meta-theorems to LDA and give a proof of Theorem \ref{Theorem_LDA_UBD}. We conclude with a brief discussion in Section~\ref{Sect.5}, highlighting several open directions for future research. Appendix~\ref{Sect.4} presents numerical studies that demonstrate the superior performance of partially grouped VI compared to fully factorized VI for MMSB. The proofs of Theorems \ref{Theorem_LDA_LBD}-\ref{Theorem_MMSB_S} as well as Theorems \ref{Theorem2.1}-\ref{Theorem2.2} are presented in Appendices~\ref{Appendix_A}-\ref{Appendix_D}. Appendix~\ref{Appendix_E} gives some auxiliary lemmas used in the proofs of the main results. In Appendix~\ref{Appendix_F} we present the implementation details of the CAVI algorithm for partially grouped VI in MMSB.

\section{Main results for general latent variable models}\label{Sect.2}

In this section, we establish general results for verifying the validity of MFVI for general latent variable models with categorical local latent variables. Throughout the remainder of the paper, for any finite set $S$, we denote by $|S|$ its cardinality. By convention, we define $0\log{0}:=0,\log{0}:=-\infty$, and $\exp(-\infty):=0$.

\subsection{Main results}\label{Sect.2.1}

Recall our basic setup from Section \ref{Sect.1.1}. We introduce some notation below. We fix a product probability measure $\gamma$ on $\prod_{j=1}^m\mathbb{R}^{d_j}\times \prod_{i=1}^n [K_i]$, the space where the posterior distribution $\mathbb{P}(\pmb{\theta},\mathbf{Z}\mid\mathbf{X})$ resides, and denote it by
\begin{equation}\label{base}
    \gamma=\nu_{1}\otimes \cdots\otimes \nu_{m}\otimes \mu_{1}\otimes\cdots\otimes\mu_{n},
\end{equation}
where $\nu_j$ is a probability distribution on $\mathbb{R}^{d_j}$ for each $j\in [m]$, and $\mu_i$ is a probability distribution on $[K_i]$ for each $i\in[n]$. We also denote
\begin{equation}\label{defmu}
    \nu:=\nu_{1}\otimes \cdots\otimes \nu_{m},\qquad \mbox{and} \qquad \mu:=\mu_{1}\otimes\cdots\otimes\mu_{n}.
\end{equation}
We assume that the conditional distribution of $\pmb{\theta} = (\pmb{\theta}_1, \pmb{\theta}_2,\cdots, \pmb{\theta}_m)$ given $\mathbf{Z}$ and $\mathbf{X}$ factorizes, i.e., 
\begin{equation*}
\frac{d\, \mathbb{P}(\pmb{\theta}\mid\mathbf{Z},\mathbf{X})}{d \nu } \propto \exp\bigg(\sum_{j=1}^m r_j(\pmb{\theta}_j,\mathbf{Z})\bigg), 
\end{equation*}
for some functions $r_1,\ldots, r_m$ (which can depend on $\mathbf{X}$; but since $\mathbf{X}$ is fixed in the posterior distribution, we suppress this dependence for notational simplicity). In particular, this is equivalent to making the following assumption on the posterior density.
\begin{Assumption}\label{Assump:Cond-ind}
    Suppose that the posterior distribution $\mathbb{P}(\pmb{\theta},\mathbf{Z}\mid\mathbf{X})$ has a density, w.r.t.~the base measure $\gamma$, given by
\begin{equation}\label{eq:Post-Form}
\frac{d\, \mathbb{P}(\pmb{\theta},\mathbf{Z}\mid\mathbf{X})}{d \gamma } \propto \exp(r(\pmb{\theta},\mathbf{Z})), \qquad \quad \mbox{where} \quad r(\pmb{\theta},\mathbf{Z}) :=r_0(\mathbf{Z})+\sum_{j=1}^m r_j(\pmb{\theta}_j,\mathbf{Z}),
\end{equation}
for some functions $r_0, r_1,\ldots, r_m$ (here, as before, we hide the dependence of $r(\pmb{\theta},\mathbf{Z})$ on $\mathbf{X}$). 
\end{Assumption}
The above assumption is satisfied by a wide range of Bayesian latent variable models, including LDA and MMSB. Let
\begin{equation}\label{def_f}
    f(\mathbf{Z}) := \log\bigg(\int_{\prod_{j=1}^m\mathbb{R}^{d_j}}\exp(r(\pmb{\theta},\mathbf{Z}))d\nu(\pmb{\theta})\bigg),
\end{equation}
and note that the collapsed posterior $\mathbb{P}(\mathbf{Z}\mid\mathbf{X})$ satisfies
\begin{equation}\label{def_pro}
    \frac{d \mathbb{P}(\mathbf{Z}\mid\mathbf{X})}{d \mu} \propto\exp(f(\mathbf{Z})).
\end{equation} 
Hereafter, we refer to $r(\cdot,\cdot)$ and $f(\cdot)$ as the \emph{Hamiltonians} of the full posterior and the collapsed posterior, respectively.

We denote $\mathbf{K}:=(K_1,\cdots,K_n)$. 
For any possible value of the local latent variables $z=(z_i)_{i\in [n]}\in\prod_{i=1}^n [K_i]$, we let $G(z) \equiv (G_{i,\ell}(z))_{i\in [n],\ell\in [K_i]}$ be the one-hot encoding of $z$ defined as 
\begin{equation}\label{eq:G}
G_{i,\ell}(z) :=\mathbbm{1}_{z_i=\ell} \qquad \mbox{ for all } \;\; i\in[n],\ell\in [K_i].
\end{equation}
Let $\mathcal{Y}_{n,\mathbf{K}}$ be the set of one-hot encodings of elements in $\prod_{i=1}^n[K_i]$ obtained this way, i.e.,
\begin{equation}
  \mathcal{Y}_{n,\mathbf{K}}:=\bigg\{y=(y_{i,\ell})_{i\in[n],\ell\in[K_i]}\in \prod_{i=1}^n\{0,1\}^{K_i}:\sum_{\ell=1}^{K_i} y_{i,\ell}=1 \text{ for every }i\in [n]\bigg\},
\end{equation}
and let $\hat{\mathcal{Y}}_{n,\mathbf{K}}$ be the continuous relaxation of $\mathcal{Y}_{n,\mathbf{K}}$, i.e., 
\begin{equation}\label{eq:Y_nK}
  \hat{\mathcal{Y}}_{n,\mathbf{K}}:=\bigg\{y=(y_{i,\ell})_{i\in[n],\ell\in[K_i]}\in \prod_{i=1}^n [0,1]^{K_i}: \sum_{\ell=1}^{K_i} y_{i,\ell}=1\text{ for every }i\in [n]\bigg\}.
\end{equation}
For any $y  =(y_{i,\ell})_{i\in[n], \ell\in[K_i]}\in \hat{\mathcal{Y}}_{n,\mathbf{K}}$, let $Q_y$ denote the corresponding product probability distribution on $\prod_{i=1}^n[K_i]$, given by, for $z=(z_i)_{i\in [n]} \in \prod_{i=1}^n[K_i]$,
\begin{equation}\label{defQy}
    Q_y(z)=\prod_{i=1}^ny_{i,z_i}.
\end{equation}

For any $y=(y_{i,\ell})_{i\in [n],\ell\in [K_i]}\in \prod_{i=1}^n[0,1]^{K_i}$, $i_0\in[n]$, and $\ell_0\in [K_{i_0}]$, we let $\UII{y}{i_0}{\ell_0} \equiv \big(\UIIsubscr{i}{\ell}{y}{i_0}{\ell_0}\big)_{i\in [n],\ell\in [K_i]}\in \prod_{i=1}^n[0,1]^{K_i}$ be such that 
\begin{equation}\label{def_U}
    \UIIsubscr{i}{\ell}{y}{i_0}{\ell_0}=\begin{cases}
       y_{i,\ell} & \text{ if }i\neq i_0\\
       \mathbbm{1}_{\ell=\ell_0} &  \text{ if }i=i_0
    \end{cases}
\end{equation}
for all $i\in[n]$ and $\ell\in [K_i]$. Thus, for any $y\in\prod_{i=1}^n[0,1]^{K_i}$, $\UII{y}{i_0}{\ell_0}$ just changes the $i_0$th row of $y$, i.e.,  $(y_{i_0,1},\cdots,$ $y_{i_0,K_{i_0}})$, to the $\ell_0$th canonical vector in $\mathbb{R}^{K_{i_0}}$. For any $y=(y_{i,\ell})_{i\in [n],\ell\in [K_i]}\in\prod_{i=1}^n[0,1]^{K_i}$ and $i_0\in [n]$, we let $\UI{y}{i_0} \equiv \big(\UIsubscr{i}{\ell}{y}{i_0}\big)_{i\in [n],\ell\in [K_i]}\in \prod_{i=1}^n [0,1]^{K_i}$ be such that 
\begin{equation}\label{def_Utilde}
    \UIsubscr{i}{\ell}{y}{i_0}=\begin{cases}
       y_{i,\ell} & \text{ if }i\neq i_0\\
       0 &  \text{ if }i=i_0
    \end{cases}
\end{equation}
for all $i\in[n]$ and $\ell\in [K_i]$. Note that for $y\in\prod_{i=1}^n[0,1]^{K_i}$, $\UI{y}{i_0}$ changes the $i_0$th row of $y$ (which is $(y_{i_0,1},\cdots,y_{i_0,K_{i_0}})$) to the zero vector and leaves the other components unchanged.

We assume the existence of a `smooth' function $R(\pmb{\theta},y)$ which extends the definition of $r(\pmb{\theta},z)$ (see~\eqref{eq:Post-Form}) for $z\in\prod_{i=1}^n[K_i]$ (a discrete vector) to $y\in \prod_{i=1}^n[0,1]^{K_i}$. The following makes this precise.
\begin{Assumption}\label{Assump:R}
We assume that there exists a real-valued function $R(\cdot,\cdot)$ on $\prod_{j=1}^m\mathbb{R}^{d_j}\times \prod_{i=1}^n[0,1]^{K_i}$ with the form (cf.~\eqref{eq:Post-Form})
\begin{equation}\label{def_form.R}
    R(\pmb{\theta},y) = R_0(y)+\sum_{j=1}^m R_j(\pmb{\theta}_j,y),\qquad \mbox{ for all } \;\;(\pmb{\theta},y)\in \prod_{j=1}^m\mathbb{R}^{d_j}\times \prod_{i=1}^n[0,1]^{K_i}, 
\end{equation}
such that for any $\pmb{\theta}\in\prod_{j=1}^m\mathbb{R}^{d_j}$, $R(\pmb{\theta},\cdot)$ is twice continuously differentiable in $\prod_{i=1}^n(0,1)^{K_i}$, $R(\pmb{\theta},\cdot)$ and all its first and second order derivatives extend continuously to the boundary, and
\begin{equation}\label{def_R}
    R(\pmb{\theta},G(z))=r(\pmb{\theta}, z), \qquad \mbox{ for all } \;\; z\in \prod_{i=1}^n [K_i].
\end{equation}
\end{Assumption}

In a similar manner, we assume that there exists a `smooth' function $F(\cdot)$ on $\prod_{i=1}^n[0,1]^{K_i}$ which extends the definition of $f(\cdot)$.
\begin{Assumption}\label{Assump:F}
We assume that $F(\cdot)$ is twice continuously differentiable in $\prod_{i=1}^n (0,1)^{K_i}$ and that $F(\cdot)$ and all its first and second order derivatives extend continuously to the boundary of its domain. We further assume that \begin{equation}\label{eq:F}
F(y)=\log\bigg(\int_{\prod_{j=1}^m\mathbb{R}^{d_j}}\exp(R(\pmb{\theta},y))d\nu(\pmb{\theta})\bigg),\qquad \mbox{ for all } \;\; y\in\hat{\mathcal{Y}}_{n,\mathbf{K}}.
\end{equation}
\end{Assumption}

By \eqref{def_f} and \eqref{def_R}, the above display implies that for any $z\in\prod_{i=1}^n [K_i]$,
\begin{equation}\label{eq:FG}
    F(G(z))=f(z).
\end{equation}
Thus, $F(y)$ is essentially an extension of $f(z)$ defined for any $y\in\prod_{i=1}^n[0,1]^{K_i}$ instead of the discrete vector $z\in\prod_{i=1}^n[K_i]$. 

For any $y\in\prod_{i=1}^n [0,1]^{K_i}$ and $i,j\in [n],\ell\in[K_i],\ell'\in [K_j]$, we define 
\begin{equation}
    F_{i,\ell}(y):=\frac{\partial F(y)}{\partial y_{i,\ell}}, \qquad \qquad F_{i,\ell;j,\ell'}(y):=\frac{\partial^2 F(y)}{\partial y_{i,\ell}\partial y_{j,\ell'}}.
\end{equation}
For any $\pmb{\theta}\in\prod_{j=1}^m\mathbb{R}^{d_j}$, $y\in\prod_{i=1}^n[0,1]^{K_i}$, and $i,j\in [n],\ell\in[K_i],\ell'\in[K_j]$, we define
\begin{equation}\label{def_Rde}
    R_{i,\ell}(\pmb{\theta},y):=\frac{\partial R(\pmb{\theta},y)}{\partial y_{i,\ell}}, \qquad \qquad R_{i,\ell;j,\ell'}(\pmb{\theta},y):=\frac{\partial^2 R(\pmb{\theta},y)}{\partial y_{i,\ell}\partial y_{j,\ell'}}.
\end{equation}
For any $y\in \hat{\mathcal{Y}}_{n,\mathbf{K}}$, we define
\begin{equation}\label{def_I}
I(y):=\sum_{i=1}^n\sum_{\ell=1}^{K_i}y_{i,\ell}\log\bigg(\frac{y_{i,\ell}}{\mu_i(\ell)}\bigg),
\end{equation}
where we recall the definition of $\mu_i$ from~\eqref{defmu}; here $y_{i,\ell}\log\Big(\frac{y_{i,\ell}}{\mu_i(\ell)}\Big):=0$ if $y_{i,\ell}=0$. Note that $I(y)$ is the KL divergence between the probability distribution $Q_y$ (defined in~\eqref{defQy}) and the base measure $\mu$ in~\eqref{defmu}.

We now introduce the main terms that will appear in the error bounds in our results. Definitions \ref{Global_GH}-\ref{Local_Hess_new} below pertain to the gradients and Hessians of the functions $F(\cdot)$ and $R(\cdot,\cdot)$.

\begin{definition}[Uniform gradient and Hessian bounds of $F$ and $R$]\label{Global_GH}
For any $i\in [n]$ and $\ell\in [K_i]$, we define
\begin{equation}\label{def_b}
    b_{i,\ell}:=\sup_{y\in\prod_{i=1}^n[0,1]^{K_i}} |F_{i,\ell}(y)|.
\end{equation}
For any $i,j\in [n]$ and $\ell\in[K_i],\ell'\in[K_j]$, we define
\begin{equation*}
    c_{i,\ell;j,\ell'}:=\sup_{y\in\prod_{i=1}^n[0,1]^{K_i}} |F_{i,\ell;j,\ell'}(y)|, \qquad \qquad  \ctilde_{i,\ell;j,\ell'}:=\sup_{\substack{\pmb{\theta}\in\,\mathrm{supp}(\nu),\\ y\in\prod_{i=1}^n[0,1]^{K_i}}} |R_{i,\ell;j,\ell'}(\pmb{\theta},y)|,
\end{equation*}
where $\mathrm{supp}(\nu)$ is the support of $\nu$.
\end{definition}

\begin{definition}[Local Hessian bound I]\label{Local_Hess}
 For any $y=(y_{i,\ell})_{i\in[n],\ell\in[K_i]}\in\prod_{i=1}^n[0,1]^{K_i}$, we define $\NI{y}$ to be the set of $y'=(y'_{i,\ell})_{i\in[n],\ell\in[K_i]}\in\prod_{i=1}^n[0,1]^{K_i}$ such that for some $i_0,i_0'\in [n]$, we have $y'_{i,\ell}=y_{i,\ell}$ for all $i\in[n]\backslash\{i_0,i_0'\}$ and $\ell\in [K_i]$. For any $i,j\in [n]$ and $\ell\in[K_i],\ell'\in [K_j]$, we define
\begin{equation*}
     \HI{i}{\ell}{j}{\ell'}{y}:=\sup_{y'\in\NI{y}} |F_{i,\ell;j,\ell'}(y')|. 
\end{equation*}
\end{definition}
\begin{definition}[Local Hessian bound II]\label{Local_Hess_new}
For any $y=(y_{i,\ell})_{i\in[n],\ell\in[K_i]}\in\prod_{i=1}^n[0,1]^{K_i}$ and $M\geq 1$, we define $\NII{y}{M}$ to be the set of $y'=(y'_{i,\ell})_{i\in[n],\ell\in[K_i]}\in\prod_{i=1}^n[0,1]^{K_i}$ such that for some $i_0\in [n]$, we have $M^{-1}y_{i,\ell}\leq y'_{i,\ell} \leq M y_{i,\ell}$ for all $i\in[n]\backslash\{i_0\}$ and $\ell\in[K_i]$. For any $i,j\in [n]$ and $\ell\in[K_i],\ell'\in [K_j]$, we define
\begin{equation*}
    \HII{i}{\ell}{j}{\ell'}{y}{M}:=\sup_{y'\in\NII{y}{M}} |F_{i,\ell;j,\ell'}(y')|. 
\end{equation*}
\end{definition} 
\begin{remark}
In Definitions \ref{Local_Hess} and \ref{Local_Hess_new}, $\NI{y}$ and $\NII{y}{M}$ can be interpreted as `neighborhoods' of $y$. $\NI{y}$ denotes the neighborhood of $y$ consisting of all $y' \in \prod_{i=1}^n[0,1]^{K_i}$ that are different from $y$ in at most two `rows', whereas $\NII{y}{M}$ denotes the neighborhood of $y$ consisting of all $y'$ where the ratio of each entry of $y'$ to the corresponding entry in $y$ is uniformly (lower and upper) bounded (by $M^{-1}$ and $M$), except for one `row' of $y$. Note that $\HI{i}{\ell}{j}{\ell'}{y}$ and $\HII{i}{\ell}{j}{\ell'}{y}{M}$ provide upper bounds on the absolute values of the second order partial derivatives of $F$ in these neighborhoods.
\end{remark}

Definitions \ref{DefSmoothnessTerm} and \ref{DefComplexityTerm} below introduce the error terms (the smoothness term $E_1$ and the complexity term $E_2(\epsilon)$) that appear in Theorems \ref{Theorem2.1} and \ref{Theorem2.2}. 

\begin{definition}[Smoothness term]\label{DefSmoothnessTerm}
For any $j\in [n]$, we define
\begin{equation}\label{defM}
    M_j:=\exp\Big(2\max_{i\in[n]}\max_{\substack{\ell\in[K_i],\ell'\in [K_j]}}\{c_{i,\ell;j,\ell'}\}\Big).
\end{equation}
For any $y\in\prod_{i=1}^n[0,1]^{K_i}$ and $i,j\in [n]$, we define
\begin{equation*}
    \Phi_{i,j}(y):=\exp\bigg(2\sum_{\ell'=1}^{K_j}\max_{\ell\in [K_i]}\big\{\HI{i}{\ell}{j}{\ell'}{y}\big\}y_{j,\ell'}\bigg)-1.  
\end{equation*}
Now we define the `smoothness term' as
\begin{eqnarray}\label{E2.1}
    E_1&:=& \sup_{y\in\hat{\mathcal{Y}}_{n,\mathbf{K}}}\bigg\{\sum_{i=1}^n\max_{j\in[n]}\big\{\Phi_{j,i}(y)\big\}\bigg\} \cdot \sup_{\substack{y\in\hat{\mathcal{Y}}_{n,\mathbf{K}},\\i\in[n],\ell\in[K_i]}}\bigg\{\sum_{j=1}^n\sum_{\ell'=1}^{K_j} \HII{i}{\ell}{j}{\ell'}{y}{M_i}y_{j,\ell'}\bigg\}\nonumber\\
    &&\hspace{0.5in} +\sup_{y\in\hat{\mathcal{Y}}_{n,\mathbf{K}}}\bigg\{\sum_{i=1}^n\max_{\ell\in[K_i]}\bigg\{\sum_{\ell'=1}^{K_i} \HII{i}{\ell}{i}{\ell'}{y}{M_i} y_{i,\ell'}\bigg\}\bigg\}.
\end{eqnarray}
For any $t>0$, we define (recall~\eqref{def_b})
\begin{equation}\label{def_L}
    \Lambda(t):=\log\bigg(2+\frac{\sum_{i=1}^n\max_{\ell\in[K_i]}\{b_{i,\ell}\}}{t}\bigg).
\end{equation}
\end{definition}
\begin{remark}\label{Remark2.3}
Note that for any $y\in\hat{\mathcal{Y}}_{n,\mathbf{K}}$ and $i,j\in [n]$, as $\sum_{\ell'=1}^{K_j} y_{j,\ell'} = 1$,
{\small
\begin{equation}\label{E3.5}
    \;\;\; \sum_{\ell'=1}^{K_j}\max_{\ell\in [K_i]}\big\{\HI{i}{\ell}{j}{\ell'}{y}\big\}y_{j,\ell'}\leq \max_{\ell\in[K_i],\ell'\in [K_j]}\{c_{i,\ell;j,\ell'}\}.
\end{equation}
}Hence, using the definition of $\Phi_{i,j}$ and~\eqref{defM}, we get
\begin{equation}\label{Bounds.Phi}
    \Phi_{i,j}(y)\leq M_{j}-1. 
\end{equation}
\end{remark}

\begin{definition}[Complexity term]\label{DefComplexityTerm}
For any $\epsilon>0$, let $\mathcal{D}(\epsilon)\subseteq\prod_{i=1}^n\mathbb{R}^{K_i}$ be a finite set such that for any $y\in\hat{\mathcal{Y}}_{n,\mathbf{K}}$, there exists $d=(d_{i,\ell})_{i\in[n],\ell\in[K_i]}\in\mathcal{D}(\epsilon)$ that satisfies 
\begin{equation}\label{Eq2.17}
   \sum_{i=1}^n \max_{\ell\in[K_i]}\big|F\big(\UII{y}{i}{\ell}\big)-F\big(\UI{y}{i}\big)-d_{i,\ell}\big|\leq \epsilon.
\end{equation}
We define the `complexity term' 
\begin{equation}\label{def_E}
    E_2(\epsilon):=2\epsilon+\log(|\mathcal{D}(\epsilon)|).
\end{equation}
\end{definition}

We denote by
\begin{equation}\label{defS}
S_{n,\mathbf{K}}:=\sum_{z\,\in\, \prod_{i=1}^n [K_i]} e^{f(z)}\mu(z)
\end{equation}
the partition function\slash normalizing constant of the collapsed posterior $\mathbb{P}(\mathbf{Z}\mid\mathbf{X})$. The following simple lemma connects the log-partition function $\log S_{n,\mathbf{K}}$ with the KL divergence between $\hat{Q}$ and $\mathbb{P}(\mathbf{Z}\mid\mathbf{X})$; see Appendix~\ref{Proof:Lemma2.1} for a proof (cf.~\cite[Section 5.2]{wainwright2008graphical}).

\begin{lemma}\label{Lemma2.1}
Let $Q_y$ denote the product probability distribution determined by $y\in\hat{\mathcal{Y}}_{n,\mathbf{K}}$ as in~\eqref{defQy}. Then we have 
\begin{itemize}
    \item[(a)] $\KL\big(\hat{Q}\,\big\| \; \mathbb{P}(\mathbf{Z}\mid\mathbf{X})\big)=\log S_{n,\mathbf{K}} -\sup_{y\in\hat{\mathcal{Y}}_{n,\mathbf{K}}}\{\mathbb{E}_{Q_y}[f(\mathbf{Z})]-I(y)\}$;
    \item[(b)] if Assumption~\ref{Assump:F} holds, then $|\mathbb{E}_{Q_y}[f(\mathbf{Z})]-F(y)|\leq 2\sum_{i=1}^n\max_{\ell,\ell'\in[K_i]} \{c_{i,\ell;i,\ell'}\}$ for every $y\in\hat{\mathcal{Y}}_{n,\mathbf{K}}$; 
 \item[(c)] if \eqref{eq:FG} holds for all $z\in\prod_{i=1}^n[K_i]$ and $F(\cdot)$ is a convex function on $\prod_{i=1}^n[0,1]^{K_i}$, then $\mathbb{E}_{Q_y}[f(\mathbf{Z})]\geq F(y)$ for every $y\in\hat{\mathcal{Y}}_{n,\mathbf{K}}$.
\end{itemize}
\end{lemma}

The main results of this section are presented in Theorems \ref{Theorem2.1} and \ref{Theorem2.2} below. Theorem \ref{Theorem2.1}, proved in Appendix~\ref{App:Proof-Upp-Bd}, provides an upper bound on the KL divergence between the collapsed variational posterior $\hat{Q}$ (defined in (\ref{vari3.2})) and the collapsed posterior $\mathbb{P}(\mathbf{Z}\mid\mathbf{X})$. 

\begin{theorem}\label{Theorem2.1}
Suppose Assumption~\ref{Assump:F} holds. There is an absolute positive constant $C$, such that for any $\epsilon,t>0$ such that $t\geq E_1$, the following holds:  
\begin{itemize}
    \item[(a)] we have the following upper and lower bounds on the log-partition function:
\begin{eqnarray*}
    \log S_{n,\mathbf{K}} &\leq & \sup_{y\in\hat{\mathcal{Y}}_{n,\mathbf{K}}}\{F(y)-I(y)\}+C(t + \Lambda(t)) +E_2(\epsilon), \quad \mbox{and} \\
    \log S_{n,\mathbf{K}} & \geq & \sup_{y\in\hat{\mathcal{Y}}_{n,\mathbf{K}}}\{F(y)-I(y)\}-2\sum_{i=1}^n\max_{\ell,\ell'\in[K_i]} \{c_{i,\ell;i,\ell'}\};
\end{eqnarray*}
\item[(b)] moreover, we can upper bound the error of MFVI by 
\begin{eqnarray*}
   \frac{1}{n}\KL\big(\hat{Q} \,\big\| \; \mathbb{P}(\mathbf{Z}\mid\mathbf{X}) \big)\nonumber &\leq& 
    \frac{C (t+ \Lambda(t))+E_2(\epsilon)+2\sum_{i=1}^n\max_{\ell,\ell'\in[K_i]} \{c_{i,\ell;i,\ell'}\}}{n};
\end{eqnarray*}
\item[(c)] finally, if $F(\cdot)$ is a convex function on $\prod_{i=1}^n[0,1]^{K_i}$, then 
\begin{eqnarray*}
    \log S_{n,\mathbf{K}} & \geq & \sup_{y\in\hat{\mathcal{Y}}_{n,\mathbf{K}}}\{F(y)-I(y)\}, \qquad \mbox{and}\\
    \frac{1}{n}\KL\big(\hat{Q} \,\big\| \; \mathbb{P}(\mathbf{Z}\mid\mathbf{X}) \big) & \leq &  \frac{C(t+\Lambda(t))+E_2(\epsilon)}{n}.
\end{eqnarray*}
\end{itemize}
\end{theorem}
\begin{remark}[The term $\Lambda(t)$]
In relevant applications, including both LDA and MMSB, $\Lambda(t)$ is typically of logarithmic order and is of smaller magnitude than $t$ (see~\eqref{Eq.A33} in Section \ref{Sect.3} and~\eqref{Step2_result} in Appendix \ref{Appendix_B}). Thus, by taking $t\approx E_1$, we expect $t+\Lambda(t) \approx t \approx E_1$. 
\end{remark}
Theorem \ref{Theorem2.1} gives results on the approximation of the log-partition function $\log S_{n,\mathbf{K}}$ and also on the (normalized) KL divergence $\frac{1}{n}\KL (\hat{Q} \, \| \; \mathbb{P}(\mathbf{Z}\mid\mathbf{X}) )$. The bound on the KL divergence follows immediately from the bound on $\log S_{n,\mathbf{K}}$, using Lemma~\ref{Lemma2.1}. Theorem \ref{Theorem2.2} below, proved in Appendix~\ref{App:Proof-Upp-Bd-Full}, provides an upper bound on the KL divergence between the variational posterior $\hat{P}$ (as defined in (\ref{vari3.1})) and the full posterior distribution $\mathbb{P}(\pmb{\theta},\mathbf{Z}\mid\mathbf{X})$. 

\begin{theorem}\label{Theorem2.2}
Suppose Assumptions~\ref{Assump:Cond-ind}-\ref{Assump:F} hold. Then for any $\epsilon,t>0$ such that $t\geq E_1$, we have 
\begin{align*}
 & \frac{1}{n} \KL\big(\hat{P} \,\big\| \; \mathbb{P}(\pmb{\theta},\mathbf{Z}\mid\mathbf{X}) \big)\nonumber\\
  \leq& \;\frac{\log S_{n,\mathbf{K}}-\sup_{y\in\hat{\mathcal{Y}}_{n,\mathbf{K}}}\{F(y)-I(y)\}}{n}+\frac{2\sum_{i=1}^n\max_{\ell,\ell'\in[K_i]}\{\ctilde_{i,\ell;i,\ell'}\}}{n}\nonumber\\
  \leq& \;\frac{C(t+\Lambda(t))+E_2(\epsilon)+2\sum_{i=1}^n\max_{\ell,\ell'\in [K_i]}\{\ctilde_{i,\ell;i,\ell'}\}}{n},
\end{align*}
where $C$ is an absolute positive constant.
\end{theorem}

Theorems \ref{Theorem2.1} and \ref{Theorem2.2} significantly improve upon previous results on nonlinear large deviations (cf.~\cite{chatterjee2016nonlinear,yan2020nonlinear}). The error bounds in these results yield sharp rates of convergence for LDA and MMSB. We provide a high-level overview of the proof strategy for Theorem \ref{Theorem2.1} in Section \ref{Sect.2.2} below. 

\subsection{Overview of the proof strategy and technical highlights}\label{Sect.2.2}

The proof of Theorem \ref{Theorem2.1} involves several novel ingredients. Let $\mathbf{Y}=(Y_{i,\ell})_{i\in[n],\ell\in [K_i]}:=G(\mathbf{Z})$ (recall~\eqref{eq:G}) be the one-hot encoding of $\mathbf{Z}$ where $\mathbf{Z}\sim \mathbb{P}(\mathbf{Z}\mid\mathbf{X})$. 
\begin{definition}[Conditional probability of one coordinate given the others]\label{DefTy}
For any $y=(y_{i,\ell})_{i\in [n],\ell\in [K_i]} $ $ \in\prod_{i=1}^n[0,1]^{K_i}$, let $T(y) \equiv (T_{i,\ell}(y))_{i\in [n], \ell\in [K_i]}\in \hat{\mathcal{Y}}_{n,\mathbf{K}}$ be defined as
\begin{equation}\label{T.exp}
    T_{i,\ell}(y) :=\frac{e^{F(\UII{y}{i}{\ell})}\mu_i(\ell)}{\sum_{s=1}^{K_i} e^{F(\UII{y}{i}{s})}\mu_i(s)},\qquad \text{for all }i\in[n],\ell\in[K_i],
\end{equation}
where $U^{(1)}(y;i,\ell)$ is as in \eqref{def_U}.
\end{definition}
Note that for any $i\in [n]$ and $\ell\in [K_i]$ we have
$
F\big(\UII{\mathbf{Y}}{i}{\ell}\big)=F(G(Z_1,\cdots,Z_{i-1},\ell,Z_{i+1},$ $\cdots, Z_n))=f(Z_1,\cdots,Z_{i-1},\ell,Z_{i+1},\cdots,Z_n),$
and consequently, $$T_{i,\ell}(\mathbf{Y})=\mathbb{P}\big(Z_i=\ell\mid(Z_j)_{j\in[n]\backslash\{i\}}, \mathbf{X}\big)=\mathbb{E}\big[Y_{i,\ell}\mid (Z_j)_{j\in[n]\backslash\{i\}}, \mathbf{X}\big].$$

\begin{definition}[Error terms]\label{Def.A3}
For any $y,y'\in\prod_{i=1}^n[0,1]^{K_i}$, let
\begin{equation}\label{def_J}
J(y,y'):=\sum_{i=1}^n\sum_{\ell=1}^{K_i} y_{i,\ell}\log\bigg(\frac{y'_{i,\ell}}{\mu_i(\ell)}\bigg).
\end{equation}
In particular, note that $J(y,y)=I(y)$ for any $y\in \hat{\mathcal{Y}}_{n,\mathbf{K}}$. For any $y\in \hat{\mathcal{Y}}_{n,\mathbf{K}}$, we define the error terms as (recall $F$ and $I$ from~\eqref{eq:F} and~\eqref{def_I})
\begin{eqnarray}
    \Delta_{1}(y) & := & F(y)-F(T(y))-\sum_{i=1}^n\sum_{\ell=1}^{K_i} F_{i,\ell}(y)(y_{i,\ell}-T_{i,\ell}(y)), \label{eq:error_term1}\\
    \Delta_{2}(y) & := & I(T(y))-J(y,T(y))+\sum_{i=1}^n\sum_{\ell=1}^{K_i} F_{i,\ell}(y)(y_{i,\ell}-T_{i,\ell}(y)).\label{eq:error_term2}
\end{eqnarray}
\end{definition}

We note that the error terms $\Delta_1(y)$ and $\Delta_2(y)$ differ from those used in \cite{chatterjee2016nonlinear,yan2020nonlinear}. In our notation, the error terms used in these works are $F(y)-F(T(y))$ and $I(T(y))-J(y,T(y))$. By keeping track of the extra term $ \sum_{i=1}^n\sum_{\ell=1}^{K_i}F_{i,\ell}(y)(y_{i,\ell}-T_{i,\ell}(y))$ we are able to obtain much more precise control over $\Delta_1(\mathbf{Y})$ and $\Delta_2(\mathbf{Y})$ (see the discussion below). In our analysis, we first show that $\Delta_1(\mathbf{Y}) \approx 0$ and $\Delta_2(\mathbf{Y}) \approx 0$. The precise bounds, involving both first and second moments of $\Delta_1(\mathbf{Y})$ and $\Delta_2(\mathbf{Y})$, are presented in the following two propositions.
\begin{proposition}\label{P2.1main}
Suppose Assumption~\ref{Assump:F} holds. Then, with $b_{i,\ell}$ and $E_1$ as in Definitions~\ref{Global_GH} and~\ref{DefSmoothnessTerm}, respectively, we have \vspace{-0.1in} 
\begin{eqnarray*}
 |\mathbb{E}[\Delta_1(\mathbf{Y})\mid\mathbf{X}]| \leq 4E_1, \qquad  & \mbox{and} &\qquad
\mathbb{E}\big[\Delta_1(\mathbf{Y})^2\mid\mathbf{X}\big] \leq  16\bigg(\sum_{i=1}^n\max_{\ell\in[K_i]}\{b_{i,\ell}\}\bigg)^2. 
\end{eqnarray*}
\end{proposition}

\begin{proposition}\label{P2.2main}
Suppose Assumption~\ref{Assump:F} holds. Then we have \vspace{-0.05in}
\begin{eqnarray*}
    |\mathbb{E}[\Delta_2(\mathbf{Y}) \mid \mathbf{X}]| \leq  2E_1, \qquad 
 & \mbox{and} & \qquad
    \mathbb{E}\big[\Delta_2(\mathbf{Y})^2 \mid \mathbf{X}\big] \leq 16\bigg(\sum_{i=1}^n\max_{\ell\in[K_i]}\{b_{i,\ell}\}\bigg)^2. \vspace{-0.1in}
\end{eqnarray*}
\end{proposition}

To obtain Propositions \ref{P2.1main} and \ref{P2.2main}, we carefully bound $\Delta_1(\mathbf{Y})$ and $\Delta_2(\mathbf{Y})$ in terms of \textit{local bounds} on the Hessian of the Hamiltonian (see Definitions \ref{Local_Hess} and \ref{Local_Hess_new}). To achieve this, a lemma (Lemma \ref{Lem2.2} in Appendix~\ref{Appendix_A}) that carefully bounds the change of $T_{i,\ell}(y)$ (see Definition \ref{DefTy}) under local perturbations of $y\in \prod_{i=1}^n[0,1]^{K_i}$ is used. The first and second moment bounds in Propositions \ref{P2.1main} and \ref{P2.2main} are then used in conjunction with Lemma \ref{Lem2.1} (in Appendix~\ref{Appendix_A}) to establish a lower bound on the probability that $\Delta_1(\mathbf{Y})+\Delta_2(\mathbf{Y})$ is small; in previous works \cite{chatterjee2016nonlinear,yan2020nonlinear}, only second moment bounds are used. These technical advancements are crucial for obtaining optimal rates of convergence of variational posteriors and characterizing the exact parameter regimes for the validity of MFVI for both LDA and MMSB.

After obtaining the error bounds on $\Delta_1(\mathbf{Y})$ and $\Delta_2(\mathbf{Y})$, we bound $\log S_{n,\mathbf{K}} $ as follows:
\begin{eqnarray}\label{Eq2.1n}
    \log S_{n,\mathbf{K}} &=&\log\bigg(\sum_{z\in\prod_{i=1}^n [K_i]}e^{F(G(z))}\mu(z)\bigg) \approx \log\bigg(\sum_{z\in\mathcal{A}}e^{F(G(z))}\mu(z)\bigg)\nonumber\\
    & \approx& \log\bigg(\sum_{z\in\mathcal{A}}e^{F(T(G(z)))-I(T(G(z)))+J(G(z),T(G(z)))}\mu(z)\bigg)\nonumber\\
    &\leq& \sup_{y\in\hat{\mathcal{Y}}_{n,\mathbf{K}}}\{F(y)-I(y)\}+\log\bigg(\sum_{z\in\prod_{i=1}^n [K_i]}e^{J(G(z),T(G(z)))}\mu(z)\bigg),
\end{eqnarray}
where $\mathcal{A}$ is a `suitable' subset of $\prod_{i=1}^n [K_i]$ (defined formally in Appendix \ref{Appendix_A}) on which the above  approximations hold. This follows from observing that  
\begin{equation*}
    F(G(z))-F(T(G(z)))+I(T(G(z)))-J(G(z),T(G(z))) \equiv \Delta_1(G(z))+\Delta_2(G(z)) \stackrel{whp}{\approx} 0,
\end{equation*}
using Propositions \ref{P2.1main} and \ref{P2.2main}, where ``whp'' denotes with high probability under the posterior. 
In order to bound the last term in~\eqref{Eq2.1n}, we make use of the covering set $\mathcal{D}(\epsilon)$ given in Definition \ref{DefComplexityTerm}. Note that in the sketch argument (\ref{Eq2.1n}), when approximating $T(G(z))$ using some $p$ constructed from the covering set $\mathcal{D}(\epsilon)$, we only need to bound $|J(G(z),T(G(z)))-J(G(z),p)|$. Previous works \cite{chatterjee2016nonlinear,yan2020nonlinear} require bounding two extra terms $|F(T(G(z)))-F(p)|$ and $|I(T(G(z)))-I(p)|$ in addition to $|J(G(z),T(G(z)))-J(G(z),p)|$, which yields suboptimal error bounds.  

\section{Variational inference for latent Dirichlet allocation}\label{Sect.3}

In this section, we apply the general meta-theorems (Theorems \ref{Theorem2.1} and \ref{Theorem2.2}) presented in Section \ref{Sect.2.1} to verify the validity of MFVI for LDA. This yields the proof of Theorem \ref{Theorem_LDA_UBD}. 

Recall the model setup and notation for LDA presented in Section \ref{Sect.1.2}. Henceforth we focus on the case where $D=1$, and suppress the subscript $d$ in subsequent notation; since the posterior~\eqref{Eq3.1} and the collapsed posterior~\eqref{Eq3.2} factorize over different documents $d\in [D]$, the proof for the general $D$ case follows from the $D=1$ case using Jensen's inequality, as
\begin{equation*}
    \sum_{d=1}^D C K\log\Big(\frac{n_d}{K}+2\Big)\leq CDK\log\Big(\frac{n}{DK}+2\Big).
\end{equation*}

To apply the general framework in Section~\ref{Sect.2.1} to LDA, we set the global latent variable to be $\pmb{\theta}=\pmb{\pi}$ and the local latent variables to be $\mathbf{Z}=(Z_1,\cdots,Z_n)$. Note that the number of categories of the local latent variables are $K_i\equiv K$ for all $i\in [n]$. We denote by $\hat{\mathcal{Y}}_{n,K}$ the set defined in~\eqref{eq:Y_nK} with $K_i\equiv K$ for all $i\in [n]$. For any $z\in[K]^n$ and any $\ell\in[K],r\in[V]$, by a slight abuse of notation (see~\eqref{defNdl} and~\eqref{defNdlr}) and noting that $D=1$, we define
\begin{equation}\label{def_N}
    N_{\ell}(z):=|\{i\in[n]:z_i=\ell\}|,\qquad \mbox{and} \qquad  N_{\ell,r}(z):=|\{i\in[n]: z_i=\ell, X_i=r\}|.
\end{equation}
Denote by $\mathcal{S}_{n,K}$ the normalizing constant\slash partition function for the collapsed posterior (\ref{Eq3.2}): 
\begin{equation}\label{eq:S_nK}
\mathcal{S}_{n,K} :=\sum_{z\in [K]^n} e^{\Upsilon(z)},
\end{equation}
where for any $z\in [K]^n$,
\begin{equation}\label{Eq.A22}
    \Upsilon(z) :=  \sum_{\ell=1}^K\sum_{r=1}^V N_{\ell,r}(z)\log \eta_{\ell,r}+\sum_{\ell=1}^K \log \Gamma(N_{\ell}(z)+\alpha_{\ell}) -\log\Gamma\Big(n+\sum_{\ell=1}^K \alpha_{\ell}\Big).
\end{equation}
Note that $\mathcal{S}_{n,K}$ is also the normalizing constant of the full posterior (\ref{Eq3.1}).

We specialize the general framework in Sections \ref{Sect.1.1} and \ref{Sect.2.1} to LDA as follows. We take $\nu$ to be the Dirichlet distribution of order $K$ with parameters $(1,\cdots,1)$, and set $\mu\equiv\mu_1\otimes\cdots\otimes\mu_n$ with $\mu_i(\ell)=\frac{\eta_{\ell,X_i}}{\sum_{s=1}^K \eta_{s,X_i}}$ for every $i\in[n]$ and $\ell\in [K]$ (recall~\eqref{base} and~\eqref{defmu}). Then, for any $z\in [K]^n$,
\begin{equation}\label{def_mu_form}
    \mu(z)=\prod_{i=1}^n \mu_i(z_i)=\prod_{i=1}^n \frac{\eta_{z_i,X_i}}{\sum_{\ell=1}^K\eta_{\ell,X_i}}=\frac{\prod_{\ell=1}^K \prod_{r=1}^V \eta_{\ell,r}^{N_{\ell,r}(z)}}{\prod_{i=1}^n\big(\sum_{\ell=1}^K\eta_{\ell,X_i}\big)}.
\end{equation}
We define
\begin{equation}\label{defrR}
    r(\pmb{\pi},z):=\sum_{\ell=1}^K(N_{\ell}(z)+\alpha_{\ell}-1)\log\pi_{\ell}-\log\Gamma(K),\quad\mbox{for any }z\in [K]^n,
\end{equation}
\begin{equation}\label{defrR2}
    R(\pmb{\pi},y):=\sum_{\ell=1}^K\big(\widetilde{N}_{\ell}(y)+\alpha_{\ell}-1\big)\log\pi_{\ell}-\log\Gamma(K),\quad\mbox{for any }y\in[0,1]^{nK},
\end{equation}
where $\widetilde{N}_{\ell}(y):=\sum_{i=1}^n y_{i,\ell}$ for every $\ell\in[K]$, and the $-\log\Gamma(K)$ term arises from the density of $\nu=\text{Dir}(1,\cdots,1)$ (see~\eqref{poste.deri} below). We keep the $-\log\Gamma(K)$ term in $r(\cdot,\cdot)$ and $R(\cdot,\cdot)$ to simplify the expressions for $f(\cdot)$ and $F(\cdot)$ in~\eqref{deffF} and~\eqref{deffF2} below. Note that Assumptions \ref{Assump:Cond-ind}-\ref{Assump:R} hold under these specifications: using~\eqref{Eq3.1} and~\eqref{def_mu_form}, and noting that $D=1$, we have
\begin{equation}\label{poste.deri}
\frac{d\mathbb{P}\left(\mathbf{Z},\pmb{\pi}\mid\mathbf{X}\right)}{d\gamma}\propto \frac{ \prod_{\ell=1}^K \pi_{\ell}^{N_{\ell}(\mathbf{Z})+\alpha_{\ell}-1} \cdot  \prod_{\ell=1}^K \prod_{r=1}^V \eta_{\ell,r}^{N_{\ell,r}(\mathbf{Z})}}{\Gamma(K)\cdot\prod_{\ell=1}^K \prod_{r=1}^V \eta_{\ell,r}^{N_{\ell,r}(\mathbf{Z})}}=\exp(r(\pmb{\pi},\mathbf{Z})).
\end{equation}
Moreover, \eqref{defrR} implies that (recall~\eqref{def_f})
\begin{equation}\label{deffF}
    f(z)=\sum_{\ell=1}^K \log\Gamma(N_{\ell}(z)+\alpha_{\ell})-\log\Gamma\Big(n+\sum_{\ell=1}^K\alpha_{\ell}\Big),\quad\mbox{for any }z\in [K]^n.
\end{equation}
We define
\begin{equation}\label{deffF2}
F(y):=\sum_{\ell=1}^K\log\Gamma\big(\widetilde{N}_{\ell}(y)+\alpha_{\ell}\big)-\log\Gamma\Big(n+\sum_{\ell=1}^K\alpha_{\ell}\Big),\quad\mbox{for any }y\in[0,1]^{nK};
\end{equation}
note that Assumption \ref{Assump:F} also holds. By~\eqref{def_I} and~\eqref{def_mu_form}, for any $y\in\hat{\mathcal{Y}}_{n,K}$,
\begin{eqnarray}\label{defI}
    I(y)
    &=&\sum_{i=1}^n\sum_{\ell=1}^K y_{i,\ell}\log y_{i,\ell}-\sum_{i=1}^n\sum_{\ell=1}^K y_{i,\ell}\log\eta_{\ell, X_i}+\sum_{i=1}^n\log\bigg(\sum_{\ell=1}^K \eta_{\ell,X_i}\bigg).
\end{eqnarray}
Moreover, with $f$ and $\mu$ as in~\eqref{deffF} and~\eqref{def_mu_form}, we have 
\begin{eqnarray}\label{partition}
    \sum_{z\in [K]^n} e^{f(z)}\mu(z)&=& \sum_{z\in [K]^n}e^{\sum_{\ell=1}^K \log\Gamma(N_{\ell}(z)+\alpha_{\ell})-\log\Gamma\big(n+\sum_{\ell=1}^K\alpha_{\ell}\big)}\cdot \frac{\prod_{\ell=1}^K \prod_{r=1}^V \eta_{\ell,r}^{N_{\ell,r}(z)}}{\prod_{i=1}^n\big(\sum_{\ell=1}^K\eta_{\ell,X_i}\big)}\nonumber\\
    &=& \sum_{z\in [K]^n}\frac{e^{\Upsilon(z)}}{\prod_{i=1}^n\big(\sum_{\ell=1}^K\eta_{\ell,X_i}\big)}=
    \frac{\mathcal{S}_{n,K}}{\prod_{i=1}^n\big(\sum_{\ell=1}^K\eta_{\ell,X_i}\big)},
\end{eqnarray}
where we use~\eqref{Eq.A22} and~\eqref{eq:S_nK} in the second and third equalities, respectively.

\subsection{Overview of the proof of Theorem~\ref{Theorem_LDA_UBD}}\label{Sect.3.1} 
The proof involves five steps (see \textbf{Steps 1}-\textbf{5} in Section~\ref{Sect.3.2}), with \textbf{Steps 2} and \textbf{3} constituting the core technical components. 

Recall the product measures $\nu,\mu$ and the functions $r,R,f,F$ specified above (after~\eqref{Eq.A22}). By Lemma \ref{Lemma2.1} and Theorem \ref{Theorem2.2}, in order to establish Theorem \ref{Theorem_LDA_UBD}, it suffices to bound the partition function $\mathcal{S}_{n,K}$ (as defined in~\eqref{eq:S_nK}); see \textbf{Step 5} for the proof of this reduction. If $N_{\ell}(z)+\alpha_{\ell}\geq 1$ for all $z\in[K]^n$ and $\ell\in[K]$, then we can approximate the $\log\Gamma(N_{\ell}(z)+\alpha_{\ell})$ term in the expression of $f(z)$ using Stirling's approximation. Then, using the arguments in \textbf{Steps 2} and \textbf{3}, we obtain upper bounds on the smoothness term $E_1$ (see Definition \ref{DefSmoothnessTerm}) and the complexity term $E_2(\epsilon)$ (see Definition \ref{DefComplexityTerm}), which combined with Theorem \ref{Theorem2.1} yields the desired bound on the partition function $\mathcal{S}_{n,K}$. 

To address the general scenario where $N_{\ell}(z)+\alpha_{\ell}$ may be close to $0$ (so that Stirling's approximation is no longer accurate), in \textbf{Step 1}, we restrict the state space $[K]^n$ of the collapsed posterior to a subset $A^n$ (where $A\subseteq[K]$) and perform Stirling's approximation to the relevant Hamiltonian~\eqref{def_Upsilon}. This yields the ``restricted partition function'' $\mathcal{S}_{n,K,A}$ defined in~\eqref{def_SnKA}. In (\ref{Eqq3.1}) below, we show that the original partition function $\mathcal{S}_{n,K}$ can be bounded in terms of $\mathcal{S}_{n,K,A}$, for $A \subseteq [K]$. In order to bound the restricted partition functions, we aim to apply Theorem \ref{Theorem2.1}, which requires bounding the smoothness term $E_1$ and the complexity term $E_2(\epsilon)$. We provide upper bounds for these terms in \textbf{Steps 2} and \textbf{3} (see (\ref{Eq.A33}) and (\ref{Eq.A8})), and obtain desired bounds on the restricted partition functions in \textbf{Step 4}. Finally, by combining the results from \textbf{Steps 1} and \textbf{4}, we attain the desired bound on $\mathcal{S}_{n,K}$ in \textbf{Step 5}.

The proof of Theorem \ref{Theorem_MMSB_UBD} follows similar steps. We present the details in Appendix \ref{Appendix_B}.  

\subsection{Proof of Theorem \ref{Theorem_LDA_UBD}}\label{Sect.3.2}

Throughout this section, we use $C,c$ to denote positive constants that depend only on $C_0$ (as specified in Theorem \ref{Theorem_LDA_UBD}). The values of these constants may change from line to line. To simplify notation, we define
\begin{equation}\label{phioriginal}
    \phi(x):=\begin{cases}
        x\log{x} & \text{ for } x>0\\
        0 & \text{ for } x=0,
    \end{cases}\qquad \mbox{and} \qquad \psi(x):=\begin{cases}
        x\log{x} & \text{ for }x\geq 1 \\
        \frac{x^2-1}{2}  & \text{ for }x\in [0,1).
    \end{cases}
\end{equation}
Here, $\psi(\cdot)$ is a regularized version of $\phi(\cdot)$. In \textbf{Step 1} of the proof, we will use Stirling's approximation to approximate the term $\log \Gamma(N_{\ell}(z)+\alpha_{\ell})$ in~\eqref{Eq.A22} by $\phi(N_{\ell}(z))$, which is further approximated by $\psi(N_{\ell}(z))$ (cf.~\eqref{def_PsiAz} below). 
It can be checked that $\psi(\cdot)$ is twice continuously differentiable on $[0,\infty)$, and for any $x\geq 0$, 
\begin{equation}\label{derivatives}
    \psi'(x)=\begin{cases}
        1+\log{x} & \text{ if } x\geq 1 \\
        x & \text{ if } x\in [0,1),
    \end{cases}\qquad \mbox{and} \qquad  \psi''(x)=\frac{1}{\max\{x,1\}}.
\end{equation}
Moreover,
\begin{equation}\label{eq3.2}
    \sup_{x\geq 0}|\phi(x)-\psi(x)|=\sup_{x\in [0,1]}|\phi(x)-\psi(x)|\leq \frac{1}{2}. 
\end{equation}

\medskip

\noindent {\bf Step 1: Restricting the partition function and Stirling's approximation.}
For any $A\subseteq [K]$ with $|A|\geq 1$, we set
\begin{equation}\label{def_SnKA}
  \mathcal{S}_{n,K,A}:=\sum_{z\in A^n} e^{\Psi_{A}(z)},
\end{equation}
where
\begin{equation}\label{def_PsiAz}
    \Psi_{A}(z):=\sum_{\ell\in A}\sum_{r=1}^V N_{\ell,r}(z)\log \eta_{\ell,r} +\sum_{\ell\in A} \psi(N_{\ell}(z)),\qquad\mbox{ for all }z\in A^n\subseteq [K]^n.
\end{equation}
In this step, we will show that {\small \begin{eqnarray}\label{Eqq3.1}
     \hspace{-0.05in} \log\mathcal{S}_{n,K} &\leq& \hspace{-0.05in}\sup_{A\subseteq [K]:|A|\geq 1}\bigg\{\sum_{\ell\in [K]\backslash A}\log\Gamma(\alpha_{\ell})+\log\mathcal{S}_{n,K,A}\bigg\} \nonumber\\
     &&\hspace{0.2in}-n-\log \Gamma\Big(n+\sum_{\ell=1}^K \alpha_{\ell}\Big)+CK\log\Big(\frac{n}{K}+2\Big).
\end{eqnarray}} 
\begin{proof}
For any $z\in [K]^n$, we define $\mathcal{A}(z):=\{\ell\in [K]: N_{\ell}(z)\geq 1\}$. For any $A\subseteq [K]$ and $z\in A^n$, we define
\begin{equation}\label{def_Upsilon}
     \Upsilon_A(z) := \sum_{\ell\in A}\sum_{r=1}^V N_{\ell,r}(z)\log\eta_{\ell,r}+\sum_{\ell\in A} \log\Gamma(N_{\ell}(z)+\alpha_{\ell})-\log\Gamma\Big(n+\sum_{\ell=1}^K \alpha_{\ell}\Big).
\end{equation}
Note that for any $A\subseteq [K]$ and $z\in [K]^n$ such that $\mathcal{A}(z)=A$ (which implies $z\in A^n$), $\Upsilon(z)=\sum_{\ell\in [K]\backslash A}\log\Gamma(\alpha_{\ell})+\Upsilon_A(z)$. Hence, using~\eqref{eq:S_nK}, with $\mathfrak{Z}_A:=\{z\in A^n: N_{\ell}(z)\geq 1\text{ for all }\ell\in A\}$, we have 
\begin{equation}\label{eq3.1}
    \mathcal{S}_{n,K} =\sum_{A\subseteq [K]:|A|\geq 1}\sum_{z\in[K]^n: \mathcal{A}(z)=A} e^{\Upsilon(z)} =\sum_{A\subseteq [K]:|A|\geq 1}e^{\sum_{\ell\in [K]\backslash A}\log\Gamma(\alpha_{\ell})}\bigg(\sum_{z\in\mathfrak{Z}_A}e^{\Upsilon_A(z)}\bigg).
\end{equation}

For any $A\subseteq [K]$ with $|A|\geq 1$, we define $\widetilde{\mathcal{S}}_{n,K,A}:=\sum_{z\in A^n} e^{\Phi_A(z)}$, where
{\small
\begin{equation*}
     \Phi_{A}(z):=\sum_{\ell\in A}\sum_{r=1}^V N_{\ell,r}(z)\log\eta_{\ell,r}+\sum_{\ell\in A} \phi(N_{\ell}(z)),\quad\text{for all } z\in A^n. 
\end{equation*}
}Now we use Stirling's approximation to bound $\Upsilon_A(z)$ with explicit error terms (see Lemma \ref{L3.2} in Appendix \ref{Appendix_E}). For any $z\in A^n$ such that $N_{\ell}(z)\geq 1$ for all $\ell\in A$, by Lemma \ref{L3.2} we get
\begin{eqnarray*}
    &&\bigg|\Upsilon_A(z)-\Phi_A(z)+\log\Gamma\Big(n+\sum_{\ell=1}^K \alpha_{\ell}\Big)+n\bigg|\\
&=&\bigg|\sum_{\ell \in A}\Big[\big(\log \Gamma (N_\ell(z)+\alpha_\ell)-\phi(N_\ell(z)+\alpha_\ell)\big)+\big(\phi(N_\ell(z)+\alpha_\ell)-\phi(N_\ell(z))\big)\Big]+n\bigg|\\
&\le &\sum_{\ell\in A}\alpha_\ell+C\Big[\sum_{\ell\in A}\big(\log (N_\ell(z)+\alpha_\ell)+1\big)\Big]\\
   &\leq& CK+C\sum_{\ell=1}^K\log(N_{\ell}(z)+C_0)\leq CK\log\Big(\frac{n}{K}+2\Big),
\end{eqnarray*}
where we use the bound $\max_{\ell\in [K]}\alpha_\ell\le C_0$, and the last step uses Jensen's inequality. Hence by (\ref{eq3.1}), we have
\begin{equation*}\label{Eqn3.26}
    \mathcal{S}_{n,K}\leq \sum_{A\subseteq [K]:|A|\geq 1}e^{\sum_{\ell\in [K]\backslash A}\log\Gamma(\alpha_{\ell})-\log\Gamma\big(n+\sum_{\ell=1}^K \alpha_{\ell}\big)-n+CK\log(n\slash K+2)}\widetilde{\mathcal{S}}_{n,K,A},
\end{equation*}
where we note that $\sum\limits_{z\in \mathfrak{Z}_A}e^{\Phi_A(z)}\leq \widetilde{\mathcal{S}}_{n,K,A}$. By (\ref{eq3.2}), for any $A\subseteq [K]$ with $|A|\geq 1$ and any $z\in A^n$, we have
\begin{equation*}
    \big|\Phi_{A}(z)-\Psi_{A}(z)\big|\leq \sum_{\ell\in A}\big|\phi(N_{\ell}(z))-\psi(N_{\ell}(z))\big|\leq \frac{K}{2} \;\; \Rightarrow \;\;\widetilde{\mathcal{S}}_{n,K,A}\leq e^{ K\slash 2}\mathcal{S}_{n,K,A}.
\end{equation*}
Thus, combining the above two displays, and possibly adjusting the constant $C$ above, we get
\begin{eqnarray*}
   \mathcal{S}_{n,K}  & \leq& \sum_{A\subseteq [K]:|A|\geq 1}e^{\sum_{\ell\in [K]\backslash A}\log\Gamma(\alpha_{\ell})-\log\Gamma\big(n+\sum_{\ell=1}^K \alpha_{\ell}\big)-n+CK\log(n\slash K+2)}\mathcal{S}_{n,K,A}. 
\end{eqnarray*}
Taking $\log$ on both sides and noting that the number of non-empty subsets of $[K]$ is $2^K-1$, the desired conclusion follows.
\end{proof}

\noindent For \textbf{Steps 2}-\textbf{4} below, we fix any $A\subseteq [K]$ such that $|A|\geq 1$. We assume that $\sum_{\ell\in A}\eta_{\ell,X_i}>0$ for every $i\in [n]$ (otherwise $\mathcal{S}_{n,K,A}=0$; see~\eqref{def_SnKA}-\eqref{def_PsiAz}). In order to bound the restricted partition function $\mathcal{S}_{n,K,A}$ (see~\eqref{def_SnKA}), we specialize the general setup in Section \ref{Sect.2.1} as follows. We let $A$ play the role of $[K_i]$ for each $i\in[n]$ in Section \ref{Sect.2.1}, and let
\begin{equation}\label{def_YhatnA}
  \hat{\mathcal{Y}}_{n,A}:=\bigg\{y=(y_{i,\ell})_{i\in[n],\ell\in A}\in [0,1]^{n|A|}: \sum_{\ell\in A} y_{i,\ell}=1\text{ for every }i\in [n]\bigg\}
\end{equation}
play the role of $\hat{\mathcal{Y}}_{n,\mathbf{K}}$. For each $i\in [n]$, we define $\muA_{i}$ to be the probability measure on $A$ such that for each $\ell\in A$, $\muA_{i}(\ell)=\frac{\eta_{\ell,X_i}}{\sum_{s\in A} \eta_{s,X_i}}$. For any $z\in A^n$, we take 
\begin{equation}
    \fA(z):=\Psi_A(z)-\sum_{\ell\in A}\sum_{r=1}^V N_{\ell,r}(z)\log \eta_{\ell,r}=\sum_{\ell\in A} \psi(N_{\ell}(z)),
\end{equation}
where the second equality uses \eqref{def_PsiAz}. With $N_\ell(z), N_{\ell,r}(z)$ as in~\eqref{def_N}, and
\begin{equation}
    \muA:=\mu^A_{1}\otimes \cdots\otimes \mu^A_{n},
\end{equation}
we have
\begin{eqnarray}\label{Eq.A11}
   \sum_{z\in A^n} e^{\fA(z)}\mu^{A}(z)  & = &  
   \sum_{z\in A^n}e^{\sum_{\ell\in A}\psi(N_{\ell}(z))}\prod_{i=1}^n\frac{\eta_{z_i,X_i}}{\sum_{\ell\in A} \eta_{\ell,X_i}}\nonumber\\
   &=&\frac{\sum_{z\in A^n} e^{\sum_{\ell\in A}\psi(N_{\ell}(z))+\sum_{\ell\in A}\sum_{r=1}^V 
 N_{\ell,r}(z)\log \eta_{\ell,r}}}{\prod_{i=1}^n\big(\sum_{\ell\in A}\eta_{\ell,X_i}\big)} \nonumber \\ 
 & = &\frac{\sum_{z\in A^n}e^{\Psi_A(z)}}{\prod_{i=1}^n\big(\sum_{\ell\in A}\eta_{\ell,X_i}\big)}=
   \frac{\mathcal{S}_{n,K,A}}{\prod_{i=1}^n\big(\sum_{\ell\in A}\eta_{\ell,X_i}\big)}, 
\end{eqnarray}
where the equalities in the last line use~\eqref{def_SnKA}-\eqref{def_PsiAz}. For any $y=(y_{i,\ell})_{i\in [n],\ell\in A}\in [0,1]^{n|A|}$, we take
\begin{equation}\label{Eq.A13}
    \qquad \FA(y) :=\sum_{\ell\in A} \psi\big(\NellA(y)\big),\qquad \mbox{where} \;\; \NellA(y):=\sum_{i=1}^{n} y_{i,\ell}\mbox{ for any }\ell\in A.
\end{equation}
For any $y=(y_{i,\ell})_{i\in [n],\ell\in A}\in \hat{\mathcal{Y}}_{n,A}$, we define
\begin{eqnarray}\label{Eq.A14}
    \IA(y)& := & \sum_{i=1}^n\sum_{\ell\in A} y_{i,\ell}\log\Big(\frac{y_{i,\ell}}{\mu_i^A(\ell)}\Big)\nonumber\\
    &=& \sum_{i=1}^n\sum_{\ell\in A} y_{i,\ell}\log y_{i,\ell} -\sum_{i=1}^n\sum_{\ell\in A}y_{i,\ell}\log \eta_{\ell,X_i} +\sum_{i=1}^n\log\bigg(\sum_{\ell\in A}\eta_{\ell,X_i}\bigg).
\end{eqnarray}
In \textbf{Steps 2}-\textbf{4} below, $\muA,\fA,\FA,\IA$ as defined above will play the roles of $\mu,f,F,I$ in Section \ref{Sect.2.1}. We aim to apply Theorem \ref{Theorem2.1} to bound $\sum_{z\in A^n} e^{\fA(z)}\muA(z)$ (cf.~\eqref{defS}), which via~\eqref{Eq.A11} yields a bound on the restricted partition function $\mathcal{S}_{n,K,A}$. The upper bounds on the smoothness term $E_1$ (see Definition \ref{DefSmoothnessTerm}) and the complexity term $E_2(\epsilon)$ (see Definition \ref{DefComplexityTerm})---as required by Theorem \ref{Theorem2.1}---are obtained in \textbf{Steps 2} and \textbf{3}, and the resulting upper bound on $\mathcal{S}_{n,K,A}$ is given in \textbf{Step~4}. 

\medskip

\noindent{\bf{Step 2: Bounding the smoothness term.}}
Let $E_1$ and $\Lambda(\cdot)$ be as defined in (\ref{E2.1}) and (\ref{def_L}), respectively, but with $\hat{\mathcal{Y}}_{n,\mathbf{K}},\mu,f,F,I$ replaced by $\hat{\mathcal{Y}}_{n,A},\mu^A,\fA,\FA,\IA$ (as in~\eqref{def_YhatnA}-\eqref{Eq.A14}). In this step, we will show that there exist positive absolute constants $C_1,C_2,C_2'$, such that 
\begin{equation}\label{Eq.A33}
    E_1 \leq C_1 K, \qquad \mbox{and} \qquad \Lambda(C_1K) \leq C_2 \log\Big(\frac{n\log{n}}{K}+2\Big) \leq C_2' K\log\Big(\frac{n}{K}+2\Big).
\end{equation}
\begin{proof}
Using the definition of $\FA$ in~\eqref{Eq.A13}, for any $y\in [0,1]^{n|A|}$, $i,j\in [n]$, and $\ell,\ell'\in A$, 
\begin{equation}\label{F_def}
    \FA_{i,\ell}(y):=\frac{\partial \FA(y)}{\partial y_{i,\ell}}= \psi'\big(\widetilde{N}^A_{\ell}(y)\big), \quad \mbox{and} \quad  \FA_{i,\ell;j,\ell'}(y):=\frac{\partial^2 \FA(y)}{\partial y_{i,\ell} \partial y_{j,\ell'}}=\mathbbm{1}_{\ell=\ell'}\psi''\big(\widetilde{N}^A_{\ell}(y)\big). 
\end{equation}
Recall $b_{i,\ell}$ and $c_{i,\ell;j,\ell'}$ from Definition \ref{Global_GH}. By (\ref{derivatives}) and (\ref{F_def}), for any $i,j\in[n]$ and $\ell,\ell'\in A$, 
\begin{equation}\label{UniformBound}
    b_{i,\ell}\leq 1+\log{n}, \qquad \mbox{and} \qquad   c_{i,\ell;j,\ell'}\leq \|\psi''\|_{\infty}\leq 1.
\end{equation}
Now recall $\NI{y}$ and $\HI{i}{\ell}{j}{\ell'}{y}$ from Definition \ref{Local_Hess}. For any $y\in [0,1]^{n|A|}$, $y'\in \NI{y}$, and $\ell\in A$, we have  
\begin{equation*}
    \NellA(y')=\sum_{i=1}^n y_{i,\ell}'\geq \sum_{i\in[n]\backslash\{i_0,i_0'\}} y_{i,\ell}\geq \NellA(y)-2,
\end{equation*}
where $i_0,i_0'$ are as in Definition \ref{Local_Hess}. Hence by (\ref{derivatives}) and (\ref{F_def}), 
\begin{equation*}
     \FA_{i,\ell;j,\ell'}(y')=\frac{\mathbbm{1}_{\ell=\ell'}}{\max\big\{\widetilde{N}^A_{\ell}(y'),1\big\}}\leq \frac{\mathbbm{1}_{\ell=\ell'}}{\max\big\{\widetilde{N}^A_{\ell}(y)-2,1\big\}}\leq \frac{3\cdot \mathbbm{1}_{\ell=\ell'}}{\max\big\{\widetilde{N}^A_{\ell}(y),1\big\}}.
\end{equation*}
Therefore,
\begin{equation}\label{E3.3}
    \HI{i}{\ell}{j}{\ell'}{y}\leq \frac{3\cdot \mathbbm{1}_{\ell=\ell'}}{\max\big\{\widetilde{N}^A_{\ell}(y),1\big\}}.
\end{equation}

Now recall $\NII{y}{M}$ and $\HII{i}{\ell}{j}{\ell'}{y}{M}$ from Definition \ref{Local_Hess_new}. For any $y\in [0,1]^{n|A|}$, $M\geq 1$, $y'\in\NII{y}{M}$, and $\ell\in A$,  
\begin{equation*}
   \NellA(y')=\sum_{i=1}^n y_{i,\ell}' \geq \sum_{i\in [n]\backslash\{i_0\}} y_{i,\ell}'\geq M^{-1}\sum_{i\in[n]\backslash\{i_0\}} y_{i,\ell} \geq \frac{\NellA(y)-1}{M},
\end{equation*}
where $i_0$ is as in Definition \ref{Local_Hess_new}. Hence
\begin{equation*}
    \max\{\NellA(y'),1\}\geq \frac{M}{M+1}\cdot \NellA(y')+\frac{1}{M+1}\geq \frac{\NellA(y)}{M+1}.
\end{equation*}
Thus by (\ref{derivatives}) and (\ref{F_def}), for any $i,j\in [n]$ and $\ell,\ell'\in A$, 
\begin{equation*}
    \FA_{i,\ell;j,\ell'}(y')=\frac{\mathbbm{1}_{\ell=\ell'}}{\max\big\{\NellA(y'),1\big\}}\leq \frac{(M+1)\mathbbm{1}_{\ell=\ell'}}{\max\big\{\NellA(y),1\big\}}.
\end{equation*}
Therefore,
\begin{equation}\label{E3.1}
    \HII{i}{\ell}{j}{\ell'}{y}{M} \leq \frac{(M+1)\mathbbm{1}_{\ell=\ell'}}{\max\big\{\NellA(y),1\big\}}.
\end{equation}
With $M_i$ as in Definition \ref{DefSmoothnessTerm}, using \eqref{UniformBound} we get $M_i\leq e^2$ for every $i\in [n]$. Fix $y\in\hat{\mathcal{Y}}_{n,A}$ arbitrary, and combine (\ref{E3.5}) and (\ref{UniformBound}) to get
\begin{equation*}
    \sum_{\ell'\in A}\max_{\ell\in A}\big\{\HI{i}{\ell}{j}{\ell'}{y}\big\}y_{j,\ell'}\leq 1. 
\end{equation*}
Thus, with $\Phi_{j,i}(y)$ as in Definition \ref{DefSmoothnessTerm}, using the inequality that $\exp(x)\leq 1+e^2x\slash 2$ for all $x\in [0,2]$, we obtain for any $i,j\in [n]$, 
\begin{equation*}
   \Phi_{j,i}(y)\leq e^2\sum_{\ell'\in A}\max_{\ell\in A}\big\{\HI{j}{\ell}{i}{\ell'}{y}\big\}y_{i,\ell'}\leq 3e^2\sum_{\ell'\in A}\frac{y_{i,\ell'}}{\max\big\{\widetilde{N}^A_{\ell'}(y),1\big\}},
\end{equation*}
where the second inequality uses
 (\ref{E3.3}). 
Hence
\begin{eqnarray}\label{E3.6}
\sum_{i=1}^n\max_{j\in[n]}\big\{\Phi_{j,i}(y)\big\}&\leq &3e^2\sum_{i=1}^n\sum_{\ell'\in  A}\frac{y_{i,\ell'}}{\max\big\{\widetilde{N}^A_{\ell'}(y),1\big\}}\nonumber\\
    &=& 3e^2\sum_{\ell' \in A}\frac{\widetilde{N}^A_{\ell'}(y)}{\max\big\{\widetilde{N}^A_{\ell'}(y),1\big\}}\leq 3e^2|A| \leq 3e^2  K. 
\end{eqnarray}
By (\ref{E3.1}) with $M=M_i\leq e^2$, for any $i\in [n]$ and $\ell\in A$ we have
\begin{eqnarray}\label{E3.8}
   \sum_{j=1}^n\sum_{\ell'\in A} \HII{i}{\ell}{j}{\ell'}{y}{M_i} y_{j,\ell'}&\leq& (e^2+1)\sum_{j=1}^n\sum_{\ell'\in A}\frac{y_{j,\ell'}\mathbbm{1}_{\ell=\ell'}}{\max\big\{\widetilde{N}_{\ell}^A(y),1\big\}} \nonumber\\
   &=& \frac{(e^2+1)\sum_{j=1}^n y_{j,\ell}}{\max\big\{\widetilde{N}^A_{\ell}(y),1\big\}}=\frac{(e^2+1)\widetilde{N}^A_{\ell}(y)}{\max\big\{\widetilde{N}^A_{\ell}(y),1\big\}}\leq 9,
\end{eqnarray}
and
\begin{eqnarray}\label{E3.9}
   \sum_{i=1}^n\max_{\ell\in A}\bigg\{\sum_{\ell'\in A} \HII{i}{\ell}{i}{\ell'}{y}{M_{i}}y_{i,\ell'}\bigg\}&\leq &\sum_{i=1}^n\sum_{\ell\in A}\sum_{\ell'\in A} \HII{i}{\ell}{i}{\ell'}{y}{M_{i}} y_{i,\ell'} \nonumber\\
   \qquad \leq (e^2+1)\sum_{i=1}^n \sum_{\ell\in A} \frac{y_{i,\ell}}{\max\big\{\widetilde{N}^A_{\ell}(y),1\big\}} \; &=& \; (e^2+1)\sum_{\ell\in A}\frac{\widetilde{N}^A_{\ell}(y)}{\max\big\{\widetilde{N}^A_{\ell}(y),1\big\}}\leq 9K.
\end{eqnarray}
Combining (\ref{E3.6})-(\ref{E3.9}) we obtain the existence of an absolute positive constant $C_1$ such that $E_1\leq  C_1 K$, where $E_1$ is as defined in (\ref{E2.1}). With $\Lambda(\cdot)$ as in (\ref{def_L}), (\ref{UniformBound}) gives
\begin{equation*}
    \Lambda(C_1 K)\leq \log\left(2+\frac{n(1+\log n)}{C_1K}\right)\le C_2\log\Big(\frac{n\log{n}}{K}+2\Big),
\end{equation*}
where $C_2>0$ is an absolute constant. Now note that there exist absolute constants $C',C''>0$, such that if $K\geq \sqrt{n}$, then $\log\big(\frac{n\log{n}}{K}+2\big)\leq \log(n\log n+2)\leq C'\sqrt{n}\leq C'K$; if $K<\sqrt{n}$, then $\log n\leq C'\sqrt{n}\leq \frac{C'n}{K}$, and thus $\log\big(\frac{n\log{n}}{K}+2\big)\leq \log\big(C'\big(\frac{n}{K}\big)^2+2\big)\leq C''\log\big(\frac{n}{K}+2\big)$. Hence $\Lambda(C_1K) \leq C_2 \log\big(\frac{n\log{n}}{K}+2\big) \leq C_2' K\log\big(\frac{n}{K}+2\big)$, where $C_2'>0$ is an absolute constant. 
\end{proof}

\noindent{\bf{Step 3: Bounding the complexity term.}}
For any $\epsilon>0$, let $E_2(\epsilon)$ be as defined in (\ref{def_E}), but with $\hat{\mathcal{Y}}_{n,\mathbf{K}},\mu,f,F,I$ replaced by $\hat{\mathcal{Y}}_{n,A},\mu^A,f^A,F^A,I^A$ (as in~\eqref{def_YhatnA}-\eqref{Eq.A14}). In this step, we will show that there exists $\epsilon>0$, such that 
\begin{equation}\label{Eq.A8}
    E_2(\epsilon) \leq CK\log\Big(\frac{n}{K}+2\Big).
\end{equation}
\begin{proof}

For any $\delta\in (0,1]$, we define 
\begin{eqnarray*}
     \mathcal{T}(\delta)& :=&\Big\{(x_{\ell})_{\ell\in A}\in (\delta\mathbb{N})^{A}:\Big|\sum_{\ell\in A}x_{\ell}-n\Big|\leq K\Big\}, \nonumber \\
     \Dtilde(\delta)&:=&\big\{(d_{i,\ell})_{i\in [n],\ell\in A}: d_{i,\ell} = \psi(x_{\ell}+1)-\psi(x_{\ell})\text{ for }i\in [n],\ell\in A,\nonumber\\
&&\hspace{1.5in}  \text{ where }(x_{\ell})_{\ell\in A}\in\mathcal{T}(\delta)\big\},
\end{eqnarray*}
where $\delta\mathbb{N}:=\{0,\delta,2\delta,\cdots\}$ (and $\mathbb{N}:=\{0,1,2,\cdots\})$. Note that by Lemma \ref{L3.3} (in Appendix \ref{Appendix_E}), for any $m\in\mathbb{N}$ and $\delta\in (0,1]$, the number of $(x_{\ell})_{\ell\in A}\in (\delta\mathbb{N})^{A}$ such that $\sum_{\ell\in A}x_{\ell}=\delta m$ is
\begin{equation}\label{combinatorial_bdd}
     \binom{m+|A|-1}{|A|-1}\leq \binom{m+K}{K}\leq \Big(\frac{e(m+K)}{K}\Big)^K.
\end{equation}
Hence for any $\delta\in (0,1]$,
\begin{eqnarray*}
     |\Dtilde(\delta)| \leq|\mathcal{T}(\delta)|&\leq &\sum_{m\in \big[\frac{n-K}{\delta},\frac{n+K}{\delta}\big]\cap\mathbb{N}}\Big(\frac{e(m+K)}{K}\Big)^K  \nonumber  \\
     &\leq&\Big(1+\frac{2K}{\delta}\Big)\Big[\frac{e(\frac{n+K}{\delta}+K)}{K}\Big]^K\leq \frac{3K}{\delta}\Big(\frac{2e(n+K)}{\delta K}\Big)^K. 
\end{eqnarray*}
Thus
\begin{equation}\label{Eqn3.14}
   \log(|\Dtilde(\delta)|)\leq   \log\Big(\frac{3K}{\delta}\Big)+K\log\Big(\frac{2e(n+K)}{\delta K}\Big)\leq CK\log\Big(\frac{n+2K}{\delta K}\Big). 
\end{equation}

Below we consider any $y\in\hat{\mathcal{Y}}_{n,A}$, and take $\delta=\min\{K\slash n,1\}\in(0,1]$. For each $\ell\in A$, there exists $x_{\ell}  \in  \delta \mathbb{N}$ such that $\big|\NellA(y)-x_{\ell}\big|\leq \delta$. As $\sum_{\ell\in A}\NellA(y)=n$, we have
\begin{equation*}
    \Big|\sum_{\ell\in A} x_{\ell}-n\Big|\leq \sum_{\ell\in A}\big|x_{\ell}-\NellA(y)\big|\leq \delta|A|\leq K.
\end{equation*}
Hence $(x_{\ell})_{\ell\in A}\in\mathcal{T}(\delta)$. Now we take $d=(d_{i,\ell})_{i\in [n],\ell\in A}\in\Dtilde(\delta)$ such that
\begin{equation*}
    d_{i,\ell}=\psi(x_{\ell}+1) -  \psi(x_{\ell})
\end{equation*}
for all $i\in [n],\ell\in A$. For any $x\geq 0$, we define $h(x) := \psi(x+1)-\psi(x)$. By (\ref{derivatives}), 
\begin{equation}\label{derivative.h}
     |h'(x)|=|\psi'(x+1)-\psi'(x)|\leq \frac{1}{\max\{x,1\}}, \quad \text{for any } x\geq 0.
\end{equation}
For each $i\in [n]$ and $\ell\in A$,
\begin{eqnarray*}
     \FA\big(\UII{y}{i}{\ell}\big)-\FA\big(\UI{y}{i}\big)&=&\psi\big(\widetilde{N}^A_{\ell}\big(\UI{y}{i}\big)+1\big)-\psi\big(\widetilde{N}^A_{\ell}\big(\UI{y}{i}\big)\big)\nonumber\\
     &=& h\big(\widetilde{N}^A_{\ell}\big(\UI{y}{i}\big)\big). 
\end{eqnarray*}
Hence for any $i\in [n]$ and $\ell\in A$,
\begin{eqnarray*}
  && \big|\FA\big(\UII{y}{i}{\ell}\big)-\FA\big(\UI{y}{i}\big)-h\big(\widetilde{N}^A_{\ell}(y)\big)\big|=\big|h\big(\widetilde{N}^A_{\ell}\big(\UI{y}{i}\big)\big)-h\big(\widetilde{N}^A_{\ell}(y)\big)\big|\nonumber\\
   &\leq&\frac{y_{i,\ell}}{\max\big\{\widetilde{N}^A_{\ell}\big(\UI{y}{i}\big),1\big\}}\leq \frac{y_{i,\ell}}{\max\big\{\widetilde{N}^A_{\ell}(y)-1,1\big\}}\leq \frac{2 y_{i,\ell}}{\max\big\{\widetilde{N}^A_{\ell}(y),1\big\}},
\end{eqnarray*}
\begin{equation*}
    \big|h\big(\widetilde{N}^A_{\ell}(y)\big)-d_{i,\ell}\big|=\big|h\big(\widetilde{N}^A_{\ell}(y)\big)-h(x_{\ell})\big|\leq \frac{\big|\widetilde{N}^A_{\ell}(y)-x_{\ell}\big|}{\max\big\{\widetilde{N}^A_{\ell}(y)-1,1\big\}}\leq \frac{2\delta}{\max\big\{\widetilde{N}^A_{\ell}(y),1\big\}}.
\end{equation*}
Therefore, for any $i\in [n]$ and $\ell\in A$,
\begin{equation*}
    \big|\FA\big(\UII{y}{i}{\ell}\big)-\FA\big(\UI{y}{i}\big)-d_{i,\ell}\big|\leq \frac{2(y_{i,\ell}+\delta)}{\max\big\{\widetilde{N}^A_{\ell}(y),1\big\}}\leq \frac{2 y_{i,\ell}}{\max\big\{\widetilde{N}^A_{\ell}(y),1\big\}}+2\delta.
\end{equation*}
Hence
\begin{eqnarray*}
  &&  \sum_{i=1}^n\max_{\ell\in A}\big\{\big|\FA\big(\UII{y}{i}{\ell}\big)-\FA\big(\UI{y}{i}\big)-d_{i,\ell}\big|\big\}\leq \sum_{i=1}^n\sum_{\ell\in A}\frac{2 y_{i,\ell}}{\max\big\{\widetilde{N}^A_{\ell}(y),1\big\}}+2\delta n \nonumber\\
  &=& \sum_{\ell\in A}\frac{2\widetilde{N}^A_{\ell}(y)}{\max\big\{\widetilde{N}^A_{\ell}(y),1\big\}}+2\delta n\leq 2K+2\delta n\leq 4 K.
\end{eqnarray*}
Thus with $\mathcal{D}(\epsilon)$ and $E_2(\epsilon)$ as in
Definition \ref{DefComplexityTerm}, choosing $\epsilon=4K$ and $\mathcal{D}(\epsilon)=\Dtilde(\delta)$, we get that  
\begin{equation*}
    E_2(4 K)= 8K+\log(|\mathcal{D}(4K)|)=8K+\log(|\Dtilde(\delta)|)\leq CK\log\Big(\frac{n}{K}+2\Big),
\end{equation*}
where the last inequality uses (\ref{Eqn3.14}).\end{proof}

\noindent{\bf{Step 4: Bounding $\mathcal{S}_{n,K,A}$.}} In this step, 
using Theorem \ref{Theorem2.1} and \textbf{Steps 2} and \textbf{3}, we will show that 
\begin{eqnarray}\label{Eq.A21}
    \log \mathcal{S}_{n,K,A}&\leq& \sup_{y\in\hat{\mathcal{Y}}_{n,A}}\bigg\{\sum_{\ell\in A} \log\Gamma\big(\NellA(y)+\alpha_{\ell}\big)+\sum_{i=1}^n\sum_{\ell\in A}y_{i,\ell}\log\eta_{\ell,X_i}-\sum_{i=1}^n\sum_{\ell\in A} y_{i,\ell}\log y_{i,\ell}\bigg\}\nonumber\\
    &&\quad\quad\quad+n+CK\log\Big(\frac{n}{K}+2\Big).
\end{eqnarray}
Here $\mathcal{S}_{n,K,A}$ is  as defined in {\bf Step 1}.
\begin{proof}
To begin, using (\ref{Eq.A11}) we have 
{\small 
\begin{align}\label{Eq.A19}
    \log \mathcal{S}_{n,K,A} & = \sum_{i=1}^n\log\Big(\sum_{\ell\in A} \eta_{\ell,X_i}\Big) +\log\bigg(\sum_{z\in A^n} e^{\fA(z)}\mu^A(z)\bigg) \nonumber\\
    &\le \sum_{i=1}^n\log\Big(\sum_{\ell\in A} \eta_{\ell,X_i}\Big)+\sup_{y\in\hat{\mathcal{Y}}_{n,A}}\Big\{\FA(y)-\IA(y)\Big\} +CK\log\Big(\frac{n}{K}+2\Big)\nonumber\\
    &= \sup_{y\in\hat{\mathcal{Y}}_{n,A}}\bigg\{\FA(y)+\sum_{i=1}^n\sum_{\ell\in A}y_{i,\ell}(\log\eta_{\ell,X_i}- \log y_{i,\ell})\bigg\}  +CK\log\Big(\frac{n}{K}+2\Big), 
\end{align}
}where the inequality in the second line uses Theorem \ref{Theorem2.1}, along with \textbf{Steps 2} and \textbf{3} (namely, (\ref{Eq.A33}) and (\ref{Eq.A8})), which bound the error terms in Theorem \ref{Theorem2.1}.
Here $F^A(\cdot)$ and $I^A(\cdot)$ are as defined in (\ref{Eq.A13}) and (\ref{Eq.A14}) respectively.

Proceeding to bound the first term in the RHS of \eqref{Eq.A19},  for any $y\in\hat{\mathcal{Y}}_{n,A}$ we have
\begin{eqnarray}\label{Eq.A15}
  && \sum_{\ell\in A}\log\big(\NellA(y)+\alpha_{\ell}\big)= \sum_{\ell\in A}\log\big(\NellA(y)+\alpha_{\ell}\big)+\sum_{\ell\in [K]\backslash A}\log(1) \nonumber\\
  &\leq& 
 K \log\bigg(\frac{\sum_{\ell\in A}\NellA(y)+\sum_{\ell\in A}\alpha_{\ell}+K-|A|}{K}\bigg)\leq C K\log\Big(\frac{n}{K}+2\Big),
\end{eqnarray}
where the last two inequalities use Jensen's inequality, and the bound $\max_{\ell\in [K]}\alpha_\ell\le C_0$, respectively.
Moreover, for any $y\in \hat{\mathcal{Y}}_{n,A}$ and $\ell\in A$ we have
 \begin{equation*}
     \big(\NellA(y)+\alpha_{\ell}\big)\log\big(\NellA(y)+\alpha_{\ell}\big)\geq
     \begin{cases}
         \inf_{x\in [0,C_0+1]}\{x\log{x}\}\geq -C & \text{ if }\NellA(y)\leq 1\\
        \NellA(y)\log\big(\NellA(y)\big) & \text{ if }\NellA(y)>1,
     \end{cases}
 \end{equation*}
and so
\begin{equation}\label{Eq.A17}
  \NellA(y)\log\big(\NellA(y)\big)\leq 
 \big(\NellA(y)+\alpha_{\ell}\big)\log\big(\NellA(y)+\alpha_{\ell}\big)+C,
\end{equation}
By applying  (\ref{eq3.2}), (\ref{Eq.A17}), Lemma \ref{L3.2} (in Appendix \ref{Appendix_E}), (\ref{Eq.A15}) chronologically, for any $y\in\hat{\mathcal{Y}}_{n,A}$ we get
{\small\begin{eqnarray}\label{Hdelta}
   \FA(y) = \sum_{\ell\in A}\psi\big(\NellA(y)\big)  &\leq&  \frac{K}{2} + \sum_{\ell\in A}\phi\big(\NellA(y)\big)=\frac{K}{2}+\sum_{\ell\in A} \NellA(y)\log\big(\NellA(y)\big)\nonumber\\
    &\leq& CK+\sum_{\ell\in A} \big(\NellA(y)+\alpha_{\ell}\big)\log\big(\NellA(y)+\alpha_{\ell}\big)\nonumber\\
    &\leq& CK + \sum_{\ell\in A}\Big(\log\Gamma\big(\NellA(y)+\alpha_{\ell}\big)+\NellA(y)+\alpha_{\ell}+\frac{1}{2}\log\big(\NellA(y)+\alpha_{\ell}\big)\Big)\nonumber\\
    &\leq& \sum_{\ell\in A}\log\Gamma\big(\NellA(y)+\alpha_{\ell}\big)+n+CK\log\Big(\frac{n}{K}+2\Big).
\end{eqnarray}}Now the desired conclusion (\ref{Eq.A21}) follows from (\ref{Eq.A19}) and (\ref{Hdelta}). 
\end{proof}

\noindent{\bf{Step 5: Bounding $\mathcal{S}_{n,K}$ and concluding.}} 
Recall the setup in~\eqref{defrR}-\eqref{partition}. Using the results from \textbf{Steps 1} and \textbf{4} (namely, (\ref{Eqq3.1}) and (\ref{Eq.A21})), we get
{\small
\begin{eqnarray*}
    \log \mathcal{S}_{n,K} &\leq& \sup_{y\in\hat{\mathcal{Y}}_{n,K}}\bigg\{\sum_{\ell=1}^K \log\Gamma\big(\widetilde{N}_{\ell}(y)+\alpha_{\ell}\big)+\sum_{i=1}^n\sum_{\ell= 1}^K y_{i,\ell}\log\eta_{\ell,X_i}-\sum_{i=1}^n\sum_{\ell=1}^K y_{i,\ell}\log y_{i,\ell}\bigg\}\nonumber\\
    &&\hspace{0.2in} -\log \Gamma\Big(n+\sum_{\ell=1}^K \alpha_{\ell}\Big)+CK\log\Big(\frac{n}{K}+2\Big)\nonumber\\
    &=& \sup_{y\in\hat{\mathcal{Y}}_{n,K}}\{F(y)-I(y)\}+\sum_{i=1}^n\log\bigg(\sum_{\ell=1}^K \eta_{\ell,X_i}\bigg)+CK\log\Big(\frac{n}{K}+2\Big).
\end{eqnarray*}
}In the above display, we use the fact that
{\small
\begin{eqnarray*}
   &&\sup_{\substack{A\subseteq [K]:|A|\geq 1,\\ y\in \hat{\mathcal{Y}}_{n,A}}}\bigg\{\sum_{\ell\in [K]\backslash A}\log\Gamma(\alpha_{\ell})+\sum_{\ell\in A} \log\Gamma\big(\NellA(y)+\alpha_{\ell}\big) +\sum_{i=1}^n\sum_{\ell\in A} y_{i,\ell}(\log\eta_{\ell,X_i}-\log y_{i,\ell})\nonumber\\
   &=& \sup_{y\in \hat{\mathcal{Y}}_{n,K}}\bigg\{\sum_{\ell=1}^K\log\Gamma\big(\widetilde{N}_{\ell}(y)+\alpha_{\ell}\big)+\sum_{i=1}^n\sum_{\ell=1}^K y_{i,\ell}\log\eta_{\ell,X_i}-\sum_{i=1}^n\sum_{\ell=1}^K y_{i,\ell}\log y_{i,\ell}\bigg\}
\end{eqnarray*}
}in the first line, and the definitions of $F(\cdot)$ and $I(\cdot)$ as in~\eqref{deffF2} and~\eqref{defI} in the second line. Hence by~\eqref{partition}, we have  
{\small \begin{equation}\label{Eq.A24}
  \frac{\log \Big(\sum_{z\in [K]^n} e^{f(z)}\mu(z)\Big) -\sup_{y\in\hat{\mathcal{Y}}_{n,K}}\{F(y)-I(y)\}}{n}\leq\frac{CK}{n}\log\Big(\frac{n}{K}+2\Big).
\end{equation}
}Part (b) of Theorem \ref{Theorem_LDA_UBD} then follows from \eqref{Eq.A24} along with parts (a) and (c) of Lemma \ref{Lemma2.1}, on noting that $F(\cdot)$ is convex (using the convexity of $\log\Gamma(\cdot)$ on $(0,\infty)$).

Now we turn to the proof of part (a) of Theorem \ref{Theorem_LDA_UBD}. With $R(\cdot,\cdot), R_{i,\ell;j,\ell'}(\cdot,\cdot), \ctilde_{i,\ell;j,\ell'}$ as defined in~\eqref{defrR2}, \eqref{def_Rde}, and Definition \ref{Global_GH}, respectively, we obtain that for any $\pmb{\pi}=(\pi_1,\cdots,\pi_K)\in\mathbb{R}^K$ such that $\sum_{\ell=1}^K\pi_{\ell}=1$ and $\pi_{\ell}\geq 0$ for every $\ell\in[K]$, and any $y\in [0,1]^{nK}, i,j\in [n], \ell,\ell'\in [K]$, we have $R_{i,\ell;j,\ell'}(\pmb{\pi},y)=0$ (since $R(\pmb{\pi},\cdot)$ is linear), hence $\ctilde_{i,\ell;j,\ell'}=0$. Hence by Theorem \ref{Theorem2.2} and~\eqref{Eq.A24}, we conclude that $\frac{1}{n} \KL\big(\hat{P} \,\big\| \; \mathbb{P}(\mathbf{Z},\pmb{\pi}\mid\mathbf{X}) \big)\leq \frac{CK}{n}\log\big(\frac{n}{K}+2\big)$.

\section{Discussion}\label{Sect.5}

In this paper, we make theoretical advancements towards a general framework for analyzing the statistical accuracy of MFVI for posterior approximation in Bayesian latent variable models with categorical local latent variables. Our framework uses (local and global) bounds on the gradient and Hessian of the Hamiltonian corresponding to the unnormalized collapsed posterior and provides explicit non-asymptotic bounds on the KL divergence between the variational and true posteriors. When applied to two widely used latent variable models, LDA and MMSB, our results characterize the exact parameter regimes for the validity of MFVI and yield optimal rates of convergence for the variational posteriors to the true posteriors. Our analysis opens the door for future work on the validity of MFVI for other latent variable models, including obtaining tight bounds for models where previously only asymptotic results were known \cite{mukherjee2022variational, mukherjee2024naive}, and other latent variable models with general types of local latent variables. Further, our tight bounds on the log-partition function (error in the mean-field approximation) can be used to develop inferential methods in the framework of empirical Bayes where the hyperparameters in the model are estimated from data~\cite{mukherjee2023mean,fan2023gradient} (e.g., for LDA this would involve the estimation of the hyperparameters $\pmb{\alpha}$ and $\pmb{\eta}$), and have implications for model selection. 


\bibliographystyle{imsart-number} 
{\footnotesize \bibliography{VI.bib} }      


\newpage
\begin{appendix}

\section{Numerical studies}\label{Sect.4}

Recall the setup in Section \ref{Sect.1.3}. In this section, we present numerical studies to further demonstrate the substantial improvement of partially grouped VI over fully factorized VI for MMSB.

In our simulation studies, we take $K=2$, $n\in\{50, 100,200,400\}$, $\pmb{\alpha}=(1,1)$, and $\mathbf{B}=\begin{bmatrix}
    0.9 & 0.3 \\
    0.3 & 0.9
\end{bmatrix}$. 
The observations $\mathbf{X}$ are generated as described in Section \ref{Sect.1.3}. We apply both partially grouped VI and fully factorized VI to the generated data. The partially grouped VI for MMSB can be efficiently implemented using CAVI, with details presented in Appendix \ref{Appendix_F}. The details of the implementation of CAVI for fully factorized VI can be found in \cite{airoldi2008mixed}. We run both algorithms until convergence to obtain the partially grouped variational posterior $\hat{P}$ and the fully factorized variational posterior $\hat{P}^{\mathrm{ff}}$. From~\eqref{eq:VI-ELBO}, we can deduce that 
{\small
\begin{eqnarray*}
    &&\frac{1}{n^2}\KL\big(\hat{P}^{\mathrm{ff}}\, \big\|\; \mathbb{P}(\mathbf{Z}_{\rightarrow},\mathbf{Z}_{\leftarrow},\{\pmb{\pi}_{i}\}_{i\in [n]}\mid\mathbf{X}) \big)=\frac{\log\mathbb{p}(\mathbf{X})-\mathrm{ELBO}(\hat{P}^{\mathrm{ff}})}{n^2}\geq
     \frac{\mathrm{ELBO}(\hat{P})-\mathrm{ELBO}(\hat{P}^{\mathrm{ff}})}{n^2},
\end{eqnarray*}
}where $\log\mathbb{p}(\mathbf{X})$ is the log-marginal density of the data $\mathbf{X}$; see~\eqref{ELBO_def} for the definition of $\mathrm{ELBO}$. In Figure~\ref{Figure}, we display the ELBO and scaled ELBO (i.e., ELBO divided by $n^2$) of $\hat{P}$ and $\hat{P}^{\mathrm{ff}}$. As can be seen from the figure, as $n$ grows, the difference in scaled ELBO for $\hat P$ and $\hat{P}^{\mathrm{ff}}$ is non-vanishing and approaches a constant value, thereby implying that the right side of the above display converges to a strictly positive number.  Thus Figure~\ref{Figure} indicates that the fully factorized variational posterior $\hat{P}^{\mathrm{ff}}$ does not give an accurate approximation of the true posterior (cf.~Theorem \ref{Theorem_MMSB_UBD} for partially grouped VI).

To further illustrate the above phenomenon we examine the correlation between $\mathbbm{1}_{Z_{i\rightarrow j}=1}$ and $\mathbbm{1}_{Z_{i\leftarrow j}=1}$ for all distinct $i,j\in [n]$ given the observed data $\mathbf{X}$ (we consider $n=400$), under the partially grouped variational posterior $\hat{P}$ (which we use as a surrogate to the true posterior as we have proved its consistency). From the simulation results, among all pairs $(i,j)\in [n]^2$ such that $i\neq j$, $40\%$ of the pairs exhibit a correlation of $-0.75$, and the remaining $60\%$ show a correlation of $0.5$. This indicates that for all distinct $i,j\in [n]$, $Z_{i\rightarrow j}$ and $Z_{i\leftarrow j}$ are \textit{not approximately independent}. This explains why the fully factorized variational posterior $\hat{P}^{\mathrm{ff}}$---which assumes the independence between $Z_{i\rightarrow j}$ and $Z_{i\leftarrow j}$ for all distinct $i,j\in [n]$---does not give a close approximation to the true posterior. 

\begin{figure}[ht]
  \centering
  \begin{subfigure}[b]{0.57\textwidth}
    \centering
    \includegraphics[width=1\textwidth]{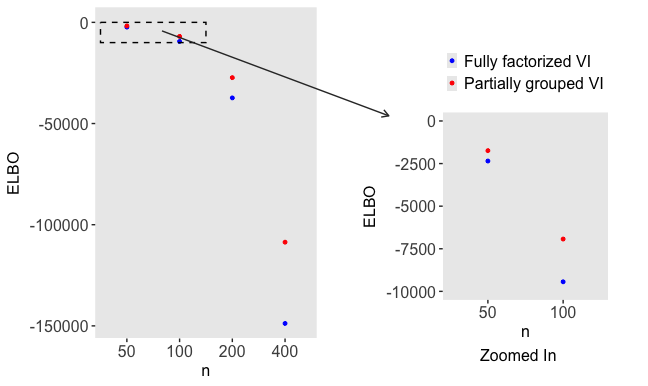}
    \caption{}
\label{Figure1.1}
  \end{subfigure}
  \hfill
  \begin{subfigure}[b]{0.4\textwidth}
    \centering
    \includegraphics[width=0.85\textwidth]{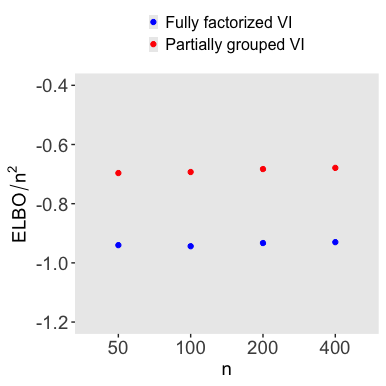}
    \caption{}
\label{Figure1.2}
  \end{subfigure}
  \caption{Comparison of ELBO and scaled ELBO for fully factorized VI and partially grouped VI. Panel (a) shows the ELBO values, with an inset providing a zoomed-in view (rescaled vertical axis) for $n=50,100$. Panel (b) presents the scaled ELBO values.}
  \label{Figure}
\end{figure}

\section{Proofs of Lemma \ref{Lemma2.1} and Theorems \ref{Theorem2.1} and \ref{Theorem2.2}}\label{Appendix_A}

In this section, we present the proofs of Lemma \ref{Lemma2.1} and Theorems \ref{Theorem2.1} and \ref{Theorem2.2}. We assume the notation and definitions presented in Sections \ref{Sect.1.1} and \ref{Sect.2}. 

Throughout this section, we use $C$ and $c$ to denote absolute positive constants. The specific values of these constants may change from line to line. 

\subsection{Preparatory results}

In this subsection, we establish preparatory results that will be used in the proofs of Proposition \ref{P2.1main} and Theorem \ref{Theorem2.1}.

\begin{lemma}\label{Lem2.1}
For any random variable $W$ satisfying $\mathbb{E}[W]\leq A$ and $\mathbb{E}[W^2]\leq B$ (where $A>0$ and $B\geq 0$), we have 
\begin{equation*}
    \mathbb{P}(W\leq 2A)\geq \frac{A^2}{4(A^2+B)}. 
\end{equation*}
\end{lemma}
\begin{proof}

If $B=0$, then $\mathbb{P}(W=0)=1$ and $\mathbb{P}(W\leq 2A)=1\geq 1\slash 4$. Below we assume that $B>0$. 

Consider any $M> 0$. When $W\leq -M$, we have $W\geq -W^2\slash M$. Hence
\begin{equation*}
    \mathbb{E}[W\mathbbm{1}_{W\leq -M}]\geq -\frac{1}{M} \mathbb{E}[W^2\mathbbm{1}_{W\leq -M}]\geq -\frac{1}{M}\mathbb{E}[W^2]\geq -\frac{B}{M}.
\end{equation*}
Therefore, we have 
\begin{eqnarray*}
   A & \geq & \mathbb{E}[W]=\mathbb{E}[W\mathbbm{1}_{W\leq -M}]+\mathbb{E}[W\mathbbm{1}_{-M<W\leq 2A}]+\mathbb{E}[W\mathbbm{1}_{W> 2A}]\nonumber\\
   &\geq& -\frac{B}{M}-M\,\mathbb{P}(-M<W\leq 2A)+2A\,\mathbb{P}(W>2A)\nonumber\\
   &=&-\frac{B}{M}-M(\mathbb{P}(W\leq 2A)-\mathbb{P}(W\leq -M))+2A(1-\mathbb{P}(W\leq 2A)) \nonumber\\
   &\geq& -\frac{B}{M}+2A-(2A+M)\,\mathbb{P}(W\leq 2A),
\end{eqnarray*}
which leads to
\begin{equation*}
    \mathbb{P}(W\leq 2A)\geq \frac{A-B\slash M}{2A+M}. 
\end{equation*}
Taking $M=2B\slash A$, we arrive at the conclusion of the lemma. 
\end{proof}

\begin{lemma}\label{Lem2.2}
With $T_{i,\ell}(y)$ as in Definition~\ref{DefTy} and $\Phi_{i,j}(y)$ as in Definition~\ref{DefSmoothnessTerm}, for any $y\in \prod_{i=1}^n[0,1]^{K_i}$, $i,j\in [n]$, and $\ell\in [K_i]$, we have 
\begin{equation*}
   \big|T_{i,\ell}\big(\UI{y}{j}\big)-T_{i,\ell}(y)\big|\leq\Phi_{i,j}(y)\min\big\{T_{i,\ell}(y),T_{i,\ell}\big(\UI{y}{j}\big)\big\}.
\end{equation*}
\end{lemma}
\begin{proof}

Consider any $y\in\prod_{i=1}^n[0,1]^{K_i}$, $i,j\in [n]$, and $\ell\in[K_i],\ell'\in [K_j]$. By (\ref{T.exp}),
\begin{equation}\label{Eq2.1}
    \frac{\partial T_{i,\ell}(y)}{\partial y_{j,\ell'}} = \begin{cases}
        \frac{e^{F(\UII{y}{i}{\ell})}\mu_i(\ell)\sum_{s=1}^{K_i} e^{F(\UII{y}{i}{s})}\mu_i(s)(F_{j,\ell'}(\UII{y}{i}{\ell})-F_{j,\ell'}(\UII{y}{i}{s}))}{\big(\sum_{s=1}^{K_i} e^{F(\UII{y}{i}{s})}\mu_i(s)\big)^2} & \text{ if }i\neq j\\
        0 & \text{ if } i=j.
    \end{cases}
\end{equation}
Note that for any $s\in [K_i]$,
\begin{eqnarray}\label{Eq2.2}
  F_{j,\ell'}\big(\UII{y}{i}{\ell}\big)-F_{j,\ell'}\big(\UII{y}{i}{s}\big)
  &=& \int_0^1\big(F_{j,\ell';i,\ell}\big(t\UII{y}{i}{\ell}+(1-t)\UII{y}{i}{s}\big)\nonumber\\
  && \quad-F_{j,\ell';i,s}\big(t\UII{y}{i}{\ell}+(1-t)\UII{y}{i}{s}\big)\big)dt.\nonumber\\
  &&
\end{eqnarray}

For any $t\in [0,1]$, we define $g(t):=T_{i,\ell}\big(t y + (1-t)\UI{y}{j}\big)$. Using (\ref{Eq2.2}) with $y$ replaced by $ty+(1-t)\UI{y}{j}$, recalling Definition \ref{Local_Hess}, we get
\begin{eqnarray*}
   && \big|F_{j,\ell'}\big(U^{(1)}\big(t y + (1-t)\UI{y}{j};i,\ell\big)\big)-F_{j,\ell'}\big(U^{(1)}\big(t y + (1-t)\UI{y}{j};i,s\big)\big)\big|\nonumber\\
   &\leq& 2\max_{s\in [K_i]}\big\{\HI{j}{\ell'}{i}{s}{y}\big\}
\end{eqnarray*}
for any $t\in[0,1]$ and $s\in [K_i]$. Plugging this into~\eqref{Eq2.1} and noting~\eqref{T.exp}, we get 
\begin{eqnarray*}
     \bigg|\frac{\partial T_{i,\ell}\big(ty+(1-t)\UI{y}{j}\big)}{\partial y_{j,\ell'}}\bigg|\leq 2\max_{s\in [K_i]}\big\{\HI{j}{\ell'}{i}{s}{y}\big\} \, T_{i,\ell}\big(ty+(1-t)\UI{y}{j}\big) 
\end{eqnarray*}
for any $t\in[0,1]$. Hence 
\begin{equation*}
  \big|g'(t)\big|=\bigg|\sum_{\ell'=1}^{K_j}\frac{\partial T_{i,\ell}\big(ty+(1-t)\UI{y}{j}\big)}{\partial y_{j,\ell'}} y_{j,\ell'}\bigg|\leq 2\bigg(\sum_{\ell'=1}^{K_j} \max_{s\in [K_i]}\big\{\HI{j}{\ell'}{i}{s}{y}\big\}y_{j,\ell'}\bigg)g(t). 
\end{equation*} 
By Gr\"onwall's inequality (\cite[Theorem 1.1 in Chapter III]{PH}) and noting that $\HI{j}{\ell'}{i}{s}{y}=\HI{i}{s}{j}{\ell'}{y}$, for any $t\in[0,1]$, we get
\begin{equation*}
   \frac{g(0)}{1+\Phi_{i,j}(y)} \leq g(t)\leq (1+\Phi_{i,j}(y))g(0). 
\end{equation*}
Taking $t=1$ in the above display, we obtain that
\begin{equation*}
     \frac{T_{i,\ell}(y)}{1+\Phi_{i,j}(y)}\leq T_{i,\ell}\big(\UI{y}{j}\big)\leq (1+\Phi_{i,j}(y))T_{i,\ell}(y). 
\end{equation*}
Hence we have
\begin{equation*}
   \big|T_{i,\ell}\big(\UI{y}{j}\big)-T_{i,\ell}(y)\big| 
   \leq  \Phi_{i,j}(y)\min\big\{T_{i,\ell}(y),T_{i,\ell}\big(\UI{y}{j}\big)\big\},
\end{equation*}
which completes the proof.
\end{proof}

\subsection{Proofs of Propositions \ref{P2.1main} and \ref{P2.2main}}

In this subsection, we establish the proofs of Propositions \ref{P2.1main} and \ref{P2.2main}.

\begin{proof}[Proof of Proposition \ref{P2.1main}]

For any $y\in\hat{\mathcal{Y}}_{n,\mathbf{K}}$, with $T_{i,\ell}(y)$ as in Definition \ref{DefTy}, we have
\begin{align}\label{Eq2.5}
    \Delta_1(y) &=\int_0^1\sum_{i=1}^n\sum_{\ell=1}^{K_i} F_{i,\ell}(ty+(1-t)T(y))(y_{i,\ell}-T_{i,\ell}(y))dt-\sum_{i=1}^n\sum_{\ell=1}^{K_i} F_{i,\ell}(y)(y_{i,\ell}-T_{i,\ell}(y))\nonumber\\
    \qquad & =\int_0^1\sum_{i=1}^n\sum_{\ell=1}^{K_i} (F_{i,\ell}(ty+(1-t)T(y))-F_{i,\ell}(y))(y_{i,\ell}-T_{i,\ell}(y))dt.
\end{align}
Hence for any $y\in\hat{\mathcal{Y}}_{n,\mathbf{K}}$, 
\begin{equation*}
    |\Delta_1(y)|\leq 2\sum_{i=1}^n\sum_{\ell=1}^{K_i} b_{i,\ell}(y_{i,\ell}+T_{i,\ell}(y))\leq 2\sum_{i=1}^n\max_{\ell\in [K_i]}\{b_{i,\ell}\}\sum_{\ell=1}^{K_i}(y_{i,\ell}+T_{i,\ell}(y))\leq 4\sum_{i=1}^n\max_{\ell\in [K_i]}\{b_{i,\ell}\}.
\end{equation*}
Therefore, we have
\begin{equation*}
    \mathbb{E}\big[\Delta_1(\mathbf{Y})^2\mid\mathbf{X}\big]\leq 16\bigg(\sum_{i=1}^n\max_{\ell\in[K_i]}\{b_{i,\ell}\}\bigg)^2.
\end{equation*}

Let $\UI{\mathbf{Y}}{i}$ be as defined in~\eqref{def_Utilde}. For any $t\in [0,1]$ and $i\in [n], \ell\in [K_i]$, as $F_{i,\ell}\big(t\UI{\mathbf{Y}}{i}+(1-t)T\big(\UI{\mathbf{Y}}{i}\big)\big)-F_{i,\ell}\big(\UI{\mathbf{Y}}{i}\big)$ does not depend on $Z_i$, and $T_{i,\ell}(\mathbf{Y})=\mathbb{E}\big[Y_{i,\ell}\mid (Z_j)_{j\in[n]\backslash\{i\}},\mathbf{X}\big]$ (see Definition \ref{DefTy}), taking expectations using the tower property gives
\begin{equation*}
\mathbb{E}\big[\big(F_{i,\ell}\big(t\UI{\mathbf{Y}}{i}+(1-t)T\big(\UI{\mathbf{Y}}{i}\big)\big)-F_{i,\ell}\big(\UI{\mathbf{Y}}{i}\big)\big)(Y_{i,\ell}-T_{i,\ell}(\mathbf{Y}))\mid\mathbf{X}\big]=0.
\end{equation*}
Hence by (\ref{Eq2.5}), 
{\small 
\begin{eqnarray}\label{Eq2.14}
   &&\mathbb{E}[\Delta_1(\mathbf{Y})\mid\mathbf{X}]\nonumber\\
   &=& \int_0^1\sum_{i=1}^n\sum_{\ell=1}^{K_i} \mathbb{E}\Big[\Big(\big(F_{i,\ell}(t\mathbf{Y}+(1-t)T(\mathbf{Y}))-F_{i,\ell}\big(t\UI{\mathbf{Y}}{i}+(1-t)T\big(\UI{\mathbf{Y}}{i}\big)\big)\big)\nonumber\\
    &&\quad\quad\quad\quad\quad\quad\quad  -\big(F_{i,\ell}(\mathbf{Y})-F_{i,\ell}\big(\UI{\mathbf{Y}}{i}\big)\big)\Big)\big(Y_{i,\ell}-T_{i,\ell}(\mathbf{Y})\big)\,\Big|\,\mathbf{X}\Big]dt\nonumber\\
    &=& \int_{0}^1\mathbb{E}\bigg[\sum_{i=1}^n\sum_{\ell=1}^{K_i} \big(F_{i,\ell}(t\mathbf{Y}+(1-t)T(\mathbf{Y}))-F_{i,\ell}\big(t\UI{\mathbf{Y}}{i}+(1-t)T\big(\UI{\mathbf{Y}}{i}\big)\big)\big)\nonumber\\
    &&\hspace{3.55in}\times\big(Y_{i,\ell}-T_{i,\ell}(\mathbf{Y})\big)\,\bigg|\,\mathbf{X}\bigg]dt\nonumber\\
    &&\hspace{0.6in}-\int_{0}^1\mathbb{E}\bigg[\sum_{i=1}^n\sum_{\ell=1}^{K_i}\big(F_{i,\ell}(\mathbf{Y})-F_{i,\ell}\big(\UI{\mathbf{Y}}{i}\big)\big)\big(Y_{i,\ell}-T_{i,\ell}(\mathbf{Y})\big)\,\bigg|\,\mathbf{X}\bigg]dt.
\end{eqnarray}
}For bounding the second term in the RHS of \eqref{Eq2.14}, fixing $i\in [n], \ell\in [K_i]$, by Definition \ref{Local_Hess_new} we have
\begin{equation*}
    \big|F_{i,\ell}(\mathbf{Y})-F_{i,\ell}\big(\UI{\mathbf{Y}}{i}\big)\big|\leq \sum_{\ell'=1}^{K_i} \HII{i}{\ell}{i}{\ell'}{\mathbf{Y}}{1} Y_{i,\ell'}. 
\end{equation*}
Hence the second term in the RHS of \eqref{Eq2.14} can be bounded as follows:
\begin{eqnarray}\label{Eq2.12}
  &&   \sum_{i=1}^n\sum_{\ell=1}^{K_i} \big|\big(F_{i,\ell}(\mathbf{Y})-F_{i,\ell}\big(\UI{\mathbf{Y}}{i}\big)\big)(Y_{i,\ell}-T_{i,\ell}(\mathbf{Y}))\big|\nonumber\\
  &\leq& \sum_{i=1}^n
\sum_{\ell=1}^{K_i}\sum_{\ell'=1}^{K_i} \HII{i}{\ell}{i}{\ell'}{\mathbf{Y}}{1}Y_{i,\ell'}(Y_{i,\ell}+T_{i,\ell}(\mathbf{Y}))\nonumber\\
&\leq& \sum_{i=1}^n\max_{\ell\in [K_i]}\bigg\{\sum_{\ell'=1}^{K_i} \HII{i}{\ell}{i}{\ell'}{\mathbf{Y}}{1} Y_{i,\ell'}\bigg\}\sum_{\ell=1}^{K_i}(Y_{i,\ell}+T_{i,\ell}(\mathbf{Y}))\nonumber\\
&=& 2\sum_{i=1}^n\max_{\ell\in [K_i]}\bigg\{\sum_{\ell'=1}^{K_i} \HII{i}{\ell}{i}{\ell'}{\mathbf{Y}}{1} Y_{i,\ell'}\bigg\}\nonumber\\
&\leq& 2\sup_{y\in\hat{\mathcal{Y}}_{n,\mathbf{K}}}\bigg\{\sum_{i=1}^n\max_{\ell\in [K_i]}\bigg\{\sum_{\ell'=1}^{K_i} \HII{i}{\ell}{i}{\ell'}{y}{1} y_{i,\ell'}\bigg\}\bigg\}. 
\end{eqnarray}

Proceeding to bound the first term in the RHS of \eqref{Eq2.14}, fix $i,j\in [n],\ell\in[K_i], \ell'\in [K_j], t\in [0,1]$, and use Lemma \ref{Lem2.2} to get 
\begin{equation*}
   \big|T_{j,\ell'}(\mathbf{Y})-T_{j,\ell'}\big(\UI{\mathbf{Y}}{i}\big)\big|\leq \Phi_{j,i}(\mathbf{Y})\min\big\{T_{j,\ell'}(\mathbf{Y}),T_{j,\ell'}\big(\UI{\mathbf{Y}}{i}\big)\big\},
\end{equation*}
where $\Phi_{j,i}(Y)$ is as defined in Definition~\ref{DefSmoothnessTerm}. Consequently, we have 
\begin{eqnarray}\label{Eq2.7}
   && \big|\big(tY_{j,\ell'}+(1-t)T_{j,\ell'}(\mathbf{Y})\big)-\big(t\UIsubscr{j}{\ell'}{\mathbf{Y}}{i}+(1-t)T_{j,\ell'}\big(\UI{\mathbf{Y}}{i}\big)\big)\big|\nonumber\\
   &\leq& (1-t)\Phi_{j,i}(\mathbf{Y})\min\big\{T_{j,\ell'}(\mathbf{Y}),T_{j,\ell'}\big(\UI{\mathbf{Y}}{i}\big)\big\}+tY_{i,\ell'}\mathbbm{1}_{j=i}.
\end{eqnarray}
Using (\ref{Eq2.7}), for any $j\in [n]\backslash \{i\}$ we have, 
\begin{eqnarray*}
    && \big|\big(tY_{j,\ell'}+(1-t)T_{j,\ell'}(\mathbf{Y})\big)-\big(t\UIsubscr{j}{\ell'}{\mathbf{Y}}{i}+(1-t)T_{j,\ell'}\big(\UI{\mathbf{Y}}{i}\big)\big)\big|\nonumber\\
    &\leq&   \Phi_{j,i}(\mathbf{Y})\min\big\{(1-t)T_{j,\ell'}(\mathbf{Y}),(1-t)T_{j,\ell'}\big(\UI{\mathbf{Y}}{i}\big)\big\}\\
    &\leq&\Phi_{j,i}(\mathbf{Y})\min\big\{tY_{j,\ell'}+(1-t)T_{j,\ell'}(\mathbf{Y}),t\UIsubscr{j}{\ell'}{\mathbf{Y}}{i}+(1-t)T_{j,\ell'}\big(\UI{\mathbf{Y}}{i}\big)\big\}\\
    &\le &(M_i-1)\min\big\{tY_{j,\ell'}+(1-t)T_{j,\ell'}(\mathbf{Y}),t\UIsubscr{j}{\ell'}{\mathbf{Y}}{i}+(1-t)T_{j,\ell'}\big(\UI{\mathbf{Y}}{i}\big)\big\},
\end{eqnarray*}
where the last inequality uses~\eqref{Bounds.Phi}, which in turn gives
$$\frac{1}{M_i}\leq\frac{tY_{j,\ell'}+(1-t)T_{j,\ell'}(\mathbf{Y})}{t\UIsubscr{j}{\ell'}{\mathbf{Y}}{i}+(1-t)T_{j,\ell'}\big(\UI{\mathbf{Y}}{i}\big)}\le M_i.$$ Using the above conclusion, and recalling $H^{(\mathrm{II})}_{i,\ell,j,\ell'}$ from Definition \ref{Local_Hess_new}, we have
\begin{eqnarray*}
  &&  \big|F_{i,\ell}(t\mathbf{Y}+(1-t)T(\mathbf{Y}))-F_{i,\ell}\big(t\UI{\mathbf{Y}}{i}+(1-t)T\big(\UI{\mathbf{Y}}{i}\big)\big)\big|\nonumber\\
  &\leq& \sum_{j=1}^n\sum_{\ell'=1}^{K_j} \HII{i}{\ell}{j}{\ell'}{t\mathbf{Y}+(1-t)T(\mathbf{Y})}{M_i}\nonumber\\
  &&\quad\quad\times \big|\big(tY_{j,\ell'}+(1-t)T_{j,\ell'}(\mathbf{Y})\big)-\big(t\UIsubscr{j}{\ell'}{\mathbf{Y}}{i}+(1-t)T_{j,\ell'}\big(\UI{\mathbf{Y}}{i}\big)\big)\big|\nonumber\\
  &\leq& \sum_{j=1}^n\sum_{\ell'=1}^{K_j} \HII{i}{\ell}{j}{\ell'}{t\mathbf{Y}+(1-t)T(\mathbf{Y})}{M_i}\times \big((1-t)\Phi_{j,i}(\mathbf{Y})T_{j,\ell'}(\mathbf{Y})+tY_{i,\ell'}\mathbbm{1}_{j=i}\big)\nonumber\\
  &\leq& \max_{j\in[n]}\big\{\Phi_{j,i}(\mathbf{Y})\big\}\sum_{j=1}^n\sum_{\ell'=1}^{K_j}(1-t)T_{j,\ell'}(\mathbf{Y})\HII{i}{\ell}{j}{\ell'}{t\mathbf{Y}+(1-t)T(\mathbf{Y})}{M_i}\nonumber\\
  && \quad\quad+ \sum_{\ell'=1}^{K_i} tY_{i,\ell'}\HII{i}{\ell}{i}{\ell'}{t\mathbf{Y}+(1-t)T(\mathbf{Y})}{M_i},
\end{eqnarray*}
where the second inequality uses
(\ref{Eq2.7}). Using the above bound and summing over $i,\ell$ we have
{\small \begin{eqnarray}\label{Eq2.15}
    && \sum_{i=1}^n\sum_{\ell=1}^{K_i} \big|\big(F_{i,\ell}(t\mathbf{Y}+(1-t)T(\mathbf{Y}))-F_{i,\ell}\big(t\UI{\mathbf{Y}}{i}+(1-t)T\big(\UI{\mathbf{Y}}{i}\big)\big)\big) \big(Y_{i,\ell}-T_{i,\ell}(\mathbf{Y})\big)\big|\nonumber\\
    &\leq& \sum_{i=1}^n\max_{\ell\in [K_i]}\big\{\big|F_{i,\ell}(t\mathbf{Y}+(1-t)T(\mathbf{Y}))-F_{i,\ell}\big(t\UI{\mathbf{Y}}{i}+(1-t)T\big(\UI{\mathbf{Y}}{i}\big)\big)\big|\big\} \sum_{\ell=1}^{K_i}\big(Y_{i,\ell}+T_{i,\ell}(\mathbf{Y})\big)\nonumber\\
    &=& 2\sum_{i=1}^n\max_{\ell\in [K_i]}\big\{\big|F_{i,\ell}(t\mathbf{Y}+(1-t)T(\mathbf{Y}))-F_{i,\ell}\big(t\UI{\mathbf{Y}}{i}+(1-t)T\big(\UI{\mathbf{Y}}{i}\big)\big)\big|\big\}\nonumber\\
    &\leq& 2\sup_{y\in\hat{\mathcal{Y}}_{n,\mathbf{K}}}\bigg\{\sum_{i=1}^n\max_{j\in[n]}\big\{\Phi_{j,i}(y)\big\}\bigg\}\sup_{\substack{y\in\hat{\mathcal{Y}}_{n,\mathbf{K}},\\i\in[n],\ell\in[K_i]}}\bigg\{\sum_{j=1}^n\sum_{\ell'=1}^{K_j} \HII{i}{\ell}{j}{\ell'}{y}{M_i} y_{j,\ell'}\bigg\}\nonumber\\
    &&\hspace{1.2in} +2\sup_{y\in\hat{\mathcal{Y}}_{n,\mathbf{K}}}\bigg\{\sum_{i=1}^n\max_{\ell\in[K_i]}\bigg\{\sum_{\ell'=1}^{K_i} \HII{i}{\ell}{i}{\ell'}{y}{M_{i}} y_{i,\ell'}\bigg\}\bigg\}.
\end{eqnarray}
}By (\ref{Eq2.14}), (\ref{Eq2.12}), and (\ref{Eq2.15}), noting that $M_i\geq 1$ for every $i\in[n]$, we have 
\begin{eqnarray*}
    |\mathbb{E}[\Delta_1(\mathbf{Y})\mid\mathbf{X}]|
   &\leq& 2\sup_{y\in\hat{\mathcal{Y}}_{n, \mathbf{K}}}\bigg\{\sum_{i=1}^n\max_{j\in[n]}\big\{\Phi_{j,i}(y)\big\}\bigg\}\cdot\sup_{\substack{y\in\hat{\mathcal{Y}}_{n,\mathbf{K}},\\i\in[n],\ell\in[K_i]}}\bigg\{\sum_{j=1}^n\sum_{\ell'=1}^{K_j} \HII{i}{\ell}{j}{\ell'}{y}{M_i} y_{j,\ell'}\bigg\}\nonumber\\
    &&\hspace{1in} +4\sup_{y\in\hat{\mathcal{Y}}_{n,\mathbf{K}}}\bigg\{\sum_{i=1}^n\max_{\ell\in[K_i]}\bigg\{\sum_{\ell'=1}^{K_i} \HII{i}{\ell}{i}{\ell'}{y}{M_i} y_{i,\ell'}\bigg\}\bigg\}\leq 4E_1,
\end{eqnarray*}
which is our desired bound.
\end{proof}

\begin{proof}[Proof of Proposition \ref{P2.2main}]

For any $y\in\hat{\mathcal{Y}}_{n,\mathbf{K}}$, using the expressions of $\Delta_2(\cdot),I(\cdot),J(\cdot,\cdot)$ in \eqref{eq:error_term2}, \eqref{def_I}, and~\eqref{def_J}, respectively, we have 
\begin{equation*}
    \Delta_2(y)
    =\sum_{i=1}^n\sum_{\ell=1}^{K_i} \log\bigg(\frac{T_{i,\ell}(y)}{\mu_{i}(\ell)}\bigg)(T_{i,\ell}(y)-y_{i,\ell})+\sum_{i=1}^n\sum_{\ell=1}^{K_i} F_{i,\ell}(y)(y_{i,\ell}-T_{i,\ell}(y)).
\end{equation*}
For any $i\in[n]$, using the definition of $T_{i,\ell}(y)$ in Definition \ref{DefTy}, we have
\begin{eqnarray*}
    &&\sum_{\ell=1}^{K_i} \log\bigg(\frac{T_{i,\ell}(y)}{\mu_{i}(\ell)}\bigg)(T_{i,\ell}(y)-y_{i,\ell})\nonumber\\
    &=& \sum_{\ell=1}^{K_i} F\big(\UII{y}{i}{\ell}\big)(T_{i,\ell}(y)-y_{i,\ell}) -\log\bigg(\sum_{s=1}^{K_i} e^{F(\UII{y}{i}{s})}\mu_i(s)\bigg)\sum_{\ell=1}^{K_i} (T_{i,\ell}(y)-y_{i,\ell})\nonumber\\
    &=& \sum_{\ell=1}^{K_i} F\big(\UII{y}{i}{\ell}\big)(T_{i,\ell}(y)-y_{i,\ell}),
\end{eqnarray*}
where we use $\sum_{\ell=1}^{K_i}T_{i,\ell}(y)=\sum_{\ell=1}^{K_i}y_{i,\ell}=1$ for all $i\in[n]$ in the last line. Hence,
\begin{eqnarray}\label{Eq2.11}
    && \hspace{-0.55in} \Delta_2(y)=\sum_{i=1}^n\sum_{\ell=1}^{K_i}\big(F\big(\UII{y}{i}{\ell}\big)-F_{i,\ell}(y)\big)(T_{i,\ell}(y)-y_{i,\ell})\nonumber\\
     &=& \sum_{i=1}^n\sum_{\ell=1}^{K_i}\big(F\big(\UII{y}{i}{\ell}\big)-F\big(\UI{y}{i}\big)-F_{i,\ell}(y)\big)(T_{i,\ell}(y)-y_{i,\ell}). 
\end{eqnarray}
By (\ref{Eq2.11}) and the mean-value theorem, we have
\begin{eqnarray*}
    |\Delta_2(y)| &\leq&2\sum_{i=1}^n\sum_{\ell=1}^{K_i} b_{i,\ell}(T_{i,\ell}(y)+y_{i,\ell})  
 \nonumber\\
    &\leq& 2\sum_{i=1}^n\max_{\ell\in[K_i]}\{b_{i,\ell}\}\sum_{\ell=1}^{K_i}(T_{i,\ell}(y)+y_{i,\ell})= 4\sum_{i=1}^n\max_{\ell\in[K_i]}\{b_{i,\ell}\}. 
\end{eqnarray*}
Therefore, we have
\begin{equation*}
    \mathbb{E}\big[\Delta_2(\mathbf{Y})^2\mid\mathbf{X}\big]\leq 16\bigg(\sum_{i=1}^n\max_{\ell\in[K_i]}\{b_{i,\ell}\}\bigg)^2.
\end{equation*}

For any $i\in[n]$ and $\ell\in [K_i]$, as $F\big(\UII{\mathbf{Y}}{i}{\ell}\big)-F\big(\UI{\mathbf{Y}}{i}\big)-F_{i,\ell}\big(\UI{\mathbf{Y}}{i}\big)$ only depends on $(Z_j)_{j\in[n]\backslash\{i\}}$, taking expectations using the tower property gives
\begin{equation*}
    \mathbb{E}\big[\big(F\big(\UII{\mathbf{Y}}{i}{\ell}\big)-F\big(\UI{\mathbf{Y}}{i}\big)-F_{i,\ell}\big(\UI{\mathbf{Y}}{i}\big)\big)\big(T_{i,\ell}(\mathbf{Y})-Y_{i,\ell}\big)\mid\mathbf{X}\big]=0, 
\end{equation*}
where we use the fact that $T_{i,\ell}(\mathbf{Y})=\mathbb{E}\big[Y_{i,\ell}\mid (Z_j)_{j\in[n]\backslash\{i\}},\mathbf{X}\big]$ (see Definition \ref{DefTy}). Hence by (\ref{Eq2.11}), 
\begin{equation*}
    \mathbb{E}[\Delta_2(\mathbf{Y})\mid\mathbf{X}]=\sum_{i=1}^n\sum_{\ell=1}^{K_i}\mathbb{E}\big[\big(F_{i,\ell}\big(\UI{\mathbf{Y}}{i}\big)-F_{i,\ell}(\mathbf{Y})\big)\big(T_{i,\ell}(\mathbf{Y})-Y_{i,\ell}\big)\mid\mathbf{X}\big],
\end{equation*}
which on invoking~\eqref{Eq2.12} gives
\begin{equation*}
    |\mathbb{E}[\Delta_2(\mathbf{Y})\mid\mathbf{X}]|\leq 2\sup_{y\in\hat{\mathcal{Y}}_{n,\mathbf{K}}}\bigg\{\sum_{i=1}^n\max_{\ell\in [K_i]}\bigg\{\sum_{\ell'=1}^{K_i} \HII{i}{\ell}{i}{\ell'}{y}{1} y_{i,\ell'}\bigg\}\bigg\}\leq 2E_1,
\end{equation*}
where the last inequality uses the fact that the quantity $M_i$ in the definition of $E_1$ (see Definition~\ref{DefSmoothnessTerm}) is at least $1$.
\end{proof}

\subsection{Proof of Lemma \ref{Lemma2.1}}\label{Proof:Lemma2.1}

For any $y\in\hat{\mathcal{Y}}_{n,\mathbf{K}}$, with the collapsed posterior $\mathbb{P}(\mathbf{Z}\mid\mathbf{X})$ as in~\eqref{def_pro} (with normalizing constant $S_{n,\mathbf{K}}$ as in~\eqref{defS}), we have 
\begin{eqnarray*}
  \KL\big(Q_y \,\big\| \; \mathbb{P}(\mathbf{Z}\mid\mathbf{X}) \big)&=&\sum_{z\in \prod_{i=1}^n [K_i]}Q_y(z)\log\bigg(\frac{Q_y(z)}{S_{n,\mathbf{K}}^{-1}e^{f(z)}\mu(z)}\bigg) \nonumber\\
  &=& \sum_{i=1}^n\sum_{\ell=1}^{K_i} y_{i,\ell}\log\Big(\frac{y_{i,\ell}}{\mu_i(\ell)}\Big)-\mathbb{E}_{Q_y}[f(\mathbf{Z})]+\log S_{n,\mathbf{K}} \nonumber\\
  &=& \log S_{n,\mathbf{K}} -\mathbb{E}_{Q_y}[f(\mathbf{Z})]+I(y),
\end{eqnarray*}
where $I$ is as defined in~\eqref{def_I}. Hence noting that $\mathcal{Q}=\{Q_y:y\in\hat{\mathcal{Y}}_{n,\mathbf{K}}\}$, using~\eqref{vari3.2}, we get
\begin{eqnarray*}
     \KL\big(\hat{Q}\,\big\| \; \mathbb{P}(\mathbf{Z}\mid\mathbf{X})\big)&=&\inf_{y\in\hat{\mathcal{Y}}_{n,\mathbf{K}}} \KL\big(Q_y \,\big\| \; \mathbb{P}(\mathbf{Z}\mid\mathbf{X}) \big) \nonumber\\
     &=&\log S_{n,\mathbf{K}} -\sup_{y\in\hat{\mathcal{Y}}_{n,\mathbf{K}}}\{\mathbb{E}_{Q_y}[f(\mathbf{Z})]-I(y)\},
\end{eqnarray*}
which proves the first conclusion of the lemma.
\\

For the second conclusion, fixing any $y\in\hat{\mathcal{Y}}_{n,\mathbf{K}},i\in[n],\ell\in [K_i],t\in [0,1]$, and with $G(\mathbf{Z})$ defined as in~\eqref{eq:G}, we have
\begin{eqnarray}\label{fact1.1}
    && \mathbb{E}_{Q_y}\big[F_{i,\ell}\big(t\UI{G(\mathbf{Z})}{i}+(1-t)y\big)\cdot(G_{i,\ell}(\mathbf{Z})-y_{i,\ell})\big]\nonumber\\
    &=& \mathbb{E}_{Q_y}\big[F_{i,\ell}\big(t\UI{G(\mathbf{Z})}{i}+(1-t)y\big)\big]\cdot \left[\mathbb{E}_{Q_y}[G_{i,\ell}(\mathbf{Z})]-y_{i,\ell}\right]=0,
\end{eqnarray}
where we have used the facts that $Z_1,\cdots,Z_n$ are mutually independent under $Q_y$ and $F_{i,\ell}\big(t\UI{G(\mathbf{Z})}{i}+(1-t)y\big)$ does not depend on $Z_i$, and $\mathbb{E}_{Q_y}[G_{i,\ell}(\mathbf{Z})]=\mathbb{E}_{Q_y}[\mathbbm{1}_{Z_i=\ell}]=Q_y(Z_i=\ell)=y_{i,\ell}$. Hence for any $y\in\hat{\mathcal{Y}}_{n,\mathbf{K}}$, noting that $f(\mathbf{Z})=F(G(\mathbf{Z}))$ (see~\eqref{eq:FG}), and using a one-term Taylor expansion with the remainder form, we get
{\small \begin{align*}\label{E.A3}
    &|\mathbb{E}_{Q_y}[f(\mathbf{Z})]-F(y)|=|\mathbb{E}_{Q_y}[F(G(\mathbf{Z}))]-F(y)|\nonumber\\
    =&\bigg|\sum_{i=1}^n\sum_{\ell=1}^{K_i} \int_0^1 \mathbb{E}_{Q_y}[F_{i,\ell}(tG(\mathbf{Z})+(1-t)y)(G_{i,\ell}(\mathbf{Z})-y_{i,\ell})]dt\bigg|\nonumber\\
   = &  \bigg|\sum_{i=1}^n\sum_{\ell=1}^{K_i} \int_0^1 \mathbb{E}_{Q_y}\big[\big(F_{i,\ell}(tG(\mathbf{Z})+(1-t)y)-F_{i,\ell}\big(t\UI{G(\mathbf{Z})}{i}+(1-t)y\big)\big)\cdot(G_{i,\ell}(\mathbf{Z})-y_{i,\ell})\big]dt\bigg| \nonumber\\
    \leq & \sum_{i=1}^n\sum_{\ell=1}^{K_i}\sum_{\ell'=1}^{K_i} c_{i,\ell;i,\ell'} \mathbb{E}_{Q_y}[G_{i,\ell'}(\mathbf{Z})(G_{i,\ell}(\mathbf{Z})+y_{i,\ell})]\leq 2\sum_{i=1}^n\max_{\ell,\ell'\in[K_i]} \{c_{i,\ell;i,\ell'}\},
\end{align*}}where we use~\eqref{fact1.1} in the third line, the mean value theorem to obtain the first inequality in the last line, and the fact that $\sum_{\ell'=1}^{K_i}G_{i,\ell'}(\mathbf{Z}) = \sum_{\ell=1}^{K_i}G_{i,\ell}(\mathbf{Z})=\sum_{\ell=1}^{K_i}y_{i,\ell}=1$ to obtain the last inequality. This verifies the second conclusion.
\\

Finally, if $F(\cdot)$ is a convex function on $\prod_{i=1}^n[0,1]^{K_i}$, by Jensen's inequality, we have
\begin{equation*}
    \mathbb{E}_{Q_y}[f(\mathbf{Z})]=\mathbb{E}_{Q_y}[F(G(\mathbf{Z}))]\geq F(\mathbb{E}_{Q_y}[G(\mathbf{Z})])=F(y).
\end{equation*}
This verifies the third conclusion, and hence completes the proof.
\qed

\subsection{Proof of Theorem \ref{Theorem2.1}}\label{App:Proof-Upp-Bd}

We begin with the proof of part (a), which is the main content of the theorem. Fix $\epsilon,t>0$ such that $t\geq E_1$, where $E_1$ is as in Definition \ref{DefSmoothnessTerm}. Let $\mathbf{Z}\sim \mathbb{P}(\mathbf{Z}\mid\mathbf{X})$, and let $\mathbf{Y}:=G(\mathbf{Z})$, where $G(\mathbf{Z})$ is as in~\eqref{eq:G}. By Propositions \ref{P2.1main} and \ref{P2.2main}, we have 
\begin{align*}
    |\mathbb{E}[\Delta_1(\mathbf{Y})+\Delta_2(\mathbf{Y})\mid\mathbf{X}]|&\leq CE_1\leq  Ct,\\
    \mathbb{E}\big[(\Delta_1(\mathbf{Y})+\Delta_2(\mathbf{Y}))^2\mid\mathbf{X}\big]&\leq 2\big(\mathbb{E}\big[\Delta_1(\mathbf{Y})^2\mid\mathbf{X}\big]+\mathbb{E}\big[\Delta_2(\mathbf{Y})^2\mid\mathbf{X}\big]\big)\leq C\bigg(\sum_{i=1}^n\max_{\ell\in[K_i]}\{b_{i,\ell}\}\bigg)^2.
\end{align*}
Hence by Lemma \ref{Lem2.1} (with $A=Ct$ and $B=C\big(\sum_{i=1}^n\max_{\ell\in[K_i]}\{b_{i,\ell}\}\big)^2$), there exist absolute positive constants $C_0$ and $c_0$, such that 
\begin{equation*}
    \mathbb{P}(\Delta_1(\mathbf{Y})+\Delta_2(\mathbf{Y})\leq C_0 t\mid\mathbf{X})\geq \frac{c_0 t^2}{t^2+\Big(\sum_{i=1}^n\max_{\ell\in[K_i]}\{b_{i,\ell}\}\Big)^2}.
\end{equation*}
Thus denoting 
\begin{equation*}
    \mathcal{A}(t):=\bigg\{z\in\prod_{i=1}^n [K_i]:\Delta_1(G(z))+\Delta_2(G(z))\leq C_0 t\bigg\},
\end{equation*}
and noting~\eqref{def_pro}, we have
\begin{equation*}
    \frac{\sum_{z\in\mathcal{A}(t)}e^{f(z)}\mu(z)}{\sum_{z\in\prod_{i=1}^n [K_i]}e^{f(z)}\mu(z)}\geq \frac{c_0  t^2}{t^2+\Big(\sum_{i=1}^n\max_{\ell\in[K_i]}\{b_{i,\ell}\}\Big)^2}.
\end{equation*}
Hence, with $S_{n,\mathbf{K}}$ as defined in~\eqref{defS}, we get 
\begin{eqnarray}\label{Eq.A4}
    \log S_{n,\mathbf{K}} &=&\log\bigg(\sum_{z\in\prod_{i=1}^n [K_i]}e^{f(z)}\mu(z)\bigg)\nonumber\\
    &\leq& \log\bigg(\sum_{z\in\mathcal{A}(t)}e^{f(z)}\mu(z)\bigg)+\log\Bigg(1+\frac{\Big(\sum_{i=1}^n\max_{\ell\in[K_i]}\{b_{i,\ell}\}\Big)^2}{t^2}\Bigg)+C \nonumber\\
     &\leq& \log\bigg(\sum_{z\in\mathcal{A}(t)}e^{f(z)}\mu(z)\bigg)+\log\Bigg(2\cdot\max\bigg\{1,\frac{\sum_{i=1}^n\max_{\ell\in[K_i]}\{b_{i,\ell}\}}{t}\bigg\}^2\Bigg)+C \nonumber\\
      &\leq& \log\bigg(\sum_{z\in\mathcal{A}(t)}e^{f(z)}\mu(z)\bigg)+2\log\bigg(\max\bigg\{1,\frac{\sum_{i=1}^n\max_{\ell\in[K_i]}\{b_{i,\ell}\}}{t}\bigg\}\bigg)+C \nonumber\\
      &\leq& \log\bigg(\sum_{z\in\mathcal{A}(t)}e^{f(z)}\mu(z)\bigg)+C\log\bigg(2+\frac{\sum_{i=1}^n\max_{\ell\in[K_i]}\{b_{i,\ell}\}}{t}\bigg)\nonumber\\
    \hspace{0.7in} &=& \log\bigg(\sum_{z\in\mathcal{A}(t)}e^{f(z)}\mu(z)\bigg)+ C \Lambda(t), 
\end{eqnarray}
where $\Lambda(t)$ is as defined in~\eqref{def_L}. Proceeding to bound the first term in the RHS of~\eqref{Eq.A4}, using the expressions of $\Delta_1(\cdot)$ and $\Delta_2(\cdot)$ in Definition~\ref{Def.A3}, for any $y\in\hat{\mathcal{Y}}_{n,\mathbf{K}}$, we get
\begin{equation*}
    \Delta_1(y)+\Delta_2(y)=F(y)-\big(F(T(y))-I(T(y))+J(y,T(y))\big).
\end{equation*}
Hence noting that $f(z)=F(G(z))$ for all $z\in\prod_{i=1}^n[K_i]$ (see~\eqref{eq:FG}), we get
\begin{eqnarray}\label{Eq.A1}
    && \log\bigg(\sum_{z\in\mathcal{A}(t)}e^{f(z)}\mu(z)\bigg)=\log\bigg(\sum_{z\in\mathcal{A}(t)}e^{F(G(z))}\mu(z)\bigg)\nonumber\\
    &\leq& \log\bigg(\sum_{z\in\mathcal{A}(t)}e^{F(T(G(z)))-I(T(G(z)))+J(G(z),T(G(z)))}\mu(z)\bigg)+Ct\nonumber\\
    &\leq& \log\bigg(\sum_{z\in\mathcal{A}(t)}e^{\sup_{w\in\hat{\mathcal{Y}}_{n,\mathbf{K}}}\{F(w)-I(w)\}+J(G(z),T(G(z)))}\mu(z)\bigg)+Ct\nonumber\\
    &=& \sup_{w\in\hat{\mathcal{Y}}_{n,\mathbf{K}}}\{F(w)-I(w)\}+\log\bigg(\sum_{z\in\mathcal{A}(t)}e^{J(G(z),T(G(z)))}\mu(z)\bigg)+Ct.
\end{eqnarray}

For any $i\in [n]$ and $x=(x_{s})_{s\in [K_i]}\in \mathbb{R}^{K_i}$, we have
\begin{equation}\label{sumto1}
    \sum_{\ell\in [K_i]}\rho_{i,\ell}(x)=1,\quad\mbox{ where }\quad\rho_{i,\ell}(x):=\frac{ e^{x_{\ell}}\mu_i(\ell)}{\sum_{s=1}^{K_i} e^{x_s}\mu_i(s)}\mbox{ for every }\ell\in [K_i].
\end{equation}
For any $i\in [n]$, $\ell,\ell'\in [K_i]$, and $x=(x_{s})_{s\in [K_i]}\in \mathbb{R}^{K_i}$, we have
\begin{equation}\label{gradlogrho}
    \frac{\partial \rho_{i,\ell}(x)}{\partial x_{\ell'}}=\rho_{i,\ell}(x)(\mathbbm{1}_{\ell=\ell'}-\rho_{i,\ell'}(x)) \quad\Rightarrow\quad  \frac{\partial \log(\rho_{i,\ell}(x))}{\partial x_{\ell'}} = \mathbbm{1}_{\ell=\ell'}-\rho_{i,\ell'}(x).
\end{equation}
For any $\epsilon>0$, with $\mathcal{D}(\epsilon)$ defined as in Definition~\ref{DefComplexityTerm}, we define
\begin{align*}
    \mathcal{D}'(\epsilon):=\Big\{(p_{i,\ell})_{i\in [n],\ell\in[K_i]}:&\, p_{i,\ell}=\rho_{i,\ell}((d_{i,s})_{s\in [K_i]})\text{ for all }i\in[n],\ell\in[K_i],\nonumber\\
    & \text{ where }(d_{i,\ell})_{i\in [n],\ell\in[K_i]}\in \mathcal{D}(\epsilon)\Big\},
\end{align*}
and note that $|\mathcal{D}'(\epsilon)|\leq |\mathcal{D}(\epsilon)|$. Fix $y\in\hat{\mathcal{Y}}_{n,\mathbf{K}}$, and define $\kappa(i,y):=(\kappa_{\ell}(i,y))_{\ell\in[K_i]}$ for every $i\in [n]$, where 
\begin{equation}\label{defkappa}
     \kappa_{\ell}(i, y):= F\big(\UII{y}{i}{\ell}\big)-F\big(\UI{y}{i}\big),\mbox{ for every }\ell\in[K_i].
\end{equation}
Then, using the definition of $\mathcal{D}(\epsilon)$ (see~\eqref{Eq2.17}), there exists $d=(d_{i,\ell})_{i\in[n],\ell\in[K_i]}\in\mathcal{D}(\epsilon)$ such that
\begin{equation}\label{kappaineq}
   \sum_{i=1}^n \max_{\ell\in[K_i]}|\kappa_{\ell}(i,y)-d_{i,\ell}|\leq \epsilon.
\end{equation}
For any $i\in [n]$ and $\ell\in [K_i]$, using the definitions of $T_{i,\ell}(y)$ and $\rho_{i,\ell}(\cdot)$ from Definition \ref{DefTy} and~\eqref{sumto1}, we have $T_{i,\ell}(y)=\rho_{i,\ell}(\kappa(i,y))$. We set $p_{i,\ell}:=\rho_{i,\ell}(d_i)$ for any $i\in [n]$ and $\ell\in [K_i]$, where $d_i:=(d_{i,\ell})_{\ell\in [K_i]}\in\mathbb{R}^{K_i}$. Setting $p:=(p_{i,\ell})_{i\in [n],\ell\in [K_i]}\in\hat{\mathcal{Y}}_{n,\mathbf{K}}$, for any $i\in[n]$ and $\ell\in[K_i]$, using~\eqref{defkappa} we get 
{\small \begin{eqnarray*}
   && |\log(T_{i,\ell}(y))-\log p_{i,\ell}|=|\log(\rho_{i,\ell}(\kappa(i,y)))-\log(\rho_{i,\ell}(d_i))| \nonumber\\
   &\leq& |\kappa_{\ell}(i,y)-d_{i,\ell}|  +
   \sum_{\ell'\in [K_i]\backslash\{\ell\}}\bigg(\int_0^1\rho_{i,\ell'}(t\kappa(i,y)+(1-t)d_i)dt\bigg)|\kappa_{\ell'}(i,y)-d_{i,\ell'}|\nonumber\\
   &\leq& \max_{s\in[K_i]}\{|\kappa_s(i,y)-d_{i,s}|\}\cdot\bigg(1+\int_0^1\bigg(\sum_{\ell'\in[K_i]\backslash\{\ell\}}\rho_{i,\ell'}(t\kappa(i,y)+(1-t)d_i)\bigg)dt\bigg)\nonumber\\
   &\leq& 2\max_{s\in[K_i]}\big\{\big|\kappa_s(i,y)-d_{i,s}\big|\big\},
\end{eqnarray*}
}where we use~\eqref{gradlogrho} in the second line and~\eqref{sumto1} in the last line. Hence using the expression of $J(\cdot,\cdot)$ in~\eqref{def_J}, we get
\begin{eqnarray}\label{Eq.A2}
    && |J(y,T(y))-J(y,p)|\leq\sum_{i=1}^n\sum_{\ell=1}^{K_i} y_{i,\ell}|\log(T_{i,\ell}(y))-\log p_{i,\ell}|\nonumber\\
    &\leq& 2\sum_{i=1}^n\max_{s\in[K_i]}\{|\kappa_s(i,y)-d_{i,s}|\}\sum_{\ell=1}^{K_i} y_{i,\ell}= 2\sum_{i=1}^n\max_{s\in[K_i]}\{|\kappa_s(i,y)-d_{i,s}|\}\leq  2\epsilon, 
\end{eqnarray}
where we use $\sum_{\ell=1}^{K_i}y_{i,\ell}=1$ for all $i\in[n]$ and~\eqref{kappaineq} in the last line. Combining, for any $z\in \prod_{i=1}^n[K_i]$ (note that $G(z)\in\hat{\mathcal{Y}}_{n,\mathbf{K}}$), there exists $p\in\mathcal{D}'(\epsilon)$, such that
\begin{equation}\label{new.eq}
    |J(G(z),T(G(z)))-J(G(z),p)|\leq 2\epsilon.
\end{equation}

For any $p\in\mathcal{D}'(\epsilon)$, we denote by $\mathcal{C}(p)$ the set of $z\in\mathcal{A}(t)$ such that~\eqref{new.eq} holds. Sequentially applying (\ref{Eq.A1}) and (\ref{Eq.A2}), we get
\begin{eqnarray}\label{Eq.A7}
    && \log\bigg(\sum_{z\in\mathcal{A}(t)}e^{f(z)}\mu(z)\bigg)\nonumber\\
    &\leq& \sup_{w\in\hat{\mathcal{Y}}_{n,\mathbf{K}}}\{F(w)-I(w)\}+\log\bigg(\sum_{z\in\mathcal{A}(t)}e^{J(G(z),T(G(z)))}\mu(z)\bigg)+Ct\nonumber\\
    &\leq& \sup_{w\in\hat{\mathcal{Y}}_{n,\mathbf{K}}}\{F(w)-I(w)\}+\log\bigg(\sum_{p\in\mathcal{D}'(\epsilon)}\sum_{z\in\mathcal{C}(p)}e^{J(G(z),T(G(z)))}\mu(z)\bigg)+Ct\nonumber\\
    &\leq& \sup_{w\in\hat{\mathcal{Y}}_{n,\mathbf{K}}}\{F(w)-I(w)\}+\log\bigg(\sum_{p\in\mathcal{D}'(\epsilon)}\sum_{z\in\mathcal{C}(p)}e^{J(G(z),p)}\mu(z)\bigg)+Ct+2\epsilon\nonumber\\
    &\leq& \sup_{w\in\hat{\mathcal{Y}}_{n,\mathbf{K}}}\{F(w)-I(w)\}+\log\bigg(\sum_{p\in\mathcal{D}'(\epsilon)}\sum_{z\in \prod_{i=1}^n [K_i]}\prod_{i=1}^n\prod_{\ell=1}^{K_i} p_{i,\ell}^{\mathbbm{1}_{z_i=\ell}}\bigg)+Ct+2\epsilon\nonumber\\
    &\leq& \sup_{w\in\hat{\mathcal{Y}}_{n,\mathbf{K}}}\{F(w)-I(w)\}+\log(|\mathcal{D}'(\epsilon)|)+Ct+2\epsilon\nonumber\\
    &\leq& \sup_{w\in\hat{\mathcal{Y}}_{n,\mathbf{K}}}\{F(w)-I(w)\}+\log(|\mathcal{D}(\epsilon)|)+Ct+2\epsilon\nonumber\\
    &=& \sup_{w\in\hat{\mathcal{Y}}_{n,\mathbf{K}}}\{F(w)-I(w)\}+Ct+E_2(\epsilon), 
\end{eqnarray}
where we recall the definition of $J(\cdot,\cdot)$ from Definition \ref{Def.A3} in the fifth line and the definition of $E_2(\epsilon)$ from Definition \ref{DefComplexityTerm} in the last line. By (\ref{Eq.A4}) and (\ref{Eq.A7}), we conclude that 
\begin{equation}\label{E.A5}
    \log S_{n,\mathbf{K}} \leq \sup_{y\in\hat{\mathcal{Y}}_{n,\mathbf{K}}}\{F(y)-I(y)\}+C(t+\Lambda(t))+E_2(\epsilon),
\end{equation}
which proves the first conclusion in part (a).

For the second conclusion in part (a), sequentially applying parts (a) and (b) of Lemma \ref{Lemma2.1}, along with the observation $\KL\big(\hat{Q}\,\big\| \; \mathbb{P}(\mathbf{Z}\mid\mathbf{X})\big)\geq 0$, we have 
\begin{equation*}
    \log S_{n,\mathbf{K}} \geq \sup_{y\in\hat{\mathcal{Y}}_{n,\mathbf{K}}}\{\mathbb{E}_{Q_y}[f(\mathbf{Z})]-I(y)\}\geq \sup_{y\in\hat{\mathcal{Y}}_{n,\mathbf{K}}}\{F(y)-I(y)\}-2\sum_{i=1}^n\max_{\ell,\ell'\in[K_i]} \{c_{i,\ell;i,\ell'}\},
\end{equation*}
which verifies the second conclusion in part (a). 

Proceeding to the proof of part (b), by parts (a) and (b) of Lemma \ref{Lemma2.1} and~\eqref{E.A5}, we get
\begin{eqnarray*}
    &&\frac{1}{n}\KL\big(\hat{Q}\,\big\| \; \mathbb{P}(\mathbf{Z}\mid\mathbf{X}) \big)=\frac{\log S_{n,\mathbf{K}}-\sup_{y\in\hat{\mathcal{Y}}_{n,\mathbf{K}}}\{\mathbb{E}_{Q_y}[f(\mathbf{Z})]-I(y)\}}{n}\nonumber\\
    &\leq& 
    \frac{\log S_{n,\mathbf{K}}-\sup_{y\in\hat{\mathcal{Y}}_{n,\mathbf{K}}}\{F(y)-I(y)\}+2\sum_{i=1}^n\max_{\ell,\ell'\in[K_i]} \{c_{i,\ell;i,\ell'}\}}{n}\nonumber\\
    &\leq& \frac{C(t+\Lambda(t))+E_2(\epsilon)+2\sum_{i=1}^n\max_{\ell,\ell'\in[K_i]} \{c_{i,\ell;i,\ell'}\}}{n}.
\end{eqnarray*} 

Finally, if $F(\cdot)$ is a convex function on $\prod_{i=1}^n [0,1]^{K_i}$, sequentially applying parts (a) and (c) of Lemma \ref{Lemma2.1} and using $\KL\big(\hat{Q}\,\big\| \; \mathbb{P}(\mathbf{Z}\mid\mathbf{X})\big)\geq 0$, we get 
\begin{equation*}
    \log S_{n,\mathbf{K}} \geq \sup_{y\in\hat{\mathcal{Y}}_{n,\mathbf{K}}}\{\mathbb{E}_{Q_y}[f(\mathbf{Z})]-I(y)\}\geq \sup_{y\in\hat{\mathcal{Y}}_{n,\mathbf{K}}}\{F(y)-I(y)\},
\end{equation*}
which proves the first conclusion in part (c). Moreover, by parts (a) and (c) of Lemma \ref{Lemma2.1} and~\eqref{E.A5}, we have
\begin{eqnarray*}
     &&\frac{1}{n}\KL\big(\hat{Q}\,\big\| \; \mathbb{P}(\mathbf{Z}\mid\mathbf{X}) \big)=\frac{\log S_{n,\mathbf{K}}-\sup_{y\in\hat{\mathcal{Y}}_{n,\mathbf{K}}}\{\mathbb{E}_{Q_y}[f(\mathbf{Z})]-I(y)\}}{n}\nonumber\\
     &\leq&  \frac{\log S_{n,\mathbf{K}}-\sup_{y\in\hat{\mathcal{Y}}_{n,\mathbf{K}}}\{F(y)-I(y)\}}{n} \leq\frac{C(t+\Lambda(t))+E_2(\epsilon)}{n},
\end{eqnarray*}
which establishes the second conclusion in part (c). 

\subsection{Proof of Theorem \ref{Theorem2.2}}\label{App:Proof-Upp-Bd-Full}

In the following, we fix an arbitrary $y\in\hat{\mathcal{Y}}_{n,\mathbf{K}}$, and let $Q_y$ be defined as in~\eqref{defQy}. Let $\mathbf{Z}\sim Q_{y}$. For any $\pmb{\theta}\in\prod_{j=1}^m\mathbb{R}^{d_j}$, $t\in [0,1]$, and $i\in [n],\ell\in[K_i]$, with $G(\mathbf{Z})$ defined as in~\eqref{eq:G}, as $Z_1,\cdots,Z_n$ are mutually independent under $Q_y$ and $R_{i,\ell}\big(\pmb{\theta},t\UI{G(\mathbf{Z})}{i}+(1-t)y\big)$ only depends on $(Z_j)_{j\in[n]\backslash [i]}$, we have
\begin{eqnarray}\label{fact1.2}
    && \mathbb{E}_{Q_y}\big[R_{i,\ell}\big(\pmb{\theta},t\UI{G(\mathbf{Z})}{i}+(1-t)y\big)(G_{i,\ell}(\mathbf{Z})-y_{i,\ell})\big]\nonumber\\
    &=& \mathbb{E}_{Q_y}\big[R_{i,\ell}\big(\pmb{\theta},t\UI{G(\mathbf{Z})}{i}+(1-t)y\big)\big]\cdot (\mathbb{E}_{Q_y}[G_{i,\ell}(\mathbf{Z})]-y_{i,\ell})=0,
\end{eqnarray}
where we use $\mathbb{E}_{Q_y}[G_{i,\ell}(\mathbf{Z})]=\mathbb{E}_{Q_y}[\mathbbm{1}_{Z_i=\ell}]=Q_y(Z_i=\ell)=y_{i,\ell}$. Thus by (\ref{def_R}), for any $\pmb{\theta}\in\mathrm{supp}(\nu)$, we have 
\begin{align}\label{def_ne}
    &\big|\mathbb{E}_{Q_y}[r(\pmb{\theta},\mathbf{Z})] - R(\pmb{\theta},y)\big| = \big|\mathbb{E}_{Q_y}[R(\pmb{\theta},G(\mathbf{Z}))]-R(\pmb{\theta},y)\big|\nonumber\\
    =& \bigg|\sum_{i=1}^n\sum_{\ell=1}^{K_i} \int_{0}^1 \mathbb{E}_{Q_y}[R_{i,\ell}(\pmb{\theta},tG(\mathbf{Z})+(1-t)y)(G_{i,\ell}(\mathbf{Z})-y_{i,\ell})]dt\bigg|\nonumber\\
    \leq& \sum_{i=1}^n\sum_{\ell=1}^{K_i} \bigg|\int_{0}^1 \mathbb{E}_{Q_y}\big[\big(R_{i,\ell}(\pmb{\theta},tG(\mathbf{Z})+(1-t)y)-R_{i,\ell}\big(\pmb{\theta},t\UI{G(\mathbf{Z})}{i}+(1-t)y\big)\big)\nonumber\\
    &\hspace{1.1in}\cdot(G_{i,\ell}(\mathbf{Z})-y_{i,\ell})\big] dt\bigg|\nonumber\\
    \leq& \sum_{i=1}^n\sum_{\ell=1}^{K_i}\sum_{\ell'=1}^{K_i} \ctilde_{i,\ell;i,\ell'} \mathbb{E}_{Q_y}[G_{i,\ell'}(\mathbf{Z})(G_{i,\ell}(\mathbf{Z})+y_{i,\ell})]\leq 2\sum_{i=1}^n\max_{\ell,\ell'\in [K_i]}\{\ctilde_{i,\ell;i,\ell'}\},
\end{align}
where $\ctilde_{i,\ell;i;\ell'}$ is as in Definition \ref{Global_GH}, and we use~\eqref{fact1.2} in the third line and the fact that $\sum_{\ell=1}^{K_i} G_{i,\ell}(\mathbf{Z})=\sum_{\ell=1}^{K_i} y_{i,\ell}=1$ for all $i\in[n]$ in the last line. 

Denote by $\mathcal{R}$ the family of mean-field distributions on $\pmb{\theta}\equiv(\pmb{\theta}_1,\cdots,\pmb{\theta}_m)$, i.e., probability distributions on $\prod_{j=1}^m\mathbb{R}^{d_j}$ whose density w.r.t. $\nu$ has form $\prod_{j=1}^m \xi_j(\pmb{\theta}_j)$ (see~\eqref{defmu}). For each $j\in [m]$, denote by $\mathscr{P}(\mathbb{R}^{d_j})$ the family of probability distributions on $\mathbb{R}^{d_j}$. By~\eqref{def_form.R}, we get
\begin{eqnarray*}
&&\sup_{\nu'\in\mathcal{R}}\big\{\mathbb{E}_{\nu'}[R(\pmb{\theta},y)]-\KL(\nu'\|\nu)\big\}\nonumber\\
&=& \sup_{\nu_j'\in \mathscr{P}(\mathbb{R}^{d_j}),\forall j\in[m]}\bigg\{R_0(y)+\sum_{j=1}^m \mathbb{E}_{\nu_j'}[R_j(\pmb{\theta}_j,y)]-\sum_{j=1}^m\KL(\nu_j'\|\nu_j)\bigg\}\nonumber\\
&=& R_0(y)+\sum_{j=1}^m \sup_{\nu_j'\in \mathscr{P}(\mathbb{R}^{d_j})}\big\{\mathbb{E}_{\nu_j'}[R_j(\pmb{\theta}_j,y)]-\KL(\nu_j'\|\nu_j)\big\}\nonumber\\
&=& R_0(y)+\sum_{j=1}^m\log\bigg(\int_{\mathbb{R}^{d_j}}e^{R_j(\pmb{\theta}_j,y)}d\nu_j(\pmb{\theta}_j)\bigg)=\log\bigg(\int_{\prod_{j=1}^m\mathbb{R}^{d_j}}e^{R_0(y)+\sum_{j=1}^m R_j(\pmb{\theta}_j,y)}d\nu(\pmb{\theta})\bigg)\nonumber\\
&=& \log\bigg(\int_{\prod_{j=1}^m\mathbb{R}^{d_j}}e^{R(\pmb{\theta},y)}d\nu(\pmb{\theta})\bigg)=F(y),
\end{eqnarray*}
where we use the Gibbs variational principle (see, for example, \cite[Lemma 4.10]{van2014probability}) in the fourth line, and use~\eqref{eq:F} in the last line. Hence
\begin{eqnarray}\label{neweq}
   && \bigg|\sup_{\nu'\in\mathcal{R}}\big\{\mathbb{E}_{\nu'\otimes Q_y}[r(\pmb{\theta},\mathbf{Z})]-\KL(\nu'\|\nu)\big\}-F(y)\bigg|\nonumber\\
   &\leq& \sup_{\nu'\in\mathcal{R}}\big\{\big|\mathbb{E}_{\nu'}[\mathbb{E}_{Q_y}[r(\pmb{\theta},\mathbf{Z})]-R(\pmb{\theta},y)]\big|\big\}\nonumber\\
   &&+\bigg|\sup_{\nu'\in\mathcal{R}}\big\{\mathbb{E}_{\nu'}[R(\pmb{\theta},y)]-\KL(\nu'\|\nu)\big\}-F(y)\bigg|\nonumber\\
   &=& \sup_{\nu'\in\mathcal{R}}\big\{\big|\mathbb{E}_{\nu'}[\mathbb{E}_{Q_y}[r(\pmb{\theta},\mathbf{Z})]-R(\pmb{\theta},y)]\big|\big\}
   \leq 2\sum_{i=1}^n\max_{\ell,\ell'\in [K_i]}\{\ctilde_{i,\ell;i,\ell'}\},
\end{eqnarray}
where $\pmb{\theta}$ and $\mathbf{Z}$ are independent in the second line, and we use~\eqref{def_ne} in the last inequality.

For any $\nu'\in\mathcal{R}$ and $y\in\hat{\mathcal{Y}}_{n,\mathbf{K}}$, using the density~\eqref{eq:Post-Form} of the posterior distribution $\mathbb{P}(\pmb{\theta},\mathbf{Z}\mid\mathbf{X})$ (note that $\gamma=\nu\otimes \mu$; see~\eqref{base}-\eqref{defmu}), we get 
\begin{eqnarray}\label{DKL_general}
 \KL\big(\nu'\otimes Q_y \,\big\| \; \mathbb{P}(\pmb{\theta},\mathbf{Z}\mid\mathbf{X}) \big)&=&\int_{\prod_{j=1}^m\mathbb{R}^{d_j}\times\prod_{i=1}^n[K_i]}\log\Bigg(\frac{\frac{d \nu'\otimes Q_y}{d\gamma}(\pmb{\theta},z)}{S_{n,\mathbf{K}}^{-1}e^{r(\pmb{\theta},z)}}\Bigg)d\nu'\otimes Q_y(\pmb{\theta},z)\nonumber\\   
  &=&  \log S_{n,\mathbf{K}}-\mathbb{E}_{\nu'\otimes Q_y}[r(\pmb{\theta},\mathbf{Z})]\nonumber\\
  &&+\int_{\prod_{j=1}^m\mathbb{R}^{d_j}\times\prod_{i=1}^n[K_i]}\log\bigg(\frac{d\nu'\otimes Q_y}{d\gamma}(\pmb{\theta},z)\bigg)d\nu'\otimes Q_y(\pmb{\theta},z)\nonumber\\
  &=& \log S_{n,\mathbf{K}}-\big(\mathbb{E}_{\nu'\otimes Q_y}[r(\pmb{\theta},\mathbf{Z})]-\KL(\nu'\otimes Q_y\|\nu\otimes \mu)\big).
\end{eqnarray}
Hence noting that $\mathcal{P}=\{\nu'\otimes \mathcal{Q}_y:\nu'\in\mathcal{R},y\in\hat{\mathcal{Y}}_{n,\mathbf{K}}\}$, using~\eqref{vari3.1}, \eqref{DKL_general}, and $I(y)=\KL(Q_y\|\mu)$ (see
the definition of $I(\cdot)$ in~\eqref{def_I}), we get for any $\epsilon,t>0$ such that $t\geq E_1$, 
\begin{eqnarray*}
   && \KL\big(\hat{P} \,\big\| \; \mathbb{P}(\pmb{\theta},\mathbf{Z}\mid\mathbf{X}) \big)=\inf_{\nu'\in\mathcal{R},y\in\hat{\mathcal{Y}}_{n,\mathbf{K}}}\KL\big(\nu'\otimes Q_y\,\big\| \; \mathbb{P}(\pmb{\theta},\mathbf{Z}\mid\mathbf{X})\big)\nonumber\\
   &=& \log S_{n,\mathbf{K}}-\sup_{\nu'\in\mathcal{R},y\in\hat{\mathcal{Y}}_{n,\mathbf{K}}}\big\{\mathbb{E}_{\nu'\otimes Q_y}[r(\pmb{\theta},\mathbf{Z})]-\KL(\nu'\otimes Q_y\|\nu\otimes \mu)\big\}\nonumber\\
   &=& \log S_{n,\mathbf{K}}-\sup_{\nu'\in\mathcal{R},y\in\hat{\mathcal{Y}}_{n,\mathbf{K}}}\big\{\mathbb{E}_{\nu'\otimes Q_y}[r(\pmb{\theta},\mathbf{Z})]-\KL(\nu'\|\nu)-I(y)\big\}\nonumber\\
   &=& \log S_{n,\mathbf{K}}-\sup_{y\in\hat{\mathcal{Y}}_{n,\mathbf{K}}}\bigg\{\sup_{\nu'\in\mathcal{R}}\big\{\mathbb{E}_{\nu'\otimes Q_y}[r(\pmb{\theta},\mathbf{Z})]-\KL(\nu'\|\nu)\big\}-I(y)\bigg\}\nonumber\\
   &\leq& \log S_{n,\mathbf{K}}-\sup_{y\in\hat{\mathcal{Y}}_{n,\mathbf{K}}}\big\{F(y)-I(y)\big\}+2\sum_{i=1}^n\max_{\ell,\ell'\in [K_i]}\{\ctilde_{i,\ell;i,\ell'}\}\nonumber\\
   &\leq& C(t + \Lambda(t)) +E_2(\epsilon)+2\sum_{i=1}^n\max_{\ell,\ell'\in [K_i]}\{\ctilde_{i,\ell;i,\ell'}\},
\end{eqnarray*}
where we use~\eqref{neweq} in the fifth line and the first conclusion in part (a) of Theorem \ref{Theorem2.1} in the last line.
This completes the proof of Theorem \ref{Theorem2.2}. 

\section{Variational inference for MMSB: Proof of Theorem \ref{Theorem_MMSB_UBD}}\label{Appendix_B}

In this section, we present the proof of Theorem \ref{Theorem_MMSB_UBD} based on the general meta-theorems (Theorems \ref{Theorem2.1} and \ref{Theorem2.2}), thus verifying the validity of partially grouped VI for MMSB.

Recall the model setup and notation for MMSB presented in Section \ref{Sect.1.3}. We denote by $\mathcal{S}_{n,K}$ the normalizing constant of the posterior (\ref{Eq4.1}). Note that $\mathcal{S}_{n,K}$ is also the normalizing constant of the collapsed posterior (\ref{mmsb}). Throughout this section, we use $C,c$ to denote positive constants that depend only on $C_0$ (as specified in Theorem \ref{Theorem_MMSB_UBD}). The values of these constants may change from line to line.

For any $z=(z_{i,j})_{i,j\in[n]:i\neq j}\in [K]^{2n(n-1)}$ (where for each $(i,j)\in [n]^2$ such that $i\neq j$, $z_{i,j}=(z_{i\rightarrow j},z_{i\leftarrow j})$ with $z_{i\rightarrow j},z_{i\leftarrow j}\in [K]$) and any $i\in[n],\ell\in[K]$, by a slight abuse of notation (see~\eqref{defNil}), we define
\begin{equation}\label{Nildef}
    N_{\rightarrow,i,\ell}(z):=|\{j\in [n]\backslash \{i\}: z_{i\rightarrow j}=\ell\}|,\quad N_{\leftarrow,i,\ell}(z):=|\{j\in [n]\backslash\{i\}: z_{j\leftarrow i}=\ell\}|,
\end{equation}
and $N_{i,\ell}(z):=N_{\rightarrow,i,\ell}(z)+N_{\leftarrow,i,\ell}(z)$. For any $z\in [K]^{2n(n-1)}$ and any $\ell,\ell'\in[K]$, we define
\begin{equation}\label{def_Al}
    A_{\ell,\ell'}(z):=|\{(i,j)\in [n]^2:i\neq j, z_{i \rightarrow j}=\ell, z_{i\leftarrow j}=\ell'\}|,
\end{equation}
\begin{equation}\label{def_Ml}
    M_{\ell,\ell'}(z):=|\{(i,j)\in [n]^2:i\neq j, z_{i \rightarrow j}=\ell, z_{i\leftarrow j}=\ell', X_{i,j}=1\}|.
\end{equation}
Recall that the normalizing constant\slash partition function of the posterior\slash collapsed posterior is (see~\eqref{mmsb})
\begin{equation}\label{eq:SnK_MMSB}
    \mathcal{S}_{n,K}=\sum_{z\in [K]^{2n(n-1)}} e^{\Upsilon(z)},
\end{equation}
where for any $z\in [K]^{2n(n-1)}$,
{\small
\begin{eqnarray}\label{defUps:MMSB} \Upsilon(z)&:=&\sum_{\ell=1}^K\sum_{\ell'=1}^K(M_{\ell,\ell'}(z)\log B_{\ell,\ell'}+(A_{\ell,\ell'}(z)-M_{\ell,\ell'}(z))\log(1-B_{\ell,\ell'}))\nonumber\\
    && \hspace{0.4in}+\sum_{i=1}^n\sum_{\ell=1}^K\log\Gamma(N_{i,\ell}(z)+\alpha_{\ell})-n\log\Gamma\Big(2n-2+\sum_{\ell=1}^K \alpha_{\ell}\Big).
\end{eqnarray}
}

To apply the general framework in Sections \ref{Sect.1.1} and \ref{Sect.2.1} to MMSB, we set the global latent variables to be $\pmb{\theta}=\pmb{\pi}:= (\pmb{\pi}_1,\cdots,\pmb{\pi}_n)$ and the local latent variables to be $\mathbf{Z}=(Z_{i,j})_{i,j\in[n]:i\neq j}$, where $Z_{i,j}:=(Z_{i\rightarrow j},Z_{i\leftarrow j})\in [K]^2$. Thus, $\{(i,j)\in [n]^2:i\neq j\}$ plays the role of $[n]$ in Sections \ref{Sect.1.1} and \ref{Sect.2.1} (for indexing the local latent variables), and $n,K_i$ in relevant bounds in Section \ref{Sect.2.1} will be replaced by $n(n-1),K^2$, respectively (since there are $n(n-1)$ local latent variables, and the number of categories of each local latent variable is $K^2$). We define
{\small\begin{eqnarray}\label{YnK:MMSB}
    \hat{\mathcal{Y}}_{n,K}  &:=&\bigg\{y=(y_{i,j;\ell,\ell'})_{i,j\in [n],i\neq j;\ell,\ell'\in[K]}\in [0,1]^{n(n-1)K^2}:\nonumber\\
    &&\quad\quad   \sum_{\ell=1}^K\sum_{\ell'=1}^K
    y_{i,j;\ell,\ell'}=1
    \text{ for every } i,j\in [n]\text{ with }i\neq j\bigg\}.
\end{eqnarray}
}For each $i\in [n]$, we define $\nu_i$ to be the Dirichlet distribution of order $K$ with parameters $(1,\cdots,1)$. For each $i,j\in [n]$ with $i\neq j$, we define $\mu_{i,j}$ to be the probability measure on $[K]^2$ such that for each $(\ell,\ell')\in [K]^2$,
\begin{equation}\label{def_muij}
    \mu_{i,j}(\ell,\ell') = \begin{cases}
        \frac{B_{\ell,\ell'}}{\sum_{s=1}^K\sum_{s'=1}^K B_{s,s'}} & \text{ if } X_{i,j}=1\\
        \frac{1-B_{\ell,\ell'}}{\sum_{s=1}^K\sum_{s'=1}^K (1-B_{s,s'})} & \text{ if } X_{i,j}=0.
    \end{cases}
\end{equation} 
We set $\nu:=\nu_1\otimes\cdots\otimes\nu_n$ and $\mu:=\bigotimes\limits_{i,j\in [n]:i\neq j}\mu_{i,j}$ (recall~\eqref{base} and~\eqref{defmu}). Note that for any $z\in[K]^{2n(n-1)}$,
\begin{eqnarray}\label{defmu:MMSB}
    \mu(z)&=&\prod_{i,j\in [n]:i\neq j}\mu_{i,j}(z_{i,j})\nonumber\\
    &=&\prod_{\substack{i,j\in [n]:\\ i\neq j,X_{i,j}=1}}\frac{B_{z_{i\rightarrow j},z_{i\leftarrow j}}}{\sum_{\ell=1}^K\sum_{\ell'=1}^K B_{\ell,\ell'}}\cdot \prod_{\substack{i,j\in [n]:\\ i\neq j,X_{i,j}=0}}\frac{1-B_{z_{i\rightarrow j},z_{i\leftarrow j}}}{\sum_{\ell=1}^K\sum_{\ell'=1}^K (1-B_{\ell,\ell'})}\nonumber\\
    &=& \frac{\prod_{\ell=1}^K\prod_{\ell'=1}^K B_{\ell,\ell'}^{M_{\ell,\ell'}(z)}(1-B_{\ell,\ell'})^{A_{\ell,\ell'}(z)-M_{\ell,\ell'}(z)}}{\big(\sum_{\ell=1}^K\sum_{\ell'=1}^K B_{\ell,\ell'}\big)^{\mathscr{M}_1}\big(\sum_{\ell=1}^K\sum_{\ell'=1}^K(1-B_{\ell,\ell'})\big)^{\mathscr{M}_0}},
\end{eqnarray}
where 
\begin{equation}\label{MathscrMde}
    \mathscr{M}_1:=|\{(i,j)\in [n]^2:i\neq j, X_{i,j}=1\}|,\quad \mathscr{M}_0:=|\{(i,j)\in [n]^2:i\neq j, X_{i,j}=0\}|. 
\end{equation}
For any $z\in[K]^{2n(n-1)}$ and $y\in [0,1]^{n(n-1)K^2}$, we define
\begin{equation}\label{defrR_MMSB}
    r(\pmb{\pi},z):=\sum_{i=1}^n\sum_{\ell=1}^K(N_{i,\ell}(z)+\alpha_{\ell}-1)\log\pi_{i,\ell}-n\log\Gamma(K),
\end{equation}
\begin{equation}\label{defrR_MMSB2}
    R(\pmb{\pi},y):=\sum_{i=1}^n\sum_{\ell=1}^K\big(\widetilde{N}_{i,\ell}(y)+\alpha_{\ell}-1\big)\log\pi_{i,\ell}-n\log\Gamma(K),
\end{equation}
where for every $i\in[n]$ and $\ell\in[K]$, 
\begin{equation}\label{Nil_y_def}
    \widetilde{N}_{i,\ell}(y):=\sum_{j\in [n]  \backslash \{i\}}\sum_{\ell'=1}^K y_{i,j;\ell,\ell'} + \sum_{j\in [n]\backslash \{i\}}\sum_{\ell'=1}^K y_{j,i;\ell',\ell}.
\end{equation}
Note that Assumptions \ref{Assump:Cond-ind}-\ref{Assump:R} hold under these specifications. Moreover, \eqref{defrR_MMSB} implies that (recall~\eqref{def_f}),
\begin{equation}\label{deffF_MMSB}
    f(z)=\sum_{i=1}^n\sum_{\ell=1}^K \log\Gamma(N_{i,\ell}(z)+\alpha_{\ell})-n\log\Gamma\Big(2n-2+\sum_{\ell=1}^K\alpha_{\ell}\Big),\quad\mbox{for any }z\in[K]^{2n(n-1)}.
\end{equation}
We define 
\begin{equation}\label{deffF_MMSB2}
    F(y):=\sum_{i=1}^n\sum_{\ell=1}^K\log\Gamma\big(\widetilde{N}_{i,\ell}(y)+\alpha_{\ell}\big)-n\log\Gamma\Big(2n-2+\sum_{\ell=1}^K\alpha_{\ell}\Big),\quad\mbox{for any }y\in [0,1]^{n(n-1)K^2};
\end{equation}
note that Assumption \ref{Assump:F} also holds. For any $y\in\hat{\mathcal{Y}}_{n,K}$, we have (recall~\eqref{def_I})
\begin{eqnarray}\label{defI_MMSB}
    I(y)&=&\sum_{i,j\in [n]: i\neq j} \sum_{\ell=1}^K\sum_{\ell'=1}^K y_{i,j;\ell,\ell'}\log\bigg(\frac{y_{i,j;\ell,\ell'}}{\mu_{i,j}(\ell,\ell')}\bigg)\nonumber\\
    &=&-\sum_{i,j\in [n]:i\neq j}\sum_{\ell=1}^K\sum_{\ell'=1}^Ky_{i,j;\ell,\ell'}(X_{i,j}\log B_{\ell,\ell'} +(1-X_{i,j})\log(1-B_{\ell,\ell'}))\nonumber\\
    &&\hspace{0.4in}+\mathscr{M}_1\log\bigg(\sum_{\ell=1}^K\sum_{\ell'=1}^K B_{\ell,\ell'}\bigg)+\mathscr{M}_0\log\bigg(\sum_{\ell=1}^K\sum_{\ell'=1}^K(1-B_{\ell,\ell'})\bigg)\nonumber\\
    &&\hspace{0.4in}+\sum_{i,j\in [n]:i\neq j}\sum_{\ell=1}^K\sum_{\ell'=1}^K y_{i,j;\ell,\ell'}\log y_{i,j;\ell,\ell'},
\end{eqnarray}
where $\mathscr{M}_1,\mathscr{M}_0$ are as in~\eqref{MathscrMde}. Moreover, with $f$ and $\mu$ as in~\eqref{deffF_MMSB} and~\eqref{defmu:MMSB}, we have
{\small
\begin{eqnarray}\label{partition_MMSB}
    \sum_{z\in [K]^{2n(n-1)}} e^{f(z)}\mu(z)&=& \sum_{z\in [K]^{2n(n-1)}}e^{\sum_{i=1}^n\sum_{\ell=1}^K \log\Gamma(N_{i,\ell}(z)+\alpha_{\ell})-n\log\Gamma\big(2n-2+\sum_{\ell=1}^K\alpha_{\ell}\big)}\nonumber\\
    &&\hspace{0.2in}\cdot \frac{\prod_{\ell=1}^K\prod_{\ell'=1}^K B_{\ell,\ell'}^{M_{\ell,\ell'}(z)}(1-B_{\ell,\ell'})^{A_{\ell,\ell'}(z)-M_{\ell,\ell'}(z)}}{\big(\sum_{\ell=1}^K\sum_{\ell'=1}^K B_{\ell,\ell'}\big)^{\mathscr{M}_1}\big(\sum_{\ell=1}^K\sum_{\ell'=1}^K(1-B_{\ell,\ell'})\big)^{\mathscr{M}_0}}\nonumber\\
    &=& \frac{\sum_{z\in [K]^{2n(n-1)}}e^{\Upsilon(z)}}{\big(\sum_{\ell=1}^K\sum_{\ell'=1}^K  B_{\ell,\ell'}\big)^{\mathscr{M}_1}\big(\sum_{\ell=1}^K\sum_{\ell'=1}^K(1-B_{\ell,\ell'})\big)^{\mathscr{M}_0}}\nonumber\\
    &=& 
    \frac{\mathcal{S}_{n,K}}{\big(\sum_{\ell=1}^K\sum_{\ell'=1}^K  B_{\ell,\ell'}\big)^{\mathscr{M}_1}\big(\sum_{\ell=1}^K\sum_{\ell'=1}^K(1-B_{\ell,\ell'})\big)^{\mathscr{M}_0}},
\end{eqnarray}
}where we use~\eqref{defUps:MMSB} and~\eqref{eq:SnK_MMSB} in the second and third equalities, respectively.

\medskip

\noindent{\bf{Step 1: Restricting the partition function and Stirling's approximation.}} Let $\mathscr{C}_{n,K}$ be the collection of $\pmb{\mathscr{A}}=(\mathscr{A}_1,\mathscr{A}_2,\cdots,\mathscr{A}_n)$ such that $\mathscr{A}_i\subseteq [K]$ and $|\mathscr{A}_i|\geq 1$ for all $i\in[n]$. For any $\pmb{\mathscr{A}}=(\mathscr{A}_1,\mathscr{A}_2,\cdots,\mathscr{A}_n)\in\mathscr{C}_{n,K}$, we define $\mathcal{Z}_{n,K,\pmb{\mathscr{A}}}$ to be the set of $z\in [K]^{2n(n-1)}$ such that $z_{i\rightarrow j}\in\mathscr{A}_i$ and $z_{i\leftarrow j}\in\mathscr{A}_j$ for all $i,j\in[n]$ with $i\neq j$, and set 
\begin{equation}\label{def_snk}
\mathcal{S}_{n,K,\pmb{\mathscr{A}}}:=\sum\limits_{z\in\mathcal{Z}_{n,K,\pmb{\mathscr{A}}}}e^{\Psi_{\pmb{\mathscr{A}}}(z)},
\end{equation}
where for any $z\in \mathcal{Z}_{n,K,\pmb{\mathscr{A}}}\subseteq [K]^{2n(n-1)}$,
{\small
\begin{equation}\label{defPsiM}
      \Psi_{\pmb{\mathscr{A}}}(z):=\sum_{\ell=1}^K\sum_{\ell'=1}^K(M_{\ell,\ell'}(z)\log B_{\ell,\ell'}+(A_{\ell,\ell'}(z)-M_{\ell,\ell'}(z))\log(1-B_{\ell,\ell'}))+\sum_{i=1}^n\sum_{\ell\in\mathscr{A}_i} \psi(N_{i,\ell}(z)),
\end{equation}
}where $\psi(\cdot)$ is as defined in~\eqref{phioriginal}. In this step, we will show that 
\begin{eqnarray}\label{Eq.A32}
    \log \mathcal{S}_{n,K} &\leq&  \sup_{\pmb{\mathscr{A}}\in \mathscr{C}_{n,K}}\bigg\{\sum_{i=1}^n\sum_{\ell\in [K]\backslash \mathscr{A}_i}\log\Gamma(\alpha_{\ell})+\log \mathcal{S}_{n,K,\pmb{\mathscr{A}}} \bigg\}\nonumber\\
    &&\hspace{0.2in} -2n(n-1)-n\log \Gamma\Big(2n-2+\sum_{\ell=1}^K \alpha_{\ell}\Big)+CnK\log\Big(\frac{n}{K}+2\Big).
\end{eqnarray}
\begin{proof}
For any $z\in [K]^{2n(n-1)}$, we define $\mathcal{A}_{i}(z):=\{\ell\in [K]: N_{i,\ell}(z) \geq 1\}$ for every $i\in [n]$, and set $\pmb{\mathcal{A}}(z):=(\mathcal{A}_1(z),\cdots,\mathcal{A}_n(z))$. For any $\pmb{\mathscr{A}}=(\mathscr{A}_1,\cdots,\mathscr{A}_n)\in\mathscr{C}_{n,K}$ and $z\in [K]^{2n(n-1)}$, we define
{\small
\begin{eqnarray*}
    \Upsilon_{\pmb{\mathscr{A}}}(z)&:=& \sum_{\ell=1}^K\sum_{\ell'=1}^K(M_{\ell,\ell'}(z)\log B_{\ell,\ell'}+(A_{\ell,\ell'}(z)-M_{\ell,\ell'}(z))\log(1-B_{\ell,\ell'}))\nonumber\\
    && \hspace{0.4in}+\sum_{i=1}^n\sum_{\ell\in\mathscr{A}_i}\log\Gamma(N_{i,\ell}(z)+\alpha_{\ell})-n\log\Gamma\Big(2n-2+\sum_{\ell=1}^K \alpha_{\ell}\Big).
\end{eqnarray*}
}Note that for any $\pmb{\mathscr{A}}\in\mathscr{C}_{n,K}$ and $z\in [K]^{2n(n-1)}$ such that $\pmb{\mathcal{A}}(z)=\pmb{\mathscr{A}}$ (which implies $z\in\mathcal{Z}_{n,K,\pmb{\mathscr{A}}}$), we have $\Upsilon(z) = \Upsilon_{\pmb{\mathscr{A}}}(z)+\sum_{i=1}^n\sum_{\ell\in [K]\backslash \mathscr{A}_i}\log\Gamma(\alpha_{\ell})$. Hence, using~\eqref{eq:SnK_MMSB}, we have 
{\small
\begin{eqnarray}\label{EEq4.1}
 \mathcal{S}_{n,K} &=& \sum_{\pmb{\mathscr{A}}\in\mathscr{C}_{n,K}}\sum_{\substack{z\in[K]^{2n(n-1)}:\\  \pmb{\mathcal{A}}(z)=\pmb{\mathscr{A}}}} e^{\Upsilon(z)}\nonumber\\
    &=& \sum_{\pmb{\mathscr{A}}\in\mathscr{C}_{n,K}}e^{\sum_{i=1}^n\sum_{\ell\in [K]\backslash \mathscr{A}_i}\log\Gamma(\alpha_{\ell})}\bigg(\sum_{\substack{z\in\mathcal{Z}_{n,K,\pmb{\mathscr{A}}}:\\ N_{i,\ell}(z)\geq 1\text{ for all }i\in [n],\ell\in\mathscr{A}_i}} e^{\Upsilon_{\pmb{\mathscr{A}}}(z)}\bigg).
\end{eqnarray}
}

For any $\pmb{\mathscr{A}}\in\mathscr{C}_{n,K}$, we define $\widetilde{\mathcal{S}}_{n,K,\pmb{\mathscr{A}}}:=\sum_{z\in\mathcal{Z}_{n,K,\pmb{\mathscr{A}}}} e^{\Phi_{\pmb{\mathscr{A}}}(z)}$, where for any $z\in\mathcal{Z}_{n,K,\pmb{\mathscr{A}}}$,
{\small
\begin{eqnarray*}
      \Phi_{\pmb{\mathscr{A}}}(z)&:=&\sum_{\ell=1}^K\sum_{\ell'=1}^K(M_{\ell,\ell'}(z)\log B_{\ell,\ell'} +(A_{\ell,\ell'}(z)-M_{\ell,\ell'}(z))\log(1-B_{\ell,\ell'}))+\sum_{i=1}^n\sum_{\ell\in\mathscr{A}_i}\phi(N_{i,\ell}(z)),
\end{eqnarray*}
}where $\phi(\cdot)$ is as in~\eqref{phioriginal}. Now we use Stirling's approximation to bound $\Upsilon_{\pmb{\mathscr{A}}}(z)$ with explicit error terms (see Lemma \ref{L3.2} in Appendix \ref{Appendix_E}). For any $z\in\mathcal{Z}_{n,K,\pmb{\mathscr{A}}}$ such that $N_{i,\ell}(z)\geq 1$ for all $i\in [n]$ and $\ell\in\mathscr{A}_i$, by Lemma \ref{L3.2} we get
{\small
\begin{eqnarray*}
    && \bigg|\Upsilon_{\pmb{\mathscr{A}}}(z)-\Phi_{\pmb{\mathscr{A}}}(z)+n\log\Gamma\Big(2n-2+\sum_{\ell=1}^K \alpha_{\ell}\Big)+2n(n-1)\bigg|\nonumber\\
    &=& \bigg|\sum_{i=1}^n\sum_{\ell\in\mathscr{A}_i}\Big[\big(\log\Gamma(N_{i,\ell}(z)+\alpha_{\ell})-\phi(N_{i,\ell}(z)+\alpha_{\ell})\big)+\big(\phi(N_{i,\ell}(z)+\alpha_{\ell})-\phi(N_{i,\ell}(z))\big)\Big]\nonumber\\
    &&\hspace{0.6in}+2n(n-1)\bigg| \nonumber\\
    &\leq&  C\sum_{i=1}^n\sum_{\ell\in\mathscr{A}_i}\big(\log(N_{i,\ell}(z)+\alpha_{\ell})+1\big)+\sum_{i=1}^n\sum_{\ell\in\mathscr{A}_i}\alpha_{\ell}\leq C\sum_{i=1}^n\sum_{\ell=1}^K \log(N_{i,\ell}(z)+C_0)+CnK\nonumber\\
    &\leq& CnK\log\bigg(\frac{\sum_{i=1}^n\sum_{\ell=1}^K(N_{i,\ell}(z)+C_0)}{nK}\bigg)+CnK\leq CnK\log\Big(\frac{n}{K}+2\Big), 
\end{eqnarray*}
}where we use the bound $\max_{\ell\in [K]}\alpha_{\ell}\leq C_0$, and the last line uses Jensen's inequality. Hence using the above bound along with (\ref{EEq4.1}), we get
{\small
\begin{eqnarray*}
 \mathcal{S}_{n,K} &\leq& e^{-n\log\Gamma\big(2n-2+\sum_{\ell=1}^K \alpha_{\ell}\big)-2n(n-1)+CnK\log(n \slash K + 2 )} \nonumber\\
    &&\hspace{0.4in}\times \sum_{\pmb{\mathscr{A}}\in\mathscr{C}_{n,K}} e^{\sum_{i=1}^n\sum_{\ell\in [K]\backslash \mathscr{A}_i}\log\Gamma(\alpha_{\ell})}\bigg(\sum_{\substack{z\in\mathcal{Z}_{n,K,\pmb{\mathscr{A}}}:\\ N_{i,\ell}(z)\geq 1 \text{ for all }i\in [n],\ell\in\mathscr{A}_i}} e^{\Phi_{\pmb{\mathscr{A}}}(z)}\bigg)\nonumber\\
    &\leq& e^{-n\log\Gamma\big(2n-2+\sum_{\ell=1}^K \alpha_{\ell}\big)-2n(n-1)+CnK\log(n \slash K + 2 )} \nonumber\\
    &&\hspace{0.4in}\times \sum_{\pmb{\mathscr{A}}\in\mathscr{C}_{n,K}} e^{\sum_{i=1}^n\sum_{\ell\in [K]\backslash \mathscr{A}_i}\log\Gamma(\alpha_{\ell})}\widetilde{\mathcal{S}}_{n,K,\pmb{\mathscr{A}}}.
\end{eqnarray*}
}Finally, using (\ref{eq3.2}), for any $\pmb{\mathscr{A}}\in\mathscr{C}_{n,K}$ and $z\in\mathcal{Z}_{n,K,\pmb{\mathscr{A}}}$, we have 
\begin{equation*}
    |\Phi_{\pmb{\mathscr{A}}}(z)-\Psi_{\pmb{\mathscr{A}}}(z)|\leq \sum_{i=1}^n\sum_{\ell\in\mathscr{A}_i}\big|\phi(N_{i,\ell}(z))-\psi(N_{i,\ell}(z))\big|\leq\frac{nK}{2}\;\; \Rightarrow \;\; \widetilde{\mathcal{S}}_{n,K,\pmb{\mathscr{A}}}\leq e^{nK\slash 2}\mathcal{S}_{n,K,\pmb{\mathscr{A}}}.
\end{equation*}
Thus, combining the above two displays, and possibly adjusting the constant $C$ above, we get
{\small
\begin{eqnarray*}
 \mathcal{S}_{n,K} 
    &\leq& e^{-n\log\Gamma\big(2n-2+\sum_{\ell=1}^K \alpha_{\ell}\big)-2n(n-1)+CnK\log(n\slash K+2)} \nonumber\\
    &&\hspace{0.4in}\times \sum_{\pmb{\mathscr{A}}\in\mathscr{C}_{n,K}} e^{\sum_{i=1}^n\sum_{\ell\in [K]\backslash \mathscr{A}_i}\log\Gamma(\alpha_{\ell})}\mathcal{S}_{n,K,\pmb{\mathscr{A}}}.
\end{eqnarray*}
}Taking $\log$ on both sides and noting that $|\mathscr{C}_{n,K}|\leq 2^{nK}$, the desired conclusion follows.
\end{proof}

\noindent For \textbf{Steps 2}-\textbf{4} below, we fix any $\pmb{\mathscr{A}}=(\mathscr{A}_1,\mathscr{A}_2,\cdots,\mathscr{A}_n)\in\mathscr{C}_{n,K}$. We assume that for every distinct $i,j\in [n]$, $\sum_{\ell\in\mathscr{A}_i}\sum_{\ell'\in\mathscr{A}_j}B_{\ell,\ell'}>0$ if $X_{i,j}=1$, and $\sum_{\ell\in\mathscr{A}_i}\sum_{\ell'\in\mathscr{A}_j}(1-B_{\ell,\ell'})>0$ if $X_{i,j}=0$ (otherwise $\mathcal{S}_{n,K,\pmb{\mathscr{A}}}=0$; see~\eqref{def_snk}-\eqref{defPsiM}). In order to bound the restricted partition function $\mathcal{S}_{n,K,\pmb{\mathscr{A}}}$ (see~\eqref{def_snk}), we specialize the general setup in Section \ref{Sect.2.1} as follows. We let $\{(i,j)\in [n]^2:i\neq j\}$ play the role of $[n]$ in Section \ref{Sect.2.1} (for indexing the local latent variables). For every $(i,j)\in [n]^2$ such that $i\neq j$, the local latent variable $Z_{i,j}\equiv (Z_{i\rightarrow j},Z_{i\leftarrow j})$ takes values in $\mathscr{A}_i\times\mathscr{A}_j$. We also let
{\small
\begin{eqnarray}\label{def_YnKAM}
    \hat{\mathcal{Y}}_{n,\pmb{\mathscr{A}}}&:=&\bigg\{y=(y_{i,j;\ell,\ell'})_{i,j\in [n],i\neq j;\ell\in\mathscr{A}_i,\ell'\in \mathscr{A}_j}\in \prod_{i,j\in [n]: i\neq j}[0,1]^{|\mathscr{A}_i||\mathscr{A}_j|} :\nonumber\\
    &&\quad\quad\quad\quad\quad   \sum_{\ell\in\mathscr{A}_i}\sum_{\ell'\in\mathscr{A}_j}
    y_{i,j;\ell,\ell'}=1
    \text{ for every } i,j\in [n]\text{ with }i\neq j\bigg\}
\end{eqnarray}
}play the role of $\hat{\mathcal{Y}}_{n,\mathbf{K}}$ in Section \ref{Sect.2.1}. For every $i,j\in [n]$ such that $i\neq j$, we define $\muAM_{i,j}$ to be the probability measure on $\mathscr{A}_i\times\mathscr{A}_j$ such that for every $(\ell,\ell')\in\mathscr{A}_i\times\mathscr{A}_j$,
\begin{equation}\label{Def_Mu}
    \muAM_{i,j}(\ell,\ell') = \begin{cases}
        \frac{B_{\ell,\ell'}}{\sum_{s\in\mathscr{A}_i}\sum_{s'\in\mathscr{A}_j} B_{s,s'}} & \text{ if } X_{i,j}=1\\
        \frac{1-B_{\ell,\ell'}}{\sum_{s\in\mathscr{A}_i}\sum_{s'\in\mathscr{A}_j} (1-B_{s,s'})} & \text{ if } X_{i,j}=0.
    \end{cases}
\end{equation} 
We define $\muAM$ to be the probability measure on $\mathcal{Z}_{n,K,\pmb{\mathscr{A}}}$ such that
\begin{equation}
     \muAM(z) = \prod_{i,j\in [n]:i\neq j} \muAM_{i,j}(z_{i\rightarrow j},z_{i\leftarrow j}), \quad \text{for all } z\in\mathcal{Z}_{n,K,\pmb{\mathscr{A}}}. 
\end{equation}
For any $z\in\mathcal{Z}_{n,K,\pmb{\mathscr{A}}}$, we take
\begin{eqnarray}
    \fAM(z)&:=&\Psi_{\pmb{\mathscr{A}}}(z)-\sum_{\ell=1}^K\sum_{\ell'=1}^K(M_{\ell,\ell'}(z)\log B_{\ell,\ell'}+(A_{\ell,\ell'}(z)-M_{\ell,\ell'}(z))\log(1-B_{\ell,\ell'}))\nonumber\\
    &=& \sum_{i=1}^n\sum_{\ell\in\mathscr{A}_i} \psi(N_{i,\ell}(z)),
\end{eqnarray}
where the second equality uses~\eqref{defPsiM}. Note that
{\small
\begin{eqnarray}\label{EEEq4.10}
  &&\sum_{z\in\mathcal{Z}_{n,K,\pmb{\mathscr{A}}}}  e^{\fAM(z)}\muAM(z)\nonumber\\
&=&\sum_{z\in\mathcal{Z}_{n,K,\pmb{\mathscr{A}}}}  e^{\sum_{i=1}^n\sum_{\ell\in\mathscr{A}_i} \psi(N_{i,\ell}(z))}\cdot\prod_{\substack{i,j\in[n]:\\ i\neq j,X_{i,j}=1}}\frac{B_{z_{i\rightarrow j},z_{i\leftarrow j}}}{\sum_{\ell\in\mathscr{A}_i}\sum_{\ell'\in\mathscr{A}_j} B_{\ell,\ell'}}\nonumber\\
&&\hspace{0.2in}\cdot\prod_{\substack{i,j\in[n]:\\ i\neq j,X_{i,j}=0}}\frac{1-B_{z_{i\rightarrow j},z_{i\leftarrow j}}}{\sum_{\ell\in\mathscr{A}_i}\sum_{\ell'\in\mathscr{A}_j} (1-B_{\ell,\ell'})}\nonumber\\
&=& \frac{\sum_{z\in\mathcal{Z}_{n,K,\pmb{\mathscr{A}}}}e^{\Psi_{\pmb{\mathscr{A}}}(z)}}{\prod\limits_{\substack{i,j\in [n]:\\i\neq j, X_{i,j}=1}}\big(\sum_{\ell\in\mathscr{A}_i}\sum_{\ell'\in\mathscr{A}_j}  B_{\ell,\ell'}\big)\prod\limits_{\substack{i,j\in [n]:\\i\neq j, X_{i,j}=0}}\big(\sum_{\ell\in\mathscr{A}_i}\sum_{\ell'\in\mathscr{A}_j}  (1-B_{\ell,\ell'})\big)}\nonumber\\
&=&\frac{\mathcal{S}_{n,K,\pmb{\mathscr{A}}}}{\prod\limits_{\substack{i,j\in [n]:\\i\neq j, X_{i,j}=1}}\big(\sum_{\ell\in\mathscr{A}_i}\sum_{\ell'\in\mathscr{A}_j}  B_{\ell,\ell'}\big)\prod\limits_{\substack{i,j\in [n]:\\i\neq j, X_{i,j}=0}}\big(\sum_{\ell\in\mathscr{A}_i}\sum_{\ell'\in\mathscr{A}_j}  (1-B_{\ell,\ell'})\big)},
\end{eqnarray}
}where we use~\eqref{defPsiM} and~\eqref{def_snk} in the last two equalities. For any $y=(y_{i,j;\ell,\ell'})_{i,j\in [n],i\neq j;\ell\in\mathscr{A}_i,\ell'\in \mathscr{A}_j}\in \prod_{i,j\in [n]:i\neq j} [0,1]^{|\mathscr{A}_i||\mathscr{A}_j|}$, we take
\begin{equation}\label{defHd}
    \FAM(y):=\sum_{i=1}^n\sum_{\ell\in\mathscr{A}_i} \psi\big(\NAM_{i,\ell}(y)\big),
\end{equation}
where for any $i\in [n]$ and $\ell\in\mathscr{A}_i$, 
\begin{equation}\label{Ntilde_d}
    \NAM_{i,\ell}(y):=\sum_{j\in [n]  \backslash \{i\}}\sum_{\ell'\in\mathscr{A}_j} y_{i,j;\ell,\ell'}+\sum_{j\in [n]\backslash \{i\}}\sum_{\ell'\in\mathscr{A}_j}y_{j,i;\ell',\ell}.
\end{equation}
Note that for any $y\in\hat{\mathcal{Y}}_{n,\pmb{\mathscr{A}}}$ and $i\in [n]$,
\begin{equation}\label{NAilsum}
    \sum_{\ell\in\mathscr{A}_i}\NAM_{i,\ell}(y)=\sum_{j\in [n]  \backslash \{i\}}\sum_{\ell\in\mathscr{A}_i}\sum_{\ell'\in\mathscr{A}_j} y_{i,j;\ell,\ell'}+\sum_{j\in [n]\backslash \{i\}}\sum_{\ell\in\mathscr{A}_i}\sum_{\ell'\in\mathscr{A}_j}y_{j,i;\ell',\ell}=2(n-1).
\end{equation}
For any $y=(y_{i,j;\ell,\ell'})_{i,j\in [n],i\neq j;\ell\in\mathscr{A}_i,\ell'\in \mathscr{A}_j}\in\hat{\mathcal{Y}}_{n,\pmb{\mathscr{A}}}$, we define
\small{
\begin{eqnarray}\label{Idefy}
     \IAM(y) &:=& \sum_{i,j\in [n]: i\neq j} \sum_{\ell\in\mathscr{A}_i}\sum_{\ell'\in\mathscr{A}_j} y_{i,j;\ell,\ell'}\log\bigg(\frac{y_{i,j;\ell,\ell'}}{\muAM_{i,j}(\ell,\ell')}\bigg)\nonumber\\
    &=&-\sum_{i,j\in [n]:i\neq j}\sum_{\ell\in\mathscr{A}_i}\sum_{\ell'\in\mathscr{A}_j}y_{i,j;\ell,\ell'}(X_{i,j}\log B_{\ell,\ell'} +(1-X_{i,j})\log(1-B_{\ell,\ell'}))\nonumber\\
    &&\hspace{0.4in}+\sum_{\substack{i,j\in [n]:\\i\neq j,X_{i,j}=1}}\log\bigg(\sum_{\ell\in\mathscr{A}_i}\sum_{\ell'\in\mathscr{A}_j}B_{\ell,\ell'}\bigg)+\sum_{\substack{i,j\in [n]:\\i\neq j,X_{i,j}=0}}\log\bigg(\sum_{\ell\in\mathscr{A}_i}\sum_{\ell'\in\mathscr{A}_j}(1-B_{\ell,\ell'})\bigg)\nonumber\\
    &&\hspace{0.4in}+\sum_{i,j\in [n]:i\neq j}\sum_{\ell\in\mathscr{A}_i}\sum_{\ell'\in\mathscr{A}_j} y_{i,j;\ell,\ell'}\log y_{i,j;\ell,\ell'}.
\end{eqnarray}
}

In \textbf{Steps 2}-\textbf{4} below, $\muAM,\fAM,\FAM,\IAM$ as defined above will play the roles of $\mu,f,F,I$ in Section \ref{Sect.2.1}. We aim to apply Theorem \ref{Theorem2.1} to bound $\sum_{z\in \mathcal{Z}_{n,K,\pmb{\mathscr{A}}}} e^{\fAM(z)}\muAM(z)$ (cf.~\eqref{defS}), which via~\eqref{EEEq4.10} yields a bound on the restricted partition function $\mathcal{S}_{n,K,\pmb{\mathscr{A}}}$. The upper bounds on the smoothness term $E_1$ (see Definition \ref{DefSmoothnessTerm}) and the complexity term $E_2(\epsilon)$ (see Definition \ref{DefComplexityTerm})---as required by Theorem \ref{Theorem2.1}---are obtained in \textbf{Steps 2} and \textbf{3}, and the resulting upper bound on $\mathcal{S}_{n,K,\pmb{\mathscr{A}}}$ is given in \textbf{Step 4}. 

\medskip

\noindent{\bf{Step 2: Bounding the smoothness term.}}
Let $E_1$ and $\Lambda(\cdot)$ be defined as in (\ref{E2.1}) and (\ref{def_L}), respectively, but with $\hat{\mathcal{Y}}_{n,\mathbf{K}},\mu,f,F,I$ replaced by $\hat{\mathcal{Y}}_{n,\pmb{\mathscr{A}}},\muAM,\fAM,\FAM,\IAM$ (as in~\eqref{def_YnKAM}-\eqref{Idefy}). In this step, we will show that there exist positive absolute constants $C_1,C_2,C_2'$, such that 
\begin{equation}\label{Step2_result}
    E_1 \leq C_1 n K,\qquad \mbox{and} \qquad \Lambda(C_1 n K) \leq C_2\log\left(\frac{n\log(nK)}{K}+2\right)\leq C_2' n.
\end{equation}
\begin{proof}

Consider any $y\in \prod_{i,j\in [n]:i\neq j} [0,1]^{|\mathscr{A}_i||\mathscr{A}_j|}$. Using the definition of $\FAM$ in~\eqref{defHd}, for any $i,j\in[n]$ with $i\neq j$ and any $(\ell,\ell')\in \mathscr{A}_i\times\mathscr{A}_j$, 
\begin{eqnarray}\label{defH1.2}
    && \FAM_{i,j,\ell,\ell'}(y) := \frac{\partial \FAM(y)}{\partial y_{i,j;\ell,\ell'}}=\psi'\big(\NAM_{i,\ell}(y)\big)+\psi'\big(\NAM_{j,\ell'}(y)\big).
\end{eqnarray}
For any $i_1,j_1,i_2,j_2\in [n]$ with $i_1\neq j_1, i_2\neq j_2$ and any $(\ell_1,\ell_1')\in\mathscr{A}_{i_1}\times \mathscr{A}_{j_1},(\ell_2,\ell_2')\in\mathscr{A}_{i_2}\times\mathscr{A}_{j_2}$,
\begin{eqnarray}\label{defH2.2}
    && \FAM_{i_1,j_1,\ell_1,\ell_1';i_2,j_2,\ell_2,\ell_2'}(y):=\frac{\partial^2 \FAM(y)}{\partial y_{i_1,j_1;\ell_1,\ell_1'}\partial y_{i_2,j_2;\ell_2,\ell_2'}}\nonumber\\
    & = & \psi''\big(\NAM_{i_1,\ell_1}(y)\big)(\mathbbm{1}_{i_1=i_2,\ell_1=\ell_2}+\mathbbm{1}_{i_1=j_2,\ell_1=\ell_2'})\nonumber\\
    &&\hspace{0.2in}+\psi''\big(\NAM_{j_1,\ell_1'}(y)\big)(\mathbbm{1}_{j_1=i_2,\ell_1'=\ell_2}+\mathbbm{1}_{j_1=j_2,\ell_1'=\ell_2'}).
\end{eqnarray}
With $b_{i_1,j_1,\ell_1,\ell_1'}$ and $c_{i_1,j_1,\ell_1,\ell_1';i_2,j_2,\ell_2,\ell_2'}$ defined analogously as in Definition \ref{Global_GH}, by (\ref{derivatives}) and (\ref{defH1.2})-(\ref{defH2.2}), for any $i_1,j_1,i_2,j_2\in [n]$ with $i_1\neq j_1, i_2\neq j_2$ and any $(\ell_1,\ell_1')\in\mathscr{A}_{i_1}\times \mathscr{A}_{j_1},(\ell_2,\ell_2')\in\mathscr{A}_{i_2}\times\mathscr{A}_{j_2}$, 
\begin{equation}\label{UnifB}
    b_{i_1,j_1,\ell_1,\ell_1'}\leq 2+2\log(2nK), \quad c_{i_1,j_1,\ell_1,\ell_1';i_2,j_2,\ell_2,\ell_2'}\leq 4\|\psi''\|_{\infty}\leq 4.
\end{equation}

With $\NIM(y)$ defined as in Definition \ref{Local_Hess}, for any $y\in \hat{\mathcal{Y}}_{n,\pmb{\mathscr{A}}}$ and $y'\in\NIM(y)$, 
\begin{equation*}
    \NAM_{i,\ell}(y')\geq \NAM_{i,\ell}(y) -2,\quad\text{for all }i\in[n],\ell\in\mathscr{A}_i.
\end{equation*}
Hence by (\ref{derivatives}) and (\ref{defH2.2}), 
\begin{eqnarray*}
    &&\FAM_{i_1,j_1,\ell_1,\ell_1';i_2,j_2,\ell_2,\ell_2'}(y')\nonumber\\
    &=&\frac{\mathbbm{1}_{i_1=i_2,\ell_1=\ell_2}+\mathbbm{1}_{i_1=j_2,\ell_1=\ell_2'}}{\max\big\{\NAM_{i_1,\ell_1}(y'),1\big\}}+\frac{\mathbbm{1}_{j_1=i_2,\ell_1'=\ell_2}+\mathbbm{1}_{j_1=j_2,\ell_1'=\ell_2'}}{\max\big\{\NAM_{j_1,\ell_1'}(y'),1\big\}}\nonumber\\
    &\leq& \frac{\mathbbm{1}_{i_1=i_2,\ell_1=\ell_2}+\mathbbm{1}_{i_1=j_2,\ell_1=\ell_2'}}{\max\big\{\NAM_{i_1,\ell_1}(y)-2,1\big\}}+\frac{\mathbbm{1}_{j_1=i_2,\ell_1'=\ell_2}+\mathbbm{1}_{j_1=j_2,\ell_1'=\ell_2'}}{\max\big\{\NAM_{j_1,\ell_1'}(y)-2,1\big\}}\nonumber\\
    &\leq& \frac{3(\mathbbm{1}_{i_1=i_2,\ell_1=\ell_2}+\mathbbm{1}_{i_1=j_2,\ell_1=\ell_2'})}{\max\big\{\NAM_{i_1,\ell_1}(y),1\big\}}+\frac{3(\mathbbm{1}_{j_1=i_2,\ell_1'=\ell_2}+\mathbbm{1}_{j_1=j_2,\ell_1'=\ell_2'})}{\max\big\{\NAM_{j_1,\ell_1'}(y),1\big\}}.
\end{eqnarray*}
Therefore, with $\HIM_{i_1,j_1,\ell_1,\ell_1';i_2,j_2,\ell_2,\ell_2'}(y)$ defined analogously as in Definition \ref{Local_Hess}, we have
\begin{equation}\label{Eq.B1}
    \HIM_{i_1,j_1,\ell_1,\ell_1';i_2,j_2,\ell_2,\ell_2'}(y)\leq\frac{3(\mathbbm{1}_{i_1=i_2,\ell_1=\ell_2}+\mathbbm{1}_{i_1=j_2,\ell_1=\ell_2'})}{\max\big\{\NAM_{i_1,\ell_1}(y),1\big\}}+\frac{3(\mathbbm{1}_{j_1=i_2,\ell_1'=\ell_2}+\mathbbm{1}_{j_1=j_2,\ell_1'=\ell_2'})}{\max\big\{\NAM_{j_1,\ell_1'}(y),1\big\}}.
\end{equation}

With $\NIIM(y,M)$ defined as in Definition \ref{Local_Hess_new}, for any $y\in \hat{\mathcal{Y}}_{n,\pmb{\mathscr{A}}}$, $M\geq 1$, $y'\in\NIIM(y,M)$, and $i\in[n],\ell\in\mathscr{A}_i$, we have
\begin{equation*}
    \NAM_{i,\ell}(y') \geq \frac{\NAM_{i,\ell}(y)-2}{M},
\end{equation*}
hence
\begin{equation*}
    \max\big\{\NAM_{i,\ell}(y'),1\big\}\geq\frac{M}{M+2}\NAM_{i,\ell}(y')+\frac{2}{M+2}
    \geq \frac{\NAM_{i,\ell}(y)}{M+2}\Rightarrow \max\big\{\NAM_{i,\ell}(y'),1\big\}\geq \frac{\max\big\{\NAM_{i,\ell}(y),1\big\}}{M+2}.
\end{equation*}
Thus by (\ref{derivatives}) and (\ref{defH2.2}),
\begin{eqnarray*}
    &&\FAM_{i_1,j_1,\ell_1,\ell_1';i_2,j_2,\ell_2,\ell_2'}(y')\nonumber\\
    &\leq& \frac{(M+2)(\mathbbm{1}_{i_1=i_2,\ell_1=\ell_2}+\mathbbm{1}_{i_1=j_2,\ell_1=\ell_2'})}{\max\big\{\NAM_{i_1,\ell_1}(y),1\big\}}+\frac{(M+2)(\mathbbm{1}_{j_1=i_2,\ell_1'=\ell_2}+\mathbbm{1}_{j_1=j_2,\ell_1'=\ell_2'})}{\max\big\{\NAM_{j_1,\ell_1'}(y),1\big\}}. 
\end{eqnarray*}
Therefore, with $\HIIM_{i_1,j_1,\ell_1,\ell_1';i_2,j_2,\ell_2,\ell_2'}(y,M)$ defined analogously as in Definition \ref{Local_Hess_new}, we have 
\begin{eqnarray}\label{Eq.B2}
    &&\HIIM_{i_1,j_1,\ell_1,\ell_1';i_2,j_2,\ell_2,\ell_2'}(y,M)\nonumber\\
    &\leq& \frac{(M+2)(\mathbbm{1}_{i_1=i_2,\ell_1=\ell_2}+\mathbbm{1}_{i_1=j_2,\ell_1=\ell_2'})}{\max\big\{\NAM_{i_1,\ell_1}(y),1\big\}}+\frac{(M+2)(\mathbbm{1}_{j_1=i_2,\ell_1'=\ell_2}+\mathbbm{1}_{j_1=j_2,\ell_1'=\ell_2'})}{\max\big\{\NAM_{j_1,\ell_1'}(y),1\big\}}.
\end{eqnarray}

With $M_{i,j}$ defined analogously as in Definition \ref{DefSmoothnessTerm}, using (\ref{UnifB}), we get $M_{i,j}\leq e^8$ for every $i,j\in[n]$ with $i\neq j$. Fix $y\in\hat{\mathcal{Y}}_{n,\pmb{\mathscr{A}}}$ arbitrary, and combine~\eqref{E3.5} and~\eqref{UnifB} to get 
\begin{equation}
    \sum_{(\ell_2,\ell_2')\in\mathscr{A}_{i_2}\times\mathscr{A}_{j_2}}\max_{(\ell_1,\ell_1')\in\mathscr{A}_{i_1}\times\mathscr{A}_{j_1}}\big\{\HIM_{i_1,j_1,\ell_1,\ell_1';i_2,j_2,\ell_2,\ell_2'}(y)\big\}y_{i_2,j_2;\ell_2,\ell_2'}\leq 4. 
\end{equation}
Thus, with $\Phi_{i_1,j_1,i_2,j_2}(y)$ defined analogously as in Definition \ref{DefSmoothnessTerm}, using the inequality that $\exp(x)\leq 1+e^8x$ for all $x\in [0,8]$, we obtain for any $i_1,j_1,i_2,j_2\in [n]$ with $i_1\neq j_1,i_2\neq j_2$,   
\begin{eqnarray*}
   &&\Phi_{i_1,j_1,i_2,j_2}(y)\nonumber\\
   &\leq& 2e^8\sum_{(\ell_2,\ell_2')\in\mathscr{A}_{i_2}\times\mathscr{A}_{j_2}}\max_{(\ell_1,\ell_1')\in\mathscr{A}_{i_1}\times\mathscr{A}_{j_1}}\big\{\HIM_{i_1,j_1,\ell_1,\ell_1';i_2,j_2,\ell_2,\ell_2'}(y)\big\}y_{i_2,j_2;\ell_2,\ell_2'}\nonumber\\
    &\leq& C\sum_{(\ell_2,\ell_2')\in \mathscr{A}_{i_2}\times\mathscr{A}_{j_2}}\Bigg(\frac{(\mathbbm{1}_{i_1=i_2}+\mathbbm{1}_{j_1=i_2})y_{i_2,j_2;\ell_2,\ell_2'}}{\max\big\{\NAM_{i_2,\ell_2}(y),1\big\}}+\frac{(\mathbbm{1}_{i_1=j_2}+\mathbbm{1}_{j_1=j_2})y_{i_2,j_2;\ell_2,\ell_2'}}{\max\big\{\NAM_{j_2,\ell_2'}(y),1\big\}}\Bigg),
\end{eqnarray*}
where the last line uses~\eqref{Eq.B1}. Hence we have
\begin{eqnarray}\label{Eq.B4}
  && \sum_{i_2,j_2\in[n]:i_2\neq j_2}\max_{i_1,j_1\in[n]:i_1\neq j_1}\{\Phi_{i_1,j_1,i_2,j_2}(y)\}\nonumber\\
  &\leq& C\sum_{i_2,j_2\in [n]:i_2\neq j_2}\sum_{(\ell_2,\ell_2')\in \mathscr{A}_{i_2}\times\mathscr{A}_{j_2}}\Bigg(\frac{y_{i_2,j_2;\ell_2,\ell_2'}}{\max\big\{\NAM_{i_2,\ell_2}(y),1\big\}}+\frac{y_{i_2,j_2;\ell_2,\ell_2'}}{\max\big\{\NAM_{j_2,\ell_2'}(y),1\big\}}\Bigg)\nonumber\\
  &=& C\sum_{i_2\in [n]}\sum_{\ell_2\in \mathscr{A}_{i_2}}\frac{\sum_{j_2\in[n]\backslash\{i_2\}}\sum_{\ell_2'\in\mathscr{A}_{j_2}} y_{i_2,j_2;\ell_2,\ell_2'}}{\max\big\{\NAM_{i_2,\ell_2}(y),1\big\}}\nonumber\\
  &&\hspace{0.4in}+C\sum_{j_2\in [n]}\sum_{\ell_2'\in\mathscr{A}_{j_2}}\frac{\sum_{i_2\in [n]\backslash\{j_2\}}\sum_{\ell_2\in\mathscr{A}_{i_2}} y_{i_2,j_2;\ell_2,\ell_2'}}{\max\big\{\NAM_{j_2,\ell_2'}(y),1\big\}}\nonumber\\
  &\leq& C\sum_{i_2\in [n]}\sum_{\ell_2\in \mathscr{A}_{i_2}}\frac{\NAM_{i_2,\ell_2}(y)}{\max\big\{\NAM_{i_2,\ell_2}(y),1\big\}}+C\sum_{j_2\in [n]}\sum_{\ell_2'\in\mathscr{A}_{j_2}}\frac{\NAM_{j_2,\ell_2'}(y)}{\max\big\{\NAM_{j_2,\ell_2'}(y),1\big\}}\leq CnK, 
\end{eqnarray}
where the last line uses the definition of $\NAM_{i,\ell}(y)$ in~\eqref{Ntilde_d}. 

Using~\eqref{Eq.B2} with $M=M_{i_1,j_1}\leq e^8$, for any $i_1,j_1\in[n]$ such that $i_1\neq j_1$ and $\ell_1\in\mathscr{A}_{i_1},\ell_1'\in\mathscr{A}_{j_1}$, we get 
\begin{eqnarray}\label{Eq.B3}
   && \sum_{i_2,j_2\in[n]:i_2\neq j_2}\sum_{(\ell_2,\ell_2')\in \mathscr{A}_{i_2}\times\mathscr{A}_{j_2}} \HIIM_{i_1,j_1,\ell_1,\ell_1';i_2,j_2,\ell_2,\ell_2'}(y,M_{i_1,j_1})y_{i_2,j_2;\ell_2,\ell_2'}\nonumber\\
   &\leq& C\sum_{i_2,j_2\in[n]:i_2\neq j_2}\sum_{(\ell_2,\ell_2')\in \mathscr{A}_{i_2}\times\mathscr{A}_{j_2}} \frac{(\mathbbm{1}_{i_1=i_2,\ell_1=\ell_2}+\mathbbm{1}_{i_1=j_2,\ell_1=\ell_2'})y_{i_2,j_2;\ell_2,\ell_2'}}{\max\big\{\NAM_{i_1,\ell_1}(y),1\big\}}\nonumber\\
   &&\hspace{0.2in}+C\sum_{i_2,j_2\in[n]:i_2\neq j_2}\sum_{(\ell_2,\ell_2')\in \mathscr{A}_{i_2}\times\mathscr{A}_{j_2}}\frac{(\mathbbm{1}_{j_1=i_2,\ell_1'=\ell_2}+\mathbbm{1}_{j_1=j_2,\ell_1'=\ell_2'})y_{i_2,j_2;\ell_2,\ell_2'}}{\max\big\{\NAM_{j_1,\ell_1'}(y),1\big\}}\nonumber\\
   &\leq&  \frac{C\big(\sum_{j_2\in[n]\backslash\{i_1\}}\sum_{\ell_2'\in\mathscr{A}_{j_2}} y_{i_1,j_2;\ell_1,\ell_2'}+\sum_{i_2\in[n]\backslash\{i_1\}}\sum_{\ell_2\in\mathscr{A}_{i_2}}y_{i_2,i_1;\ell_2,\ell_1}\big)}{\max\big\{\NAM_{i_1,\ell_1}(y),1\big\}}\nonumber\\
   &&\hspace{0.2in}+\frac{C\big(\sum_{j_2\in[n]\backslash\{j_1\}}\sum_{\ell_2'\in\mathscr{A}_{j_2}}y_{j_1,j_2;\ell_1',\ell_2'}+\sum_{i_2\in[n]\backslash\{j_1\}}\sum_{\ell_2\in\mathscr{A}_{i_2}}y_{i_2,j_1;\ell_2,\ell_1'}\big)}{\max\big\{\NAM_{j_1,\ell_1'}(y),1\big\}}\nonumber\\
   &=& \frac{C\NAM_{i_1,\ell_1}(y)}{\max\big\{\NAM_{i_1,\ell_1}(y),1\big\}}+\frac{C\NAM_{j_1,\ell_1'}(y)}{\max\big\{\NAM_{j_1,\ell_1'}(y),1\big\}}\leq C, 
\end{eqnarray}
where the last line uses the definition of $\NAM_{i,\ell}(y)$ in~\eqref{Ntilde_d}. 

Using~\eqref{Eq.B2} again with $M=M_{i_1,j_1}\leq e^8$, we get
\begin{eqnarray}\label{Eq.B6}
&&\sum_{\substack{i_1,j_1\in[n]:\\i_1\neq j_1}}\max_{\substack{\ell_1\in\mathscr{A}_{i_1},\\\ell_1'\in\mathscr{A}_{j_1}}}\bigg\{\sum_{\substack{\ell_2\in\mathscr{A}_{i_1},\\\ell_2'\in\mathscr{A}_{j_1}}} \HIIM_{i_1,j_1,\ell_1,\ell_1';i_1,j_1,\ell_2,\ell_2'}(y,M_{i_1,j_1})y_{i_1,j_1;\ell_2,\ell_2'}\bigg\}\nonumber\\
&\leq& C\sum_{\substack{i_1,j_1\in[n]:\\i_1\neq j_1}}\max_{\substack{\ell_1\in\mathscr{A}_{i_1},\\\ell_1'\in\mathscr{A}_{j_1}}}\Bigg\{\sum_{\substack{\ell_2\in\mathscr{A}_{i_1},\\\ell_2'\in\mathscr{A}_{j_1}}}\Bigg(\frac{y_{i_1,j_1;\ell_2,\ell_2'}\mathbbm{1}_{\ell_1=\ell_2}}{\max\big\{\NAM_{i_1,\ell_1}(y),1\big\}}+\frac{ y_{i_1,j_1;\ell_2,\ell_2'}\mathbbm{1}_{\ell_1'=\ell_2'}}{\max\big\{\NAM_{j_1,\ell_1'}(y),1\big\}}\Bigg)\Bigg\}\nonumber\\
&\leq& C\sum_{\substack{i_1,j_1\in[n]:\\i_1\neq j_1}}\sum_{\ell_1\in\mathscr{A}_{i_1}}\sum_{\substack{\ell_2\in\mathscr{A}_{i_1},\\\ell_2'\in\mathscr{A}_{j_1}}}\frac{y_{i_1,j_1;\ell_2,\ell_2'}\mathbbm{1}_{\ell_1=\ell_2}}{\max\big\{\NAM_{i_1,\ell_1}(y),1\big\}}+C\sum_{\substack{i_1,j_1\in[n]:\\i_1\neq j_1}}\sum_{\ell_1'\in\mathscr{A}_{j_1}}\sum_{\substack{\ell_2\in\mathscr{A}_{i_1},\\\ell_2'\in\mathscr{A}_{j_1}}}\frac{ y_{i_1,j_1;\ell_2,\ell_2'}\mathbbm{1}_{\ell_1'=\ell_2'}}{\max\big\{\NAM_{j_1,\ell_1'}(y),1\big\}}\nonumber\\
&=& C\sum_{i_1\in[n]}\sum_{\ell_1\in\mathscr{A}_{i_1}}\frac{\sum_{j_1\in[n]\backslash\{i_1\}}\sum_{\ell_2'\in\mathscr{A}_{j_1}}y_{i_1,j_1;\ell_1,\ell_2'}}{\max\big\{\NAM_{i_1,\ell_1}(y),1\big\}}+C\sum_{j_1\in[n]}\sum_{\ell_1'\in\mathscr{A}_{j_1}}\frac{\sum_{i_1\in[n]\backslash\{j_1\}}\sum_{\ell_2\in\mathscr{A}_{i_1}}y_{i_1,j_1;\ell_2,\ell_1'}}{\max\big\{\NAM_{j_1,\ell_1'}(y),1\big\}}\nonumber\\
&&\hspace{0.6in}\leq C\sum_{i_1\in[n]}\sum_{\ell_1\in\mathscr{A}_{i_1}}\frac{\NAM_{i_1,\ell_1}(y)}{\max\big\{\NAM_{i_1,\ell_1}(y),1\big\}}+C\sum_{j_1\in[n]}\sum_{\ell_1'\in\mathscr{A}_{j_1}}\frac{\NAM_{j_1,\ell_1'}(y)}{\max\big\{\NAM_{j_1,\ell_1'}(y),1\}} \leq CnK.
\end{eqnarray}
Combining (\ref{Eq.B4})-(\ref{Eq.B6}), we obtain the existence of an absolute positive constant $C_1$ such that $E_1\leq C_1nK$, where $E_1$ is defined as in~\eqref{E2.1}. With $\Lambda(\cdot)$ as in (\ref{def_L}), (\ref{UnifB}) gives
\begin{equation*}
    \Lambda(C_1 n K)\leq \log\left(2+\frac{n(n-1)\cdot(2+2\log(2nK))}{C_1nK}\right)
    \leq
    C_2\log\left(\frac{n\log(nK)}{K}+2\right),
\end{equation*}
where $C_2>0$ is an absolute constant. Now note that there exist absolute constants $C',C''>0$, such that $\log\left(\frac{n\log(nK)}{K}+2\right)=\log\left(\frac{n(\log n+\log K)}{K}+2\right)\leq \log\left(n\log n+C'n+2\right)\leq C'' n$. Hence $\Lambda(C_1 n K) \leq C_2\log\left(\frac{n\log(nK)}{K}+2\right)\leq C_2' n$, where $C_2'>0$ is an absolute constant.
\end{proof}

\noindent{\bf{Step 3: Bounding the complexity term.}}
For any $\epsilon>0$, let $E_2(\epsilon)$ be defined as in (\ref{def_E}), but with $\hat{\mathcal{Y}}_{n,\mathbf{K}},\mu,f,F,I$ replaced by $\hat{\mathcal{Y}}_{n,\pmb{\mathscr{A}}},\muAM,\fAM,\FAM,\IAM$ (as in~\eqref{def_YnKAM}-\eqref{Idefy}). In this step, we will show that there exists $\epsilon>0$, such that  
\begin{equation}\label{Step3_result}
    E_2(\epsilon) \leq CnK\log\Big(\frac{n}{K}+2\Big).
\end{equation}
\begin{proof}

For any $x\geq 0$, we define $h(x) := \psi(x+1)-\psi(x)$, where $\psi(\cdot)$ is as in~\eqref{phioriginal}. For any $\delta\in(0,1]$, we define
\begin{equation*}
    \mathcal{T}(\delta):=\bigg\{(x_{i,\ell})_{i\in [n],\ell\in \mathscr{A}_i}\in(\delta\mathbb{N})^{\sum_{i=1}^n|\mathscr{A}_i|}:\text{for every }i\in [n],\Big|\sum_{\ell\in\mathscr{A}_i} x_{i,\ell}-2(n-1)\Big|\leq K\bigg\},
\end{equation*}
\begin{eqnarray*}
    \Dtilde(\delta)&:=&\Big\{(d_{i,j;\ell,\ell'})_{i,j\in [n]:i\neq j;\ell\in \mathscr{A}_i,\ell'\in\mathscr{A}_j}: \text{ for every }i,j\in [n]\text{ with }i\neq j\text{ and }(\ell,\ell')\in \mathscr{A}_i\times \mathscr{A}_j,\nonumber\\
    &&\hspace{1.45in} d_{i,j;\ell,\ell'} = h(x_{i,\ell})+h(x_{j,\ell'}), \text{ where }(x_{i,\ell})_{i\in [n],\ell\in \mathscr{A}_i}\in\mathcal{T}(\delta)\Big\},
\end{eqnarray*}
where $\delta\mathbb{N}:=\{0,\delta,2\delta,\cdots\}$. Note that by arguments similar to (\ref{combinatorial_bdd}), for any $\delta\in (0,1]$, we have
\begin{eqnarray*}
  && |\Dtilde(\delta)| \leq |\mathcal{T}(\delta)|\leq \bigg(\sum_{m\in\big[\frac{2(n-1)-K}{\delta},\frac{2(n-1)+K}{\delta}\big]\cap\mathbb{N}} \Big(\frac{e(m+K)}{K}\Big)^K \bigg)^{n}\nonumber\\
  &\leq& \bigg(\Big(1+\frac{2K}{\delta}\Big)\Big[\frac{e\big(\frac{2n+K}{\delta}+K\big)}{K}\Big]^K \bigg)^{n}\leq\bigg( \frac{3K}{\delta}\Big[\frac{2e(n+K)}{\delta K}\Big]^K \bigg)^{n}.
\end{eqnarray*}
Hence
\begin{equation}\label{covering.set}
    \log(|\Dtilde(\delta)|)\leq n\log\Big(\frac{3K}{\delta}\Big)+nK\log\Big(\frac{2e(n+K)}{\delta K}\Big)
    \leq CnK\log\Big(\frac{n+2K}{\delta K}\Big). 
\end{equation}

Fix $y\in\hat{\mathcal{Y}}_{n,\pmb{\mathscr{A}}}$, and take $\delta =\min\{K\slash n,1\} \in (0,1]$. For every $i\in [n]$ and $\ell\in\mathscr{A}_i$, there exists $x_{i,\ell}\in\delta\mathbb{N}$ such that $\big|\NAM_{i,\ell}(y)-x_{i,\ell}\big|\leq \delta$. As $\sum_{\ell\in\mathscr{A}_i}\NAM_{i,\ell}(y)=2(n-1)$, we have 
\begin{equation*}
    \Big|\sum_{\ell\in\mathscr{A}_i}x_{i,\ell}-2(n-1)\Big|\leq \sum_{\ell\in\mathscr{A}_i} \big|x_{i,\ell}-\NAM_{i,\ell}(y)\big|\leq \delta|\mathscr{A}_i|\leq K.
\end{equation*}
Hence $x:=(x_{i,\ell})_{i\in [n],\ell\in\mathscr{A}_i}\in\mathcal{T}(\delta)$. With $x$ as above, we take $d=(d_{i,j;\ell,\ell'})_{i,j\in [n]:i\neq j;\ell\in\mathscr{A}_i,\ell'\in\mathscr{A}_j}\in\mathscr{D}(\delta)$ such that $d_{i,j;\ell,\ell'}=h(x_{i,\ell})+h(x_{j,\ell'})$ for all $i,j\in [n]$ with $i\neq j$ and $\ell\in \mathscr{A}_i,\ell'\in\mathscr{A}_j$. For any $i_0,j_0\in [n]$ with $i_0\neq j_0$ and $\ell_0\in\mathscr{A}_{i_0},\ell_0'\in\mathscr{A}_{j_0}$, with $\UIIM(y;i_0,j_0,\ell_0,\ell'_0)$ and $\UIM(y;i_0,j_0)$ defined analogously as in~\eqref{def_U} and~\eqref{def_Utilde}, respectively, we have 
\begin{eqnarray*}
   && \FAM\big(\UIIM(y;i_0,j_0,\ell_0,\ell'_0)\big)-\FAM\big(\UIM(y;i_0,j_0)\big)\nonumber\\
   &=& \sum_{i=1}^n\sum_{\ell\in\mathscr{A}_i}\Big(\psi\big(\NAM_{i,\ell}\big(\UIIM(y;i_0,j_0,\ell_0,\ell'_0)\big)\big)-\psi\big(\NAM_{i,\ell}\big(\UIM(y;i_0,j_0)\big)\big)\Big)\nonumber\\
   &=& \psi\big(\NAM_{i_0,\ell_0}\big(\UIM(y;i_0,j_0)\big)+1\big)-\psi\big(\NAM_{i_0,\ell_0}\big(\UIM(y;i_0,j_0)\big)\big)\nonumber\\
   &&\hspace{0.4in}+\psi\big(\NAM_{j_0,\ell_0'}\big(\UIM(y;i_0,j_0)\big)+1\big)-\psi\big(\NAM_{j_0,\ell_0'}\big(\UIM(y;i_0,j_0)\big)\big)\nonumber\\
    &=& h\big(\NAM_{i_0,\ell_0}\big(\UIM(y;i_0,j_0)\big)\big)+h\big(\NAM_{j_0,\ell_0'}\big(\UIM(y;i_0,j_0)\big)\big).
\end{eqnarray*}
Hence by (\ref{derivative.h}), we get
\begin{eqnarray*}
   && \Big|\FAM\big(\UIIM(y;i_0,j_0,\ell_0,\ell'_0)\big)-\FAM\big(\UIM(y;i_0,j_0)\big)-h\big(\NAM_{i_0,\ell_0}(y)\big)-h\big(\NAM_{j_0,\ell_0'}(y)\big)\Big|\nonumber\\
   &\leq& \frac{\sum_{\ell'\in\mathscr{A}_{j_0}}y_{i_0,j_0;\ell_0,\ell'}}{\max\big\{\NAM_{i_0,\ell_0}\big(\UIM(y;i_0,j_0)\big)
   ,1\big\}}+ \frac{\sum_{\ell\in \mathscr{A}_{i_0}}y_{i_0,j_0;\ell,\ell_0'}}{\max\big\{\NAM_{j_0,\ell_0'}\big(\UIM(y;i_0,j_0)\big),1\big\}}\nonumber\\
   &\leq& \frac{\sum_{\ell'\in\mathscr{A}_{j_0}}y_{i_0,j_0;\ell_0,\ell'}}{\max\big\{\NAM_{i_0,\ell_0}(y)-1,1\big\}}+\frac{\sum_{\ell\in \mathscr{A}_{i_0}}y_{i_0,j_0;\ell,\ell_0'}}{\max\big\{\NAM_{j_0,\ell_0'}(y)-1,1\big\}}\nonumber\\
   &\leq& \frac{2\sum_{\ell'\in\mathscr{A}_{j_0}}y_{i_0,j_0;\ell_0,\ell'}}{\max\big\{\NAM_{i_0,\ell_0}(y),1\big\}}+\frac{2\sum_{\ell\in \mathscr{A}_{i_0}}y_{i_0,j_0;\ell,\ell_0'}}{\max\big\{\NAM_{j_0,\ell_0'}(y),1\big\}}.
\end{eqnarray*}
Moreover, with $d_{i_0,j_0;\ell_0,\ell_0'}=h(x_{i_0,\ell_0})+h(x_{j_0,\ell_0'})$, we have
\begin{eqnarray*}
    && \Big|h\big(\NAM_{i_0,\ell_0}(y)\big)+h\big(\NAM_{j_0,\ell_0'}(y)\big)-d_{i_0,j_0;\ell_0,\ell_0'}\Big| \nonumber\\
    &\leq& \Big|h\big(\NAM_{i_0,\ell_0}(y)\big)-h(x_{i_0,\ell_0})\Big|+ \Big|h\big(\NAM_{j_0,\ell_0'}(y)\big)-h(x_{j_0,\ell_0'})\Big|\nonumber\\
    &\leq& \frac{\delta}{\max\big\{\NAM_{i_0,\ell_0}(y)-\delta,1\big\}}+\frac{\delta}{\max\big\{\NAM_{j_0,\ell_0'}(y)-\delta,1\big\}}\leq\frac{2\delta}{\max\big\{\NAM_{i_0,\ell_0}(y),1\big\}}+\frac{2\delta}{\max\big\{\NAM_{j_0,\ell_0'}(y),1\big\}}.
\end{eqnarray*}
Therefore,
\begin{eqnarray*}
    &&\Big|\FAM\big(\UIIM(y;i_0,j_0,\ell_0,\ell'_0)\big)-\FAM\big(\UIM(y;i_0,j_0)\big)-d_{i_0,j_0;\ell_0,\ell_0'}\Big|\nonumber\\
    &\leq& \frac{2\sum_{\ell'\in\mathscr{A}_{j_0}}y_{i_0,j_0;\ell_0,\ell'}+2\delta}{\max\big\{\NAM_{i_0,\ell_0}(y),1\big\}}+\frac{2\sum_{\ell\in \mathscr{A}_{i_0}}y_{i_0,j_0;\ell,\ell_0'}+2\delta}{\max\big\{\NAM_{j_0,\ell_0'}(y),1\big\}}\nonumber\\
    &\leq& \frac{2\sum_{\ell'\in\mathscr{A}_{j_0}}y_{i_0,j_0;\ell_0,\ell'}}{\max\big\{\NAM_{i_0,\ell_0}(y),1\big\}}+\frac{2\sum_{\ell\in \mathscr{A}_{i_0}}y_{i_0,j_0;\ell,\ell_0'}}{\max\big\{\NAM_{j_0,\ell_0'}(y),1\big\}}+4\delta.
\end{eqnarray*}
Hence
\begin{eqnarray*}
    && \sum_{i_0,j_0\in[n]:i_0\neq j_0}\max_{\ell_0\in\mathscr{A}_{i_0},\ell_0'\in\mathscr{A}_{j_0}}\Big\{\Big|\FAM\big(\UIIM(y;i_0,j_0,\ell_0,\ell_0')\big)-\FAM\big(\UIM(y;i_0,j_0)\big)-d_{i_0,j_0;\ell_0,\ell_0'}\Big|\Big\}\nonumber\\
    &\leq& \sum_{i_0\in [n]}\sum_{\ell_0\in\mathscr{A}_{i_0}}\frac{2\sum_{j_0\in[n]\backslash\{i_0\}}\sum_{\ell'\in\mathscr{A}_{j_0}}y_{i_0,j_0;\ell_0,\ell'}}{\max\big\{\NAM_{i_0,\ell_0}(y),1\big\}}\nonumber\\
&&\hspace{0.4in}+\sum_{j_0\in[n]}\sum_{\ell_0'\in\mathscr{A}_{j_0}}\frac{2\sum_{i_0\in[n]\backslash\{j_0\}}\sum_{\ell\in \mathscr{A}_{i_0}}y_{i_0,j_0;\ell,\ell_0'}}{\max\big\{\NAM_{j_0,\ell_0'}(y),1\big\}}+4\delta n^2\nonumber\\
    &\leq& \sum_{i_0\in [n]}\sum_{\ell_0\in\mathscr{A}_{i_0}}\frac{2\NAM_{i_0,\ell_0}(y)}{\max\big\{\NAM_{i_0,\ell_0}(y),1\big\}}+\sum_{j_0\in[n]}\sum_{\ell_0'\in\mathscr{A}_{j_0}}\frac{2\NAM_{j_0,\ell_0'}(y)}{\max\big\{\NAM_{j_0,\ell_0'}(y),1\big\}}+4\delta n^2\nonumber\\
    &\leq& 4nK+4\delta n^2\leq 8 nK,
\end{eqnarray*}
where the fourth line uses the definition of $\NAM_{i,\ell}(y)$ in~\eqref{Ntilde_d}. Thus with $\mathcal{D}(\epsilon)$ and $E_2(\epsilon)$ as in
Definition \ref{DefComplexityTerm}, choosing $\epsilon=8nK$ and $\mathcal{D}(\epsilon)=\Dtilde(\delta)$, we get that  
\begin{equation*}
    E_2(8nK)\leq 16 nK+\log(|\mathcal{D}(8 nK)|)=16 nK+\log(|\mathscr{D}(\delta)|)\leq CnK\log\Big(\frac{n}{K}+2\Big), 
\end{equation*}
where the last inequality uses~\eqref{covering.set}.
\end{proof}

\noindent{\bf{Step 4: Bounding $\log \mathcal{S}_{n,K,\pmb{\mathscr{A}}}$.}} In this step, using Theorem \ref{Theorem2.1} and \textbf{Steps 2} and \textbf{3}, we will show that 
\begin{eqnarray}\label{SnK_A}
    &&\log\mathcal{S}_{n,K,\pmb{\mathscr{A}}}\nonumber\\
    &\leq& \sup_{y\in\hat{\mathcal{Y}}_{n,\pmb{\mathscr{A}}}}\bigg\{\sum_{\substack{i,j\in [n]: \\ i\neq j}}\sum_{\ell\in\mathscr{A}_i}\sum_{\ell'\in\mathscr{A}_j}y_{i,j;\ell,\ell'}(X_{i,j}\log B_{\ell,\ell'}+(1-X_{i,j})\log(1-B_{\ell,\ell'}))\nonumber\\
    &&\hspace{0.6in}+\sum_{i=1}^n\sum_{\ell\in\mathscr{A}_i}\log\Gamma\big(\NAM_{i,\ell}(y)+\alpha_{\ell}\big)-\sum_{\substack{i,j\in [n]:\\i\neq j}}\sum_{\ell\in\mathscr{A}_i}\sum_{\ell'\in\mathscr{A}_j} y_{i,j;\ell,\ell'}\log y_{i,j;\ell,\ell'}\bigg\}\nonumber\\
    &&\hspace{0.4in}+2n(n-1)+CnK\log\Big(\frac{n}{K}+2\Big).
\end{eqnarray}
Here $\mathcal{S}_{n,K,\pmb{\mathscr{A}}}$ is as defined in {\bf Step 1}.
\begin{proof}
To begin, using (\ref{EEEq4.10}) we have 
\begin{eqnarray}\label{Eq.B10}
    \log \mathcal{S}_{n,K,\pmb{\mathscr{A}}}&=& \sum_{\substack{i,j\in [n]:\\i\neq j,X_{i,j}=1}}\log\bigg(\sum_{\ell\in\mathscr{A}_i}\sum_{\ell'\in\mathscr{A}_j}B_{\ell,\ell'}\bigg)+\sum_{\substack{i,j\in [n]:\\i\neq j,X_{i,j}=0}}\log\bigg(\sum_{\ell\in\mathscr{A}_i}\sum_{\ell'\in\mathscr{A}_j}(1-B_{\ell,\ell'})\bigg)\nonumber\\
    &&\hspace{0.4in}+\log\bigg(\sum_{z\in\mathcal{Z}_{n,K,\pmb{\mathscr{A}}}}  e^{\fAM(z)}\muAM(z)\bigg)\nonumber\\
    &\leq&\sum_{\substack{i,j\in [n]:\\i\neq j,X_{i,j}=1}}\log\bigg(\sum_{\ell\in\mathscr{A}_i}\sum_{\ell'\in\mathscr{A}_j}B_{\ell,\ell'}\bigg)+\sum_{\substack{i,j\in [n]:\\i\neq j,X_{i,j}=0}}\log\bigg(\sum_{\ell\in\mathscr{A}_i}\sum_{\ell'\in\mathscr{A}_j}(1-B_{\ell,\ell'})\bigg)\nonumber\\
    &&\hspace{0.4in}+\sup_{y\in\hat{\mathcal{Y}}_{n,\pmb{\mathscr{A}}}}\Big\{\FAM(y)-\IAM(y)\Big\}+CnK\log\Big(\frac{n}{K}+2\Big)\nonumber\\
    &=& \sup_{y\in\hat{\mathcal{Y}}_{n,\pmb{\mathscr{A}}}}\bigg\{\sum_{\substack{i,j\in [n]: \\ i\neq j}}\sum_{\ell\in\mathscr{A}_i}\sum_{\ell'\in\mathscr{A}_j}y_{i,j;\ell,\ell'}(X_{i,j}\log B_{\ell,\ell'}+(1-X_{i,j})\log(1-B_{\ell,\ell'}))\nonumber\\
    &&\hspace{0.6in}+\sum_{i=1}^n\sum_{\ell\in\mathscr{A}_i} \psi\big(\NAM_{i,\ell}(y)\big)-\sum_{\substack{i,j\in [n]:\\i\neq j}}\sum_{\ell\in\mathscr{A}_i}\sum_{\ell'\in\mathscr{A}_j} y_{i,j;\ell,\ell'}\log y_{i,j;\ell,\ell'}\bigg\}\nonumber\\
    &&\hspace{0.4in}+CnK\log\Big(\frac{n}{K}+2\Big),
\end{eqnarray}
where the inequality in the third line uses Theorem \ref{Theorem2.1}, along with \textbf{Steps 2} and \textbf{3} (namely, (\ref{Step2_result}) and (\ref{Step3_result})), which bound the error terms in Theorem \ref{Theorem2.1}. Here, $\FAM$ and $\IAM$ are defined as in~\eqref{defHd} and~\eqref{Idefy}, respectively. 

Proceeding to bound $\sum_{i=1}^n\sum_{\ell\in\mathscr{A}_i} \psi\big(\NAM_{i,\ell}(y)\big)$ in the RHS of~\eqref{Eq.B10}, for any $y\in\hat{\mathcal{Y}}_{n,\pmb{\mathscr{A}}}$ we have 
\begin{eqnarray}\label{Eq3.11n}
  && \sum_{i=1}^n\sum_{\ell\in\mathscr{A}_i} \log\big(\NAM_{i,\ell}(y) +\alpha_{\ell}\big)  = \sum_{i=1}^n\bigg(\sum_{\ell\in \mathscr{A}_i}\log\big(\NAM_{i,\ell}(y)+\alpha_{\ell}\big)+\sum_{\ell\in [K]\backslash \mathscr{A}_i}\log(1)\bigg) \nonumber\\
  &\leq& 
\sum_{i=1}^n K \log\bigg(\frac{\sum_{\ell\in \mathscr{A}_i}\NAM_{i,\ell}(y)+\sum_{\ell\in\mathscr{A}_i}\alpha_{\ell}+K-|\mathscr{A}_i|}{K}\bigg)\leq CnK\log\Big(\frac{n}{K}+2\Big),
\end{eqnarray}
where the last two inequalities use Jensen's inequality, \eqref{NAilsum}, and the bound $\max_{\ell\in[K]}\alpha_{\ell}\leq C_0$. Moreover, for any $y\in \hat{\mathcal{Y}}_{n,\pmb{\mathscr{A}}}$ and $i\in [n],\ell\in \mathscr{A}_i$, we have
\begin{equation*}
    \big(\NAM_{i,\ell}(y) +\alpha_{\ell}\big)\log\big(\NAM_{i,\ell}(y)+\alpha_{\ell}\big)\geq
    \begin{cases}
        \inf_{x\in [0,C_0+1]}\{x\log{x}\}\geq -C & \text{ if }\NAM_{i,\ell}(y)\leq 1\\
        \NAM_{i,\ell}(y)\log\big(\NAM_{i,\ell}(y)\big) &\text{ if }\NAM_{i,\ell}(y)>1,
    \end{cases}
\end{equation*}
hence
\begin{equation}\label{Eq3.10n}
    \NAM_{i,\ell}(y) \log\big(\NAM_{i,\ell}(y)\big)\leq \big(\NAM_{i,\ell}(y)+\alpha_{\ell}\big)\log\big(\NAM_{i,\ell}(y)+\alpha_{\ell}\big)+C.
\end{equation}
By applying (\ref{eq3.2}), (\ref{Eq3.10n}), Lemma \ref{L3.2} (in Appendix \ref{Appendix_E}), \eqref{NAilsum}, \eqref{Eq3.11n} chronologically, for any $y\in\hat{\mathcal{Y}}_{n,\pmb{\mathscr{A}}}$ we get
\begin{eqnarray}\label{EEEq4.5}
  &&  \sum_{i=1}^n\sum_{\ell\in\mathscr{A}_i} \psi\big(\NAM_{i,\ell}(y)\big)\leq\frac{nK}{2}+\sum_{i=1}^n\sum_{\ell\in\mathscr{A}_i} \phi\big(\NAM_{i,\ell}(y)\big)\nonumber\\
  &\leq& CnK+\sum_{i=1}^n\sum_{\ell\in\mathscr{A}_i} \big(\NAM_{i,\ell}(y) +\alpha_{\ell}\big)\log\big(\NAM_{i,\ell}(y)+\alpha_{\ell}\big)\nonumber\\
  &\leq&
  CnK+\sum_{i=1}^n\sum_{\ell\in\mathscr{A}_i}\bigg(\log \Gamma\big(\NAM_{i,\ell}(y)+\alpha_{\ell}\big)+\NAM_{i,\ell}(y)+\alpha_{\ell}+\frac{1}{2}\log\big(\NAM_{i,\ell}(y)+\alpha_{\ell}\big)\bigg)\nonumber\\
  &\leq& \sum_{i=1}^n\sum_{\ell\in\mathscr{A}_i}\log\Gamma\big(\NAM_{i,\ell}(y)+\alpha_{\ell}\big)+2n(n-1)+CnK\log\Big(\frac{n}{K}+2\Big).
\end{eqnarray}
Now the desired conclusion (\ref{SnK_A}) follows from (\ref{Eq.B10}) and (\ref{EEEq4.5}).  
\end{proof} 

\noindent{\bf{Step 5: Bounding $\log\mathcal{S}_{n,K}$ and concluding.}} Recall the setup in~\eqref{def_muij}-\eqref{partition_MMSB}. Using the results from \textbf{Steps 1} and \textbf{4} (namely, (\ref{Eq.A32}) and (\ref{SnK_A})), we get
\begin{eqnarray*}
    \log \mathcal{S}_{n,K} &\leq& \sup_{y\in\hat{\mathcal{Y}}_{n,K}}\bigg\{\sum_{i=1}^n\sum_{\ell=1}^K\log\Gamma\big(\widetilde{N}_{i,\ell}(y)+\alpha_{\ell}\big)-\sum_{\substack{i,j\in [n]:\\i\neq j}}\sum_{\ell=1}^K\sum_{\ell'=1}^K y_{i,j;\ell,\ell'}\log y_{i,j;\ell,\ell'}\nonumber\\
    &&\hspace{0.6in}+\sum_{\substack{i,j\in [n]: \\ i\neq j}}\sum_{\ell=1}^K\sum_{\ell'=1}^K y_{i,j;\ell,\ell'}(X_{i,j}\log B_{\ell,\ell'}+(1-X_{i,j})\log(1-B_{\ell,\ell'}))\bigg\}\nonumber\\
    &&\hspace{0.2in}-n\log \Gamma\Big(2n-2+\sum_{\ell=1}^K \alpha_{\ell}\Big)+CnK\log\Big(\frac{n}{K}+2\Big)\nonumber\\
    &=& \sup_{y\in\hat{\mathcal{Y}}_{n,K}}\{F(y)-I(y)\}+\mathscr{M}_1\log\Big(\sum_{\ell=1}^K\sum_{\ell'=1}^K  B_{\ell,\ell'}\Big)+\mathscr{M}_0\log\Big(\sum_{\ell=1}^K\sum_{\ell'=1}^K (1-B_{\ell,\ell'})\Big)\nonumber\\
    &&\hspace{0.2in}+CnK\log\Big(\frac{n}{K}+2\Big).
\end{eqnarray*}
In the above display, we use the fact that 
\begin{eqnarray*}
    &&\sup_{\substack{\pmb{\mathscr{A}}\in\mathscr{C}_{n,K}:\\ y\in\hat{\mathcal{Y}}_{n,\pmb{\mathscr{A}}}}}\bigg\{\sum_{i=1}^n\sum_{\ell\in [K]\backslash \mathscr{A}_i}\log\Gamma(\alpha_{\ell})+\sum_{i=1}^n\sum_{\ell\in\mathscr{A}_i}\log\Gamma\big(\NAM_{i,\ell}(y)+\alpha_{\ell}\big)\nonumber\\
    &&\hspace{0.65in}-\sum_{\substack{i,j\in [n]:\\i\neq j}}\sum_{\ell\in\mathscr{A}_i}\sum_{\ell'\in\mathscr{A}_j} y_{i,j;\ell,\ell'}\log y_{i,j;\ell,\ell'}\nonumber\\&&\hspace{0.65in}+\sum_{\substack{i,j\in [n]: \\ i\neq j}}\sum_{\ell\in\mathscr{A}_i}\sum_{\ell'\in\mathscr{A}_j}y_{i,j;\ell,\ell'}(X_{i,j}\log B_{\ell,\ell'}+(1-X_{i,j})\log(1-B_{\ell,\ell'}))\bigg\}\nonumber\\
    &=& \sup_{y\in\hat{\mathcal{Y}}_{n,K}}\bigg\{\sum_{i=1}^n\sum_{\ell=1}^K\log\Gamma\big(\widetilde{N}_{i,\ell}(y)+\alpha_{\ell}\big)-\sum_{\substack{i,j\in [n]:\\i\neq j}}\sum_{\ell=1}^K\sum_{\ell'=1}^K y_{i,j;\ell,\ell'}\log y_{i,j;\ell,\ell'}\nonumber\\
    &&\hspace{0.6in}+\sum_{\substack{i,j\in [n]: \\ i\neq j}}\sum_{\ell=1}^K\sum_{\ell'=1}^K y_{i,j;\ell,\ell'}(X_{i,j}\log B_{\ell,\ell'}+(1-X_{i,j})\log(1-B_{\ell,\ell'}))\bigg\}
\end{eqnarray*}
in the first line, and the definitions of $F(\cdot)$ and $I(\cdot)$ as in~\eqref{deffF_MMSB2} and~\eqref{defI_MMSB} in the second line. Hence by~\eqref{partition_MMSB}, we have
\begin{equation}\label{resu}
    \frac{\log\Big(\sum_{z\in [K]^{2n(n-1)}} e^{f(z)}\mu(z)\Big)-\sup_{y\in\hat{\mathcal{Y}}_{n,K}}\{F(y)-I(y)\}}{n^2}\leq \frac{CK}{n}\log\Big(\frac{n}{K}+2\Big).
\end{equation}
Part (b) of Theorem \ref{Theorem_MMSB_UBD} then follows on using \eqref{resu} along with parts (a) and (c) of Lemma 2.1, on noting that $F(\cdot)$ is convex (which follows from the convexity of $\log\Gamma(\cdot)$ on $(0,\infty)$).

Now we turn to the proof of part (a) of Theorem \ref{Theorem_MMSB_UBD}. Let
\begin{equation*}
    \Pi_K:=\bigg\{(u_{\ell})_{\ell\in [K]}: u_{\ell}\geq 0 \text{ for every }\ell\in [K], \sum_{\ell=1}^K u_{\ell}=1\bigg\}.
\end{equation*}
With $R(\cdot,\cdot)$ as in~\eqref{defrR_MMSB2}, and with $R_{i_1,j_1,\ell_1,\ell_1';i_2,j_2,\ell_2,\ell_2'}$ and  $\ctilde_{i_1,j_1,\ell_1,\ell_1';i_2,j_2,\ell_2,\ell_2'}$ defined analogously as in~\eqref{def_Rde} and Definition \ref{Global_GH}, we obtain that for any $\pmb{\pi}=(\pmb{\pi}_1,\cdots,\pmb{\pi}_n)$ such that $\pmb{\pi}_i\in \Pi_K$ for every $i\in [n]$, $y\in [0,1]^{n(n-1)K^2}$, and $i_1,j_1,i_2,j_2\in [n], \ell_1,\ell_1',\ell_2,\ell_2'\in [K]$ such that $i_1\neq j_1,i_2\neq j_2$, we have $R_{i_1,j_1,\ell_1,\ell_1';i_2,j_2,\ell_2,\ell_2'}(\pmb{\pi},y)=0$ (using linearity of $R(\pmb{\pi},\cdot)$), hence $\ctilde_{i_1,j_1,\ell_1,\ell_1';i_2,j_2,\ell_2,\ell_2'}=0$. Hence by Theorem \ref{Theorem2.2} and~\eqref{resu}, we conclude that $\frac{1}{n^2} \KL\big(\hat{P} \,\big\| \; \mathbb{P}(\mathbf{Z},\pmb{\pi}\mid\mathbf{X}) \big)\leq \frac{CK}{n}\log\big(\frac{n}{K}+2\big)$.

\section{Lower bounds for LDA and MMSB: Proofs of Theorems \ref{Theorem_LDA_LBD} and \ref{Theorem_MMSB_LBD}}\label{Appendix_C}

In this section, we present the proofs of Theorems \ref{Theorem_LDA_LBD} and \ref{Theorem_MMSB_LBD}. Throughout this section, we use $C$ to denote an absolute positive constant. The specific value of $C$ may vary from line to line.

\subsection{Proof of Theorem \ref{Theorem_LDA_LBD}}

We recall the model setup and notation for LDA in Section \ref{Sect.1.2}. For any $z=(z_{d,i})_{d\in[D],i\in[n_d]}\in[K]^n$ and any $d\in[D],\ell\in[K],r\in[V]$, similar to~\eqref{def_N}, we define
\begin{equation*}
    N_{d,\ell}(z):=|\{i\in[n_d]:z_{d,i}=\ell\}|,\qquad \mbox{and} \qquad  N_{d,\ell,r}(z):=|\{i\in[n_d]: z_{d,i}=\ell, X_{d,i}=r\}|.
\end{equation*}
For any $z\in [K]^n$, similar to~\eqref{Eq.A22} and~\eqref{deffF}, and noting that $\alpha_{\ell}=1\slash 2$ and $\eta_{\ell,r}=V^{-1}$ for every $\ell\in[K],r\in[V]$, we define
{\small
\begin{equation}\label{def_Up}
    \Upsilon(z):=-n\log V +\sum_{d=1}^D\sum_{\ell=1}^K \log \Gamma(N_{d,\ell}(z)+1\slash 2) -\sum_{d=1}^D\log\Gamma(n_d+K\slash 2),
\end{equation}
\begin{equation}\label{def_f_new}
    f(z):=\sum_{d=1}^D\sum_{\ell=1}^K \log \Gamma(N_{d,\ell}(z)+1\slash 2) -\sum_{d=1}^D\log\Gamma(n_d+K\slash 2).
\end{equation}
}We denote by $\mathcal{S}_{n,K}$ the normalizing constant of the posterior\slash collapsed posterior, as in~\eqref{eq:S_nK}:
\begin{equation}\label{S_nK:LDA}
  \mathcal{S}_{n,K}=\sum_{z\in [K]^n} e^{\Upsilon(z)}.
\end{equation}

\medskip

\noindent{\bf{Step 1.}} We first derive a lower bound on $\log \mathcal{S}_{n,K}$. Note that by~\eqref{def_Up} and~\eqref{S_nK:LDA}, we have 
\begin{eqnarray}\label{E5.14}
    \mathcal{S}_{n,K}&=&
     \frac{1}{V^n\prod_{d=1}^D\Gamma(n_d+K\slash 2)}\sum_{z\in [K]^n}\prod_{d=1}^D\prod_{\ell=1}^K \Gamma(N_{d,\ell}(z)+1 \slash 2) \nonumber\\
     &=&   \frac{1}{V^n\prod_{d=1}^D\Gamma(n_d+K\slash 2)}\sum_{\substack{(n_{d,\ell})_{d\in[D],\ell\in[K]}\in\mathbb{N}^{DK}:\\\sum_{\ell=1}^K n_{d,\ell}=n_d,\forall d\in[D]}} \sum_{\substack{z\in [K]^n:\\ N_{d,\ell}(z)=n_{d,\ell},\\\forall d\in[D],\ell\in[K]}}  \prod_{d=1}^D\prod_{\ell=1}^K \Gamma(n_{d,\ell}+1 \slash 2)  \nonumber\\
     &=&   \frac{1}{V^n\prod_{d=1}^D\Gamma(n_d+K\slash 2)}\sum_{\substack{(n_{d,\ell})_{d\in[D],\ell\in[K]}\in\mathbb{N}^{DK}:\\\sum_{\ell=1}^K n_{d,\ell}=n_d,\forall d\in[D]}}  \prod_{d=1}^D \frac{ n_d!\prod_{\ell=1}^K\Gamma(n_{d,\ell}+1 \slash 2)}{\prod_{\ell=1}^K n_{d,\ell}!} \nonumber\\
    &=& \frac{\prod_{d=1}^D n_d!}{V^n\prod_{d=1}^D\Gamma(n_d+K\slash 2)}\prod_{d=1}^D\bigg(\sum_{\substack{(n_{d,\ell})_{\ell\in[K]}\in\mathbb{N}^K:\\\sum_{\ell=1}^K n_{d,\ell}=n_d}}\prod_{\ell=1}^K\frac{\Gamma(n_{d,\ell}+1 \slash 2)}{n_{d,\ell}!}\bigg).
\end{eqnarray}

By Stirling's approximation in Lemma \ref{L3.2} (in Appendix \ref{Appendix_E}), for any $d\in [D]$, $\ell\in[K]$, and $n_{d,\ell}\in\mathbb{N}$, 
\begin{eqnarray}\label{stirlingapprox}
    \Big|\log\Gamma(n_{d,\ell}+1\slash 2)-n_{d,\ell}\log n_{d,\ell}+n_{d,\ell}\Big|\leq C,
\end{eqnarray}
\begin{equation*}
    \bigg|\log(n_{d,\ell}!)-n_{d,\ell}\log n_{d,\ell}+n_{d,\ell}-\frac{1}{2}\log(\max\{n_{d,\ell},1\})\bigg|\leq C.
\end{equation*}
Hence for any $d\in [D]$, 
\begin{eqnarray}\label{E5.10}
&& \sum_{\substack{(n_{d,\ell})_{\ell\in[K]}\in\mathbb{N}^K:\\\sum_{\ell=1}^K n_{d,\ell}=n_d}}\prod_{\ell=1}^K\frac{\Gamma(n_{d,\ell}+1 \slash 2)}{n_{d,\ell}!}\nonumber\\
&\geq&  \exp(-CK)\sum_{\substack{(n_{d,\ell})_{\ell\in[K]}\in\mathbb{N}^K:\\\sum_{\ell=1}^K n_{d,\ell}=n_d}}\exp\bigg(-\frac{1}{2}\sum_{\ell=1}^K\log(\max\{n_{d,\ell},1\})\bigg). 
\end{eqnarray}
Throughout the rest of the proof, we assume that $DK\leq n$, which implies $n_d=n\slash D\geq K$. By Jensen's inequality, for any $(n_{d,\ell})_{\ell\in[K]}\in\mathbb{N}^K$ such that $\sum_{\ell=1}^K n_{d,\ell}=n_d$, we have  
\begin{eqnarray}\label{E5.11}
    && -\frac{1}{K}\sum_{\ell=1}^K\log(\max\{n_{d,\ell},1\})\geq -\log\bigg(\frac{1}{K}\sum_{\ell=1}^K \max\{n_{d,\ell},1\}  \bigg)
    \nonumber\\
    &\geq&-\log\Big(\frac{n_d}{K}+1\Big)\geq -\log\Big(\frac{n_d}{K}\Big)-\log{2}. 
\end{eqnarray}
Moreover,
\begin{equation}\label{E5.12}
    \binom{n_d+K-1}{K-1}=\frac{(n_d+K-1)(n_d+K-2)\cdots (n_d+1)}{(K-1)!}\geq \frac{n_d^{K-1}}{(K-1)!}\geq \Big(\frac{n_d}{K}\Big)^{K-1}. 
\end{equation}
By (\ref{E5.10})-(\ref{E5.12}),
\begin{eqnarray}\label{E5.13}
&&\sum_{\substack{(n_{d,\ell})_{\ell\in[K]}\in\mathbb{N}^K:\\\sum_{\ell=1}^K n_{d,\ell}=n_d}}\prod_{\ell=1}^K\frac{\Gamma(n_{d,\ell}+1 \slash 2)}{n_{d,\ell}!}\nonumber\\
&\geq&\exp\bigg(-CK-\frac{K}{2}\log(n_d\slash K)\bigg)\bigg|\bigg\{(n_{d,\ell})_{\ell\in[K]}\in\mathbb{N}^K:\sum_{\ell=1}^K n_{d,\ell}=n_d\bigg\}\bigg|\nonumber\\
&=& \exp\bigg(-CK-\frac{K}{2}\log(n_d\slash K)\bigg)\binom{n_d+K-1}{K-1}\nonumber\\
&\geq& \exp\bigg(\bigg(\frac{K}{2}-1\bigg)\log(n_d\slash K)-CK\bigg).
\end{eqnarray}

By (\ref{E5.14}) and (\ref{E5.13}), we get 
\begin{eqnarray}\label{E5.15}
   \log \mathcal{S}_{n,K} &\geq& \sum_{d=1}^D\log(n_d!)-\sum_{d=1}^D\log\Gamma(n_d+K\slash 2)-n\log V+\bigg(\frac{K}{2}-1\bigg)\sum_{d=1}^D\log\Big(\frac{n_d}{K}\Big)-CDK\nonumber\\
   &\geq& \sum_{d=1}^D n_d\log n_d -n+\frac{1}{2}\sum_{d=1}^D \log n_d-\sum_{d=1}^D\log\Gamma(n_d+K\slash 2)\nonumber\\
   &&\hspace{0.4in}-n\log V+\bigg(\frac{K}{2}-1\bigg)\sum_{d=1}^D\log\Big(\frac{n_d}{K}\Big)-CDK\nonumber\\
   &\geq& \sum_{d=1}^D n_d\log n_d-n-\sum_{d=1}^D\log\Gamma(n_d+K\slash 2)-n\log V+\frac{(K-1)}{2}\sum_{d=1}^D\log\Big(\frac{n_d}{K}\Big)-CDK,
\end{eqnarray}
where in the second inequality we use Stirling's approximation in Lemma \ref{L3.2}. 

\medskip

\noindent{\bf{Step 2.}} We define
\begin{equation*}
  \hat{\mathcal{Y}}_{n,K}:=\bigg\{y=(y_{d,i,\ell})_{d\in[D],i\in[n_d],\ell\in[K]}\in [0,1]^{nK}: \sum_{\ell=1}^{K} y_{d,i,\ell}=1\text{ for every }d\in[D],i\in [n_d]\bigg\}.
\end{equation*}
For any $y=(y_{d,i,\ell})_{d\in[D],i\in[n_d],\ell\in[K]}\in\hat{\mathcal{Y}}_{n,K}$, similar to~\eqref{defQy}, we denote by $Q_y$ the product probability distribution on $[K]^n$
given by $Q_y(z)=\prod_{d=1}^D\prod_{i=1}^{n_d}y_{d,i, z_{d,i}}$ for all $z=(z_{d,i})_{d\in[D],i\in[n_d]}\in[K]^n$. For any $y\in\hat{\mathcal{Y}}_{n,K}$ and $d\in[D],\ell\in [K]$, we define $\widetilde{N}_{d,\ell}(y):=\sum_{i=1}^{n_d}y_{d,i,\ell}$. Similar to~\eqref{def_mu_form}, we set $\mu:=\bigotimes\limits_{d\in[D],i\in[n_d]}\mu_{d,i}$, where $\mu_{d,i}(\ell)=\frac{\eta_{\ell,X_{d,i}}}{\sum_{s=1}^K\eta_{s,X_{d,i}}}=\frac{1}{K}$ for every $d\in [D],i\in[n_d],\ell\in[K]$ (note that $\eta_{\ell,r}=V^{-1}$ for every $\ell\in[K],r\in[V]$). For any $y\in\hat{\mathcal{Y}}_{n,K}$, analogously as in~\eqref{def_I} and~\eqref{defI}, we define
\begin{equation*}
    I(y):=\sum_{d=1}^D\sum_{i=1}^{n_d}\sum_{\ell=1}^K y_{d,i,\ell}\log\bigg( \frac{y_{d,i,\ell}}{\mu_{d,i}(\ell)}\bigg) = \sum_{d=1}^D\sum_{i=1}^{n_d}\sum_{\ell=1}^K y_{d,i,\ell}\log y_{d,i,\ell}+n\log K.
\end{equation*}

In this step, with $f(\cdot)$ as in~\eqref{def_f_new}, we derive an upper bound on $\mathbb{E}_{Q_y}[f(\mathbf{Z})]-I(y)$ for any $y\in  \hat{\mathcal{Y}}_{n,K}$. Note that
\begin{eqnarray}\label{E5.4}
    \mathbb{E}_{Q_y}[f(\mathbf{Z})]-I(y)
   &=& \sum_{d=1}^D\sum_{\ell=1}^K\mathbb{E}_{Q_y}[\log\Gamma(N_{d,\ell}(\mathbf{Z})+1\slash 2)]-\sum_{d=1}^D\log\Gamma(n_d+K\slash 2)\nonumber\\
   &&\hspace{0.4in}-\sum_{d=1}^D\sum_{i=1}^{n_d}\sum_{\ell=1}^K y_{d,i,\ell}\log y_{d,i,\ell}-n\log K.
\end{eqnarray}

By Stirling's approximation in~\eqref{stirlingapprox}, we have
\begin{eqnarray}\label{E5.1}
    &&\sum_{d=1}^D\sum_{\ell=1}^K\mathbb{E}_{Q_y}[\log\Gamma(N_{d,\ell}(\mathbf{Z})+1\slash 2)]\nonumber\\
    &\leq& \sum_{d=1}^D\sum_{\ell=1}^K \mathbb{E}_{Q_y}[N_{d,\ell}(\mathbf{Z})\log(N_{d,\ell}(\mathbf{Z}))]-\sum_{d=1}^D\sum_{\ell=1}^K \mathbb{E}_{Q_y}[N_{d,\ell}(\mathbf{Z})]+CDK\nonumber\\
    &=& \sum_{d=1}^D\sum_{\ell=1}^K \mathbb{E}_{Q_y}[N_{d,\ell}(\mathbf{Z})\log(N_{d,\ell}(\mathbf{Z}))]-n+CDK.
\end{eqnarray}
Below we consider any $d\in[D]$ and $\ell\in[K]$. If $\widetilde{N}_{d,\ell}(y)=0$, then $y_{d,i,\ell}=0$ for all $i\in [n_d]$, and so $N_{d,\ell}(\mathbf{Z})=0$ under $Q_y$. Consequently, we have  
\begin{equation}\label{E5.2}
    \mathbb{E}_{Q_y}[N_{d,\ell}(\mathbf{Z})\log(N_{d,\ell}(\mathbf{Z}))]=0=\widetilde{N}_{d,\ell}(y)\log\big(\widetilde{N}_{d,\ell}(y)\big).
\end{equation}
Below we assume that $\widetilde{N}_{d,\ell}(y)>0$. Using the inequality $\log x\leq x-1$ for all $x>0$, we obtain that 
\begin{equation}\label{boundNlog}
     \mathbb{E}_{Q_y}\bigg[N_{d,\ell}(\mathbf{Z})\log\bigg(\frac{N_{d,\ell}(\mathbf{Z})}{\widetilde{N}_{d,\ell}(y)}\bigg)\bigg]
    \leq \mathbb{E}_{Q_y}\bigg[N_{d,\ell}(\mathbf{Z})\bigg(\frac{N_{d,\ell}(\mathbf{Z})}{\widetilde{N}_{d,\ell}(y)}-1\bigg)\bigg]=\frac{\mathbb{E}_{Q_y}\big[N_{d,\ell}(\mathbf{Z})^2\big]}{\widetilde{N}_{d,\ell}(y)}-\widetilde{N}_{d,\ell}(y), 
\end{equation}
where we use the fact that
\begin{equation}\label{ENdl}
    \mathbb{E}_{Q_y}[N_{d,\ell}(\mathbf{Z})]=
    \widetilde{N}_{d,\ell}(y)=\sum_{i=1}^{n_d} Q_y(Z_{d,i}=\ell).
\end{equation}
Proceeding to bound the RHS in~\eqref{boundNlog}, we have 
\begin{eqnarray*}
    &&\mathbb{E}_{Q_y}\big[N_{d,\ell}(\mathbf{Z})^2\big]=\sum_{i=1}^{n_d}Q_y(Z_{d,i}=\ell)+\sum_{i,j\in[n_d]:i\neq j}Q_y(Z_{d,i}=\ell)Q_y(Z_{d,j}=\ell)\nonumber\\
    &=&\sum_{i=1}^{n_d}Q_y(Z_{d,i}=\ell)(1-Q_y(Z_{d,i}=\ell))+\widetilde{N}_{d,\ell}(y)^2\leq \widetilde{N}_{d,\ell}(y)+\widetilde{N}_{d,\ell}(y)^2,
\end{eqnarray*}
where we use~\eqref{ENdl} in the last inequality. Using~\eqref{boundNlog}, this gives
\begin{equation*}
    \mathbb{E}_{Q_y}\bigg[N_{d,\ell}(\mathbf{Z})\log\bigg(\frac{N_{d,\ell}(\mathbf{Z})}{\widetilde{N}_{d,\ell}(y)}\bigg)\bigg]\leq 1,
\end{equation*}
which in turn gives
\begin{eqnarray}\label{E5.3}
    \mathbb{E}_{Q_y}[N_{d,\ell}(\mathbf{Z})\log(N_{d,\ell}(\mathbf{Z}))]&\leq& \mathbb{E}_{Q_y}[N_{d,\ell}(\mathbf{Z})]\log\big(\widetilde{N}_{d,\ell}(y)\big)+1\nonumber\\
    &=& \widetilde{N}_{d,\ell}(y)\log\big(\widetilde{N}_{d,\ell}(y)\big)+1,
\end{eqnarray}
where we use~\eqref{ENdl} in the last equality. Noting that the conclusion of~\eqref{E5.3} also holds if $\widetilde{N}_{d,\ell}(y)=0$ (using~\eqref{E5.2}), combining~\eqref{E5.1} and~\eqref{E5.3} we get 
\begin{equation}\label{E5.6}
    \sum_{d=1}^D\sum_{\ell=1}^K\mathbb{E}_{Q_y}[\log\Gamma(N_{d,\ell}(\mathbf{Z})+1\slash 2)]\leq\sum_{d=1}^D\sum_{\ell=1}^K \widetilde{N}_{d,\ell}(y)\log\big(\widetilde{N}_{d,\ell}(y)\big)-n+CDK.
\end{equation}

Combining (\ref{E5.4}) and (\ref{E5.6}), we obtain  
\begin{eqnarray}\label{E5.5}
    \mathbb{E}_{Q_y}[f(\mathbf{Z})]-I(y) 
   &\leq& -\sum_{d=1}^D\log\Gamma(n_d+K\slash 2)-n\log K -n+CDK\nonumber\\
   &&\hspace{0.4in}+\sum_{d=1}^D\sum_{\ell=1}^K \widetilde{N}_{d,\ell}(y)\log\big(\widetilde{N}_{d,\ell}(y)\big)-\sum_{d=1}^D\sum_{i=1}^{n_d}\sum_{\ell=1}^K y_{d,i,\ell}\log y_{d,i,\ell}.
\end{eqnarray}
By Jensen's inequality, for any $d\in[D]$ and $\ell\in[K]$, we have
\begin{eqnarray*}
    \frac{1}{n_d}\sum_{i=1}^{n_d}y_{d,i,\ell}\log y_{d,i,\ell} &\geq& \bigg(\frac{\sum_{i=1}^{n_d}y_{d,i,\ell}}{n_d}\bigg)\log\bigg(\frac{\sum_{i=1}^{n_d}y_{d,i,\ell}}{n_d}\bigg)=\frac{\widetilde{N}_{d,\ell}(y)}{n_d}\log\bigg(\frac{\widetilde{N}_{d,\ell}(y)}{n_d}\bigg),
\end{eqnarray*}
and so
\begin{eqnarray}\label{E5.9}
\sum_{d=1}^D\sum_{i=1}^{n_d}\sum_{\ell=1}^K y_{d,i,\ell}\log y_{d,i,\ell}
     &\geq&\sum_{d=1}^D \sum_{\ell=1}^K\widetilde{N}_{d,\ell}(y)\log\big(\widetilde{N}_{d,\ell}(y)\big)-\sum_{d=1}^D\log n_d \bigg(\sum_{\ell=1}^K\widetilde{N}_{d,\ell}(y)\bigg)\nonumber\\
     &=& \sum_{d=1}^D \sum_{\ell=1}^K\widetilde{N}_{d,\ell}(y)\log\big(\widetilde{N}_{d,\ell}(y)\big)-\sum_{d=1}^D n_d\log n_d. 
\end{eqnarray}
By (\ref{E5.5}) and (\ref{E5.9}), for any $y\in\hat{\mathcal{Y}}_{n,K}$, we have
\begin{equation}\label{E5.16}
      \mathbb{E}_{Q_y}[f(\mathbf{Z})]-I(y)\leq -\sum_{d=1}^D\log\Gamma(n_d+K\slash 2)+\sum_{d=1}^Dn_d\log n_d -n\log K -n+CDK.
\end{equation}

\noindent{\bf{Step 3.}} 
By the data processing inequality (see, for example, \cite[Theorem 2.8.1]{cover2012elements}), denoting by $Q'$ the marginal of $\hat{P}$ on $\mathbf{Z}$, we have 
\begin{equation}\label{eqnew.2}
    \KL\big(\hat{P} \,\big\| \; \mathbb{P}(\mathbf{Z},\{\pmb{\pi}_d\}_{d\in [D]}|\mathbf{X}) \big)\geq \KL\big(Q'\,\big\|\;\mathbb{P}(\mathbf{Z}|\mathbf{X})\big) \geq \KL\big(\hat{Q}\,\big\|\;\mathbb{P}(\mathbf{Z}|\mathbf{X})\big),
\end{equation}
where in the second inequality we use the definition of $\hat{Q}$ in~\eqref{vari3.2} and note that $Q'\in\mathcal{Q}$. Similar to~\eqref{partition}, noting that $\eta_{\ell,r}=V^{-1}$ for every $\ell\in[K],r\in[V]$, we have
\begin{equation}\label{eqnew.1}
    \sum_{z\in [K]^n}e^{f(z)}\mu(z)=\frac{\mathcal{S}_{n,K}}{\prod_{d=1}^D\prod_{i=1}^{n_d}\big(\sum_{\ell=1}^K\eta_{\ell,X_{d,i}}\big)}=\mathcal{S}_{n,K}\cdot \Big(\frac{V}{K}\Big)^n.
\end{equation}
Sequentially applying~\eqref{eqnew.2}, Lemma \ref{Lemma2.1}, and~\eqref{eqnew.1}, we get for any $\mathbf{X}\in[V]^n$, 
\begin{eqnarray*}
  &&  \frac{1}{n}\KL\big(\hat{P} \,\big\| \; \mathbb{P}(\mathbf{Z},\{\pmb{\pi}_d\}_{d\in [D]}|\mathbf{X}) \big)\geq \frac{1}{n}\KL\big(\hat{Q}\,\big\|\;\mathbb{P}(\mathbf{Z}|\mathbf{X})\big)\nonumber\\
  &=& \frac{\log\big(\sum_{z\in [K]^n}e^{f(z)}\mu(z)\big)-\sup_{y\in\hat{\mathcal{Y}}_{n,K}}\{\mathbb{E}_{Q_y}[f(\mathbf{Z})]-I(y)\}}{n} \nonumber\\
  &=&\frac{\log \mathcal{S}_{n,K}+n\log V - n \log K-\sup_{y\in\hat{\mathcal{Y}}_{n,K}}\{\mathbb{E}_{Q_y}[f(\mathbf{Z})]-I(y)\}}{n}\nonumber\\
  &\geq& \frac{(K-1)D}{2 n}\log\Big(\frac{n}{DK}\Big)-\frac{CDK}{n} \geq \frac{DK}{4n}\log\Big(\frac{n}{DK}\Big)-\frac{CDK}{n},
\end{eqnarray*}
where we use~\eqref{E5.15} and \eqref{E5.16} in the fourth line, along with the fact that $n_d=n\slash D$ for all $d\in [D]$. Hence there exists an absolute positive constant $c_0$, such that when $DK\leq c_0 n$, for any $\mathbf{X}\in[V]^n$, 
\begin{equation*}
    \frac{1}{n}\KL\big(\hat{P} \,\big\| \; \mathbb{P}(\mathbf{Z},\{\pmb{\pi}_d\}_{d\in [D]}|\mathbf{X}) \big)\geq \frac{1}{n}\KL\big(\hat{Q}\,\big\|\;\mathbb{P}(\mathbf{Z}|\mathbf{X})\big)\geq \frac{DK}{5n}\log\Big(\frac{n}{DK}+2\Big). 
\end{equation*}

\subsection{Proof of Theorem \ref{Theorem_MMSB_LBD}} 

We recall the model setup and notation for MMSB in Section \ref{Sect.1.3}, and assume the definitions given at the beginning of Appendix \ref{Appendix_B}. We denote by $\mathcal{S}_{n,K}$ the normalizing constant of the posterior\slash collapsed posterior (see~\eqref{eq:SnK_MMSB}). 

\medskip

\noindent{\bf{Step 1.}} We first derive a lower bound on $\log \mathcal{S}_{n,K}$. Under the assumptions of Theorem \ref{Theorem_MMSB_LBD}, the functions $\Upsilon(\cdot)$ and $f(\cdot)$ defined in~\eqref{defUps:MMSB} and~\eqref{deffF_MMSB} reduce to
\begin{equation}\label{def_Upsilon_new}
    \Upsilon(z)=\mathscr{M}_1\log B+\mathscr{M}_0\log(1-B)+\sum_{i=1}^n\sum_{\ell=1}^K\log\Gamma(N_{i,\ell}(z)+1\slash 2)-n\log\Gamma(2n-2+K\slash 2),
\end{equation}
\begin{equation}\label{def_f_new.MMSB}
    f(z)=\sum_{i=1}^n\sum_{\ell=1}^K\log\Gamma(N_{i,\ell}(z)+1\slash 2)-n\log\Gamma(2n-2+K\slash 2),
\end{equation}
where $\mathscr{M}_1,\mathscr{M}_0$ are as in~\eqref{MathscrMde}. Consequently, using~\eqref{eq:SnK_MMSB} we have 
\begin{eqnarray}\label{E6.30}
\mathcal{S}_{n,K} &=& \frac{B^{\mathscr{M}_1}(1-B)^{\mathscr{M}_0}}{\Gamma(2n-2+K\slash 2)^n}\cdot\sum_{z\in[K]^{2n(n-1)}}\prod_{i=1}^n \prod_{\ell=1}^K\Gamma(N_{i,\ell}(z)+1\slash 2).
\end{eqnarray}

In the following, denote
\begin{equation*}
    \mathfrak{S}_{n,K}:=\bigg\{(u_{\ell})_{\ell\in [K]}\in\mathbb{N}^K:\sum_{\ell=1}^K  u_{\ell}=n-1\bigg\}.
\end{equation*}
Also, for any $u,v\geq 0$, set
\begin{equation*}
     \rho(u,v):=(u+v)\log(u+v)-u\log u-v\log v. 
\end{equation*}
Recalling that
\begin{equation*}
    N_{i,\ell}(z)=N_{\rightarrow,i,\ell}(z)+N_{\leftarrow,i,\ell}(z),\quad\mbox{where}\quad\sum_{\ell=1}^K N_{\rightarrow,i,\ell}(z)=\sum_{\ell=1}^K N_{\leftarrow,i,\ell}(z)=n-1
\end{equation*}
(see~\eqref{Nildef}), we have 
\begin{eqnarray}\label{E6.29}
   && \sum_{z\in[K]^{2n(n-1)}}\prod_{i=1}^n \prod_{\ell=1}^K\Gamma(N_{i,\ell}(z)+1\slash 2)\nonumber\\
     &=& \sum_{\substack{(u_{i,\ell})_{i\in[n],\ell\in[K]},\\(v_{i,\ell})_{i\in[n],\ell\in [K]}\in(\mathfrak{S}_{n,K})^n}} \sum_{\substack{z\in[K]^{2n(n-1)}:N_{\rightarrow,i,\ell}(z)=u_{i,\ell},\\N_{\leftarrow,i,\ell}(z)=v_{i,\ell},\forall i\in[n],\ell\in[K]}}\prod_{i=1}^n \prod_{\ell=1}^K\Gamma(u_{i,\ell}+v_{i,\ell}+1\slash 2)\nonumber\\
&=&\sum_{\substack{(u_{i,\ell})_{i\in[n],\ell\in[K]},\\(v_{i,\ell})_{i\in[n],\ell\in [K]}\in(\mathfrak{S}_{n,K})^n}}\prod_{i=1}^n \frac{((n-1)!)^2\prod_{\ell=1}^K\Gamma(u_{i,\ell}+v_{i,\ell}+1\slash 2)}{\prod_{\ell=1}^K u_{i,\ell}!\prod_{\ell=1}^K v_{i,\ell}!}\nonumber\\
   &=& ((n-1)!)^{2n} \bigg(\sum_{\substack{(u_{\ell})_{\ell\in[K]},\\(v_{\ell})_{\ell\in[K]}\in\mathfrak{S}_{n,K}}} \prod_{\ell=1}^K\frac{\Gamma(u_{\ell}+v_{\ell}+1\slash 2)}{\Gamma(u_{\ell}+1)\Gamma(v_{\ell}+1)}\bigg)^n.
\end{eqnarray}
By Stirling's approximation in Lemma \ref{L3.2} (in Appendix \ref{Appendix_E}), for any $\ell\in[K]$ and $u_{\ell},v_{\ell}\in\mathbb{N}$, we have  
\begin{eqnarray}\label{stirling_new}
    \Big|\log\Gamma(u_{\ell}+v_{\ell}+1\slash 2)-(u_{\ell}+v_{\ell})\log(u_{\ell}+v_{\ell})+(u_{\ell}+v_{\ell})\Big|\leq C,
\end{eqnarray}
\begin{equation*}
    \Big|\log\Gamma(u_{\ell}+1)-u_{\ell}\log u_{\ell}+u_{\ell}-\frac{1}{2}\log(\max\{u_{\ell},1\})\Big|\leq C,
\end{equation*}
\begin{equation*}
    \Big|\log\Gamma(v_{\ell}+1)-v_{\ell}\log v_{\ell}+v_{\ell}-\frac{1}{2}\log(\max\{v_{\ell},1\})\Big|\leq C.
\end{equation*}
Hence
\begin{align}\label{E6.27}
& \sum_{\substack{(u_{\ell})_{\ell\in[K]},\\(v_{\ell})_{\ell\in[K]}\in\mathfrak{S}_{n,K}}} \prod_{\ell=1}^K\frac{\Gamma(u_{\ell}+v_{\ell}+1\slash 2)}{\Gamma(u_{\ell}+1)\Gamma(v_{\ell}+1)} \nonumber\\
\geq& \exp(-CK) \sum_{\substack{(u_{\ell})_{\ell\in[K]},\\(v_{\ell})_{\ell\in[K]}\in\mathfrak{S}_{n,K}}} \exp\bigg(\sum_{\ell=1}^K\bigg(\rho(u_{\ell},v_{\ell})-\frac{1}{2}\log(\max\{u_{\ell},1\})-\frac{1}{2}\log(\max\{v_{\ell},1\})\bigg)\bigg).
\end{align}

Now for any $\mathbf{u}=(u_{\ell})_{\ell\in[K]}\in\mathfrak{S}_{n,K}$, we define  
\begin{align*}
    \mathcal{D}(\mathbf{u}):=\bigg\{&\mathbf{v}=(v_{\ell})_{\ell\in [K]}\in\mathfrak{S}_{n,K}:\sum_{\ell=1}^K\rho(u_{\ell},v_{\ell})\geq 2\log(2)(n-1)-\frac{K}{4},\nonumber\\
    &\hspace{1.4in}\sum_{\ell=1}^K\log(\max\{v_{\ell},1\})\leq K\log{2}+\sum_{\ell=1}^K\log(\max\{u_{\ell},1\})\bigg\}.
\end{align*}
The following lemma provides a lower bound on the cardinality of $\mathcal{D}(\mathbf{u})$. The proof of this lemma will be given at the end of this subsection.

\begin{lemma}\label{Lem6.1}
For any $\mathbf{u}=(u_{\ell})_{\ell\in[K]}\in\mathfrak{S}_{n,K}$, we have
\begin{equation*}
    |\mathcal{D}(\mathbf{u})|\geq \exp\bigg(\frac{1}{2}\bigg(\sum_{\ell=1}^K\log(\max\{u_{\ell},1\})-\max_{\ell\in[K]}\big\{\log(\max\{u_{\ell},1\})\big\}\bigg)-CK\bigg),
\end{equation*}
where $C$ is a positive absolute constant. 
\end{lemma}

By (\ref{E6.27}) and Lemma \ref{Lem6.1},
\begin{eqnarray*}
    && \sum_{\substack{(u_{\ell})_{\ell\in[K]},\\(v_{\ell})_{\ell\in[K]}\in\mathfrak{S}_{n,K}}} \prod_{\ell=1}^K\frac{\Gamma(u_{\ell}+v_{\ell}+1\slash 2)}{\Gamma(u_{\ell}+1)\Gamma(v_{\ell}+1)} \nonumber\\
    &\geq& \exp(2\log(2)(n-1)-CK)\sum_{(u_{\ell})_{\ell\in[K]}\in\mathfrak{S}_{n,K}}\sum_{(v_{\ell})_{\ell\in[K]}\in\mathcal{D}(\mathbf{u})} \exp\bigg(-\sum_{\ell=1}^K\log(\max\{u_{\ell},1\})\bigg)\nonumber\\
    &\geq& \exp(2\log(2)(n-1)-CK) 
 \nonumber\\
    &&\hspace{0.4in}\times  \sum_{(u_{\ell})_{\ell\in[K]}\in\mathfrak{S}_{n,K}}\exp\bigg(-\frac{1}{2}\max_{\ell\in[K]}\big\{\log(\max\{u_{\ell},1\})\big\}-\frac{1}{2}\sum_{\ell=1}^K\log(\max\{u_{\ell},1\})\bigg).
\end{eqnarray*}
Throughout the rest of the proof, we assume that $n\geq K$. Note that
\begin{equation*}
    |\mathfrak{S}_{n,K}|=\binom{n+K-2}{K-1}=\frac{(n+K-2)\cdots (n+1) n}{(K-1)!}\geq \Big(\frac{n}{K}\Big)^{K-1}.
\end{equation*}
Moreover, for any $(u_{\ell})_{\ell\in[K]}\in\mathfrak{S}_{n,K}$, 
by Jensen's inequality,
\begin{eqnarray*}
    \sum_{\ell=1}^K\log(\max\{u_{\ell},1\})&\leq& K\log\bigg(\frac{\sum_{\ell=1}^K \max\{u_{\ell},1\}}{K}\bigg)\leq K\log\Big(\frac{n}{K}+1\Big)\leq K\log\Big(\frac{n}{K}\Big)+K\log{2},
\end{eqnarray*}
\begin{equation*}
    \max_{\ell\in[K]}\big\{\log(\max\{u_{\ell},1\})\big\}\leq \log{n}.
\end{equation*}
Combining the four displays as above, we have
\begin{eqnarray}\label{E6.31}
 &&\sum_{\substack{(u_{\ell})_{\ell\in[K]},\\(v_{\ell})_{\ell\in[K]}\in\mathfrak{S}_{n,K}}} \prod_{\ell=1}^K\frac{\Gamma(u_{\ell}+v_{\ell}+1\slash 2)}{\Gamma(u_{\ell}+1)\Gamma(v_{\ell}+1)}\nonumber\\
 &\geq&\Big(\frac{n}{K}\Big)^{K-1}\cdot \exp\bigg(2\log(2)(n-1)-CK-\frac{1}{2}\log n-\frac{K}{2}\log\Big(\frac{n}{K}\Big)-\frac{\log 2}{2} K\bigg)  \nonumber\\
 &\geq&\exp\bigg(2\log(2)(n-1)-\frac{1}{2}\log n+\Big(\frac{K}{2}-1\Big)\log\Big(\frac{n}{K}\Big)-CK\bigg).
\end{eqnarray}

By (\ref{E6.30}), (\ref{E6.29}), and (\ref{E6.31}), 
\begin{eqnarray}\label{E6.10}
    \log\mathcal{S}_{n,K}&\geq&\mathscr{M}_1\log B+\mathscr{M}_0\log(1-B)-n\log\Gamma(2n-2+K\slash 2)+2n\log((n-1)!)\nonumber\\
   &&\hspace{0.4in}+2\log(2)n(n-1)-\frac{1}{2}n\log n+n\Big(\frac{K}{2}-1\Big)\log\Big(\frac{n}{K}\Big)-CnK.
\end{eqnarray}

\medskip

\noindent{\bf{Step 2.}} Let $\hat{\mathcal{Y}}_{n,K}$ be as defined in~\eqref{YnK:MMSB}. For any $y\in\hat{\mathcal{Y}}_{n,K}$, we denote by $Q_y$ the product probability distribution on $[K]^{2n(n-1)}$ given by
\begin{equation}\label{def_Qy_MMSB}
    Q_y(z)=\prod_{i,j\in [n]:i\neq j}y_{i,j;z_{i\rightarrow j}, z_{i\leftarrow j}},\qquad \mbox{for all }z\in[K]^{2n(n-1)}.
\end{equation}
Under the assumptions of Theorem \ref{Theorem_MMSB_LBD}, for any $y\in\hat{\mathcal{Y}}_{n,K}$,  $I(y)$ defined in~\eqref{defI_MMSB} reduces to
\begin{equation*}
    I(y)=\sum_{i,j\in [n]:i\neq j}\sum_{\ell=1}^K\sum_{\ell'=1}^K y_{i,j;\ell,\ell'}\log y_{i,j;\ell,\ell'}+2n(n-1)\log K.
\end{equation*}

In this step, with $f(\cdot)$ as in~\eqref{def_f_new.MMSB}, we derive an upper bound on $\mathbb{E}_{Q_y}[f(\mathbf{Z})]-I(y)$ for any $y\in \hat{\mathcal{Y}}_{n,K}$. Note that  
\begin{eqnarray}\label{E6.4}
    \mathbb{E}_{Q_y}[f(\mathbf{Z})]-I(y)
   &=& \sum_{i=1}^n\sum_{\ell=1}^K\mathbb{E}_{Q_y}\big[\log\Gamma(N_{i,\ell}(\mathbf{Z})+1\slash 2)\big]-n\log\Gamma(2n-2+K\slash 2)\nonumber\\
   &&\hspace{0.2in}-\sum_{i,j\in [n]:i\neq j}\sum_{\ell=1}^K\sum_{\ell'=1}^K y_{i,j;\ell,\ell'}\log y_{i,j;\ell,\ell'}-2n(n-1)  \log K.
\end{eqnarray}

Similar to~\eqref{stirling_new}, by Stirling's approximation we have 
\begin{eqnarray}\label{E6.3}
    &&\sum_{i=1}^n\sum_{\ell=1}^K\mathbb{E}_{Q_y}[\log\Gamma(N_{i,\ell}(\mathbf{Z})+1\slash 2)]\nonumber\\
    &\leq& \sum_{i=1}^n\sum_{\ell=1}^K \mathbb{E}_{Q_y}[N_{i,\ell}(\mathbf{Z})\log(N_{i,\ell}(\mathbf{Z}))]-\sum_{i=1}^n\sum_{\ell=1}^K\mathbb{E}_{Q_y}[N_{i,\ell}(\mathbf{Z})]+CnK\nonumber\\
    &=&\sum_{i=1}^n\sum_{\ell=1}^K \mathbb{E}_{Q_y}[N_{i,\ell}(\mathbf{Z})\log(N_{i,\ell}(\mathbf{Z}))]-2n(n-1)+CnK.
\end{eqnarray}
Fix $i\in [n]$ and $\ell\in [K]$. With $\widetilde{N}_{i,\ell}(y)$ as in~\eqref{Nil_y_def}, if $\widetilde{N}_{i,\ell}(y)=0$, then $y_{i,j;\ell,\ell'}=y_{j,i;\ell',\ell}=0$ for all $j\in[n]\backslash\{i\},\ell'\in[K]$, and so $N_{i,\ell}(\mathbf{Z})=0$ under $Q_y$. Consequently, we have 
\begin{eqnarray}\label{E6.12}
    &&\mathbb{E}_{Q_y}[N_{i,\ell}(\mathbf{Z})\log(N_{i,\ell}(\mathbf{Z}))]=0=\widetilde{N}_{i,\ell}(y)\log\big(\widetilde{N}_{i,\ell}(y)\big).
\end{eqnarray}
Below we assume that $\widetilde{N}_{i,\ell}(y)>0$. Using the inequality $\log x\leq x-1$ for all $x>0$, we obtain that 
\begin{eqnarray}\label{E6.1}
 \mathbb{E}_{Q_y}\bigg[N_{i,\ell}(\mathbf{Z}) \log\bigg(\frac{N_{i,\ell}(\mathbf{Z})}{\widetilde{N}_{i,\ell}(y)}\bigg)\bigg]&\leq&  \mathbb{E}_{Q_y}\bigg[N_{i,\ell}(\mathbf{Z})\bigg(\frac{N_{i,\ell}(\mathbf{Z})}{\widetilde{N}_{i,\ell}(y)}-1\bigg)\bigg]\nonumber\\
 &=&\frac{\mathbb{E}_{Q_y}\big[N_{i,\ell}(\mathbf{Z})^2\big]}{\widetilde{N}_{i,\ell}(y)}-\widetilde{N}_{i,\ell}(y),
\end{eqnarray}
where we use the fact that
\begin{equation}\label{EQY}
    \mathbb{E}_{Q_y}[N_{i,\ell}(\mathbf{Z})]=\widetilde{N}_{i,\ell}(y)=\sum_{j\in[n]\backslash\{i\}}Q_y(Z_{i\rightarrow j}=\ell)+\sum_{j\in[n]\backslash\{i\}}Q_y(Z_{j\leftarrow i}=\ell).
\end{equation}
Now note that (recall~\eqref{Nildef}) $N_{i,\ell}(z) =\sum_{j\in[n]\backslash\{i\}}\mathbbm{1}_{z_{i\rightarrow j}=\ell}+\sum_{j\in[n]\backslash\{i\}}\mathbbm{1}_{z_{j\leftarrow i}=\ell}$ for any $z\in [K]^{2n(n-1)}$. Hence
\begin{eqnarray}\label{E6.2}
\mathbb{E}_{Q_y}\big[N_{i,\ell}(\mathbf{Z})^2\big] 
    &=& \sum_{j,j'\in[n]\backslash \{i\}}\mathbb{E}_{Q_y}[\mathbbm{1}_{Z_{i\rightarrow j}=\ell}\mathbbm{1}_{Z_{i\rightarrow j'}=\ell}]+\sum_{j,j'\in[n]\backslash \{i\}}\mathbb{E}_{Q_y}[\mathbbm{1}_{Z_{j\leftarrow i}=\ell}\mathbbm{1}_{Z_{j'\leftarrow i}=\ell}]\nonumber\\
    &&\hspace{0.4in}+2\sum_{j,j'\in[n]\backslash\{i\}}\mathbb{E}_{Q_y}[\mathbbm{1}_{Z_{i\rightarrow j}=\ell}\mathbbm{1}_{Z_{j'\leftarrow i}=\ell}]\nonumber\\
    &=& \sum_{j\in[n]\backslash\{i\}}Q_y(Z_{i\rightarrow j}=\ell)+\sum_{j\in[n]\backslash\{i\}}Q_y(Z_{j\leftarrow i}=\ell)\nonumber\\
    &&\hspace{0.4in}+\sum_{j,j'\in[n]\backslash\{i\}:j\neq j'}Q_y(Z_{i\rightarrow j}=\ell) Q_y(Z_{i\rightarrow j'}=\ell)\nonumber\\
    &&\hspace{0.4in}+\sum_{j,j'\in[n]\backslash\{i\}:j\neq j'}Q_y(Z_{j\leftarrow i}=\ell) Q_y(Z_{j'\leftarrow i}=\ell)\nonumber\\
    &&\hspace{0.4in} +2\sum_{j,j'\in[n]\backslash\{i\}} Q_y(Z_{i\rightarrow j}=\ell) Q_y(Z_{j'\leftarrow i}=\ell)\nonumber\\
    &=& \widetilde{N}_{i,\ell}(y)^2+\sum_{j\in[n]\backslash\{i\}}Q_y(Z_{i\rightarrow j}=\ell)(1-Q_y(Z_{i\rightarrow j}=\ell))\nonumber\\
    &&\hspace{0.4in}+\sum_{j\in[n]\backslash\{i\}}Q_y(Z_{j\leftarrow i}=\ell)(1-Q_y(Z_{j\leftarrow i}=\ell))\nonumber\\
    &\leq& \widetilde{N}_{i,\ell}(y)^2+\widetilde{N}_{i,\ell}(y),
\end{eqnarray}
where we use~\eqref{EQY} in the last line. Combining (\ref{E6.1}) and (\ref{E6.2}), we obtain that
\begin{equation*} \mathbb{E}_{Q_y}\bigg[N_{i,\ell}(\mathbf{Z}) \log\bigg(\frac{N_{i,\ell}(\mathbf{Z})}{\widetilde{N}_{i,\ell}(y)}\bigg)\bigg]\leq 1, 
\end{equation*}
which in turn gives
\begin{equation}\label{E6.16}
   \mathbb{E}_{Q_y}[N_{i,\ell}(\mathbf{Z})\log(N_{i,\ell}(\mathbf{Z}))]\leq \mathbb{E}_{Q_y}[N_{i,\ell}(\mathbf{Z})] \log\big(\widetilde{N}_{i,\ell}(y)\big) +1 = \widetilde{N}_{i,\ell}(y)\log\big(\widetilde{N}_{i,\ell}(y)\big)+1,
\end{equation}
where we use~\eqref{EQY} in the last equality. Noting that the conclusion of~\eqref{E6.16} also holds if $\widetilde{N}_{i,\ell}(y)=0$ (using~\eqref{E6.12}), combining~\eqref{E6.3} and~\eqref{E6.16} we get
\begin{eqnarray}\label{E6.5}
    &&\sum_{i=1}^n\sum_{\ell=1}^K\mathbb{E}_{Q_y}[\log\Gamma(N_{i,\ell}(\mathbf{Z})+1\slash 2)]\nonumber\\
    &\leq& \sum_{i=1}^n\sum_{\ell=1}^K \widetilde{N}_{i,\ell}(y)\log\big(\widetilde{N}_{i,\ell}(y)\big) -2n(n-1)+CnK. 
\end{eqnarray}

By~\eqref{E6.4} and~\eqref{E6.5}, we have 
\begin{eqnarray}\label{E6.7}
    \mathbb{E}_{Q_y}[f(\mathbf{Z})]-I(y)
&\leq&\sum_{i=1}^n\sum_{\ell=1}^K \widetilde{N}_{i,\ell}(y)\log\big(\widetilde{N}_{i,\ell}(y)\big)-n\log\Gamma(2n-2+K\slash 2)\nonumber\\
&&\hspace{0.2in}-2n(n-1)-2n(n-1)\log K+CnK\nonumber\\
&&\hspace{0.2in}-\sum_{i,j\in [n]:i\neq j}\sum_{\ell=1}^K\sum_{\ell'=1}^K y_{i,j;\ell,\ell'}\log y_{i,j;\ell,\ell'}. 
\end{eqnarray}
For any $i,j\in [n]$ such that $i\neq j$, using the subadditivity of Shannon entropy, we have
{\small
\begin{eqnarray*}
    -\sum_{\ell=1}^K\sum_{\ell'=1}^Ky_{i,j;\ell,\ell'}\log y_{i,j;\ell,\ell'}
  &\leq& -\sum_{\ell=1}^K\bigg(\sum_{\ell'=1}^K y_{i,j;\ell,\ell'}\bigg)\log\bigg(\sum_{\ell'=1}^K y_{i,j;\ell,\ell'}\bigg)\nonumber\\
  &&\hspace{0.2in} -\sum_{\ell'=1}^K\bigg(\sum_{\ell=1}^K y_{i,j;\ell,\ell'}\bigg)\log\bigg(\sum_{\ell=1}^K y_{i,j;\ell,\ell'}\bigg).
\end{eqnarray*}
}Hence
{\small
\begin{eqnarray}
    -\sum_{i,j\in [n]:i\neq j}\sum_{\ell=1}^K\sum_{\ell'=1}^K y_{i,j;\ell,\ell'}\log y_{i,j;\ell,\ell'}
   &\leq& -\sum_{i=1}^n\sum_{\ell=1}^K\sum_{j\in[n]\backslash\{i\}}\bigg(\sum_{\ell'=1}^K y_{i,j;\ell,\ell'}\bigg)\log\bigg(\sum_{\ell'=1}^K y_{i,j;\ell,\ell'}\bigg)\nonumber\\
   && \hspace{0.2in} -\sum_{i=1}^n\sum_{\ell=1}^K\sum_{j\in[n]\backslash\{i\}}\bigg(\sum_{\ell'=1}^K y_{j,i;\ell',\ell}\bigg)\log\bigg(\sum_{\ell'=1}^K y_{j,i;\ell',\ell}\bigg).
\end{eqnarray}
}Let $\phi(x)=x\log x$ (where $x\geq 0$) be defined as in~\eqref{phioriginal}. Note that $\phi$ is convex, and so by Jensen's inequality, for any $i\in [n]$ and $\ell\in [K]$,
\begin{eqnarray*}
   && \frac{1}{2(n-1)}\sum_{j\in[n]\backslash\{i\}}\phi\bigg(\sum_{\ell'=1}^K y_{i,j;\ell,\ell'}\bigg)+\frac{1}{2(n-1)}\sum_{j\in[n]\backslash\{i\}}\phi\bigg(\sum_{\ell'=1}^K y_{j,i;\ell',\ell}\bigg)\nonumber\\
   &\geq& \phi\bigg(\frac{1}{2(n-1)}\sum_{j\in[n]\backslash\{i\}}\sum_{\ell'=1}^K y_{i,j;\ell,\ell'}+\frac{1}{2(n-1)}\sum_{j\in[n]\backslash\{i\}}\sum_{\ell'=1}^K y_{j,i;\ell',\ell}\bigg)= \phi\bigg(\frac{\widetilde{N}_{i,\ell}(y)}{2(n-1)}\bigg).
\end{eqnarray*}
Hence
\begin{eqnarray}\label{E6.8}
   && -\sum_{i=1}^n\sum_{\ell=1}^K\sum_{j\in[n]\backslash\{i\}}\bigg(\sum_{\ell'=1}^K y_{i,j;\ell,\ell'}\bigg)\log\bigg(\sum_{\ell'=1}^K y_{i,j;\ell,\ell'}\bigg)\nonumber\\
   && \hspace{0.4in} -\sum_{i=1}^n\sum_{\ell=1}^K\sum_{j\in[n]\backslash\{i\}}\bigg(\sum_{\ell'=1}^K y_{j,i;\ell',\ell}\bigg)\log\bigg(\sum_{\ell'=1}^K y_{j,i;\ell',\ell}\bigg)\nonumber\\
   &\leq& -\sum_{i=1}^n\sum_{\ell=1}^K \widetilde{N}_{i,\ell}(y)\log\big(\widetilde{N}_{i,\ell}(y)\big)+\log(2(n-1))\sum_{i=1}^n\sum_{\ell=1}^K\widetilde{N}_{i,\ell}(y)\nonumber\\
   &=& -\sum_{i=1}^n\sum_{\ell=1}^K \widetilde{N}_{i,\ell}(y)\log\big(\widetilde{N}_{i,\ell}(y)\big)+2n(n-1)\log(2(n-1)). 
\end{eqnarray}
Combining (\ref{E6.7})-(\ref{E6.8}), we obtain that for any $y\in\hat{\mathcal{Y}}_{n,K}$,
\begin{eqnarray}\label{E6.11}
     \mathbb{E}_{Q_y}[f(\mathbf{Z})]-I(y)
   &\leq& -n\log\Gamma(2n-2+K\slash 2)-2n(n-1)-2n(n-1)\log K\nonumber\\
   &&\hspace{0.2in}+CnK+2n(n-1)\log(2(n-1)).
\end{eqnarray}

\medskip

\noindent{\bf{Step 3.}} Similar to~\eqref{eqnew.2}, using the data processing inequality we get
\begin{equation}\label{eq.new.2.2}
\KL\big(\hat{P}\,\big\|\;\mathbb{P}(\mathbf{Z}_{\rightarrow},\mathbf{Z}_{\leftarrow},\{\pmb{\pi}_{i}\}_{i\in [n]}|\mathbf{X})\big)\geq \KL\big(\hat{Q}\,\big\|\;\mathbb{P}(\mathbf{Z}_{\rightarrow},\mathbf{Z}_{\leftarrow}|\mathbf{X})\big).
\end{equation}
Moreover, under the assumptions of Theorem \ref{Theorem_MMSB_LBD}, \eqref{partition_MMSB} reduces to
\begin{equation}\label{eq.new.2.1}
    \sum_{z\in [K]^{2n(n-1)}} e^{f(z)}\mu(z)=\frac{\mathcal{S}_{n,K}}{K^{2n(n-1)}
    B^{\mathscr{M}_1}(1-B)^{\mathscr{M}_0}},
\end{equation}
where $\mathscr{M}_1,\mathscr{M}_0$ are as in~\eqref{MathscrMde}. Sequentially applying~\eqref{eq.new.2.2}, Lemma \ref{Lemma2.1}, and~\eqref{eq.new.2.1}, we get for any $\mathbf{X}\in\{0,1\}^{n(n-1)}$, 
\begin{eqnarray}
  &&  \frac{1}{n^2}\KL\big(\hat{P}\,\big\|\;\mathbb{P}(\mathbf{Z}_{\rightarrow},\mathbf{Z}_{\leftarrow},\{\pmb{\pi}_{i}\}_{i\in [n]}|\mathbf{X})\big)
  \geq \frac{1}{n^2} \KL\big(\hat{Q}\,\big\|\;\mathbb{P}(\mathbf{Z}_{\rightarrow},\mathbf{Z}_{\leftarrow}|\mathbf{X})\big)\nonumber\\
  &=&\frac{\log\big(\sum_{z\in [K]^{2n(n-1)}}e^{f(z)}\mu(z)\big)-\sup_{y\in\hat{\mathcal{Y}}_{n,K}}\{\mathbb{E}_{Q_y}[f(\mathbf{Z})]-I(y)\}}{n^2}\nonumber\\
  &=&\frac{\log \mathcal{S}_{n,K} - 2n(n-1)\log K-\mathscr{M}_1\log B-\mathscr{M}_0\log(1-B)-\sup_{y\in\hat{\mathcal{Y}}_{n,K}}\{\mathbb{E}_{Q_y}[f(\mathbf{Z})]-I(y)\}}{n^2}\nonumber\\
  &\geq& \frac{2n\log((n-1)!)+2\log(2)n(n-1)-\frac{1}{2}n\log n+n\big(\frac{K}{2}-1\big)\log\big(\frac{n}{K}\big)-CnK}{n^2} \nonumber\\
  && \hspace{0.2in}-\frac{-2n(n-1)+2n(n-1)\log(2(n-1))}{n^2}\nonumber\\
  &=& \frac{2\log((n-1)!)-2(n-1)\log(n-1)+2(n-1)-\frac{1}{2}\log n +\big(\frac{K}{2}-1\big)\log\big(\frac{n}{K}\big)-CK}{n}\nonumber\\
  &\geq&\frac{\frac{1}{2}\log n +\big(\frac{K}{2}-1\big)\log\big(\frac{n}{K}\big)-CK}{n}\nonumber\\
  &&\hspace{0.2in}\geq\frac{K-1}{2n}\log\Big(\frac{n}{K}\Big)-\frac{CK}{n}\geq \frac{K}{4n}\log\Big(\frac{n}{K}\Big)-\frac{CK}{n},
\end{eqnarray}
where we use (\ref{E6.10}) and (\ref{E6.11}) in the fourth line, and Stirling's approximation as in Lemma \ref{L3.2} in the second to last line. Hence there exists a positive absolute constant $c_0$, such that when $K\leq c_0 n$, for any $\mathbf{X}\in\{0,1\}^{n(n-1)}$, we have
\begin{equation*}
     \frac{1}{n^2}\KL\big(\hat{P}\,\big\|\;\mathbb{P}(\mathbf{Z}_{\rightarrow},\mathbf{Z}_{\leftarrow},\{\pmb{\pi}_{i}\}_{i\in [n]}|\mathbf{X})\big)
    \geq\frac{1}{n^2} \KL\big(\hat{Q}\,\big\|\;\mathbb{P}(\mathbf{Z}_{\rightarrow},\mathbf{Z}_{\leftarrow}|\mathbf{X})\big)\geq\frac{K}{5n}\log\Big(\frac{n}{K}+2\Big).
\end{equation*}

\bigskip

Now we give the proof of Lemma \ref{Lem6.1}.

\begin{proof}[Proof of Lemma \ref{Lem6.1}]
For any $u,v>0$, we define
\begin{equation*}
    h_u(v):=\rho(u,v)=(u+v)\log(u+v)-u\log u-v\log v.
\end{equation*}
Then we have 
\begin{equation*}
    h'_u(v)=\log(u+v)-\log v,\qquad h''_u(v)=\frac{1}{u+v}-\frac{1}{v}=-\frac{u}{v(v+u)}.
\end{equation*}
By second-order Taylor expansion, there exists $\xi$ between $u$ and $v$, such that
\begin{eqnarray}\label{E6.17}
    \rho(u,v)&=&h_u(v)=h_u(u)+h_u'(u)(v-u)+\frac{1}{2}h_u''(\xi)(v-u)^2\nonumber\\
    &=& \log(2)(u+v)-\frac{u(v-u)^2}{2\xi(\xi+u)}\geq \log(2)(u+v)-\frac{(v-u)^2}{2\min\{u,v\}}.
\end{eqnarray}

Let $\sigma\in S_K$ be a permutation of $[K]$ such that $u_{\sigma(1)}\leq \cdots\leq u_{\sigma(K)}$. Define $\mathcal{D}'(\mathbf{u})$ to be the set of $\mathbf{v}=(v_{\ell})_{\ell\in[K]}\in\mathbb{Z}^K$ that can be obtained by the following sequential procedure. For each $\ell=1,2,\cdots,K-1$, assume that $v_{\sigma(1)},\cdots,v_{\sigma(\ell-1)}$ have been determined. If $\sum_{j=1}^{\ell-1}v_{\sigma(j)}\geq \sum_{j=1}^{\ell-1}u_{\sigma(j)}$, we take $v_{\sigma(\ell)}\in\mathbb{N}$ such that $u_{\sigma(\ell)}-\sqrt{u_{\sigma(\ell)}}\slash 2\leq v_{\sigma(\ell)}\leq u_{\sigma(\ell)}$; otherwise, we take $v_{\sigma(\ell)}\in\mathbb{N}$ such that $u_{\sigma(\ell)}\leq v_{\sigma(\ell)}\leq u_{\sigma(\ell)}+\sqrt{u_{\sigma(\ell)}}\slash 2$. Finally, we take $v_{\sigma(K)}=n-1-\sum_{\ell=1}^{K-1}v_{\sigma(\ell)}$. By the definition of $\mathcal{D}'(\mathbf{u})$, 
\begin{eqnarray}\label{E6.24}
   &&|\mathcal{D}'(\mathbf{u})|\geq \prod_{\ell=1}^{K-1}\max\bigg\{\frac{\sqrt{u_{\sigma(\ell)}}}{2},1\bigg\}\nonumber\\
   &\geq&\exp\bigg(\frac{1}{2}\bigg(\sum_{\ell=1}^K\log(\max\{u_{\ell},1\})-\max_{\ell\in[K]}\Big\{\log(\max\{u_{\ell},1\})\Big\}\bigg)-CK\bigg).
\end{eqnarray}

Now we show by induction that for any $\mathbf{v}=(v_{\ell})_{\ell\in[K]}\in\mathcal{D}'(\mathbf{u})$,
\begin{equation}\label{E6.20}
    \bigg|\sum_{j=1}^{\ell} v_{\sigma(j)}- \sum_{j=1}^{\ell} u_{\sigma(j)}\bigg|\leq \frac{\sqrt{u_{\sigma(\ell)}}}{2}
\end{equation}
for all $\ell\in[K-1]$. When $\ell=1$, (\ref{E6.20}) holds by the definition of $\mathcal{D}'(\mathbf{u})$. Now consider any $\ell\in [2,K-1]\cap\mathbb{N}^{*}$, and assume that (\ref{E6.20}) holds with $\ell$ replaced by $\ell-1$. If $\sum_{j=1}^{\ell-1}v_{\sigma(j)}\geq \sum_{j=1}^{\ell-1}u_{\sigma(j)}$, we have
\begin{equation}\label{neweq1}
    -\sqrt{u_{\sigma(\ell)}}\slash 2\leq v_{\sigma(\ell)}-u_{\sigma(\ell)}\leq 0.
\end{equation}
Hence 
\begin{equation*}
    \sum_{j=1}^{\ell} v_{\sigma(j)}- \sum_{j=1}^{\ell} u_{\sigma(j)}\leq \sum_{j=1}^{\ell-1} v_{\sigma(j)}- \sum_{j=1}^{\ell-1} u_{\sigma(j)}\leq \frac{\sqrt{u_{\sigma(\ell-1)}}}{2}\leq \frac{\sqrt{u_{\sigma(\ell)}}}{2},
\end{equation*}
where the three inequalities use~\eqref{neweq1}, the induction hypothesis, and the fact that $\ell\mapsto u_{\sigma(\ell)}$ is non-decreasing. Also, using the inequality $\sum_{j=1}^{\ell-1}v_{\sigma(j)}\geq \sum_{j=1}^{\ell-1}u_{\sigma(j)}$ and~\eqref{neweq1}, 
we have
\begin{equation*}
    \sum_{j=1}^{\ell} v_{\sigma(j)}- \sum_{j=1}^{\ell} u_{\sigma(j)}\geq v_{\sigma(\ell)}-u_{\sigma(\ell)}\geq   -\sqrt{u_{\sigma(\ell)}}\slash 2.
\end{equation*}
A similar proof works for the case where $\sum_{j=1}^{\ell-1}v_{\sigma(j)}<\sum_{j=1}^{\ell-1}u_{\sigma(j)}$. Hence we have verified (\ref{E6.20}) for all $\ell\in [K-1]$ by induction.

By (\ref{E6.20}) (taking $\ell=K-1$), for any $\mathbf{v}=(v_{\ell})_{\ell\in[K]}\in\mathcal{D}'(\mathbf{u})$,
\begin{equation*}
    |v_{\sigma(K)}-u_{\sigma(K)}|=\bigg|\sum_{j=1}^{K-1} v_{\sigma(j)}- \sum_{j=1}^{K-1} u_{\sigma(j)}\bigg|\leq \frac{\sqrt{u_{\sigma(K-1)}}}{2}\leq \frac{\sqrt{u_{\sigma(K)}}}{2}.
\end{equation*}
From this and the definition of $\mathcal{D}'(\mathbf{u})$, we conclude that 
\begin{equation}\label{E6.22}
    |v_{\sigma(\ell)}-u_{\sigma(\ell)}|\leq \frac{\sqrt{u_{\sigma(\ell)}}}{2}
\end{equation}
for all $\ell\in [K]$. Note that (\ref{E6.22}) implies $v_{\sigma(K)}\in\mathbb{N}$. Since $v_{\sigma(\ell)}\in\mathbb{N}$ for all $\ell\in[K-1]$ by construction, we have $\mathbf{v}\in \mathfrak{S}_{n,K}$. Hence $\mathcal{D}'(\mathbf{u})\subseteq \mathfrak{S}_{n,K}$. 

By (\ref{E6.17}) and (\ref{E6.22}), for any $\mathbf{v}=(v_{\ell})_{\ell\in[K]}\in\mathcal{D}'(\mathbf{u})$ and $\ell\in [K]$, if $u_{\ell}>0$, then $v_{\ell}\geq u_{\ell}-\sqrt{u_{\ell}}\slash 2\geq u_{\ell}\slash 2$, and so 
\begin{equation}\label{E6.25}
    \rho(u_{\ell},v_{\ell})\geq \log(2)(u_{\ell}+v_{\ell})-\frac{(v_{\ell}-u_{\ell})^2}{2\min\{u_{\ell},v_{\ell}\}}\geq \log(2)(u_{\ell}+v_{\ell})-\frac{1}{4}.
\end{equation}
If $u_{\ell}=0$, then $v_{\ell}=0$ and $\rho(u_{\ell},v_{\ell})=0$, hence (\ref{E6.25}) still holds. Hence for any $\mathbf{v}=(v_{\ell})_{\ell\in[K]}\in\mathcal{D}'(\mathbf{u})$, we have
\begin{equation}\label{E6.26}
    \sum_{\ell=1}^K \rho(u_{\ell},v_{\ell})\geq \log(2)\sum_{\ell=1}^K(u_{\ell}+v_{\ell})-\frac{K}{4}=2\log(2)(n-1)-\frac{K}{4}. 
\end{equation}
By (\ref{E6.22}), for any $\mathbf{v}=(v_{\ell})_{\ell\in[K]} \in \mathcal{D}'(\mathbf{u})$ and $\ell\in [K]$, $\max\{v_{\ell},1\}\leq 2\max\{u_{\ell},1\}$. Hence for any $\mathbf{v}=(v_{\ell})_{\ell\in[K]} \in \mathcal{D}'(\mathbf{u})$,
\begin{equation}\label{E6.28}
    \sum_{\ell=1}^K\log(\max\{v_{\ell},1\})\leq K\log 2+\sum_{\ell=1}^K\log(\max\{u_{\ell},1\}). 
\end{equation}
By (\ref{E6.26}) and (\ref{E6.28}), $\mathcal{D}'(\mathbf{u})\subseteq\mathcal{D}(\mathbf{u})$, and so $|\mathcal{D}(\mathbf{u})|\geq |\mathcal{D}'(\mathbf{u})|$. The conclusion of the lemma follows from this and~\eqref{E6.24}. 
\end{proof}

\section{Proof of Theorem \ref{Theorem_MMSB_S}}\label{Appendix_D}

In this section, we present the proof of Theorem \ref{Theorem_MMSB_S}. We assume the notation in Section \ref{Sect.1.3} and the definitions given in~\eqref{Nildef}-\eqref{YnK:MMSB} and~\eqref{Nil_y_def} (in Appendix \ref{Appendix_B}). Throughout this section, we use $C$ to denote an absolute positive constant. The specific value of $C$ may vary from line to line.

We define $\mathcal{Q}^{\mathrm{ff}}$ to be the family of fully factorized mean-field distributions on $(\mathbf{Z}_{\rightarrow},\mathbf{Z}_{\leftarrow})$; i.e., $Q\in\mathcal{Q}^{\mathrm{ff}}$ if and only if the density of $Q$ w.r.t. a product base measure has the form
\begin{equation*}
    \prod_{i,j\in [n]:i\neq j}\zeta_{i \rightarrow j}(z_{i\rightarrow j})\cdot \prod_{i,j\in [n]: i\neq j}\zeta_{i \leftarrow j}(z_{i\leftarrow j}),\qquad \text{for all }z\in[K]^{2n(n-1)}.
\end{equation*} 
We let
\begin{equation*}
    \hat{Q}^{\mathrm{ff}}\in\arg\min_{Q\in \mathcal{Q}^{\mathrm{ff}}} \KL\big(Q\,\big\|\;\mathbb{P}(\mathbf{Z}_{\rightarrow},\mathbf{Z}_{\leftarrow}\mid\mathbf{X})\big).
\end{equation*}
With $\hat{\mathcal{Y}}_{n,K}$ as in~\eqref{YnK:MMSB}, we define
{\small
\begin{eqnarray}\label{Ydefff}
    \hat{\mathcal{Y}}^{\mathrm{ff}}_{n,K}&:=&\bigg\{y=(y_{i,j;\ell,\ell'})_{i,j\in[n],i\neq j;\ell,\ell'\in[K]}\in\hat{\mathcal{Y}}_{n,K}:\text{ there exist }\nonumber\\
    &&\hspace{0.1in} y_{\rightarrow}=(y_{i\rightarrow j,\ell})_{i,j\in[n],i\neq j;\ell\in [K]}\in \hat{\mathcal{Y}}^{\mathrm{ff},\rightarrow}_{n,K}\text{ and }y_{\leftarrow}=(y_{i\leftarrow j,\ell})_{i,j\in[n],i\neq j;\ell\in [K]}\in \hat{\mathcal{Y}}^{\mathrm{ff},\leftarrow}_{n,K},\nonumber\\
    &&\hspace{0.1in}\text{such that for every }i,j\in[n]\text{ with }i\neq j\text{ and every }\ell,\ell'\in[K],y_{i,j;\ell,\ell'}=y_{i\rightarrow j,\ell}y_{i\leftarrow j,\ell'}\bigg\},
\end{eqnarray}
}where
{\small
\begin{equation*}
    \hat{\mathcal{Y}}^{\mathrm{ff},\rightarrow}_{n,K}:=
    \bigg\{y_{\rightarrow}=(y_{i\rightarrow j,\ell})_{i,j\in[n],i\neq j;\ell\in [K]}\in [0,1]^{n(n-1)K}: \sum_{\ell =1}^K y_{i\rightarrow j,\ell}=1\text{ for all distinct }i,j\in [n]\bigg\},
\end{equation*}
\begin{equation*}
    \hat{\mathcal{Y}}^{\mathrm{ff},\leftarrow}_{n,K}:=
    \bigg\{y_{\leftarrow}=(y_{i\leftarrow j,\ell})_{i,j\in[n],i\neq j;\ell\in [K]}\in [0,1]^{n(n-1)K}: \sum_{\ell =1}^K y_{i\leftarrow j,\ell}=1\text{ for all distinct }i,j\in [n]\bigg\}.
\end{equation*}
}For every $i,j\in [n]$ such that $i\neq j$, we set $\mu_{i,j}$ to be the probability measure on $[K]^2$ such that $\mu_{i,j}(\ell,\ell')=K^{-2}$ for every $(\ell,\ell')\in [K]^2$. We also set $\mu:=\bigotimes\limits_{i,j\in [n]:i\neq j}\mu_{i,j}$ (recall~\eqref{defmu}). For any $y\in\hat{\mathcal{Y}}_{n,K}$, we let $Q_y$ be defined as in~\eqref{def_Qy_MMSB}, and set (note~\eqref{defI_MMSB})
\begin{eqnarray}\label{def_Iy}
    I(y)&:=&\sum_{i,j\in [n]: i\neq j} \sum_{\ell=1}^K\sum_{\ell'=1}^K y_{i,j;\ell,\ell'}\log\bigg(\frac{y_{i,j;\ell,\ell'}}{\mu_{i,j}(\ell,\ell')}\bigg)\nonumber\\
    &=& \sum_{i,j\in [n]: i\neq j} \sum_{\ell=1}^K\sum_{\ell'=1}^K y_{i,j;\ell,\ell'}\log y_{i,j;\ell,\ell'} +2n(n-1)\log K.
\end{eqnarray} 
Note that for any $y\in \hat{\mathcal{Y}}_{n,K}^{\mathrm{ff}}\subseteq \hat{\mathcal{Y}}_{n,K}$, with $y_{\rightarrow}\in\hat{\mathcal{Y}}^{\mathrm{ff},\rightarrow}_{n,K}$ and $y_{\leftarrow}\in\hat{\mathcal{Y}}^{\mathrm{ff},\leftarrow}_{n,K}$ satisfying the condition in~\eqref{Ydefff}, $Q_y$ is the product distribution on $[K]^{2n(n-1)}$ given by
\begin{equation*}
    Q_y(z)=\prod_{i,j\in [n]:i\neq j}y_{i\rightarrow j,z_{i\rightarrow j}}y_{i\leftarrow j, z_{i\leftarrow j}},\qquad \mbox{for all }z\in[K]^{2n(n-1)},
\end{equation*} 
and 
\begin{equation}\label{def_Iyn}
    I(y)=\sum_{i,j\in [n]:i\neq j}\sum_{\ell=1}^K y_{i\rightarrow j,\ell} \log y_{i\rightarrow j,\ell} +   \sum_{i,j\in [n]:i\neq j}\sum_{\ell=1}^K y_{i\leftarrow j,\ell} \log y_{i\leftarrow j,\ell}+2n(n-1)\log K.
\end{equation}

With $\Upsilon(\cdot)$ as in~\eqref{defUps:MMSB}, we set $f(z):=\Upsilon(z)$ for every $z\in [K]^{2n(n-1)}$. By the data processing inequality and Lemma~\ref{Lemma2.1}, we have 
\begin{eqnarray}\label{Thm1.5.eqq1}
   && \KL\big(\hat{P}^{\mathrm{ff}}\, \big\|\; \mathbb{P}(\mathbf{Z}_{\rightarrow},\mathbf{Z}_{\leftarrow},\{\pmb{\pi}_{i}\}_{i\in [n]}\mid\mathbf{X})\big)\geq \KL\big(\hat{Q}^{\mathrm{ff}}\, \big\|\; \mathbb{P}(\mathbf{Z}_{\rightarrow},\mathbf{Z}_{\leftarrow}\mid\mathbf{X})\big)\nonumber\\
   &=&\log\Big(\sum_{z\in[K]^{2n(n-1)}}e^{f(z)}\mu(z)\Big)-\sup_{y \in \hat{\mathcal{Y}}_{n,K}^{\mathrm{ff}} }\big\{\mathbb{E}_{Q_y}[\Upsilon(\mathbf{Z})]-I(y)\big\}.
\end{eqnarray}
Moreover, by Lemma~\ref{Lemma2.1}, we have
\begin{eqnarray}\label{Thm1.5.eqq2}
&&\log\Big(\sum_{z\in[K]^{2n(n-1)}}e^{f(z)}\mu(z)\Big)-\sup_{y \in \hat{\mathcal{Y}}_{n,K}}\big\{\mathbb{E}_{Q_y}[\Upsilon(\mathbf{Z})]-I(y)\big\} \nonumber\\
   &=& \KL\big(\hat{Q}\, \big\|\; \mathbb{P}(\mathbf{Z}_{\rightarrow},\mathbf{Z}_{\leftarrow}\mid\mathbf{X})\big)\geq 0.
\end{eqnarray}
Combining~\eqref{Thm1.5.eqq1} and~\eqref{Thm1.5.eqq2}, we get
\begin{eqnarray}\label{res1}
   && \KL\big(\hat{P}^{\mathrm{ff}}\, \big\|\; \mathbb{P}(\mathbf{Z}_{\rightarrow},\mathbf{Z}_{\leftarrow},\{\pmb{\pi}_{i}\}_{i\in [n]}\mid\mathbf{X})\big)\nonumber\\
   &\geq& \sup_{y\in\hat{\mathcal{Y}}_{n,K}}\big\{\mathbb{E}_{Q_y}[\Upsilon(\mathbf{Z})]-I(y)\big\}-\sup_{y \in \hat{\mathcal{Y}}_{n,K}^{\mathrm{ff}} }\big\{\mathbb{E}_{Q_y}[\Upsilon(\mathbf{Z})]-I(y)\big\}.
\end{eqnarray}

Now note that under the assumptions of Theorem \ref{Theorem_MMSB_S}, $\Upsilon(z)$ as defined in~\eqref{defUps:MMSB} reduces to
\begin{eqnarray}\label{def_Upsilon_newn}
    \Upsilon(z)&=&\sum_{\ell=1}^2\sum_{\ell'=1}^2 |\{(i,j)\in [n]^2:i\neq j, z_{i \rightarrow j}=\ell, z_{i\leftarrow j}=\ell'\}|\log B_{\ell,\ell'}\nonumber\\
    && \hspace{0.2in}+\sum_{i=1}^n\sum_{\ell=1}^2\log\Gamma(N_{i,\ell}(z)+1\slash 2)-n\log\Gamma(2n-1).
\end{eqnarray}
Hence for any $y\in\hat{\mathcal{Y}}_{n,K}$, with $\widetilde{N}_{i,\ell}(y)$ as in~\eqref{Nil_y_def}, using Jensen's inequality (noting that $\log\Gamma(\cdot)$ is convex on $(0,\infty)$) we get
\begin{eqnarray*}
   && \mathbb{E}_{Q_y}[\Upsilon(\mathbf{Z})]\nonumber\\
   &=&\sum_{i,j\in[n]:i\neq j}\sum_{\ell=1}^2\sum_{\ell'=1}^2 y_{i,j;\ell,\ell'}\log B_{\ell,\ell'}+\sum_{i=1}^n\sum_{\ell=1}^2\mathbb{E}_{Q_y}[\log\Gamma(N_{i,\ell}(\mathbf{Z})+1\slash 2)]-n\log\Gamma(2n-1)\nonumber\\
    &\geq& \sum_{i,j\in[n]:i\neq j}\sum_{\ell=1}^2\sum_{\ell'=1}^2 y_{i,j;\ell,\ell'}\log B_{\ell,\ell'}+\sum_{i=1}^n\sum_{\ell=1}^2\log\Gamma\big(\widetilde{N}_{i,\ell}(y)+1\slash 2\big)-n\log\Gamma(2n-1)\nonumber\\
    &\geq& \sum_{i,j\in[n]:i\neq j}\sum_{\ell=1}^2\sum_{\ell'=1}^2 y_{i,j;\ell,\ell'}\log B_{\ell,\ell'}+\sum_{i=1}^n\sum_{\ell=1}^2\phi\big(\widetilde{N}_{i,\ell}(y)\big)-2n(n-1)-Cn-n\log\Gamma(2n-1),
\end{eqnarray*}
where the last line uses Stirling's approximation as in Lemma \ref{L3.2}. Combining the above display with~\eqref{def_Iy} and noting that $K=2$ for this example, we get
\begin{eqnarray}\label{conclu1}
  && \sup_{y\in\hat{\mathcal{Y}}_{n,K}} \big\{\mathbb{E}_{Q_y}[\Upsilon(\mathbf{Z})]-I(y)\big\}\nonumber\\
  &\geq& \sup_{y\in\hat{\mathcal{Y}}_{n,K}}\{\Omega(y)\}-n\log\Gamma(2n-1) -2n(n-1)\log(2e)-Cn,
\end{eqnarray}
where for any $y\in\hat{\mathcal{Y}}_{n,K}$,
\begin{eqnarray}\label{def_Omega}
    \Omega(y)&:=&\sum_{i,j\in[n]:i\neq j}\sum_{\ell=1}^2\sum_{\ell'=1}^2 y_{i,j;\ell,\ell'}\log B_{\ell,\ell'}+\sum_{i=1}^n\sum_{\ell=1}^2\phi\big(\widetilde{N}_{i,\ell}(y)\big)\nonumber\\
    &&\hspace{0.2in}-\sum_{i,j\in [n]: i\neq j} \sum_{\ell=1}^2\sum_{\ell'=1}^2 \phi(y_{i,j;\ell,\ell'}).
\end{eqnarray}

Now fix $y\in\hat{\mathcal{Y}}^{\mathrm{ff}}_{n,K}$, with $y_{\rightarrow}\in\hat{\mathcal{Y}}^{\mathrm{ff},\rightarrow}_{n,K}$ and $y_{\leftarrow}\in\hat{\mathcal{Y}}^{\mathrm{ff},\leftarrow}_{n,K}$ satisfying the condition in~\eqref{Ydefff}. Note that for any $i,j\in [n]$ such that $i\neq j$, we have
\begin{eqnarray*}
    \sum_{\ell=1}^2\sum_{\ell'=1}^2 \phi(y_{i,j;\ell,\ell'})&=&\sum_{\ell=1}^2\sum_{\ell'=1}^2 \phi(y_{i\rightarrow j,\ell}y_{i\leftarrow j,\ell'})=\sum_{\ell=1}^2\sum_{\ell'=1}^2 y_{i\rightarrow j,\ell}y_{i\leftarrow j,\ell'}(\log y_{i\rightarrow j,\ell}+\log y_{i\leftarrow j,\ell'})\nonumber\\
    &=&\sum_{\ell=1}^2\phi(y_{i\rightarrow j,\ell})+\sum_{\ell'=1}^2\phi(y_{i\leftarrow j,\ell'})=\sum_{\ell=1}^2\phi(y_{i\rightarrow j,\ell})+\sum_{\ell=1}^2\phi(y_{i\leftarrow j,\ell}).
\end{eqnarray*}
Hence by~\eqref{def_Omega}, we have
\begin{equation}\label{OmegaEq}
    \Omega(y)=\omega(y_{\rightarrow},y_{\leftarrow}),
\end{equation}
where 
\begin{eqnarray}\label{def_omega}
    \omega(y_{\rightarrow},y_{\leftarrow})&:=&\sum_{i,j\in[n]:i\neq j}\sum_{\ell=1}^2\sum_{\ell'=1}^2 y_{i \rightarrow j,\ell}y_{i\leftarrow j,\ell'}\log B_{\ell,\ell'}\nonumber\\
    &&\hspace{0.2in}+\sum_{i=1}^n\sum_{\ell=1}^2\phi\bigg(\sum_{j\in[n]\backslash\{i\}}y_{i\rightarrow j,\ell}+\sum_{j\in[n]\backslash\{i\}}y_{j\leftarrow i,\ell}\bigg)\nonumber\\
    &&\hspace{0.2in}-\sum_{i,j\in [n]:i\neq j}\sum_{\ell=1}^2 \phi(y_{i\rightarrow j,\ell})-   \sum_{i,j\in [n]:i\neq j}\sum_{\ell=1}^2 \phi(y_{i\leftarrow j,\ell}).
\end{eqnarray}
By~\eqref{def_Upsilon_newn}, we have 
\begin{eqnarray}\label{EUpsilon}
    \mathbb{E}_{Q_y}[\Upsilon(\mathbf{Z})]&=&\sum_{i,j\in[n]:i\neq j}\sum_{\ell=1}^2\sum_{\ell'=1}^2 y_{i \rightarrow j,\ell}y_{i\leftarrow j,\ell'}\log B_{\ell,\ell'}+\sum_{i=1}^n\sum_{\ell=1}^2\mathbb{E}_{Q_y}[\log\Gamma(N_{i,\ell}(\mathbf{Z})+1\slash 2)]\nonumber\\
    &&\hspace{0.2in}-n\log\Gamma(2n-1).
\end{eqnarray}
Arguing similarly as in~\eqref{E6.3}-\eqref{E6.5}, we get 
\begin{eqnarray}\label{Thm1.5.eqq3}
&&\sum_{i=1}^n\sum_{\ell=1}^2\mathbb{E}_{Q_y}[\log\Gamma(N_{i,\ell}(\mathbf{Z})+1\slash 2)]\nonumber\\
&\leq&\sum_{i=1}^n\sum_{\ell=1}^2 
\phi\bigg(\sum_{j\in[n]\backslash\{i\}}y_{i\rightarrow j,\ell}+\sum_{j\in[n]\backslash\{i\}}y_{j\leftarrow i,\ell}\bigg)
-2n(n-1)+Cn,
\end{eqnarray}
where $\phi(x)=x\log x$ as in~\eqref{phioriginal}. By~\eqref{EUpsilon}, \eqref{Thm1.5.eqq3}, and~\eqref{def_Iyn}, noting that $K=2$ in this example, we have 
\begin{eqnarray*}
 \mathbb{E}_{Q_y}[\Upsilon(\mathbf{Z})]-I(y)
 &\leq& \omega(y_{\rightarrow},y_{\leftarrow})-n\log\Gamma(2n-1)-2n(n-1)\log(2e)+Cn\nonumber\\
    &=& \Omega(y)-n\log\Gamma(2n-1)-2n(n-1)\log(2e)+Cn,
\end{eqnarray*}
where $w(y_{\rightarrow},y_{\leftarrow})$ is as in~\eqref{def_omega}. Hence
\begin{eqnarray}\label{conclusion_new}
    && \sup_{y\in\hat{\mathcal{Y}}_{n,K}^{\mathrm{ff}}}\big\{\mathbb{E}_{Q_y}[\Upsilon(\mathbf{Z})]-I(y)\big\}\nonumber\\
    &\leq& \sup_{y\in\hat{\mathcal{Y}}_{n,K}^{\mathrm{ff}}}\{\Omega(y)\}-n\log\Gamma(2n-1)-2n(n-1)\log(2e)+Cn.
\end{eqnarray}
Combining~\eqref{res1}, \eqref{conclu1}, and \eqref{conclusion_new}, we get
\begin{eqnarray*}
   && \frac{1}{n^2}\KL\big(\hat{P}^{\mathrm{ff}}\, \big\|\; \mathbb{P}(\mathbf{Z}_{\rightarrow},\mathbf{Z}_{\leftarrow},\{\pmb{\pi}_{i}\}_{i\in [n]}\mid\mathbf{X})\big)\nonumber\\
   &\geq&\frac{\sup_{y\in\hat{\mathcal{Y}}_{n,K}}\{\Omega(y)\}-\sup_{y\in\hat{\mathcal{Y}}_{n,K}^{\mathrm{ff}}}\{\Omega(y)\}-Cn}{n^2}.
\end{eqnarray*}
Denote $\Delta:=\log B_{1,2}-\log B_{1,1}>0$. Combining the above display with Lemma \ref{LE2} below yields the desired conclusion of Theorem \ref{Theorem_MMSB_S}. 

\begin{lemma}\label{LE2}
There exists a constant $\delta>0$ that only depends on $\mathbf{B} \equiv (B_{\ell,\ell'})_{\ell,\ell'\in [2]}$, such that 
\begin{equation*}
    \sup_{y\in\hat{\mathcal{Y}}_{n,K}}\{\Omega(y)\}-\sup_{y\in\hat{\mathcal{Y}}_{n,K}^{\mathrm{ff}}}\{\Omega(y)\}
    \geq \delta n(n-1)-\Delta n.
\end{equation*}
\end{lemma}

The rest of this section is devoted to the proof of Lemma \ref{LE2}. As $\hat{\mathcal{Y}}_{n,K}^{\mathrm{ff}}$ is a compact set and $\Omega(\cdot)$ is continuous on $\hat{\mathcal{Y}}_{n,K}^{\mathrm{ff}}$, there exists a maximizer of $\Omega(\cdot)$ on $\hat{\mathcal{Y}}_{n,K}^{\mathrm{ff}}$, which we denote by $y^{*}=(y^{*}_{i,j;\ell,\ell'})_{i,j\in[n],i\neq j;\ell,\ell'\in [K]}$. Suppose that
\begin{equation*}
    y_{\rightarrow}^{*}=(y^{*}_{i\rightarrow j,\ell})_{i,j\in[n],i\neq j;\ell\in [K]}\in \hat{\mathcal{Y}}^{\mathrm{ff},\rightarrow}_{n,K},\qquad y_{\leftarrow}^{*}=(y^{*}_{i\leftarrow j,\ell})_{i,j\in[n],i\neq j;\ell\in [K]}\in \hat{\mathcal{Y}}^{\mathrm{ff},\leftarrow}_{n,K}
\end{equation*}
satisfy the condition in~\eqref{Ydefff}. We have the following lemma. 

\begin{lemma}\label{LE1}
For any $\tau\in \Big(0,\frac{1}{1+e^{2\Delta}}\Big)$, there exists $\varepsilon(\tau)>0$ that only depends on $\tau,\Delta$, such that $\lim_{\tau\rightarrow 0^{+}}\varepsilon(\tau)=0$ and at least one of the following holds true:
\begin{itemize}
    \item[(i)] There exist at least $\tau n(\tau n-1)$ pairs $(i,j)\in [n]^2$ with $i\neq j$, such that
\begin{equation}\label{cond1}
    \tau\leq y^{*}_{i\rightarrow j,1}\leq 1-\tau, \qquad \tau\leq y^{*}_{i\leftarrow j,1}\leq 1-\tau. 
\end{equation}
    \item[(ii)] $\sup_{y\in \hat{\mathcal{Y}}_{n,K}^{\mathrm{ff}}}\{\Omega(y)\}\leq \Big(\log B_{1,1}+\frac{\Delta}{2}+2\log(2(n-1))+\varepsilon(\tau)\Big)n(n-1)+\Delta n$.
\end{itemize}
\end{lemma} 

\begin{proof}[Proof of Lemma \ref{LE1}]

For any $y_{\rightarrow}\in\hat{\mathcal{Y}}^{\mathrm{ff},\rightarrow}_{n,K},y_{\leftarrow}\in\hat{\mathcal{Y}}^{\mathrm{ff},\leftarrow}_{n,K}$ and $i\in [n],\ell\in [2]$, we define 
\begin{equation*}
    \widetilde{N}_{i,\ell}(y_{\rightarrow},y_{\leftarrow}):=\sum_{j\in[n]\backslash\{i\}}y_{i\rightarrow j,\ell}+\sum_{j\in[n]\backslash\{i\}}y_{j\leftarrow i,\ell}.
\end{equation*}
With $\omega(\cdot,\cdot)$ as defined in~\eqref{def_omega}, for any $y_{\rightarrow}\in\hat{\mathcal{Y}}^{\mathrm{ff},\rightarrow}_{n,K},y_{\leftarrow}\in\hat{\mathcal{Y}}^{\mathrm{ff},\leftarrow}_{n,K}$ and $i,j\in [n]$ with $i\neq j$, if $y_{i\rightarrow j,1}\in (0,1)$, then 
\begin{equation*}
    \frac{\partial \omega(y_{\rightarrow},y_{\leftarrow})}{\partial y_{i\rightarrow j,1}}=y_{i\leftarrow j,1}\log B_{1,1}+y_{i\leftarrow j,2}\log B_{1,2}+\log\big(\widetilde{N}_{i,1}(y_{\rightarrow},y_{\leftarrow})\big)-\log y_{i\rightarrow j,1},
\end{equation*}
\begin{equation*}
    \frac{\partial \omega(y_{\rightarrow},y_{\leftarrow})}{\partial y_{i\rightarrow j,2}}=y_{i\leftarrow j,1}\log B_{2,1}+y_{i\leftarrow j,2}\log B_{2,2}+\log\big(\widetilde{N}_{i,2}(y_{\rightarrow},y_{\leftarrow})\big)-\log y_{i\rightarrow j,2}.
\end{equation*}
Taking the difference of the above two displays, noting that $B_{1,1}=B_{2,2}<B_{1,2}=B_{2,1}$ and $\Delta=\log B_{1,2}-\log B_{1,1}$, we get
\begin{equation}\label{deriv1}
     \frac{\partial \omega(y_{\rightarrow},y_{\leftarrow})}{\partial y_{i\rightarrow j,1}}-\frac{\partial \omega(y_{\rightarrow},y_{\leftarrow})}{\partial y_{i\rightarrow j,2}}=
        -(y_{i\leftarrow j,1}-y_{i\leftarrow j,2})\Delta+ \log\bigg(\frac{\widetilde{N}_{i,1}(y_{\rightarrow},y_{\leftarrow})}{\widetilde{N}_{i,2}(y_{\rightarrow},y_{\leftarrow})}\bigg)-\log\bigg(\frac{y_{i\rightarrow j,1}}{y_{i\rightarrow j,2}}\bigg).
\end{equation}
Similarly, if $y_{i\leftarrow j,1}\in (0,1)$, then 
\begin{equation}\label{deriv2}
    \frac{\partial \omega(y_{\rightarrow},y_{\leftarrow})}{\partial y_{i\leftarrow j,1}}-\frac{\partial \omega(y_{\rightarrow},y_{\leftarrow})}{\partial y_{i\leftarrow j,2}}= -(y_{i\rightarrow j,1}-y_{i\rightarrow j,2})\Delta+ \log\bigg(\frac{\widetilde{N}_{j,1}(y_{\rightarrow},y_{\leftarrow})}{\widetilde{N}_{j,2}(y_{\rightarrow},y_{\leftarrow})}\bigg)-\log\bigg(\frac{y_{i\leftarrow j,1}}{y_{i\leftarrow j,2}}\bigg).
\end{equation}

Below we fix $i,j\in [n]$ with $i\neq j$. By~\eqref{deriv1}, if $y^{*}_{i\rightarrow j, 1}=0$ (which implies $y^{*}_{i\rightarrow j,2}=1$) and $\widetilde{N}_{i,1}(y^{*}_{\rightarrow},y^{*}_{\leftarrow})>0$, then we have $\frac{\partial \omega(y^{*}_{\rightarrow},y^{*}_{\leftarrow})}{\partial y_{i\rightarrow j,1}}-\frac{\partial \omega(y^{*}_{\rightarrow},y^{*}_{\leftarrow})}{\partial y_{i\rightarrow j,2}}=+\infty$, and we can increase the value of $y^{*}_{i\rightarrow j,1}$ by a sufficiently small amount (and decrease $y^{*}_{i\rightarrow j,2}=1-y^{*}_{i\rightarrow j,1}$ meanwhile) to strictly increase the value of $\omega(y^{*}_{\rightarrow},y^{*}_{\leftarrow})$. This contradicts with the definition of $(y^{*}_{\rightarrow},y^{*}_{\leftarrow})$. Hence if $y_{i\rightarrow j,1}^{*}=0$, then $\widetilde{N}_{i,1}(y^{*}_{\rightarrow},y^{*}_{\leftarrow})=0$. Similarly, if $y_{i\rightarrow j,1}^{*}=1$, then $\widetilde{N}_{i,2}(y^{*}_{\rightarrow},y^{*}_{\leftarrow})=0$; if $y_{i\leftarrow j,1}^{*}=0$, then $\widetilde{N}_{j,1}(y^{*}_{\rightarrow},y^{*}_{\leftarrow})=0$; if $y_{i\leftarrow j,1}^{*}=1$, then $\widetilde{N}_{j,2}(y^{*}_{\rightarrow},y^{*}_{\leftarrow})=0$. 

Now if $y_{i\rightarrow j,1}^{*}\in (0,1)$, we have $\frac{\partial \omega(y^{*}_{\rightarrow},y^{*}_{\leftarrow})}{\partial y_{i\rightarrow j,1}}-\frac{\partial \omega(y^{*}_{\rightarrow},y^{*}_{\leftarrow})}{\partial y_{i\rightarrow j,2}}=0$. Hence by~\eqref{deriv1}, 
\begin{eqnarray*}
   & -(y^{*}_{i\leftarrow j,1}-y^{*}_{i\leftarrow j,2})\Delta+ \log\bigg(\frac{\widetilde{N}_{i,1}(y^{*}_{\rightarrow},y^{*}_{\leftarrow})}{\widetilde{N}_{i,2}(y^{*}_{\rightarrow},y^{*}_{\leftarrow})}\bigg)-\log\bigg(\frac{y^{*}_{i\rightarrow j,1}}{y^{*}_{i\rightarrow j,2}}\bigg)=0,\nonumber\\
   & \Rightarrow \bigg|\log\bigg(\frac{\widetilde{N}_{i,1}(y^{*}_{\rightarrow},y^{*}_{\leftarrow})}{\widetilde{N}_{i,2}(y^{*}_{\rightarrow},y^{*}_{\leftarrow})}\frac{y^{*}_{i\rightarrow j,2}}{y^{*}_{i\rightarrow j,1}}\bigg)\bigg|=|(y^{*}_{i\leftarrow j,1}-y^{*}_{i\leftarrow j,2})\Delta|\leq \Delta,
\end{eqnarray*}
which gives
\begin{equation}\label{neq1.1}
  e^{-\Delta}\cdot\frac{\widetilde{N}_{i,1}(y^{*}_{\rightarrow},y^{*}_{\leftarrow})}{\widetilde{N}_{i,2}(y^{*}_{\rightarrow},y^{*}_{\leftarrow})}  \leq \frac{y^{*}_{i\rightarrow j,1}}{y^{*}_{i\rightarrow j,2}} \leq e^{\Delta}\cdot\frac{\widetilde{N}_{i,1}(y^{*}_{\rightarrow},y^{*}_{\leftarrow})}{\widetilde{N}_{i,2}(y^{*}_{\rightarrow},y^{*}_{\leftarrow})}.
\end{equation}
Similarly, if $y_{i\leftarrow j,1}^{*}\in (0,1)$, then $\frac{\partial \omega(y^*_{\rightarrow},y^*_{\leftarrow})}{\partial y_{i\leftarrow j,1}}-\frac{\partial \omega(y^*_{\rightarrow},y^*_{\leftarrow})}{\partial y_{i\leftarrow j,2}}=0$, and from~\eqref{deriv2} we can deduce that
\begin{equation}\label{neq1.2}
  e^{-\Delta}\cdot\frac{\widetilde{N}_{j,1}(y^{*}_{\rightarrow},y^{*}_{\leftarrow})}{\widetilde{N}_{j,2}(y^{*}_{\rightarrow},y^{*}_{\leftarrow})}  \leq \frac{y^{*}_{i\leftarrow j,1}}{y^{*}_{i\leftarrow j,2}} \leq e^{\Delta}\cdot\frac{\widetilde{N}_{j,1}(y^{*}_{\rightarrow},y^{*}_{\leftarrow})}{\widetilde{N}_{j,2}(y^{*}_{\rightarrow},y^{*}_{\leftarrow})}.
\end{equation}

In the following, we fix $\tau\in \Big(0,\frac{1}{1+e^{2\Delta}}\Big)$, and suppose that condition (i) in the statement of the lemma does not hold. Denote $M(\tau):=\log\big(\frac{1-\tau}{\tau}\big)-\Delta>0$, and let $\mathfrak{A}(\tau)$ be the set of $i\in [n]$ such that
\begin{equation}\label{good}
   -M(\tau)\leq \log\bigg(\frac{\widetilde{N}_{i,1}(y^*_{\rightarrow},y^*_{\leftarrow})}{\widetilde{N}_{i,2}(y^*_{\rightarrow},y^*_{\leftarrow})}\bigg)\leq M(\tau). 
\end{equation}
From the above analysis, for any $(i,j)\in \mathfrak{A}(\tau)^2$ such that $i\neq j$, we have
\begin{equation*}
     \frac{\tau}{1-\tau} \leq  \frac{y^{*}_{i\rightarrow j,1}}{y^{*}_{i\rightarrow j,2}}\leq \frac{1-\tau}{\tau} \Rightarrow \tau\leq y^{*}_{i\rightarrow j,1}\leq 1-\tau,
\end{equation*}
\begin{equation*}
     \frac{\tau}{1-\tau} \leq  \frac{y^{*}_{i\leftarrow j,1}}{y^{*}_{i\leftarrow j,2}}\leq \frac{1-\tau}{\tau} \Rightarrow \tau\leq y^{*}_{i\leftarrow j,1}\leq 1-\tau.
\end{equation*}
As condition (i) does not hold, we have $|\mathfrak{A}(\tau)|\leq \tau n$. 

Now we proceed to bound $\sup_{y\in \hat{\mathcal{Y}}_{n,K}^{\mathrm{ff}}}\{\Omega(y)\}=\omega(y^*_{\rightarrow},y^*_{\leftarrow})$ (see \eqref{OmegaEq}). As $\sum_{\ell=1}^2\widetilde{N}_{i,\ell}(y_{\rightarrow}^*,y_{\leftarrow}^*)=2(n-1)$ for every $i\in [n]$ and $\max_{x\in [0,2(n-1)]}\{\phi(x)+\phi(2(n-1)-x)\}=\phi(2(n-1))$, we have
\begin{equation}\label{resu1}
    \sum_{i=1}^n\sum_{\ell=1}^2\phi\big(\widetilde{N}_{i,\ell}(y^{*}_{\rightarrow},y^{*}_{\leftarrow})\big)\leq n\phi(2(n-1))=2n(n-1)\log(2(n-1)).
\end{equation}
We denote by $\mathfrak{A}_1'(\tau)$ the set of $i\in [n]$ such that $\log\Big(\frac{\widetilde{N}_{i,1}(y^*_{\rightarrow},y^*_{\leftarrow})}{\widetilde{N}_{i,2}(y^*_{\rightarrow},y^*_{\leftarrow})}\Big)> M(\tau)$, and by $\mathfrak{A}_2'(\tau)$ the set of $i\in [n]$ such that $\log\Big(\frac{\widetilde{N}_{i,1}(y^*_{\rightarrow},y^*_{\leftarrow})}{\widetilde{N}_{i,2}(y^*_{\rightarrow},y^*_{\leftarrow})}\Big)< -M(\tau)$. Note that by~\eqref{good}, $[n]=\mathfrak{A}(\tau)\cup\mathfrak{A}_1'(\tau)\cup\mathfrak{A}_2'(\tau)$. Below we denote $a(\tau):=\frac{(1-\tau)e^{-2\Delta}}{\tau+(1-\tau)e^{-2\Delta}}\in (1\slash 2,1)$ (since $\tau\in \Big(0,\frac{1}{1+e^{2\Delta}}\Big)$). By~\eqref{neq1.1} and~\eqref{neq1.2}, for any $i\in \mathfrak{A}_1'(\tau)$ and $j\in [n]\backslash\{i\}$, 
\begin{equation*}
   \min\bigg\{ \log\bigg(\frac{y^{*}_{i\rightarrow j,1}}{y^{*}_{i\rightarrow j,2}}\bigg),\log\bigg(\frac{y^{*}_{j\leftarrow i,1}}{y^{*}_{j\leftarrow i,2}}\bigg)\bigg\}\geq M(\tau)-\Delta=\log\bigg(\frac{1-\tau}{\tau}\bigg)-2\Delta>0, 
\end{equation*}
which implies 
\begin{equation}\label{minbdd}
    \min\{y^{*}_{i\rightarrow j,1},y^{*}_{j\leftarrow i,1}\}\geq \frac{(1-\tau)e^{-2\Delta}}{\tau+(1-\tau)e^{-2\Delta}}=a(\tau)>\frac{1}{2} \Rightarrow \max\{y^{*}_{i\rightarrow j,2},y^{*}_{j\leftarrow i,2}\}\leq 1-a(\tau).
\end{equation}
By~\eqref{minbdd}, for any $i\in \mathfrak{A}_1'(\tau)$ and $j\in [n]\backslash\{i\}$, 
\begin{eqnarray}\label{ent}
   && \phi(y^{*}_{i\rightarrow j,1})+\phi(y^{*}_{i\rightarrow j,2})=\phi(y^{*}_{i\rightarrow j,1})+\phi(1-y^{*}_{i\rightarrow j,1})\geq \phi(a(\tau))+\phi(1-a(\tau)),\nonumber\\
   &&\phi(y^{*}_{j\leftarrow i,1})+\phi(y^{*}_{j\leftarrow i,2})=\phi(y^{*}_{j\leftarrow i,1})+\phi(1-y^{*}_{j\leftarrow i,1})\geq \phi(a(\tau))+\phi(1-a(\tau)),
\end{eqnarray}
where we note that $\phi(x)+\phi(1-x)$ is monotonically decreasing on $[0,1\slash 2]$ and monotonically increasing on $[1\slash 2,1]$. By~\eqref{neq1.1} and~\eqref{neq1.2}, for any $i\in\mathfrak{A}_2'(\tau)$ and $j\in [n]\backslash\{i\}$,
\begin{equation*}
   \max\bigg\{ \log\bigg(\frac{y^{*}_{i\rightarrow j,1}}{y^{*}_{i\rightarrow j,2}}\bigg),\log\bigg(\frac{y^{*}_{j\leftarrow i,1}}{y^{*}_{j\leftarrow i,2}}\bigg)\bigg\}\leq -M(\tau)+\Delta=-\log\bigg(\frac{1-\tau}{\tau}\bigg)+2\Delta, 
\end{equation*}
which implies
\begin{equation}\label{minbdd.1}
    \max\{y^{*}_{i\rightarrow j,1},y^{*}_{j\leftarrow i,1}\}\leq \frac{\tau}{\tau+(1-\tau)e^{-2\Delta}}=1-a(\tau)<\frac{1}{2}.
\end{equation}
Hence for any $i\in\mathfrak{A}_2'(\tau)$ and $j\in [n]\backslash\{i\}$,
\begin{equation}\label{ent.1}
    \min\{\phi(y^{*}_{i\rightarrow j,1})+\phi(y^{*}_{i\rightarrow j,2}),\phi(y^{*}_{j\leftarrow i,1})+\phi(y^{*}_{j\leftarrow i,2})\}\geq \phi(a(\tau))+\phi(1-a(\tau)).
\end{equation}
Combining~\eqref{ent} and~\eqref{ent.1}, and noting that $\phi(x)+\phi(1-x)\geq 2\phi(1\slash 2)=-\log 2$ for any $x\in [0,1]$, we get
\begin{eqnarray}\label{resu2}
  &&  -\sum_{i,j\in [n]:i\neq j}\sum_{\ell=1}^2 \phi(y^{*}_{i\rightarrow j,\ell})-   \sum_{i,j\in [n]:i\neq j}\sum_{\ell=1}^2 \phi(y^{*}_{i\leftarrow j,\ell})\nonumber\\
  &=& -\sum_{i\in \mathfrak{A}_1'(\tau)\cup\mathfrak{A}_2'(\tau)}\sum_{j\in[n]\backslash \{i\}}\sum_{\ell=1}^2 \phi(y^{*}_{i\rightarrow j,\ell})-\sum_{i\in \mathfrak{A}_1'(\tau)\cup\mathfrak{A}_2'(\tau)}\sum_{j\in[n]\backslash \{i\}}\sum_{\ell=1}^2 \phi(y^{*}_{j\leftarrow i,\ell})\nonumber\\
  &&\hspace{0.2in}-\sum_{i\in \mathfrak{A}(\tau)}\sum_{j\in[n]\backslash \{i\}}\sum_{\ell=1}^2 \phi(y^{*}_{i\rightarrow j,\ell})-\sum_{i\in \mathfrak{A}(\tau)}\sum_{j\in[n]\backslash \{i\}}\sum_{\ell=1}^2 \phi(y^{*}_{j\leftarrow i,\ell})\nonumber\\
  &\leq& -2(n-|\mathfrak{A}(\tau)|)(n-1)(\phi(a(\tau))+\phi(1-a(\tau)))+2|\mathfrak{A}(\tau)|(n-1)\log 2\nonumber\\
  &\leq& \big(-2(1-\tau)(\phi(a(\tau))+\phi(1-a(\tau)))+2\tau\log 2\big)n(n-1),
\end{eqnarray}
where the last inequality uses the fact that $|\mathfrak{A}(\tau)|\leq \tau n$, along with the fact that $\phi(a(\tau))+\phi(1-a(\tau))\geq -\log 2$. Using the fact that $y^*_{i\rightarrow j,\ell},y^*_{i\leftarrow j,\ell}\in [0,1]$, for any $(i,j)\in (\mathfrak{A}_1'(\tau)\times\mathfrak{A}_1'(\tau))\cup(\mathfrak{A}_2'(\tau)\times\mathfrak{A}_2'(\tau))$ such that $i\neq j$, we have
\begin{equation*}
    y_{i\rightarrow j,1}^{*}y^{*}_{i\leftarrow j,2}+y_{i\rightarrow j,2}^{*}y^{*}_{i\leftarrow j,1}  \leq \min\{y_{i\rightarrow j,1}^{*}+y^{*}_{i\leftarrow j,1},y^{*}_{i\leftarrow j,2}+y_{i\rightarrow j,2}^{*}\}
    \leq
    2(1-a(\tau))<1,
\end{equation*}
where the last but one inequality uses~\eqref{minbdd} and~\eqref{minbdd.1}. Moreover, for any $i,j\in [n]$ such that $i\neq j$, we have 
\begin{equation*}
    y_{i\rightarrow j,1}^{*}y^{*}_{i\leftarrow j,2}+y_{i\rightarrow j,2}^{*}y^{*}_{i\leftarrow j,1}  \leq (y^*_{i\rightarrow j,1}+y^*_{i\rightarrow j,2})(y^{*}_{i\leftarrow j,1}+y^{*}_{i\leftarrow j,2})=1.
\end{equation*}
Hence noting that $B_{1,1}=B_{2,2}<B_{1,2}=B_{2,1}$ and $\Delta=\log B_{1,2}-\log B_{1,1}$, we get   
\begin{eqnarray}\label{resu3}
    && \sum_{i,j\in[n]:i\neq j}\sum_{\ell=1}^2\sum_{\ell'=1}^2 y^*_{i \rightarrow j,\ell}y^*_{i\leftarrow j,\ell'}\log B_{\ell,\ell'}\nonumber\\
    &=& n(n-1)\log B_{1,1}+\sum_{i,j\in [n]:i\neq j}(y^*_{i \rightarrow j,1}y^*_{i\leftarrow j,2}+y^*_{i\rightarrow j,2}y^*_{i\leftarrow j,1})\Delta\nonumber\\
    &\leq& n(n-1)\log B_{1,1}\nonumber\\
    &&\hspace{0.1in}+2(1-a(\tau))\Delta\big|\big\{(i,j)\in [n]^2:i\neq j,(i,j)\in (\mathfrak{A}_1'(\tau)\times\mathfrak{A}_1'(\tau))\cup(\mathfrak{A}_2'(\tau)\times\mathfrak{A}_2'(\tau))\big\}\big|\nonumber\\
    &&\hspace{0.1in}+\Delta\big|\big\{(i,j)\in [n]^2:i\neq j,(i,j)\notin (\mathfrak{A}_1'(\tau)\times\mathfrak{A}_1'(\tau))\cup(\mathfrak{A}_2'(\tau)\times\mathfrak{A}_2'(\tau))\big\}\big|\nonumber\\
    &=& n(n-1)\log B_{1,1}+n(n-1)\Delta\nonumber\\
    &&\hspace{0.1in}
    -(1-2(1-a(\tau)))\Delta\big(|\mathfrak{A}'_1(\tau)|(|\mathfrak{A}'_1(\tau)|-1)+|\mathfrak{A}'_2(\tau)|(|\mathfrak{A}'_2(\tau)|-1)\big)\nonumber\\
    &\leq& n(n-1)\log B_{1,1}+n(n-1)\Delta
    +(1-2(1-a(\tau)))\Delta n\nonumber\\
    &&\hspace{0.1in} -(1-2(1-a(\tau)))\Delta\frac{(1-\tau)^2n^2}{2}\nonumber\\
    &\leq& n(n-1)\log B_{1,1}+\Delta  n^2 -\frac{(1-\tau)^2(1-2(1-a(\tau)))}{2}\Delta n^2,
\end{eqnarray}
where the last but one inequality uses the bound $|\mathfrak{A}(\tau)|\leq\tau n$ along with the inequality
\begin{eqnarray*}
   |\mathfrak{A}'_1(\tau)|(|\mathfrak{A}'_1(\tau)|-1)+|\mathfrak{A}'_2(\tau)|(|\mathfrak{A}'_2(\tau)|-1)&\geq& \frac{(|\mathfrak{A}'_1(\tau)|+|\mathfrak{A}'_2(\tau)|)^2}{2}-(|\mathfrak{A}'_1(\tau)|+|\mathfrak{A}'_2(\tau)|)\nonumber\\
   &=&\frac{(n-|\mathfrak{A}(\tau)|)^2}{2}-(n-|\mathfrak{A}(\tau)|)\geq \frac{(1-\tau)^2n^2}{2}-n.
\end{eqnarray*}
With $\omega(\cdot,\cdot)$ as in~\eqref{def_omega},  using~\eqref{resu3}, \eqref{resu1}, and~\eqref{resu2}, we get
\begin{eqnarray*}
   && \sup_{y\in \hat{\mathcal{Y}}_{n,K}^{\mathrm{ff}}}\{\Omega(y)\}=\omega(y^*_{\rightarrow},y^*_{\leftarrow})\nonumber\\
   &\leq& n(n-1)\big(\log B_{1,1}+2\log(2(n-1))-2(1-\tau)(\phi(a(\tau))+\phi(1-a(\tau)))+2\tau\log 2\big)\nonumber\\
   &&\hspace{0.2in} +\Delta n^2-
\frac{(1-\tau)^2(1-2(1-a(\tau)))}{2}\Delta n^2\nonumber\\
&\leq& n(n-1)\Big(\log B_{1,1}+2\log(2(n-1))-2(1-\tau)(\phi(a(\tau))+\phi(1-a(\tau)))+2\tau\log 2\Big)\nonumber\\
   &&\hspace{0.2in} +
   \bigg(1-\frac{(1-\tau)^2(1-2(1-a(\tau)))}{2}\bigg)\Delta n^2.
\end{eqnarray*}
Taking 
\begin{eqnarray*}
    \varepsilon(\tau)&:=&\bigg(\frac{1}{2}-\frac{(1-\tau)^2(1-2(1-a(\tau)))}{2}\bigg)\Delta-2(1-\tau)(\phi(a(\tau))+\phi(1-a(\tau)))+2\tau\log 2,
\end{eqnarray*}
we thus get 
\begin{equation*}
    \sup_{y\in \hat{\mathcal{Y}}_{n,K}^{\mathrm{ff}}}\{\Omega(y)\}\leq \Big(\log B_{1,1}+\frac{\Delta}{2}+2\log(2(n-1))+\varepsilon(\tau)\Big)n(n-1)+\Delta n.
\end{equation*}
Moreover, as $\lim_{\tau\rightarrow 0^+}a(\tau)=1$, we have $\lim_{\tau\rightarrow 0^+}\varepsilon(\tau)=0$. Hence condition (ii) in the statement of the lemma holds.
\end{proof}

Now we complete the proof of Lemma \ref{LE2} using Lemma \ref{LE1}.

\begin{proof}[Proof of Lemma \ref{LE2}]  

We first note that, with $\Omega(\cdot)$ as in~\eqref{def_Omega}, for any $t\in [0,1]$, if we take $y\in \hat{\mathcal{Y}}_{n,K}$ such that for every $i,j\in [n]$ with $i\neq j$, $y_{i,j;1,1}=y_{i,j;2,2}=\frac{1-t}{2}$ and $y_{i,j;1,2}=y_{i,j;2,1}=\frac{t}{2}$, then $\widetilde{N}_{i,\ell}(y)=n-1$ (see~\eqref{Nil_y_def} for the definition of $\widetilde{N}_{i,\ell}(y)$) for any $i\in[n],\ell\in[2]$, and consequently
\begin{eqnarray*}
    \Omega(y)&=&n(n-1)\bigg((1-t)\log B_{1,1}+t\log B_{1,2}\bigg)+2n(n-1)\log(n-1)\nonumber\\
    &&\hspace{0.2in}-n(n-1)\bigg(2\phi\Big(\frac{1-t}{2}\Big)+2\phi\Big(\frac{t}{2}\Big)\bigg)\nonumber\\
    &=& n(n-1)\bigg(\log B_{1,1}+2\log(n-1)+t\Delta-2\phi\Big(\frac{1-t}{2}\Big)-2\phi\Big(\frac{t}{2}\Big)\bigg).
\end{eqnarray*}
Let $h(t):=t\Delta-2\phi\big(\frac{1-t}{2}\big)-2\phi\big(\frac{t}{2}\big)$ for any $t\in[0,1]$. As $h'(t)=\Delta+\log\big(\frac{1-t}{t}\big)$, $h(t)$ takes its maximum value on $[0,1]$ at $t=\frac{e^{\Delta}}{1+e^{\Delta}}$, and $h\Big(\frac{e^{\Delta}}{1+e^{\Delta}}\Big)>h\Big(\frac{1}{2}\Big)$. Hence
\begin{eqnarray}\label{nneq1}
    \sup_{y\in\hat{\mathcal{Y}}_{n,K}}\{\Omega(y)\} &\geq& n(n-1)\bigg(\log B_{1,1}+2\log(n-1)+h\bigg(\frac{e^{\Delta}}{1+e^{\Delta}}\bigg)\bigg)\nonumber\\
    &=& n(n-1)\bigg(\log B_{1,1}+2\log(n-1)+h\Big(\frac{1}{2}\Big)+h\bigg(\frac{e^{\Delta}}{1+e^{\Delta}}\bigg)-h\Big(\frac{1}{2}\Big)\bigg)\nonumber\\
    &=& n(n-1)\bigg(\log B_{1,1}+\frac{\Delta}{2}+2\log(2(n-1))+h\bigg(\frac{e^{\Delta}}{1+e^{\Delta}}\bigg)-h\Big(\frac{1}{2}\Big)\bigg).
\end{eqnarray}

In Lemma \ref{LE1}, we take $\tau\in \Big(0,\frac{1}{1+e^{2\Delta}}\Big)$ sufficiently small (depending on $\Delta$) so that $h\Big(\frac{e^{\Delta}}{1+e^{\Delta}}\Big)>h\big(\frac{1}{2}\big)+2\varepsilon(\tau)$. If condition (ii) in Lemma \ref{LE1} holds, by the lemma along with~\eqref{nneq1}, we get  
\begin{equation*}
    \sup_{y\in\hat{\mathcal{Y}}_{n,K}}\{\Omega(y)\}-\sup_{y\in\hat{\mathcal{Y}}_{n,K}^{\mathrm{ff}}}\{\Omega(y)\}\geq \varepsilon(\tau)n(n-1)-\Delta n,
\end{equation*}
which yields the desired conclusion of Lemma \ref{LE2}. 

Thus, without loss of generality, we can assume that condition (i) in Lemma \ref{LE1} holds. Let $\mathfrak{B}$ be the set of $(i,j)\in [n]^2$ such that $i\neq j$ and~\eqref{cond1} holds. Note that
\begin{equation}\label{Bdds}
    |\mathfrak{B}|\geq \tau n(\tau n-1). 
\end{equation}
In the following, for any $t\in (0,\tau^2\slash 2)$, we construct $y^{\dagger}(t)=(y_{i,j;\ell,\ell'}^{\dagger}(t))_{i,j\in [n],i\neq j;\ell,\ell'\in [K]}\in\hat{\mathcal{Y}}_{n,K}$. For every $(i,j)\in [n]^2$ such that $i\neq j$ and $(i,j)\notin\mathfrak{B}$, we take $y^{\dagger}_{i,j;\ell,\ell'}(t):=y^{*}_{i\rightarrow j,\ell}y^{*}_{i\leftarrow j,\ell'}$ for every $\ell,\ell'\in[2]$. For every $(i,j)\in\mathfrak{B}$ and $\ell,\ell'\in [2]$, we take $y^{\dagger}_{i,j;\ell,\ell'}(t):=y^{*}_{i\rightarrow j,\ell}y^{*}_{i\leftarrow j,\ell'}+(-1)^{\mathbbm{1}_{\ell=\ell'}}t$ (note that $y^{\dagger}_{i,j;\ell,\ell'}(t)\geq 0$ thanks to~\eqref{cond1}). Note that for every $(i,j)\in [n]^2$ such that $i\neq j$ and every $\ell\in [2]$,
\begin{equation*}
    \sum_{\ell'=1}^2y^{\dagger}_{i,j;\ell,\ell'}(t)=y^*_{i\rightarrow j,\ell},\qquad \sum_{\ell'=1}^2y^{\dagger}_{i,j;\ell',\ell}(t)=y^*_{i\leftarrow j,\ell}. 
\end{equation*}
Hence with $\widetilde{N}_{i,\ell}(\cdot)$ as in~\eqref{Nil_y_def}, we have
\begin{equation}\label{ress1}
    \sum_{i=1}^n\sum_{\ell=1}^2\phi\big(\widetilde{N}_{i,\ell}\big(y^{\dagger}(t)\big)\big)=\sum_{i=1}^n\sum_{\ell=1}^2\phi\bigg(\sum_{j\in[n]\backslash\{i\}}y^*_{i\rightarrow j,\ell}+\sum_{j\in[n]\backslash\{i\}}y^*_{j\leftarrow i,\ell}\bigg).
\end{equation}

For any $a>0,s\in\mathbb{R}$ such that $|s|\leq a\slash 2$, we define $\varphi_a(s):=\phi(a+s)-\phi(a)-(1+\log a)s$. As $\varphi'_a(s)=\log\Big(1+\frac{s}{a}\Big)$, using the inequality $|\log(1+x)|\leq 2|x|,\forall x\in [-1\slash 2,1\slash 2]$ (which can be proved by elementary calculus), we get 
\begin{equation}\label{lem1}
  |\phi(a+s)-\phi(a)-(1+\log a)s|=|\varphi_a(s)|= \bigg|\int_0^s\varphi_a'(s')ds'\bigg|\leq \int_0^{|s|}\frac{2s'}{a}ds'=\frac{s^2}{a}.
\end{equation}
Now for any $(i,j)\in [n]^2$ such that $i\neq j$ and $(i,j)\notin\mathfrak{B}$,
\begin{eqnarray}\label{ress2}
   && \sum_{\ell=1}^2\sum_{\ell'=1}^2 y^{\dagger}_{i,j;\ell,\ell'}(t)\log B_{\ell,\ell'}-\sum_{\ell=1}^2\sum_{\ell'=1}^2\phi(y^{\dagger}_{i,j;\ell,\ell'}(t))\nonumber\\
   &=& \sum_{\ell=1}^2\sum_{\ell'=1}^2 y^*_{i \rightarrow j,\ell}y^*_{i\leftarrow j,\ell'}\log B_{\ell,\ell'}-\sum_{\ell=1}^2 \phi(y^*_{i\rightarrow j,\ell})-\sum_{\ell=1}^2 \phi(y^*_{i\leftarrow j,\ell}). 
\end{eqnarray}
For any $(i,j)\in\mathfrak{B}$, 
\begin{eqnarray}\label{ress3}
    &&\sum_{\ell=1}^2\sum_{\ell'=1}^2 y^{\dagger}_{i,j;\ell,\ell'}(t)\log B_{\ell,\ell'}-\sum_{\ell=1}^2\sum_{\ell'=1}^2\phi(y^{\dagger}_{i,j;\ell,\ell'}(t))\nonumber\\
    &&\hspace{0.2in}-\Bigg(\sum_{\ell=1}^2\sum_{\ell'=1}^2 y^*_{i \rightarrow j,\ell}y^*_{i\leftarrow j,\ell'}\log B_{\ell,\ell'}-\sum_{\ell=1}^2 \phi(y^*_{i\rightarrow j,\ell})-\sum_{\ell=1}^2 \phi(y^*_{i\leftarrow j,\ell})\Bigg)\nonumber\\
    &=& 2t\Delta-\sum_{\ell=1}^2\sum_{\ell'=1}^2\phi\big(y^{*}_{i\rightarrow j,\ell}y^{*}_{i\leftarrow j,\ell'}+(-1)^{\mathbbm{1}_{\ell=\ell'}}t\big)+\sum_{\ell=1}^2 \phi(y^*_{i\rightarrow j,\ell})+\sum_{\ell=1}^2 \phi(y^*_{i\leftarrow j,\ell})\nonumber\\
    &\geq&  2t\Delta-\sum_{\ell=1}^2\sum_{\ell'=1}^2\bigg(\phi(y^{*}_{i\rightarrow j,\ell}y^{*}_{i\leftarrow j,\ell'})+\big(1+\log\big(y^{*}_{i\rightarrow j,\ell}y^{*}_{i\leftarrow j,\ell'}\big)\big)(-1)^{\mathbbm{1}_{\ell=\ell'}}t+\frac{t^2}{\tau^2}\bigg)\nonumber\\
    &&\hspace{0.2in}+\sum_{\ell=1}^2 \phi(y^*_{i\rightarrow j,\ell})+\sum_{\ell=1}^2 \phi(y^*_{i\leftarrow j,\ell})\nonumber\\
    &=& 2t\Delta-\frac{4t^2}{\tau^2},
\end{eqnarray}
where the inequality in the fourth line uses~\eqref{lem1} (upon noting that $y^{*}_{i\rightarrow j,\ell}y^{*}_{i\leftarrow j,\ell'}\geq \tau^2$ for every $\ell,\ell'\in [2]$ (by~\eqref{cond1}) and $t\in(0,\tau^2\slash 2)$), and the equality in the last line uses the fact that
\begin{eqnarray*}
&&\sum_{\ell=1}^2\sum_{\ell'=1}^2\big(1+\log\big(y^{*}_{i\rightarrow j,\ell}y^{*}_{i\leftarrow j,\ell'}\big)\big)(-1)^{\mathbbm{1}_{\ell=\ell'}}\nonumber\\
&=&\big(1+\log\big(y^*_{i\rightarrow j,1}y^*_{i\leftarrow j,2}\big)\big)+\big(1+\log\big(y^*_{i\rightarrow j,  2}y^*_{i\leftarrow j,  1})\big)\nonumber\\
&&\hspace{0.2in}-\big(1+\log\big(y^*_{i\rightarrow j,  1}y^*_{i\leftarrow j,  1}\big)\big)-\big(1+\log\big(y^*_{i\rightarrow j,  2}y^*_{i\leftarrow j,  2}\big)\big)\nonumber\\
&=& \log y^*_{i\rightarrow j,1}+\log y^*_{i\leftarrow j,2}+\log  y^*_{i\rightarrow j,  2}+\log y^*_{i\leftarrow j,  1}\nonumber\\
&&\hspace{0.2in}-\log y^*_{i\rightarrow j,  1}-\log y^*_{i\leftarrow j,  1}-\log y^*_{i\rightarrow j,  2}-\log y^*_{i\leftarrow j,  2}=0.
\end{eqnarray*}
With $\Omega(\cdot)$ and $\omega(\cdot,\cdot)$ as in~\eqref{def_Omega} and~\eqref{def_omega}, combining~\eqref{ress1}, \eqref{ress2}, and~\eqref{ress3}, we get for any $t\in (0,\tau^2\slash 2)$, 
\begin{equation*}
    \Omega\big(y^{\dagger}(t)\big)-\sup_{y\in\hat{\mathcal{Y}}_{n,K}^{\mathrm{ff}}}\{\Omega(y)\}=\Omega\big(y^{\dagger}(t)\big)-\omega(y^*_{\rightarrow},y^*_{\leftarrow})\geq |\mathfrak{B}|\bigg(2t\Delta-\frac{4t^2}{\tau^2}\bigg).
\end{equation*}
Taking $t=\min\{\Delta,1\}\tau^2\slash 4$ in the above display (note that $4t^2 \slash\tau^2\leq t\Delta$), and noting~\eqref{Bdds}, we get
\begin{eqnarray*}
    \sup_{y\in\hat{\mathcal{Y}}_{n,K}}\{\Omega(y)\}-\sup_{y\in\hat{\mathcal{Y}}_{n,K}^{\mathrm{ff}}}\{\Omega(y)\}\geq \Omega\big(y^{\dagger}(t)\big)-\sup_{y\in\hat{\mathcal{Y}}_{n,K}^{\mathrm{ff}}}\{\Omega(y)\}\geq|\mathfrak{B}|t\Delta\geq\tau n (\tau n-1)t\Delta,
\end{eqnarray*}
from which the desired conclusion of Lemma \ref{LE2} follows. 
\end{proof}

\section{Auxiliary lemmas}\label{Appendix_E}
In this section, we present some auxiliary lemmas used in the proofs of Theorems \ref{Theorem_LDA_UBD}-\ref{Theorem_MMSB_S}.  

\begin{lemma}\label{L3.2}
There is an absolute positive constant $C$, such that for any $x\geq 1\slash 2$, 
\begin{equation*}
    \left|\log\Gamma(x)-\Big(x\log{x}-x-\frac{1}{2}\log{x}\Big)\right|\le  C,
\end{equation*}
and for any $x>0$,
\begin{equation*}
    \log\Gamma(x)\geq -1-\log x.
\end{equation*} 
\end{lemma}
\begin{proof}

By Stirling's approximation \cite{karatsuba2001asymptotic}, 
\begin{equation*}
    \lim_{x\rightarrow\infty} \frac{\Gamma(x)}{x^x e^{-x} x^{-1\slash 2}} = \sqrt{2\pi}. 
\end{equation*}
As $\Gamma(x)>0$ and $\Gamma(x)$ is continuous for $x\in (0,\infty)$, there exist absolute positive constants $C,c$, such that for any $x\geq 1\slash 2$,
\begin{equation}\label{Bdd1}
  c x^x e^{-x}x^{-1\slash 2}\leq  \Gamma(x)\leq Cx^x e^{-x}x^{-1\slash 2}. 
\end{equation}

Now note that for any $x>0$, 
\begin{equation}\label{Bdd2}
  \Gamma(x)=\int_0^{\infty} t^{x-1} e^{-t}dt \geq  \int_0^1t^{x-1} e^{-t}dt \geq e^{-1}\int_0^1 t^{x-1}dt = e^{-1} x^{-1}.
\end{equation}

The conclusion of the lemma follows from~\eqref{Bdd1} and~\eqref{Bdd2}. 
\end{proof}

\begin{lemma}\label{L3.3}
For any $m,i\in\mathbb{N}^{*}$ such that $i\leq m$, 
\begin{equation*}
     \binom{m}{i}\leq \Big(\frac{em}{i}\Big)^i.
\end{equation*}
\end{lemma}
\begin{proof}
We have
\begin{equation*}
    \binom{m}{i}=\frac{m(m-1)\cdots(m-i+1)}{i!}\leq \frac{i^i}{i!}\cdot\frac{m^i}{i^i}\leq e^i\cdot\frac{m^i}{i^i}=\Big(\frac{em}{i}\Big)^i.
\end{equation*}
\end{proof}

\section{Implementation of CAVI for partially grouped VI in MMSB}\label{Appendix_F}

In this section, we present the implementation details of the CAVI algorithm for partially grouped VI in the MMSB model. 

Let $\Pi_K:=\big\{(u_{\ell})_{\ell\in [K]}: u_{\ell}\geq 0 \text{ for every }\ell\in [K], \sum_{\ell=1}^K u_{\ell}=1\big\}$. For every $i,j\in [n]$ such that $i\neq j$, let $Z_{i,j}:=(Z_{i\rightarrow j}, Z_{i\leftarrow j})\in [K]^2$. For partially grouped VI, the variational distribution is of the form $\big(\bigotimes\limits_{i,j\in[n]:i\neq j}Q_{i,j}\big)\bigotimes \big(\bigotimes_{i\in [n]} R_i\big)$, where $Q_{i,j}$ is a probability distribution on $[K]^2$ (representing the distribution of $Z_{i,j}$) for every distinct $i,j\in [n]$, and $R_i$ is a probability distribution on $\Pi_K$ (representing the distribution of $\pmb{\pi}_{i}$) for every $i\in [n]$.  

For each $t\in\mathbb{N}^{*}$, we denote by $Q_{i,j}^{(t)}$ (where $i,j\in[n]$ and $i\neq j$) and $R_i^{(t)}$ (where $i\in [n]$) the components of the variational distribution for the $t$th iteration of the CAVI algorithm. For each $i,j\in [n]$ with $i\neq j$, we let $y_{i,j}^{(t)}=(y_{i,j;\ell,\ell'}^{(t)})_{\ell,\ell'\in [K]}$ with $y_{i,j;\ell,\ell'}^{(t)}=Q_{i,j}^{(t)}(\ell,\ell')$ for each $\ell,\ell'\in [K]$. For each $i\in [n]$, it can be shown that (see \cite[Section 2.4]{blei2017variational}) $R_i^{(t)}$ is a Dirichlet distribution, and we denote by $\gamma_{i}^{(t)}=\big(\gamma_{i,1}^{(t)},\gamma_{i,2}^{(t)},\cdots,\gamma_{i,K}^{(t)}\big)$ the parameters of this Dirichlet distribution. For the $(t+1)$st iteration, we update the parameters $y_{i,j}^{(t+1)}$ and $\gamma_i^{(t+1)}$ based on those from the $t$th iteration as follows:
\begin{itemize}
    \item For each $i\in [n]$ and $\ell\in [K]$, we compute 
    \begin{equation*}
         \mathbb{E}_{R_i^{(t)}}[\log \pi_{i,\ell}]=\frac{\Gamma'\Big(\gamma_{i,\ell}^{(t)}\Big)}{\Gamma\Big(\gamma_{i,\ell}^{(t)}\Big)}-\frac{\Gamma'\Big(\sum_{s=1}^K\gamma_{i,s}^{(t)}\Big)}{\Gamma\Big(\sum_{s=1}^K\gamma_{i,s}^{(t)}\Big)}.
    \end{equation*}
    Then for each $i,j\in [n]$ with $i\neq j$, we update $y_{i,j}^{(t+1)}$ so that 
    \begin{equation*}
        y_{i,j;\ell,\ell'}^{(t+1)}  \propto \exp\Big(\mathbb{E}_{R_i^{(t)}}[\log \pi_{i,\ell}]+\mathbb{E}_{R_j^{(t)}}[\log \pi_{j,\ell'}]\Big) B_{\ell,\ell'}^{X_{i,j}} (1-B_{\ell,\ell'})^{1-X_{i,j}}
    \end{equation*}
    for all $\ell,\ell'\in [K]$, where $\big(y_{i,j;\ell,\ell'}^{(t+1)}\big)_{\ell,\ell'=1}^K$ are normalized to sum to $1$.
    \item For each $i\in [n]$ and $\ell\in [K]$, we update \begin{equation*}
        \gamma_{i,\ell}^{(t+1)} = \sum_{j\in [n]\backslash \{i\}}\sum_{\ell'=1}^K y_{i,j;\ell,\ell'}^{(t+1)}+\sum_{j\in [n]\backslash \{i\}}\sum_{\ell'=1}^K y_{j,i;\ell',\ell}^{(t+1)}+\alpha_{\ell}.
    \end{equation*}
\end{itemize}

\end{appendix}

\end{document}